%% file: thesis.tex
\documentclass[12pt,leqno,letterpaper]{report}

\newcommand{\thesistitle}{Nonlinear Hawkes Processes}
\newcommand{\thesisauthor}{Lingjiong Zhu}
\newcommand{\thesisadvisor}{Professor S. R. S. Varadhan}
\newcommand{\graddate}{May 2013}
\newcommand{\thesisdedication}{To the memory of my grandpa
\\
Zhixuan Zhu (1923-2001)}

\usepackage{setspace}

\input{definitions}

\begin{document}

\pagenumbering{roman}
%
\thispagestyle{empty}
\begin{center}
  {\large\textbf{\thesistitle}}
  \vspace{.7in}

  by
  \vspace{.7in}

  \thesisauthor
  \vfill

\begin{doublespace}
  A dissertation submitted in partial fulfillment\\
  of the requirements for the degree of\\
  Doctor of Philosophy\\
  Department of Mathematics\\
  New York University\\
  \graddate
\end{doublespace}
\end{center}
\vfill

\noindent\makebox[\textwidth]{\hfill\makebox[2.5in]{\hrulefill}}\\
\makebox[\textwidth]{\hfill\makebox[2.5in]{\hfill\thesisadvisor\hfill}}
\newpage
\thispagestyle{empty}
\vspace*{2in} 
\begin{center}
\copyright \hspace{.2cm} \thesisauthor\\
All Rights Reserved, 2013
\end{center}
\newpage

\begin{center}
Dedication
\end{center}
\vspace*{\fill}
\begin{center}
  \thesisdedication\addcontentsline{toc}{section}{Dedication}
\end{center}
\vfill
\newpage
\section*{Acknowledgements}\addcontentsline{toc}{section}{Acknowledgements}
\input{acknowledge}
\newpage
\section*{Abstract}\addcontentsline{toc}{section}{Abstract}
\input{abstract}
\newpage
\tableofcontents

\listoffigures\addcontentsline{toc}{section}{List of Figures}
\newpage

\pagenumbering{arabic} 
\section*{Introduction}\addcontentsline{toc}{section}{Introduction}
\input{intro}
\input{chap1}
\input{chap2} 
\input{chap3}
\input{chap4}

\input{chap5}

\input{chap6}

\input{biblio}

\end{document}

%% file: definitions.tex
\usepackage[final]{graphicx}

\usepackage{indentfirst}

\usepackage{amsthm} 
\usepackage[tbtags]{amsmath} 

\usepackage{amsfonts}
\usepackage{amssymb}

\usepackage{subfigure}

\renewcommand{\emptyset}{\ensuremath{\varnothing}}

\newtheorem{theorem}{Theorem}
\newtheorem{assumption}{Assumption}

\newtheorem{corollary}{Corollary}

\newtheorem{example}{Example}
\newtheorem{lemma}{Lemma}

\newtheorem{proposition}{Proposition}

\newtheorem{remark}{Remark}

\def\P{\mathbb P} \def\E{\mathbb E} 



%% file: acknowledge.tex
%
It is difficult to overstate my gratitude to my adviser Professor Varadhan. 
Working with Professor Varadhan has been an absolutely amazing experience for me. 
I thank him for always keeping his door open and patiently answering my questions.
I thank him for his superb guidance, understanding, and generosity.
I thank him for suggesting the topic for my thesis, which would not be possible
without his deep wisdom and sharing of many new ideas.
He has been everything that one can reasonably ask for in an advisor and more, and I am
truly grateful to him.

I want to thank the Courant community for guiding me through this
process and for putting up with me in general. Tamar Arnon does her job
exceptionally well and her efforts are much appreciated. I want to thank the faculty for
many well taught and interesting classes. 
I am indebted to G\'{e}rard Ben Arous, Sourav Chatterjee and Raghu Varadhan
for writing me recommendations for my first academic job. I also want to thank Peter Carr for his interest in my thesis.

I remember a joke told by Jalal Shatah that the most important thing as an undergraduate student is to go to a top
graduate program. But once you are already at a graduate school, the most important thing is to get out of it! This would not be
possible without the final step, i.e. thesis defense! I am grateful to have Henry McKean, Chuck Newman and Raghu Varadhan
as the three readers and G\'{e}rard Ben Arous and Lai-Sang Young as the two non-readers on my thesis committee.

Most importantly, I want to thank my fellow
colleagues for all the fun memories that I take with me from Courant. New York
City, without good friends, can be the most populated lonely place in the world,
but thankfully the constant friendship of my fellow Courant colleagues has made
these five years some of the most entertaining and pleasurable of my life. I thank Antoine Cerfon, Shirshendu Chatterjee, Oliver Conway,
S\^{i}nziana Datcu, Partha Dey, Thomas Fai, Max Fathi, Mert G\"{u}rb\"{u}zbalaban, Matan Harel, Miranda Holmes-Cerfon, Arjun Krishnan, 
Shoshana Leffler, Sandra May, Jim Portegies, Alex Rozinov, 
Patrick Stewart, Adam Stinchcombe, Jordan Thomas, Chen-Hung Wu and many others for their friendship.
In particular, I want to thank Dmytro Karabash, Behzad Mehrdad and Sanchayan Sen. 
They are not only my good friends, but coauthors as well. I also thank my office neighbor Cheryl Sylivant
for her friendship.

By living in New York City, I had the great opportunities to visit as many museums
and go to as many concerts as possible. I am grateful to the
New York Philharmonic and Metropolitan Opera House for their student ticket offers and also many wonderful student recitals and
concerts at Juilliard School, which have made my stay in New York City
much more enjoyable.

I also want to thank the professors at the University of Cambridge, who provided me a solid
undergraduate education. In particular, I am grateful to Rachel Camina, as well as Houshang Ardavan
and Tom K\"{o}rner. I also want to thank Stefano Luzzatto for supervising me on an
undergraduate research project at Imperial College, London.

I am very much indebted to my family back home. 
I thank my parents for so many years of love and understanding. They are truly the best parents one could ask for.
I also thank my grandmas, uncles and aunts for their support. Finally, I dedicate
this thesis to the memory of my late grandpa. I miss him dearly.

%% file: abstract.tex
%
The Hawkes process is a simple point process
that has long memory, clustering effect, self-exciting property and is in
general non-Markovian. The future evolution of a self-exciting point process is
influenced by the timing of the past events. There are applications in finance, neuroscience,
genome analysis, seismology, sociology, criminology and many other fields.
We first survey the known results about the theory and applications of both linear and nonlinear Hawkes processes.
Then, we obtain the central limit theorem and process-level, i.e. level-3 large deviations for nonlinear Hawkes processes.
The level-1 large deviation principle holds as a result of the contraction principle.
We also provide an alternative variational formula for the rate function of the level-1 large deviations in the Markovian case.
Next, we drop the usual assumptions on the nonlinear Hawkes process and categorize it into different regimes, 
i.e. sublinear, sub-critical, critical, super-critical and explosive regimes.
We show the different time asymptotics in different regimes and obtain other properties as well.
Finally, we study the limit theorems of linear Hawkes processes with random marks.

%% file: intro.tex
%
This thesis is about the nonlinear Hawkes process, a simple point processes, that has long memory,
the clustering effect, the self-exciting property and is in general non-Markovian.
The future evolution of a self-exciting point process
is influenced by the timing of the past events. 
There are applications in finance, neuroscience, genome analysis, sociology, criminology, seismology, and many other fields.

Chapter \ref{chap:one} includes the introduction of the model and the survey of the results already known in the
literature about Hawkes processes. That includes the stability results, limit theorems, power spectra
of linear Hawkes processes and stability results of nonlinear Hawkes processes.

Chapter \ref{chap:two} is about the functional central limit theorem of nonlinear Hawkes processes. 
A Strassen's invariance holds under the same assumptions. The work in Chapter \ref{chap:two} is based
on Zhu \cite{ZhuIII}.

Chapter \ref{chap:three} is dedicated to the process-level large deviations, i.e. level-3 large deviations, of
the nonlinear Hawkes processes. The proofs consist of the proofs of the lower bound, the upper bound and
the superexponential estimates. The level-1 large deviation principle is derived as a result of the contraction principle.
This chapter is based on Zhu \cite{ZhuII}.

Chapter \ref{chap:four} is dedicated to the study of level-1 large deviation principle for nonlinear Hawkes processes
when the exciting functions are exponential or sums of exponentials. It is based on the observation that when the exciting
functions are exponential or sums of exponentials, the process is Markovian and a combination of Feynman-Kac formula
for the upper bound of large deviations of Markov processes and tilting of the intensity function of Hawkes processes
for the lower bound will establish a level-1 large deviation principle with the rate function expressed in terms of 
some variational formula. This chapter is based on Zhu \cite{ZhuI}.

Chapter \ref{chap:five} is about the asymptotics for nonlinear Hawkes processes. In this chapter, we drop the usual
assumptions on nonlinear Hawkes processes, and study the phase transitions in different regimes.
We categorize nonlinear Hawkes processes into the following regimes: sublinear regime, sub-critical regime, critical regime,
super-critical regime and explosive regime. Different time asymptotics and various properties are obtained in different regimes.
This chapter is based on Zhu \cite{ZhuV}.

Chapter \ref{chap:six} is about the limit theorems for linear Hawkes processes with random marks. The Central limit theorem
and the large deviation principle are derived. We end this chapter with a simple application to a risk model.
This is based on the joint work with my colleague Dmytro Karabash, see \cite{Karabash}.

During my time as a PhD student at Courant Institute, I have the joy to work on some other problems either by myself
or with my colleagues. For example, I studied the large deviations of self-correcting point processes with Sanchayan Sen, see \cite{Sen}
and also did some work on biased random walks on Galton-Watson trees without leaves with Behzad Mehrdad and Sanchayan Sen,
see \cite{Mehrdad}. But since they are not closely related to the topics of my thesis, I do not include them here.

%% file: chap1.tex
\chapter{Hawkes Processes\label{chap:one}}

\section{Introduction}

Hawkes process is a self-exciting simple point process first introduced by Hawkes \cite{Hawkes}. The future
evolution of a self-exciting point process is influenced by the timing of past events. The process is non-Markovian except
for some very special cases. In other words, Hawkes process depends on the entire past history and has a long memory. 
Hawkes process has wide applications in neuroscience, see e.g. Johnson \cite{Johnson}, Chornoboy et al. \cite{Chornoboy}, Pernice et al. \cite{PerniceI},
Pernice et al. \cite{PerniceII}, Reynaud et al. \cite{ReynaudIII};
seismology, see e.g. Hawkes and Adamopoulos \cite{HawkesIV}, Ogata \cite{OgataII}, Ogata \cite{OgataIII}, Ogata et al. \cite{OgataV};
genome analysis, see e.g. Gusto and Schbath \cite{Gusto}, Reynaud-Bouret and Schbath \cite{Reynaud}; psycology, see 
e.g. Halpin and De Boeck \cite{Halpin}; 
spread of infectious disease, see e.g. Meyer et al. \cite{Meyer};
finance, see e.g. Bauwens and Hautsch \cite{Bauwens}, 
Bowsher \cite{Bowsher}, Hewlett \cite{Hewlett}, Large \cite{Large}, Cartea et al. \cite{Cartea},
Chavez-Demoulin et al. \cite{Chavez}, Errais et al. \cite{Errais}.
Embrechts et al. \cite{Embrechts},  
Muni Toke and Pomponio \cite{Muni},
Bacry et al. \cite{BacryII}, \cite{BacryIII}, \cite{BacryIV}; and in many other fields.

Let $N$ be a simple point process on $\mathbb{R}$ and $\mathcal{F}^{-\infty}_{t}:=\sigma(N(C),C\in\mathcal{B}(\mathbb{R}), C\subset(-\infty,t])$ be
an increasing family of $\sigma$-algebras. Any nonnegative $\mathcal{F}^{-\infty}_{t}$-progressively measurable process $\lambda_{t}$ with
\begin{equation}
\mathbb{E}\left[N(a,b]|\mathcal{F}^{-\infty}_{a}\right]=\mathbb{E}\left[\int_{a}^{b}\lambda_{s}ds\big|\mathcal{F}^{-\infty}_{a}\right],
\end{equation}
a.s. for all intervals $(a,b]$ is called the $\mathcal{F}^{-\infty}_{t}$-intensity of $N$. We use the notation $N_{t}:=N(0,t]$ to denote the number of
points in the interval $(0,t]$. 

A nonlinear Hawkes process is a simple point process $N$ admitting an $\mathcal{F}^{-\infty}_{t}$-intensity
\begin{equation}
\lambda_{t}:=\lambda\left(\int_{-\infty}^{t}h(t-s)N(ds)\right),\label{dynamics}
\end{equation}
where $\lambda(\cdot):\mathbb{R}^{+}\rightarrow\mathbb{R}^{+}$ is locally integrable and left continuous, 
$h(\cdot):\mathbb{R}^{+}\rightarrow\mathbb{R}^{+}$. 
We always assume that $\Vert h\Vert_{L^{1}}=\int_{0}^{\infty}h(t)dt<\infty$ unless otherwise specified. 
Here $\int_{-\infty}^{t}h(t-s)N(ds)$ stands for $\int_{(-\infty,t)}h(t-s)N(ds)$, which is important for $\mathcal{F}^{-\infty}_{t}$-predictability.
The local integrability assumption of $\lambda(\cdot)$ is to avoid explosion and the left continuity assumption 
of $\lambda(\cdot)$ is to ensure that the process is $\mathcal{F}^{-\infty}_{t}$-predictable.

In the literature, $h(\cdot)$ and $\lambda(\cdot)$ are usually referred to
as exciting function and rate function respectively.

A Hawkes process is said to be linear if $\lambda(\cdot)$ is linear and it is nonlinear otherwise.
For a linear Hawkes process, we can assume that the intensity is
\begin{equation}
\lambda_{t}:=\nu+\int_{(-\infty,t)}h(t-s)N(ds).
\end{equation}

In this thesis, unless otherwise specified, we assume the following.
\begin{itemize}
\item
$\lambda(\cdot):\mathbb{R}^{+}\rightarrow\mathbb{R}^{+}$ is continuous and non-decreasing.
\item
$h(\cdot):\mathbb{R}^{+}\rightarrow\mathbb{R}^{+}$ is continuous and non-increasing.
\item
$N(-\infty,0]=0$, i.e. the Hawkes process has empty past history.
\end{itemize}

Throughout, we define $Z_{t}$ as $Z_{t}:=\int_{0}^{t}h(t-s)N(ds)$. Thus, $\lambda_{t}=\lambda(Z_{t})$. 

The first assumption says that the occurence of the past and present events have positive impact on the occurence of the future events.
The second assumption says that as time evolves, the impact of the past events is decreasing. For most of the results in this paper,
these two assumptions may not be necessary. We nevertheless make them to avoid some technical difficulties.

If one looks at \eqref{dynamics}, it is clear that if you witness some events occuring, $\lambda_{t}$ increases since $\lambda(\cdot)$
is increasing and you would expect even more events occuring. This is called the self-exciting property.
Because of this, you would expect to see some clustering effects.

Figure \ref{Comparison} shows the histograms of a Hawkes process and a usual Poisson process. A Poisson process is stationary with
independent increments. On the contrary, the Hawkes process has dependent increments and has clustering effects. As a result,
in the picture, the Poisson process is more or less flat whilst the Hawkes process has peaks when it gets ``excited'' and has valleys
when it ``cools down''. 
Figure \ref{IntensityPlot} shows
the plot of the intensity $\lambda_{t}$ of a Hawkes process. Unlike the usual Poisson process for which the intensity
is a positive constant, the intensity of Hawkes process increases when you witness arrivals of points and it
decays when there are no arrivals of points.

The self-exciting and clutstering properties of the Hawkes process make it ideal to characterize
the correlations in some complex systems, including the default clustering effect in finance.

One generalization of classical linear Hawkes process is the so-called multivariate Hawkes process.
We will define the multivariate Hawkes process and discuss some
basic results in Section \ref{multivariate} of Chapter \ref{chap:one}.
The multivariate Hawkes process has been well studied in the literature and we would like to point out that
if you have the result for the univariate Hawkes process, mathematically, 
it is not too difficult to generalize your result to multivariate Hawkes process.

Unlike the univariate Hawkes process, which only has the self-exciting property, the multivariate Hawkes process also has
the mutually-exciting property. In the context of industry, consider that you have a large portfolio
of companies, then the failure of one company can have impact on the performance of other companies.
In other words, multivariate Hawkes process captures the cross-sectional clustering effect.
That is why in most applications of Hawkes processes in finance, people usually consider multivariate Hawkes processes.
We will review some basic results about multivariate linear Hawkes process in Chapter \ref{chap:one}.

Another possible
generalization to Hawkes process is the marked Hawkes process, i.e. Hawkes process with random marks. Just like univariate Hawkes process vesus
multivariate Hawkes process, if you have the results in unmarked Hawkes process, usually it can be generalized to marked Hawkes process
without much difficulty. 
For instance, the large deviations for linear Hawkes process is proved in Bordenave and Torrisi \cite{Bordenave} and
the large deviations for linear marked Hawkes process is then proved in Karabash and Zhu \cite{Karabash}.
We will discuss the details of limit theorems of linear marked Hawkes process in Chapter \ref{chap:six}.

Most of the literature on Hawkes processes studies only the linear case, which has an immigration-birth representation (see Hawkes and Oakes \cite{HawkesII}).
The stability, law of large numbers, central limit theorem, large deviations, Bartlett spectrum etc. have all been studied and understood very well.
Almost all of the applications of Hawkes processes in the literature consider exclusively the linear case. Daley and Vere-Jones \cite{Daley} 
and Liniger \cite{Liniger} provide nice surveys about the theory and applications of Hawkes processes.

One special case of the Hawkes process is when the exciting function $h(\cdot)$ is exponential. 
In this case, the Hawkes process is a continuous time Markov process. If $\lambda(\cdot)$ is linear, the process
is a special case of affine jump-diffusion process and is analytically tractable.
This special case was for example studied in Oakes \cite{Oakes} and Errais et al. \cite{Errais}.

Because of the lack of computational tractability and immigration-birth representation, nonlinear Hawkes process is much less studied. 
However, some efforts have already been made in this direction. For instance, see Br\'{e}maud and Massouli\'{e} \cite{Bremaud} for stability results,
and Bremaud et al. \cite{BremaudII} for the rate of convergence to stationarity. Karabash \cite{KarabashII} recently
proved the stability results for a wider class of nonlinear Hawkes processes. 

As to the limit theorems, Bacry et al. \cite{Bacry} proved the central limit theorem for linear Hawkes process
and Bordenave and Torrisi \cite{Bordenave} proved the large deviation principle for linear Hawkes process.

For nonlinear Hawkes process, there is no explicit expression for the
variance in the central limit theorem or the rate function for the large deviation principle. The method is more abstract and much
more involved. Zhu \cite{ZhuIII} proved a central limit theorem for ergodic nonlinear Hawkes processes. 
Zhu \cite{ZhuI} studied the large deviations in the Markovian case, i.e. when $h(\cdot)$ is exponential
or sum of exponentials. And Zhu \cite{ZhuII} proved the large deviation principle for more general nonlinear Hawkes processes
at the process-level, i.e. level-3. 

\begin{figure}
\centering
\begin{subfigure}
  \centering
  \includegraphics[scale=0.50]{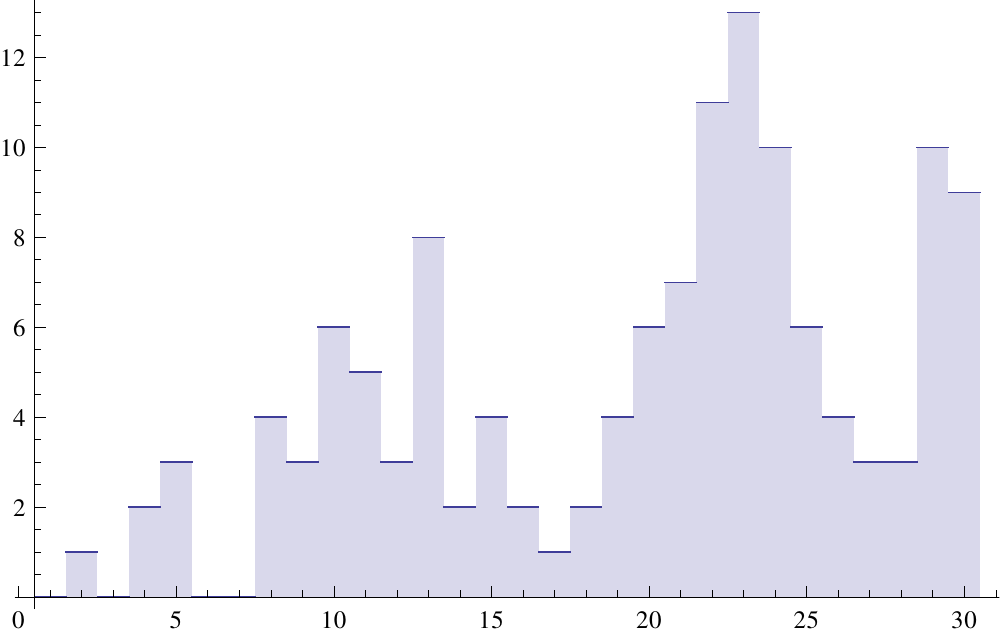}
\end{subfigure}
\begin{subfigure}
  \centering
  \includegraphics[scale=0.50]{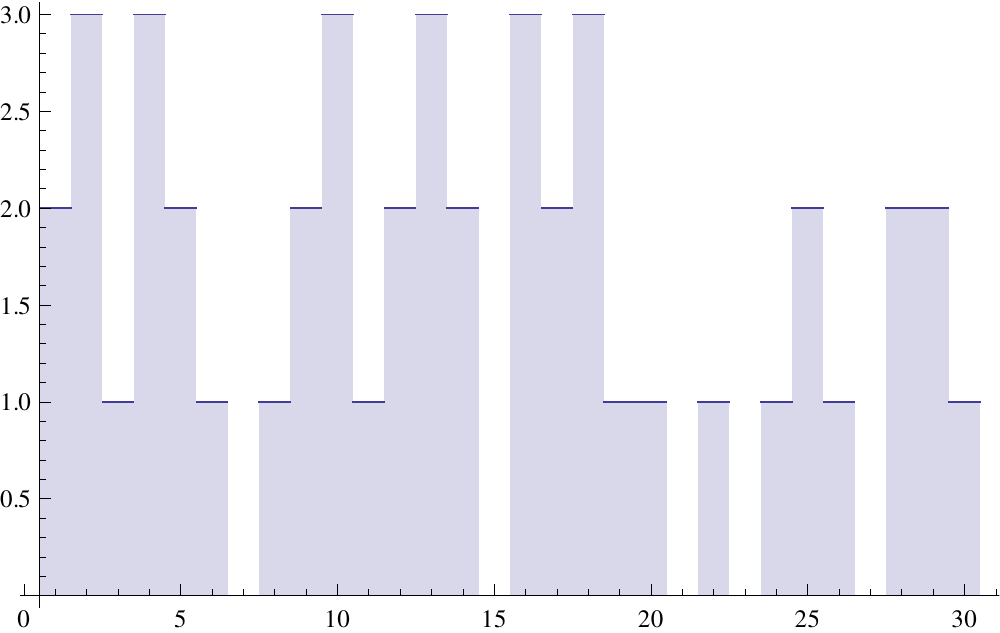}
\end{subfigure}
\caption{This is a comparison of a Hawkes process with a Poisson process.
The figure on the left shows the histogram of a Hawkes process with $h(t)=\frac{1}{(t+1)^{2}}$ and $\lambda(z)=1+\frac{9}{10}z$
and the figure on the right the histogram of a Poisson process with constant intensity $\lambda\equiv\frac{3}{2}$. In the figure,
each column represents the number of points that arrived in that unit time subinterval.}
\label{Comparison}
\end{figure}

\begin{figure}[htb]
\begin{center}
\includegraphics[scale=0.70]{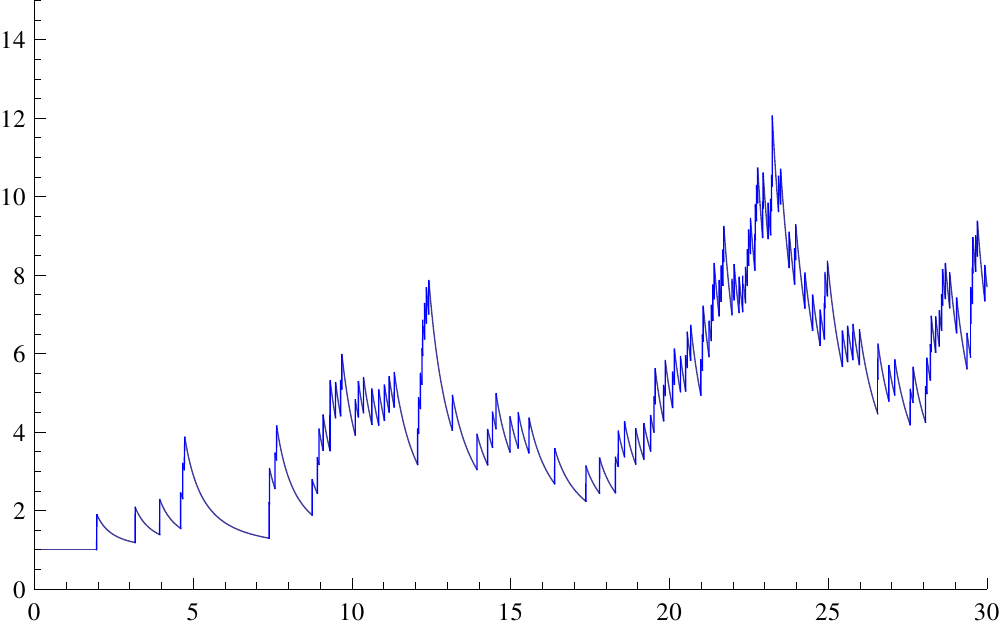}
\caption{Plot of intensity $\lambda_{t}$ for a realization of Hawkes process.
Here $h(t)=\frac{1}{(t+1)^{2}}$ and $\lambda(z)=1+0.9z$.}
\label{IntensityPlot}
\end{center}
\end{figure}

\section{Applications of Hawkes Processes}

\subsection{Applications in Finance}

The applications of Hawkes processes in finance include market orders modelling, see e.g. Bauwens and Hautsch \cite{Bauwens}, 
Bowsher \cite{Bowsher}, Hewlett \cite{Hewlett}, Large \cite{Large} and Cartea et al. \cite{Cartea};
value-at-risk, see e.g. Chavez-Demoulin et al. \cite{Chavez}; and credit risk, see e.g. Errais et al. \cite{Errais}.
Embrechts et al. \cite{Embrechts}
applied Hawkes processes to model the financial data. 
Muni Toke and Pomponio \cite{Muni} applied Hawkes processes to model the trade-through.
Bacry et al. \cite{BacryII} used Hawkes processes to reproduce empirically microstructure noise
and discussed the Epps effect and lead-lag.
The self-exciting and clustering properties of Hawkes processes 
are especially appealing in financial applications. 

Currently, most of the applications of Hawkes process in the finance literature are about market orders modelling, see
e.g. Bauwens and 
Hautsch \cite{Bauwens}, Bowsher \cite{Bowsher} and Large \cite{Large}.

Recently, Chavez-Demoulin and McGill \cite{ChavezII} used Hawkes processes to study the extremal returns in high-frequency trading.
The Hawkes process captures the volatility clustering behavior of the intraday extremal returns.
and provides a suitable estimation of high-quantile based risk measures (e.g. VaR, ES) for financial time series.

Filimonov and Sornette \cite{Filimonov} used Hawkes process to model
market events, with the aim of quantifying precisely endogeneity and exogeneity in market activity.
By using Hawkes process, Filimonov and Sornette \cite{Filimonov} analyzed E-mini S\&P futures contract 
over the period 1998-2010 and discovered that the degree of self-reflexivity has increased steadily in the
last decade, an effect they attribute to the increased deployment of high-frequency and algorithmic trading. 
When they calibrated over much shorter time intervals (10 minutes), the Hawkes process analysis is found
to detect precursors of the flash-crash that happened on May 6th, 2010. An early detection can benefit market regulators.

Very recently, Hardiman et al. \cite{Hardiman} used (linear) Hawkes process to model the arrival of mid-price changes in the E-Mini S\&P futures contract. 
Using several estimation methods, they found that the exciting function $h(\cdot)$ has a power-law decay and $\Vert h\Vert_{L^{1}}$
is close to $1$. They pointed out that markets are and have always been close to criticality, 
challenging the studies of Filimonov and Sornette \cite{Filimonov} which indicates 
that self-reflexivity (endogeneity) has increased in recent years 
as a result of increased automation of trading.

Egami et al. \cite{Egami} studied the credit default swap (CDS) markets in both Japan and U.S. They made a dynamic
analysis of the bid-ask spreads in both countries, which surged dramatically
during the 2008-2009 financial crisis and they used the Hawkes process to predict the bid-ask spreads.

As pointed out in Errais et al. \cite{Errais}, ``The collapse of Lehman Brothers brought the financial system to the brink of a breakdown.
The dramatic repercussions point to the exisence of feedback phenomena that are channeled through the complex web of informational and contractual
relationships in the economy... This and related episodes motivate the design of models of correlated default timing that incorporate
the feedback phenomena that plague credit markets.''
According to Peng and Kou \cite{Peng}, ``We need better models to incorporate the default clustering effect, i.e., one default
event tends to trigger more default events both across time and cross-sectionally.'' The Hawkes process provides a model
to characterize default events across time and if one uses a multivariate Hawkes process, that would describe the cross-sectional
clustering effect as well. 

Hawkes processes have been proposed as models for the arrival of company defaults in
a bond portfolio, starting with the papers Giesecke and Tomecek \cite{Giesecke} and Giesecke et al. \cite{GieseckeII}.
It is not hard to see that when the exciting function $h(\cdot)$ is exponential, the linear Hawkes processes are
affine jump-diffusion processes, see for instance Errais et al. \cite{Errais}.
With the help of the theory of affine jump-diffusions, one can then analyze
price processes related to certain credit derivatives analytically.

\subsection{Applications in Sociology}

The Hawkes process has also been applied to the study of social interactions. 
Crane and Sornette \cite{Crane} analysed the viewing of YouTube videos as an example
of a nonlinear social system. They identified peaks in the time series of viewing figures
for around half a million videos and studied the subsequent decay of the peak to a
background viewing level. In Crane and Sornette \cite{Crane}, the Hawkes process was
proposed as a model of the video-watching dynamics, and a plausible link made to
the social interactions that create strong correlations between the viewing actions of
different people. Individual viewing is not random but influenced by various channels
of communication about what to watch next. Mitchell and Cates \cite{Mitchell} used computer simulation
to test the the claims in Crane and Sornette \cite{Crane} that robust identification is possible
for classes of dynamic response following activity bursts. They also pointed out some limitations
of the analysis based on the Hawkes process.

In sociology, Hawkes process has also been used by Blundell et al. \cite{Blundell}
to study the reciprocating relationships.
Reciprocity is a common social norm, where one person's actions towards another increases the probability of the same type
of action being returned, e.g., if Bob emails Alice, it increases the probability that Alice
will email Bob in the near future. The mutually-exciting processes, e.g. multivariate Hawkes processes, are able to capture
the causal nature of reciprocal interactions.

\subsection{Applications in Seismology}

Ogata \cite{OgataII} used a particular case of the Hawkes process to predict earthquakes and the Hawkes process
appears to be superior to other models by residual analysis. The specific model used by Ogata \cite{OgataII}
is now known as ETAS (Epidemic Type Aftershock-Sequences) model. The discussions of ETAS model
can be found in Daley and Vere-Jones \cite{Daley}.

\subsection{Applications in Genome Analysis}

Gusto and Schbath \cite{Gusto} used the Hawkes process to model the occurences along the genome and studied
how the occurences of a given process along a genome, genes or motifs for instance, 
may be influenced by the occurrences of a second process. More precisely, 
the aim is to detect avoided and/or favored distances between two motifs, for instance, suggesting possible 
interactions at a molecular level. The statistical method proposed by Gusto and Schbath \cite{Gusto}
is useful for functional motif detection or to improve knowledge of some biological mechanisms.

Reynaud-Bouret and Schbath \cite{Reynaud}
provided a new method for the detection of either favored or avoided distances between genomic 
events along DNA sequences. These events are modeled by the Hawkes process. The biological problem is actually complex 
enough to need a non-asymptotic penalized model selection approach and Reynaud-Bouret and Schbath \cite{Reynaud} 
provided a theoretical penalty that satisfies an oracle inequality even for quite complex families of models.

\subsection{Applications in Neuroscience}

Chornoboy et al. \cite{Chornoboy} used the Hawkes process to detect and model the functional relationships
between the neurons. The estimates are based on the maximum likelihood principle.

In most neural systems, neurons communicate via sequences of action potentials. Johnson \cite{Johnson}
used various point processes, including Poisson process, renewal process and the Hawkes process and 
showed that neural discharges patterns convey time-varying information intermingled with the neuron's response characteristics.
By applying information theory and estimation theory to point processes, Johnson \cite{Johnson} 
described the fundamental limits on how well information can be extracted from neural discharges. 

More recently, Pernice et al. \cite{PerniceI} and Pernice et al. \cite{PerniceII} have used Hawkes process to model the spike train dynamics
in the studies of neuronal networks. As pointed out in Pernice et al. \cite{PerniceI},
``Hawkes' point process theory allows the treatment of correlations on the level of spike trains as well as the 
understanding of the relation of complex connectivity patterns to the statistics of pairwise correlations.''
Reynaud et al. \cite{ReynaudIII} proposed new non-parametric adaptive estimation methods and adapted other recent similar results 
to the setting of spike trains analysis in neuroscience. They tested homogeneuous Poisson process, inhomogeneous Poisson process
and the Hawkes process. A complete analysis was performed on 
single unit activity recorded on a monkey during a sensory-motor task. Reynaud et al. \cite{ReynaudIII} showed that 
the homogeneous Poisson process hypothesis is always rejected and that the inhomogeneous Poisson process hypothesis is rarely accepted. 
The Hawkes model seems to fit most of the data. 

The application of the Hawkes process in neuroscience has also been
mentioned in Br\'{e}maud and Massouli\'{e} \cite{Bremaud}.

\subsection{Applications in Criminology}

Hawkes processes have also been used in criminology.
Violence among gangs exhibits retaliatory behavior, i.e. given that an event has
happened between two gangs, the likelihood that another event will happen shortly afterwards is increased.
A problem like this can be modeled naturally by a self-exciting point process. 
Mohler et al. \cite{Mohler} and Egesdal et al. \cite{Egesdal} have successfully modeled the pairwise gang violence as a Hawkes
process. As pointed out in Hegemann et al. \cite{Hegemann}, in real-life situations, data is incomplete and law-enforcement
agencies may not know which gang is involved. However, even when gang activity is highly stochastic,
localized excitations in parts of the known dataset can help identify gangs responsible for unsolved
crimes. The works before Hegemann et al. \cite{Hegemann} incorporated the observed clustering in time of the data to identify
gangs responsible for unsolved crimes by assuming that the parameters of the model are
known, when in reality they have to be estimated from the data itself. Hegemann et al. \cite{Hegemann} proposed an iterative method
that simultaneously estimates the parameters in the underlying point process and assigns weights to the
unknown events with a directly calculable score function.

Hawkes processses have also been used in the studies of terrorist activities. For example, Porter and White \cite{Porter}
used Hawkes process to examine the daily number of terrorist attacks in Indonesia from 1994 through 2007.
Their model explains the self-exciting nature of the terrorist activities. It estimates the probability
of future attacks as a function of the times since the past attacks. 

Lewis et al. \cite{LewisIII} used Hawkes process
to model the temporal dynamics of violence and civilian deaths in Iraq.

\section{Related Models}

There are other generalizations or variations of the Hawkes processes in the literature.
For example, Bormetti et al. \cite{Bormetti}
introduced a one factor model where both the factor and the idiosyncratic jump components
are described by a Hawkes process. Their model is a better candidate than classical Poisson or Hawkes models
to describe the dynamics of jumps in a multi-asset framework.
Another example is a multivariate Hawkes process with constraints on its conditional density introduced by Zheng et al. \cite{Zheng}.
Their study is mainly motivated by the stochastic modelling of a limit order book for high frequency financial data analysis.
Dassios and Zhao \cite{Dassios} proposed a dynamic contagion process. 
It is basically a combination of a marked Hawkes process with exponential exciting function
and an external shot noise process. Their model is Markovian. 
They also applied their model to insurance, see e.g. Dassios and Zhao \cite{DassiosII}.

In \cite{ZhuVII}, Zhu incorporated Hawkes jumps into the classical Cox-Ingersoll-Ross model and obtained limit theorems
and various other properties.

In seismology, Wang et al. \cite{Wang} proposed a new model, i.e. the Markov-modulated Hawkes process with stepwise decay (MMHPSD),
to investigate the variation in seismicity rate during a series of earthquakes sequence including multiple main shocks. 
The MMHPSD is a self-exciting process which switches among different states, 
in each of which the process has distinguishable background seismicity and decay rates.
Stress release models are often used in seismology. In Br\'{e}maud and Foss \cite{BremaudIV}, they created a new earthquake model
combining the classical stress release model for primary shocks with the Hawkes model for aftershocks and studied the ergodicity
of this new model. 

In addition to the classical Hawkes process, 
one can also study the spatial Hawkes process, see e.g. M\o ller and Torrisi \cite{MollerIII},
M\o ller and Torrisi \cite{MollerIV}, Bordenave and Torrisi \cite{Bordenave}.        
In addition, the space-time Hawkes process has been used, see 
e.g. Musmeci and Vere-Jones \cite{Musmeci} and Ogata \cite{OgataIII}.

\section{Linear Hawkes Processes}

In this section, let us review some known results about linear Hawkes process.
Unlike the nonlinear Hawkes process, the linear Hawkes process has been very well studied in the literature.
Hawkes and Oakes \cite{HawkesII} introduced an immigration-birth representation of the linear Hawkes process,
which can be viewed as a special case of the Poisson cluster process. 
The stability results of the linear Hawkes process, i.e. existence and uniqueness of a stationary linear
Hawkes process have been summarised in Chapter 12 of Daley and Vere-Jones \cite{Daley}.
The rate of convergence to equilibrium has been stuided by Br\'{e}maud et al. \cite{BremaudII}.
The second-order analysis, i.e. the Bartlett spectrum etc. have been studied in Hawkes \cite{Hawkes} and Hawkes \cite{HawkesIII}.
Reynaud-Bouret and Roy \cite{ReynaudII} considered the linear Hawkes process
as a special case of Poisson cluster process and studied the non-asymptotic tail estimates
of the extinction time, the length of a cluster, and the number of points in an interval.
Reynaud-Bouret and Roy \cite{ReynaudII} also obtained some so-called non-asymptotic ergodic theorems.
The limit theorems have also been studied for linear Hawkes process. The central limit theorem was
considered in Bacry et al. \cite{Bacry}, the large deviation principle was obtained in Bordenave and Torrisi \cite{Bordenave},
and very recently the moderate deviation principle was proved in Zhu \cite{ZhuIV}.
The simulations and calibrations of linear Hawkes process have been studied in 
Ogata \cite{OgataIV}, M\o ller and Rasmussen \cite{MollerII}, \cite{MollerI}, Vere-Jones \cite{Vere}, Ozaki \cite{Ozaki}
and many others.

\subsection{Immigration-Birth Representation}

Consider the linear Hawkes process $N$ with empty history, i.e. $N(-\infty,0]=0$ and intensity
\begin{equation}
\lambda_{t}=\nu+\int_{0}^{t}h(t-u)N(du),\quad\nu>0,
\end{equation}
where $\int_{0}^{\infty}h(t)dt<1$. 
It is well known that it has the following immigration-birth representation; see for example Hawkes and Oakes \cite{HawkesII}.
The immigrant arrives according to a homogeneous
Poisson process with constant rate $\nu$. Each immigrant reproduces children and the number
of children has a Poisson distribution with parameter $\Vert h\Vert_{L^{1}}$. Conditional on the number of
the children of an immigrant, the time that a child was born
has probability density function $\frac{h(t)}{\Vert h\Vert_{L^{1}}}$. Each child produces children 
according to the same laws, independent of other children. All the immigrants produce children independently.
Now, $N(0,t]$ is the same as the total number of immigrants and children in the time interval $(0,t]$.

\subsection{Stability Results}

Consider the linear Hawkes process $N$ with empty history, i.e. $N(-\infty,0]=0$ and intensity
\begin{equation}
\lambda_{t}=\nu+\int_{0}^{t}h(t-u)N(du),
\end{equation}
where $\int_{0}^{\infty}h(t)dt<1$. We review here the known results of existence and uniqueness
of a stationary version of the process. We follow the arguments of Chapter 12 of Daley and Vere-Jones \cite{Daley}.

The existence of a stationary version of the process can be seen from the immigration-birth representation of the linear
Hawkes process. To show uniqueness, let us do the following. Let $N^{\dagger}$ be a stationary version with
intensity
\begin{equation}
\lambda^{\dagger}_{t}=\nu+\int_{-\infty}^{t}h(t-u)N^{\dagger}(du),
\end{equation}
and mean intensity $\mu:=\mathbb{E}[\lambda^{\dagger}_{t}]=\frac{\nu}{1-\Vert h\Vert_{L^{1}}}$.
For both $N$ and $N^{\dagger}$, we consider the shifted versions $\theta_{s}N$ and $\theta_{s}N^{\dagger}$ that
bring the origin back to zero. $\theta_{s}N^{\dagger}$ can be split into two components, the one with the
same structure as $\theta_{s}N$, being generated from the clusters initiated by immigrants arriving after time $-s$
and the component $N_{-s}^{\dagger}$ that counts the children of the immigrants that arrived before time $-s$.
On $\mathbb{R}^{+}$, the contribution from the latter form a Poisson process with intensity
\begin{equation}
\lambda_{-s}^{\dagger}(t)=\int_{-\infty}^{-s}h(t-u)N_{-s}^{\dagger}(du).
\end{equation}
For any $T<\infty$,
\begin{align}
\mathbb{P}(N_{-s}^{\dagger}(0,T)>0)
&=\mathbb{E}\left[1-e^{-\int_{0}^{T}\lambda_{-s}^{\dagger}(t)dt}\right]
\\
&\leq\mathbb{E}\left[\int_{0}^{T}\lambda_{-s}^{\dagger}(t)dt\right]
\leq\mu T\int_{s}^{\infty}h(u)du\rightarrow 0,\nonumber
\end{align}
as $s\rightarrow\infty$. Let $\mathcal{P}$ and $\mathcal{P}^{\dagger}$ represent the probability measures corresponding
to $N$ and $N^{\dagger}$. For any $T>0$, we have
\begin{equation}
\Vert\theta_{-s}\mathcal{P}-\mathcal{P}^{\dagger}\Vert_{[0,T]}
\leq\mathbb{P}(N_{-s}^{\dagger}(0,T)>0)\rightarrow 0,
\end{equation}
as $s\rightarrow\infty$, where $\Vert\cdot\Vert$ denotes the variation norm. This implies
the weak convergence and thus the weak asymptotic stationarity of $N$.

Under a stronger assumption $\int_{0}^{\infty}th(t)dt<\infty$, i.e. the mean time to
the appearance of a child is finite. Since the mean number of offspring is also finite (because $\Vert h\Vert_{L^{1}}<1$),
the random time $T$ from the appearance of an ancestor to the last of its descendants has finite mean, i.e. $\mathbb{E}[T]<\infty$.
Thus, we have
\begin{equation}
\mathbb{P}(N_{-s}^{\dagger}[0,\infty)>0)=1-e^{-\nu\int_{s}^{\infty}\mathbb{P}(T>u)du}\rightarrow 0,
\end{equation}
as $s\rightarrow\infty$ and $\Vert\theta_{-s}\mathcal{P}-\mathcal{P}^{\dagger}\Vert_{[0,\infty]}\rightarrow 0$
as $s\rightarrow\infty$, which implies that the process starting from empty history is strongly asymptotically stationary.

Br\'{e}maud et al. \cite{BremaudII} studied the rate of convergence to the
equilibrium in a more general setting, i.e. Hawkes process with random marks.
Here, we only consider the unmarked case.
Assume $N(-\infty,0]=0$ and let $N^{\dagger}$ denote the unique stationary Hawkes process.
The convergence in variation is seen via coupling, namely, $N$ and $N^{\dagger}$ are constructed
on the same space and there exists a finite random time $T$ such that
\begin{equation}
\mathbb{P}(N(t,\infty)=N^{\dagger}(t,\infty)\text{ for all $t\geq T$})=1.
\end{equation}

In the exponential case, there exists some $\beta>0$ such that $\int_{0}^{\infty}e^{\beta t}h(t)dt=1$.
Let us define
\begin{equation}
H(t):=\frac{\nu}{1-\Vert h\Vert_{L^{1}}}\int_{t}^{\infty}h(s)ds.
\end{equation}
If $e^{\beta t}H(t)$ is directly Riemann integrable on $\mathbb{R}^{+}$, then for any 
\begin{equation}
K>\frac{\int_{0}^{\infty}e^{\beta t}H(t)dt}{\beta\int_{0}^{\infty}te^{\beta t}h(t)dt},
\end{equation}
there exists $t_{0}(K)$ such that $\mathbb{P}(T>t)\leq Ke^{-\beta t}$ for any $t\geq t_{0}(K)$.

In the subexponential case, the distribution function $G$ with density $g(t)=\frac{h(t)}{\Vert h\Vert_{L^{1}}}$
is subexponential, in the sense that,
\begin{equation}
\lim_{t\rightarrow\infty}\frac{1-G^{\ast n}(t)}{1-G(t)}=n,\quad\text{for any $n\in\mathbb{N}$}.
\end{equation}
Further assume that $\int_{0}^{\infty}th(t)dt<\infty$. Then, for any
\begin{equation}
K>\frac{\nu\Vert h\Vert_{L^{1}}}{(1-\Vert h\Vert_{L^{1}})^{2}},
\end{equation}
there exists some $t_{0}(K)$ such that for any $t\geq t_{0}(K)$, we have
\begin{equation}
\mathbb{P}(T>t)\leq K\int_{t}^{\infty}\overline{G}(u)du,
\end{equation}
where $\overline{G}=1-G$.

\subsection{Bartlett Spectrum for Linear Hawkes Processes}

The methods of analysis for point processes by spectrum were introduced by Bartlett \cite{BartlettI} and \cite{BartlettII}.
We refer to Chapter 8 of Daley and Vere-Jones \cite{Daley} for a detailed discussion.

Let $N$ be a second-order stationary point process on $\mathbb{R}$. (For the definition
of second-order stationary point process, we refer to Daley and Vere-Jones \cite{Daley}.)
Define the set $\mathcal{S}$ as the space of functions of rapid decay, i.e. $\phi\in\mathcal{S}$ if
\begin{equation}
\left|\frac{d^{k}\phi(x)}{dx^{k}}\right|\leq\frac{C(k,r)}{(1+|x|)^{r}},
\end{equation}
for some constants $C(k,r)<\infty$ and all positive integers $r$ and $k$.

For bounded measurable $\phi$ with bounded support and also $\phi\in\mathcal{S}$, there exists
a measure $\Gamma$ on $\mathcal{B}$ such that
\begin{equation}
\text{Var}\left(\int_{\mathbb{R}}\phi(x)N(dx)\right)
=\int_{\mathbb{R}}|\hat{\phi}(\omega)|\Gamma(d\omega),
\end{equation}
where $\hat{\phi}(\omega)=\int_{\mathbb{R}}e^{i\omega u}\phi(u)du$ is the Fourier transform of $\phi$.
$\Gamma$ is refered to as the Bartlett spectrum. 
We also have
\begin{equation}
\text{Cov}\left(\int_{\mathbb{R}}\phi(x)N(dx),\int_{\mathbb{R}}\psi(x)N(dx)\right)
=\int_{\mathbb{R}}\hat{\phi}(\omega)\hat{\psi}(\omega)\Gamma(d\omega).
\end{equation}
Hawkes \cite{HawkesIII} proved that
for the linear stationary Hawkes process with
\begin{equation}
\lambda_{t}=\nu+\int_{-\infty}^{t}h(t-s)N(ds),
\end{equation}
$\nu>0$ and $\Vert h\Vert_{L^{1}}<1$, the Bartlett spectrum is given by
\begin{equation}
\Gamma(d\omega)=\frac{\nu}{2\pi(1-\Vert h\Vert_{L^{1}})|1-\hat{h}(\omega)|^{2}}d\omega.
\end{equation}

Moreover, if $\mu(\tau):=\mathbb{E}[dN(t+\tau)dN(t)]/(dt)^{2}-\mu^{2}$ is the
covariance density, where $\mu:=\frac{\nu}{1-\Vert h\Vert_{L^{1}}}$, then Hawkes \cite{Hawkes}
proved that $\mu(\tau)=\mu(-\tau)$, $\tau>0$, satisfies the equation
\begin{equation}
\mu(\tau)=\mu h(\tau)+\int_{-\infty}^{\tau}h(t-v)\mu(v)dv.
\end{equation}
Since $\mu(\tau)=\mu(-\tau)$, we have
\begin{equation}
\mu(\tau)=\mu h(\tau)+\int_{0}^{\infty}h(\tau+v)\mu(v)dv
+\int_{0}^{\tau}h(\tau-v)\mu(v)dv,\quad\tau>0.
\end{equation}
In general, $\mu(\tau)$ may not have an analytical form. However, when $h(\cdot)$
is exponential, say $h(t)=\alpha e^{-\beta t}$, Hawkes \cite{Hawkes} showed that
\begin{equation}
\mu(\tau)=\frac{\nu\alpha\beta(2\beta-\alpha)}{2(\beta-\alpha)^{2}}e^{-(\beta-\alpha)\tau},\quad\tau>0.
\end{equation}

The Bartlett spectrum analysis has later been generalized to marked linear Hawkes processes and some more general models.
We refer to Br\'{e}maud and Massouli\'{e} \cite{BremaudV} and Br\'{e}maud and Massouli\'{e} \cite{BremaudVI}.

\subsection{Limit Theorems for Linear Hawkes Processes}

When $\lambda(\cdot)$ is linear, say $\lambda(z)=\nu+z$, for some $\nu>0$ and $\Vert h\Vert_{L^{1}}<1$, 
the Hawkes process has a very nice immigration-birth representation, see
for example Hawkes and Oakes \cite{HawkesII}. For such a linear Hawkes process, the limit theorems are very well understood. 
Consider a stationary Hawkes process $N^{\dagger}$ with intensity
\begin{equation}
\lambda^{\dagger}_{t}=\nu+\int_{-\infty}^{t}h(t-s)N^{\dagger}(ds).
\end{equation}
Taking expecatations on the both sides of the above equation and using stationarity, we get
\begin{equation}
\mu:=\mathbb{E}[\lambda^{\dagger}_{t}]=\nu+\int_{-\infty}^{t}h(t-s)\mathbb{E}[\lambda^{\dagger}_{s}]ds=\nu+\mu\Vert h\Vert_{L^{1}},
\end{equation}
which implies that $\mu=\frac{\nu}{1-\Vert h\Vert_{L^{1}}}$. By ergodic theorem, we have
\begin{equation}
\frac{N_{t}}{t}\rightarrow\frac{\nu}{1-\Vert h\Vert_{L^{1}}},\quad\text{as $t\rightarrow\infty$ a.s.}
\end{equation}
Moreover, Bordenave and Torrisi \cite{Bordenave} proved a large deviation principle for $(\frac{N_{t}}{t}\in\cdot)$.

\begin{theorem}[Bordenave and Torrisi 2007]
$(N_{t}/t\in\cdot)$ satisfies a large deviation principle with the rate function
\begin{equation}
I(x)=
\begin{cases}
x\log\left(\frac{x}{\nu+x\Vert h\Vert_{L^{1}}}\right)-x+x\Vert h\Vert_{L^{1}}+\nu &\text{if $x\in[0,\infty)$}
\\
+\infty &\text{otherwise}
\end{cases}.
\end{equation}
\end{theorem}

Recently, Bacry et al. \cite{Bacry} proved a functional central limit theorem for linear multivariate Hawkes process under certain assumptions.
That includes the linear Hawkes process as a special case and they proved that

\begin{theorem}[Bacry et al. 2011]
\begin{equation}
\frac{N_{\cdot t}-\cdot\mu t}{\sqrt{t}}\rightarrow\sigma B(\cdot),\quad\text{as $t\rightarrow\infty$,}
\end{equation}
where $B(\cdot)$ is a standard Brownian motion. The convergence
is weak convergence on $D[0,1]$, the space of c\'{a}dl\'{a}g functions on $[0,1]$, equipped with Skorokhod topology.
Here, 
\begin{equation}
\mu=\frac{\nu}{1-\Vert h\Vert_{L^{1}}}\quad\text{and}\quad\sigma^{2}=\frac{\nu}{(1-\Vert h\Vert_{L^{1}})^{3}}.
\end{equation}
\end{theorem}

Unlike the central limit theorem and the law of the iterated logarithm, there are not as many good crietria one can use to 
prove the moderate deviation principle
for nonlinear Hawkes processes,
which would fill in the gap between the central limit theorem and the large deviation principle. 
Nevertheless, due to the analytical tractability and birth-immigration representation
of linear Hawkes process, Zhu \cite{ZhuIV} proved the moderate deviations for linear Hawkes processes.

\begin{theorem}\label{MDP}
Assume $\lambda(z)=\nu+z$, $\nu>0$, $\Vert h\Vert_{L^{1}}<1$ and $\sup_{t>0}t^{3/2}h(t)=C<\infty$.
For any Borel set $A$ and time sequence $a(t)$ such that $\sqrt{t}\ll a(t)\ll t$, we have
the following moderate deviation principle.
\begin{align}
-\inf_{x\in A^{o}}J(x)&\leq\liminf_{t\rightarrow\infty}\frac{t}{a(t)^{2}}
\log\mathbb{P}\left(\frac{N_{t}-\mu t}{a(t)}\in A\right)
\\
&\leq\limsup_{t\rightarrow\infty}\frac{t}{a(t)^{2}}
\log\mathbb{P}\left(\frac{N_{t}-\mu t}{a(t)}\in A\right)
\leq-\inf_{x\in\overline{A}}J(x),\nonumber
\end{align}
where $J(x)=\frac{x^{2}(1-\Vert h\Vert_{L^{1}})^{3}}{2\nu}$.
\end{theorem}

The proof of Theorem \ref{MDP} will be given in Section \ref{SectionMDP}.

In a nutshell, linear Hawkes processes satisfy very nice limit theorems and the limits can be computed more or less explicitly.

\subsection{Proof of Theorem \ref{MDP}}\label{SectionMDP}

Since a Hawkes process has a long memory and is in general non-Markovian, there is no good criterion
in the literature for moderate deviations that we can use directly. For example, Bacry et al. 
\cite{Bacry} used a central limit theorem for martingales
to obtain a central limit theorem for linear Hawkes processes. But there is no criterion
for moderate deviations for martingales that can fit into the context of Hawkes processes.
Our strategy relies on the fact that for linear Hawkes processes there is a nice immigration-birth representation
from which we can obtain a semi-explicit formula for the moment generating function of $N_{t}$
in Lemma \ref{expectationII}. A careful asymptotic analysis of this formula would lead to the proof
of Theorem \ref{MDP}.

\begin{proof}[Proof of Theorem \ref{MDP}]
Let us first prove that for any $\theta\in\mathbb{R}$,
\begin{equation}
\lim_{t\rightarrow\infty}\frac{t}{a(t)^{2}}\log\mathbb{E}
\left[e^{\frac{a(t)}{t}\theta(N_{t}-\mu t)}\right]
=\frac{\nu\theta^{2}}{2(1-\Vert h\Vert_{L^{1}})^{3}}.
\end{equation}
By Lemma \ref{expectationII}, for fixed $\theta\in\mathbb{R}$ and $t$ sufficiently large, we have
\begin{equation}\label{Gformula}
\mathbb{E}\left[e^{\frac{a(t)}{t}\theta N_{t}}\right]=e^{\nu\int_{0}^{t}G_{t}(s)ds},
\end{equation}
where $G_{t}(s)=e^{\frac{a(t)}{t}\theta+\int_{0}^{s}h(u)G_{t}(s-u)du}-1$, $0\leq s\leq t$. Here,
$G_{t}(s)$ is simply the $F(s)-1$ in Lemma \ref{expectationII}. Because $\frac{a(t)}{t}\theta$ depends on $t$,
we write $G_{t}(s)$ instead of $G(s)$ to indicate its dependence on $t$.
Clearly, $G_{t}(s)$ is increasing in $s$ and $G_{t}(\infty)$ is the minimal
solution to the equation $x_{t}=e^{\frac{a(t)}{t}\theta+\Vert h\Vert_{L^{1}}x_{t}}-1$. (See the proof of
Lemma \ref{expectationII} and the reference therein.)
Since $\Vert h\Vert_{L^{1}}<1$, it is easy to see that $x_{t}=O(a(t)/t)$.
Since $x_{t}=O(a(t)/t)$, we have $G_{t}(s)=O(a(t)/t)$ uniformly in $s$. By Taylor's expansion,
\begin{align}
G_{t}(s)&=\frac{a(t)\theta}{t}+\int_{0}^{s}h(u)G_{t}(s-u)du\label{G}
\\
&\qquad\qquad+\frac{1}{2}\left(\frac{a(t)\theta}{t}\right)^{2}
+\frac{1}{2}\left(\int_{0}^{s}h(u)G_{t}(s-u)du\right)^{2}\nonumber
\\
&\qquad\qquad\qquad\qquad+\frac{a(t)\theta}{t}\int_{0}^{s}h(u)G_{t}(s-u)du+O\left((a(t)/t)^{3}\right).\nonumber
\end{align}
Let $G_{t}(s)=\frac{a(t)\theta}{t}G_{1}(s)+\left(\frac{a(t)}{t}\right)^{2}G_{2}(s)+\epsilon_{t}(s)$, where
\begin{equation}
G_{1}(s):=1+\int_{0}^{s}h(u)G_{1}(s-u)du,\label{G1}
\end{equation}
and
\begin{equation}
G_{2}(s):=\int_{0}^{s}h(u)G_{2}(s-u)du+\frac{\theta^{2}}{2}+\theta^{2}(G_{1}(s)-1)
+\frac{\theta^{2}}{2}(G_{1}(s)-1)^{2}.\label{G2}
\end{equation}
Substituting \eqref{G1} and \eqref{G2} back into \eqref{G} and using the fact $G_{t}(s)=O(a(t)/t)$ uniformly in $s$, 
we get $\epsilon_{t}(s)=O((a(t)/t)^{3})$ uniformly in $s$.
Moreover, we claim that
\begin{align}
&\lim_{t\rightarrow\infty}\frac{1}{a(t)}\left[\theta\nu\int_{0}^{t}G_{1}(s)ds
-\theta\mu t\right]=0,\label{ClaimI}
\\
&\lim_{t\rightarrow\infty}\frac{1}{t}\int_{0}^{t}G_{2}(s)ds=\frac{\theta^{2}}{2(1-\Vert h\Vert_{L^{1}})^{3}}.\label{ClaimII}
\end{align}
To prove \eqref{ClaimI}, notice first that
\begin{align}
\frac{1}{t}\int_{0}^{t}G_{1}(s)ds
&=1+\frac{1}{t}\int_{0}^{t}\int_{0}^{s}h(u)G_{1}(s-u)duds
\\
&=1+\frac{1}{t}\int_{0}^{t}h(u)\int_{u}^{t}G_{1}(s-u)dsdu\nonumber
\\
&=1+\frac{1}{t}\int_{0}^{t}h(u)\int_{0}^{t-u}G_{1}(s)dsdu\nonumber
\\
&=1+\frac{1}{t}\int_{0}^{t}h(u)\int_{0}^{t}G_{1}(s)dsdu
-\frac{1}{t}\int_{0}^{t}h(u)\int_{t-u}^{t}G_{1}(s)dsdu.\nonumber
\end{align}
Therefore, 
\begin{equation}
\frac{1}{t}\int_{0}^{t}G_{1}(s)ds=
\frac{1-\frac{1}{t}\int_{0}^{t}h(u)\int_{t-u}^{t}G_{1}(s)dsdu}{1-\int_{0}^{t}h(u)du}.
\end{equation}
Hence,
\begin{align}
&\frac{1}{a(t)}\left[\theta\nu\int_{0}^{t}G_{1}(s)ds
-\theta\mu t\right]\label{TwoTerms}
\\
&=\frac{\theta\nu}{a(t)}\int_{0}^{t}\left(G_{1}(s)-\frac{1}{1-\Vert h\Vert_{L^{1}}}\right)ds\nonumber
\\
&=\frac{\theta\nu t}{a(t)}
\left[\frac{1}{1-\int_{0}^{t}h(u)du}-\frac{1}{1-\int_{0}^{\infty}h(u)du}\right]
-\frac{\theta\nu}{a(t)}\frac{\int_{0}^{t}h(u)\int_{t-u}^{t}G_{1}(s)dsdu}{1-\int_{0}^{t}h(u)du}.\nonumber
\end{align}

For the first term in \eqref{TwoTerms}, we have
\begin{equation}
\left|\frac{\theta\nu t}{a(t)}
\left[\frac{1}{1-\int_{0}^{t}h(u)du}-\frac{1}{1-\int_{0}^{\infty}h(u)du}\right]\right|
\leq\frac{|\theta|\nu t}{a(t)}\frac{\int_{t}^{\infty}h(u)du}{(1-\Vert h\Vert_{L^{1}})^{2}}\rightarrow 0,
\end{equation}
as $t\rightarrow\infty$, since by our assumption, $\sup_{t>0}t^{3/2}h(t)\leq C<\infty$, which implies that 
$\frac{t}{a(t)}\int_{t}^{\infty}h(u)du\leq\frac{t}{a(t)}\int_{t}^{\infty}\frac{C}{u^{3/2}}du=\frac{2C\sqrt{t}}{a(t)}\rightarrow 0$ 
as $t\rightarrow\infty$.

For the second term in \eqref{TwoTerms}, we have
\begin{equation}
\limsup_{t\rightarrow\infty}\left|\frac{\theta\nu}{a(t)}\frac{\int_{0}^{t}h(u)\int_{t-u}^{t}G_{1}(s)dsdu}{1-\int_{0}^{t}h(u)du}\right|
\leq\lim_{t\rightarrow\infty}G_{1}(t)\limsup_{t\rightarrow\infty}\frac{|\theta|\nu}{a(t)}\frac{\int_{0}^{t}h(u)udu}{1-\Vert h\Vert_{L^{1}}}=0.
\end{equation}
This is because \eqref{G1} is a renewal equation and $\Vert h\Vert_{L^{1}}<1$.
By the application of the Tauberian theorem to the renewal equation, (see Chapters XIII and XIV of Feller \cite{Feller}),
$\lim_{t\rightarrow\infty}G_{1}(t)=\frac{1}{1-\Vert h\Vert_{L^{1}}}$.
Moreover, our assumptions $\sup_{t>0}t^{3/2}h(t)\leq C<\infty$ and $\Vert h\Vert_{L^{1}}<\infty$ imply that 
\begin{equation}
\frac{1}{a(t)}\int_{0}^{t}h(u)udu\leq\frac{1}{a(t)}\int_{0}^{1}h(u)udu+\frac{1}{a(t)}\int_{1}^{t}\frac{C}{u^{1/2}}du\rightarrow 0,
\end{equation}
as $t\rightarrow\infty$.

To prove \eqref{ClaimII}, notice that 
$\lim_{t\rightarrow\infty}G_{1}(t)=\frac{1}{1-\Vert h\Vert_{L^{1}}}$ and
again by the application of the Tauberian theorem to the renewal equation, (see Chapters XIII and XIV of Feller \cite{Feller}),
we have
\begin{align}
\lim_{t\rightarrow\infty}\frac{1}{t}\int_{0}^{t}G_{2}(s)ds
&=\lim_{t\rightarrow\infty}G_{2}(t)
\\
&=\frac{\theta^{2}}{2}\frac{1+2\left(\frac{1}{1-\Vert h\Vert_{L^{1}}}-1\right)
+\left(\frac{1}{1-\Vert h\Vert_{L^{1}}}-1\right)^{2}}{1-\Vert h\Vert_{L^{1}}}\nonumber
\\
&=\frac{\theta^{2}}{2(1-\Vert h\Vert_{L^{1}})^{3}}.\nonumber
\end{align}
Finally, from \eqref{Gformula} and the definitions of $G_{1}(s)$, $G_{2}(s)$ and $\epsilon_{t}(s)$,
we have
\begin{align}
&\frac{t}{a(t)^{2}}\log\mathbb{E}\left[e^{\frac{a(t)}{t}\theta(N_{t}-\mu t)}\right]
\\
&=\frac{t}{a(t)^{2}}\nu\int_{0}^{t}G_{t}(s)ds-\theta\mu\frac{t}{a(t)}\nonumber
\\
&=\frac{1}{a(t)}\left[\nu\theta\int_{0}^{t}G_{1}(s)ds-\theta\mu t\right]
+\frac{1}{t}\nu\int_{0}^{t}G_{2}(s)ds+\frac{t}{a(t)^{2}}\int_{0}^{t}\epsilon_{t}(s)ds.\nonumber
\end{align}
Hence, by \eqref{ClaimI}, \eqref{ClaimII} and the fact that $\epsilon_{t}(s)=O((a(t)/t)^{3})$ uniformly in $s$,
we conclude that, for any $\theta\in\mathbb{R}$,
\begin{equation}
\lim_{t\rightarrow\infty}\frac{t}{a(t)^{2}}\log\mathbb{E}
\left[e^{\frac{a(t)}{t}\theta(N_{t}-\mu t)}\right]
=\frac{\nu\theta^{2}}{2(1-\Vert h\Vert_{L^{1}})^{3}}.
\end{equation}
Applying the G\"{a}rtner-Ellis theorem (see for example \cite{Dembo}), we conclude that, for any Borel set $A$,
\begin{align}
-\inf_{x\in A^{o}}J(x)&\leq\liminf_{t\rightarrow\infty}\frac{t}{a(t)^{2}}
\log\mathbb{P}\left(\frac{N_{t}-\mu t}{a(t)}\in A\right)
\\
&\leq\limsup_{t\rightarrow\infty}\frac{t}{a(t)^{2}}
\log\mathbb{P}\left(\frac{N_{t}-\mu t}{a(t)}\in A\right)
\leq-\inf_{x\in\overline{A}}J(x),\nonumber
\end{align}
where 
\begin{equation}
J(x)=\sup_{\theta\in\mathbb{R}}\left\{\theta x-\frac{\nu\theta^{2}}{2(1-\Vert h\Vert_{L^{1}})^{3}}\right\}
=\frac{x^{2}(1-\Vert h\Vert_{L^{1}})^{3}}{2\nu}.
\end{equation}
\end{proof}

\begin{lemma}\label{expectationII}
For $\theta\leq\Vert h\Vert_{L^{1}}-1-\log\Vert h\Vert_{L^{1}}$, 
\begin{equation}
\mathbb{E}[e^{\theta N_{t}}]=e^{\nu\int_{0}^{t}(F(s)-1)ds},
\end{equation}
where $F(s)=e^{\theta+\int_{0}^{s}h(u)(F(s-u)-1)du}$ for any $0\leq s\leq t$.
\end{lemma}

\begin{proof}
The Hawkes process has a very nice immigration-birth representation, see
for example Hawkes and Oakes \cite{HawkesII}. The immigrant arrives according to a homogeneous
Poisson process with constant rate $\nu$. Each immigrant produces a number of children, this being
Poisson distributed with parameter $\Vert h\Vert_{L^{1}}$. Conditional on the number of
the children of an immigrant, the time that a child is born
has probability density function $\frac{h(t)}{\Vert h\Vert_{L^{1}}}$. Each child produces children 
according to the same laws independent of other children. All the immigrants produce children independently.
Let $F(t)=\mathbb{E}[e^{\theta S(t)}]$, where $S(t)$ is the number of descendants an immigrant
generates up to time $t$. Hence, we have
\begin{align}
\mathbb{E}\left[e^{\theta N_{t}}\right]
&=\sum_{k=0}^{\infty}\frac{(\nu t)^{k}}{k!}e^{-\nu t}\frac{1}{t^{k}/k!}
\idotsint_{t_{1}<t_{2}<\cdots<t_{k}}F(t_{1})\cdots F(t_{k})dt_{1}\cdots dt_{k}
\\
&=e^{\nu\int_{0}^{t}(F(s)-1)ds}.\nonumber
\end{align}

By page 39 of Jagers \cite{Jagers}, 
for all $\theta\in(-\infty,\Vert h\Vert_{L^{1}}-1-\log\Vert h\Vert_{L^{1}}]$, $\mathbb{E}[e^{\theta S(\infty)}]$
is the minimal positive solution of 
\begin{equation}
\mathbb{E}[e^{\theta S(\infty)}]=e^{\theta}\exp\left\{\mu(\mathbb{E}[e^{\theta S(\infty)}]-1)\right\}.
\end{equation}

Let $K$ be the number of
children of an immigrant and let $S^{(1)}_t,S^{(2)}_t,\ldots,S^{(K)}_t$ be the number of descendants 
of immigrant's $k$th child that were born before time $t$ (including the
$k$th child if and only if it was born before time $t$). Then
\begin{align}
F(t)&= \sum_{k=0}^{\infty}\mathbb{E}\left[e^{\theta S(t)}|K=k\right]\mathbb{P}(K=k) 
\\
&=e^{\theta}\sum_{k=0}^{\infty}\mathbb{E}\left[e^{\theta S_{t}^{(1)}}\right]^{k}\mathbb{P}(K=k)\nonumber
\\
&=e^{\theta}\sum_{k=0}^{\infty}\left(\int_{0}^{t}\frac{h(s)}{\Vert h\Vert_{L^{1}}}F(t-s)ds\right)^k 
e^{-\Vert h\Vert_{L^{1}}}\frac{\Vert h\Vert_{L^{1}}^k}{k!}\nonumber
\\
&=e^{\theta+\int_{0}^{t}h(s)(F(t-s)-1)ds}.\nonumber
\end{align}
\end{proof}

\subsection{Simulations and Calibrations}

Assume the past of a Hawkes process is known up to present time zero, say
the configuration of the history is $\omega^{-}$. Let $\tau_{1}$ be the first jump after time zero.
Then, it is easy to see that
\begin{equation}
\mathbb{P}(\tau_{1}\geq t)=e^{-\int_{0}^{t}\lambda^{\omega^{-}}_{s}ds},
\end{equation}
where $\lambda^{\omega^{-}}_{s}=\nu+\sum_{\tau\in\omega^{-}}h(s-\tau)$.
This leads to a straight forward simulation method which is applicable for any simple point process. 
This algorithm and its theoretical foundation go back to a thinning
procedure given Lewis and Shedler \cite{Lewis}. In the context of Hawkes processes, this simulation method
was first used in Ogata \cite{OgataIV}. It is sometimes called Ogata's modified
thinning algorithm. 

If we want to simulate the stationary version of the Hawkes process on a finite time interval,
then the standard method for the simulation method described above does not work
as the past of the process is not known and cannot be simulated,
at least not completely.

If one ignores the past of the process and simply starts to simulate the
process at some given time, one speaks about an approximate simulation. In
this case, one is actually simulating a transient version and not the stationary
version of the process. 
But if one simulates for a long enough time interval, 
then the transient version converges to the stationary
one. Such an approximate simulation method of Hawkes processes was discussed in M\o ller and Rasmussen \cite{MollerII}.
A simulation method which directly simulates the stationary
version without approximation is a so-called perfect simulation method.
The idea is to incorporate somehow the effect of past observations without actually
simulating the past of the process. For point processes, this
type of simulation has first been described in Brix and Kendall \cite{Brix}.
In the context of Hawkes processes, the perfect simulation method was discussed in M\o ller and Rasmussen \cite{MollerI}.

The calibrations, i.e. the estimation of the parameters of Hawkes processes,
was first studied in Vere-Jones \cite{Vere} and Ozaki \cite{Ozaki}, based on a maximum likelihood method for point
processes introduced by Rubin \cite{Rubin}.
The properties of the maximum likelihood estimator was analyzed in Ogata \cite{Ogata}.

In Marsan and Lengline \cite{Marsan}, an Expectation-Maximization (EM) algorithm, called ``Model
Independent Stochastic Declustering" (MISD), is introduced for the nonparametric estimation of
self-exciting point processes with time-homogeneous background rate (For linear Hawkes process with
intensity $\lambda_{t}=\nu_{t}+\sum_{\tau<t}h(t-\tau)$, $\nu_{t}$ is the background rate
and $h(\cdot)$ is the exciting function). 

The efficacy of the MISD algorithm was studied in Sornette and Utkin \cite{Sornette}, where the authors
found that the ability of MISD to recover key parameters such as $\Vert h\Vert_{L^{1}}$
depends on the values of the model parameters. In particular, they pointed out that the accuracy
of MISD improves as the timescale over which the exciting function $h(\cdot)$ decays shortens.
In Lewis and Mohler \cite{LewisII}, they introduced a Maximum Penalized Likelihood Estimation (MPLE) approach for the
nonparametric estimation of Hawkes processes. The method is capable of estimating
$\nu_{t}$ and $h(t)$ simultaneously, without prior knowledge of their form. Analogous to MPLE in the
context of density estimation, the added regularity of the estimates allows for higher accuracy
and/or lower sample sizes in comparison to MISD.

\section{Nonlinear Hawkes Processes}

Consider a simple point process with intensity
\begin{equation}
\lambda_{t}=\lambda\left(\int_{-\infty}^{t}h(t-s)N(ds)\right),\label{nonlineardynamics}
\end{equation}
where $\lambda(\cdot):\mathbb{R}^{+}\rightarrow\mathbb{R}^{+}$ and $h(\cdot):\mathbb{R}^{+}\rightarrow[0,\infty)$.
Br\'{e}maud and Massouli\'{e} \cite{Bremaud} studied the existence and uniqueness of a stationary nonlinear 
Hawkes process that satisfies the dynamics \eqref{nonlineardynamics} as well as its stability in distribution
and in variation. They allow $h(\cdot)$ to take negative values as well. In this thesis, we always 
consider $h(\cdot)$ to be nonnegative.

The following result is about the existence of a stationary nonlinear Hawkes process
satisfying the dynamics \eqref{nonlineardynamics}. We do not need $\lambda(\cdot)$ to be Lipschitz.

\begin{theorem}[Br\'{e}maud and Massouli\'{e} \cite{Bremaud}]
Let $\lambda(\cdot)$ be a nonnegative, nondecreasing and left-continuous function, 
satisfying $\lambda(z)\leq C+\alpha z$ for any $z\geq 0$, for some $C>0$ and $\alpha\geq 0$ and 
let $h(\cdot):\mathbb{R}^{+}\rightarrow\mathbb{R}^{+}$ 
be such that $\alpha\int_{0}^{\infty}h(t)dt<1$. Then there exists a stationary point process $N$ with dynamics \eqref{nonlineardynamics}.
\end{theorem}

The following results concerns the uniqueness and stability in distribution and in variation of a nonlinear Hawkes process.

\begin{theorem}[Br\'{e}maud and Massouli\'{e} \cite{Bremaud}]
Let $\lambda(\cdot)$ be $\alpha$-Lipschitz such that $\alpha\Vert h\Vert_{L^{1}}<1$.

(i) There exists a unique stationary distribution of $N$ with finite average intensity $\mathbb{E}[N(O, 1]]$ 
and with dynamics \eqref{nonlineardynamics}.

(ii) Let $\epsilon_{a}(t):=\int_{t-a}^{t}\int_{\mathbb{R}^{-}}h(s-u)N(du)ds$. 
The dynamics \eqref{nonlineardynamics} are stable in distribution with respect to either the initial
condition \eqref{condition1} or the condition \eqref{condition2} below,
\begin{equation}\label{condition1}
\sup_{t\geq 0}\epsilon_{a}(t)<\infty\text{ a.s. and }\lim_{t\rightarrow\infty}\epsilon_{a}(t)=0\text{ a.e. for every $a>0$},
\end{equation}
\begin{equation}\label{condition2}
\sup_{t\geq 0}\mathbb{E}[\epsilon_{a}(t)]<\infty\text{ and }\lim_{t\rightarrow\infty}\mathbb{E}[\epsilon_{a}(t)]=0\text{ for every $a>0$}.
\end{equation}

(iii) The dynamics \eqref{nonlineardynamics} are stable in variation with respect to the initial condition,
\begin{equation}
\int_{\mathbb{R}^{+}}h(t)N[-t,0)dt=\int_{\mathbb{R}^{+}}\int_{-\infty}^{0}h(t-s)N(ds)<\infty,\text{ a.s.}
\end{equation}
if we assume further that $\int_{0}^{\infty}th(t)dt<\infty$.
\end{theorem}

Massouli\'{e} \cite{Massoulie} extended the stability results to nonlinear Hawkes processes with random marks. 
He also considered the Markovian case and proved stability results without the Lipschitz condition for $\lambda(\cdot)$.

Very recently, Karabash \cite{KarabashII} proved stability results for a much wider class
of nonlinear Hawkes process, including the case when $\lambda(\cdot)$ is not Lipschitz. 

Moreover, Br\'{e}maud et al. \cite{BremaudII} considered
the rate of extinction for nonlinear Hawkes process, that is the rate of convergence to the
equilibrium when the stationary process is an empty process. 
Indeed, they considered a more general setting, i.e. Hawkes process with random marks.
Let $N$ be a nonlinear Hawkes process which is empty
on $(-\infty,0]$, i.e. $N(-\infty,0]=0$ which satisfies the dynamics
\begin{equation}
\lambda_{t}:=\nu(t)+
\phi\left(\int_{0}^{t}h(t-s)N(ds)\right),
\end{equation}
where $\nu:\mathbb{R}^{+}\rightarrow\mathbb{R}^{+}$ is locally integrable, $\phi:\mathbb{R}\rightarrow[0,\infty)$,
$\phi(0)=0$, $\phi$ is $1$-Lipschitz and $h:\mathbb{R}^{+}\rightarrow\mathbb{R}$ is measurable and not necessarily nonnegative
and $\int_{0}^{\infty}|h(t)|dt<1$. The unique stationary process $N^{0}$
corresponding to the dynamics
\begin{equation}
\phi\left(\int_{0}^{t}h(t-s)N^{0}(ds)\right),
\end{equation}
is the empty process. Assume $\int_{0}^{\infty}\nu(t)dt<\infty$, $\int_{0}^{\infty}th(t)dt<\infty$
and $t\mapsto|h(t)|$ is locally bounded.
 
Then $\theta_{t}N$ converges in variation to the empty process. The convergence in variation takes place
via coupling in the sense that there exists a finite random time $T$ so that,
\begin{equation}
\mathbb{P}(N(t,\infty)=0\text{ for any $t\geq T$ })=1.
\end{equation}

Depending on whether the tail of $|h(t)|$ is exponential or subexponentail, the following was
obtained by Br\'{e}maud et al. \cite{BremaudII}.

In the exponential case, let $\beta>0$ be such that $\int_{0}^{\infty}e^{\beta t}|h(t)|dt=1$.
Assume $e^{\beta t}\nu(t)$ is directly Riemann integrable. Then,
for any $K$ with
\begin{equation}
K>\frac{\int_{0}^{\infty}e^{\beta t}\nu(t)dt}
{\beta\int_{0}^{\infty}te^{\beta t}|h(t)|dt},
\end{equation}
there exists $t_{0}(K)$, for any $t\geq t_{0}(K)$,
\begin{equation}
\mathbb{P}(T>t)\leq Ke^{-\beta t}.
\end{equation}

In the subexponential case, assume that distribution functinon $G$ with density $g(t)=\frac{|h(t)|}{\int_{0}^{\infty}|h(t)|dt}$
is subexponential, $\nu(\cdot)$ is bounded and that $B=\limsup_{t\rightarrow\infty}\frac{\nu(t)}{\overline{G}(t)}<\infty$,
where $\overline{G}=1-G$. Then for any
\begin{equation}
K>\frac{B}{1-\int_{0}^{\infty}|h(t)|dt},
\end{equation}
there exists $t_{0}(K)$ such that for any $t\geq t_{0}$,
\begin{equation}
\mathbb{P}(T>t)\leq K\int_{t}^{\infty}\overline{G}(s)ds.
\end{equation}

Kwieci\'{n}ski and Szekli \cite{Kwiecinski} considered the nonlinear Hawkes process
as a special case of self-exciting process. Let $\mathcal{N}(\mathbb{R}^{+})$ be the space
of point processes on $\mathbb{R}^{+}$, which can be regarded as an element of $\mathcal{D}(\mathbb{R}^{+})$,
the space of functions which are right-continuous with left limits, equipped with Skorohod topology.
For any $\mu,\nu\in\mathcal{N}(\mathbb{R}^{+})$, $\mu\prec_{\mathcal{N}}\nu$ if $\mu(B)\leq\nu(B)$
for any bounded set $B\in\mathcal{B}(\mathbb{R}^{+})$. For any $\mu,\nu\in\mathcal{N}(\mathbb{R}^{+})$,
$\mu\prec_{\mathcal{D}}\nu$ if and only if $(\mu_{t})\prec_{\mathcal{D}}(\nu_{t})$ for the corresponding
functions $\mu_{t}:=\mu((0,t]),\nu_{t}:=\nu((0,t])\in\mathcal{D}(\mathbb{R}^{+})$, i.e. $\mu_{t}\leq\nu_{t}$ for all $t>0$.

Now, for a simple point process $N$ with intensity $\lambda(t,N)$ and compensator $\Lambda(t,N):=\int_{0}^{t}\lambda(s,N)ds$,
we say that $N$ is positively self-exciting w.r.t. $\prec_{\mathcal{N}}$ if for any $\mu,\nu\in\mathcal{N}(\mathbb{R}^{+})$,
\begin{equation}
\mu\prec_{\mathcal{N}}\nu\text{ implies that for any $t>0$, }\lambda(t,\mu)\leq\lambda(t,\nu),
\end{equation}
and $N$ is positively self-exciting w.r.t. $\prec_{\mathcal{D}}$ if for any $\mu,\nu\in\mathcal{N}(\mathbb{R}^{+})$,
\begin{equation}
\mu\prec_{\mathcal{D}}\nu\text{ implies that for any $t>0$, }\Lambda(t,\mu)\leq\Lambda(t,\nu).
\end{equation}

Kwieci\'{n}ski and Szekli \cite{Kwiecinski} pointed out that if $h(\cdot)$ is nonnegative
and $\lambda(\cdot)$ nondecreasing, then $N$ is positively self-exciting with respect to $\prec_{\mathcal{N}}$,
and that if $h(\cdot)$ is nonnegative and nondecreasing and $\lambda(\cdot)$ is nondecreasing, then $N$
is positively self-exciting with respect to $\prec_{\mathcal{D}}$.

Let $(\Omega,\mathcal{F})$ be a Polish space with a closed partical ordering $\prec$. A probability
measure on $(\Omega,\mathcal{F})$ is associated $(\prec)$ if
\begin{equation}
P(C_{1}\cap C_{2})\geq P(C_{1})P(C_{2}),
\end{equation}
for all increasing sets $C_{1},C_{2}\in\mathcal{F}$ (a set $C$ is increasing if $x\in C$ and $x\prec y$ implies
$y\in C$).

Kwieci\'{n}ski and Szekli \cite{Kwiecinski} proved that if $N$ is positively self-exciting point process
w.r.t. $\prec_{\mathcal{N}}$ (resp. $\prec_{\mathcal{D}}$), then $N$ is associated $(\prec_{\mathcal{N}})$
(resp. $(\prec_{\mathcal{D}})$). Therefore, it implies that for a nonlinear Hawkes process,
if $h(\cdot)$ is nonnegative and $\lambda(\cdot)$ nondecreasing, then $N$ is associated $(\prec_{\mathcal{N}})$
and if $h(\cdot)$ is nonnegative and nondecreasing and $\lambda(\cdot)$ is nondecreasing, then $N$
is associated $(\prec_{\mathcal{D}})$.

Next, let us consider the limit theorems for nonlinear Hawkes process.
When $\lambda(\cdot)$ is nonlinear, the usual immigration-birth representation no longer works and you may have to use some
abstract theory to obtain limit theorems. Some progress has already been made. 

Br\'{e}maud and Massouli\'{e} \cite{Bremaud}'s stability result implies that by the ergodic theorem,
\begin{equation}
\frac{N_{t}}{t}\rightarrow\mu:=\mathbb{E}[N[0,1]],
\end{equation}
as $t\rightarrow\infty$, where $\mathbb{E}[N[0,1]]$ is the mean of $N[0,1]$ under the stationary and ergodic measure. 

In this thesis, we will obtain a functional central limit theorem and a Strassen's invariance principle in 
Chapter \ref{chap:two} and a process-level, i.e. level-3 large deviation principle in Chapter \ref{chap:three}
and thus a level-1 large deviation principle by contraction principle.
We will also obtain an alternative expression for the rate function for level-1 large deviation principle
of Markovian nonlinear Hawkes process as a variational formula in Chapter \ref{chap:four}.

\section{Multivariate Hawkes Processes}\label{multivariate}

We say $N=(N_{1},\ldots,N_{d})$ is a multivariate Hawkes process if for any $1\leq i\leq d$, 
$N_{i}$ is a simple point process with intensity
\begin{equation}
\lambda_{i,t}:=\nu_{i}+\int_{0}^{t}\sum_{j=1}^{d}h_{ij}(t-s)dN_{j,s},\label{multidynamics}
\end{equation}
where $\nu_{i}\in\mathbb{R}^{+}$ and $h_{ij}(\cdot):\mathbb{R}^{+}\rightarrow\mathbb{R}^{+}$.
Then, $\mathbf{\nu}:=(\nu_{1},\ldots,\nu_{d})$ is a vector and 
$\mathbf{h}:=(h_{ij})_{1\leq i,j\leq d}$ is a $d\times d$ matrix-valued function.

Let us assume that for any $i,j$, $\int_{0}^{\infty}h_{ij}(t)dt<\infty$ and that the spectral radius $\rho(\mathbf{K})$
of the matrix $\mathbf{K}=\int_{0}^{\infty}\mathbf{h}(t)dt$ satisfies $\rho(\mathbf{K})<1$.
Then, Bacry et al. \cite{Bacry} proved a law of large numbers, i.e.
\begin{equation}
\sup_{u\in[0,1]}\Vert T^{-1}N_{Tu}-u(\mathbf{I}-\mathbf{K})^{-1}\mathbf{\nu}\Vert\rightarrow 0,
\end{equation} 
as $T\rightarrow\infty$ almost surely and also in $L^{2}(\mathbb{P})$.
If we assume further that for any $1\leq i,j\leq d$,
\begin{equation}
\int_{0}^{\infty}h_{ij}(t)t^{1/2}dt<\infty.
\end{equation}
Then, Bacry et al. \cite{Bacry} proved the following central limit theorem:
\begin{equation}
\sqrt{T}\left(\frac{1}{T}N_{Tu}-u(\mathbf{I}-\mathbf{K})^{-1}\mathbf{\nu}\right),\quad u\in[0,1]
\end{equation}
converges in law as $T\rightarrow\infty$ under the Skorohod topology to
\begin{equation}
(\mathbf{I}-\mathbf{K})^{-1}\Sigma^{1/2}W_{u},\quad u\in[0,1],
\end{equation} 
where $\Sigma$ is the diagonal matrix with $\Sigma_{ii}=((\mathbf{I}-\mathbf{K})^{-1}\mathbf{\nu})_{i}$, $1\leq i\leq d$.

It is well known that under the assumption that $\rho(\mathbf{K})<1$, there exists a unique
stationary version of the multivariate Hawkes process satisfying the dynamics \eqref{multidynamics}.
The rate of convergence to the stationary version of the multivariate Hawkes process was obtained in Torrisi \cite{Torrisi}. 
The Bartlett spectrum of the multivariate Hawkes process was derived in Hawkes \cite{HawkesIII}. Some non-asymptotics
estimates for multivariate Hawkes processes were obtained in Hansen et al. \cite{Hansen}.

A nice survey on multivariate linear Hawkes processes can be found in Liniger \cite{Liniger}.

%% file: chap2.tex
\chapter{Central Limit Theorem for Nonlinear Hawkes Processes\label{chap:two}}

\section{Main Results}

In this chapter, we obtain a functional central limit theorem for the nonlinear Hawkes process
under Assumption \ref{Assumption}. Under the same assumption, a Strassen's invariance principle also holds.
Let us recall that $N$ is a nonlinear Hawkes process with intensity
\begin{equation}
\lambda_{t}:=\lambda\left(\int_{(-\infty,t)}h(t-s)N(ds)\right).
\end{equation}

\begin{assumption}\label{Assumption}
We assume that
\begin{itemize}
\item
$h(\cdot):[0,\infty)\rightarrow\mathbb{R}^{+}$ is a decreasing function and $\int_{0}^{\infty}th(t)dt<\infty$.

\item
$\lambda(\cdot)$ is positive, increasing and $\alpha$-Lipschitz (i.e. $|\lambda(x)-\lambda(y)|\leq\alpha|x-y|$ for any $x,y$) 
and $\alpha\Vert h\Vert_{L^{1}}<1$.
\end{itemize}
\end{assumption}
Br\'{e}maud and Massouli\'{e} \cite{Bremaud} proved that if $\lambda(\cdot)$ is $\alpha$-Lipschitz
with $\alpha\Vert h\Vert_{L^{1}}<1$, 
there exists a unique stationary and ergodic Hawkes process satisfying the dynamics \eqref{dynamics}.
Hence, under our Assumption \ref{Assumption} (which is slightly stronger than \cite{Bremaud}), there exists a unique
stationary and ergodic Hawkes process satisfying the dynamics \eqref{dynamics}.

Let $\mathbb{P}$ and $\mathbb{E}$ denote the probability measure and expectation for a stationary, ergodic Hawkes process, 
and let $\mathbb{P}(\cdot|\mathcal{F}^{-\infty}_{0})$
and $\mathbb{E}(\cdot|\mathcal{F}^{-\infty}_{0})$ denote the conditional probability measure and expectation 
given the past history.

The following are the main results of this chapter.

\begin{theorem}\label{mainthm}
Under Assumption \ref{Assumption}, let $N$ be the stationary and ergodic nonlinear Hawkes process with dynamics \eqref{dynamics}. We have
\begin{equation}
\frac{N_{\cdot t}-\cdot\mu t}{\sqrt{t}}\rightarrow\sigma B(\cdot),\quad\text{as $t\rightarrow\infty$,}\label{convergence}
\end{equation}
where $B(\cdot)$ is a standard Brownian motion and $0<\sigma<\infty$, where
\begin{equation}
\sigma^{2}:=\mathbb{E}[(N[0,1]-\mu)^{2}]+2\sum_{j=1}^{\infty}\mathbb{E}[(N[0,1]-\mu)(N[j,j+1]-\mu)].\label{sigmasquare}
\end{equation}
The convergence in \eqref{convergence}
is weak convergence on $D[0,1]$, the space of c\'{a}dl\'{a}g functions on $[0,1]$, equipped with Skorokhod topology.
\end{theorem}

\begin{remark}
By a standard central limit theorem for martingales, i.e. Theorem \ref{BTheoremII}, it is easy to see that
\begin{equation}
\frac{N_{\cdot t}-\int_{0}^{\cdot t}\lambda_{s}ds}{\sqrt{t}}\rightarrow\sqrt{\mu} B(\cdot),\quad\text{as $t\rightarrow\infty$,}
\end{equation}
where $\mu=\mathbb{E}[N[0,1]]$. In the linear case, say $\lambda(z)=\nu+z$, Bacry et al. \cite{Bacry} proved that $\sigma^{2}$ in 
\eqref{sigmasquare} satisfies $\sigma^{2}=\frac{\nu}{(1-\Vert h\Vert_{L^{1}})^{3}}>\mu=\frac{\nu}{1-\Vert h\Vert_{L^{1}}}$. That is
not surprising because $N_{\cdot t}-\cdot\mu t$ ``should'' have more fluctuations than $N_{\cdot t}-\int_{0}^{\cdot t}\lambda_{s}ds$.
Therefore, we guess that for nonlinear $\lambda(\cdot)$, $\sigma^{2}$ defined in \eqref{sigmasquare} should also satisfy $\sigma^{2}>\mu=\mathbb{E}[N[0,1]]$.
However, it might not be very easy to compute and say something about $\sigma^{2}$ in such a case.
\end{remark}

In the classical case for a sequence of i.i.d. random variables $X_{i}$ with mean $0$ and variance $1$, we have
the central limit theorem $\frac{1}{\sqrt{n}}\sum_{i=1}^{n}X_{i}\rightarrow N(0,1)$ as $n\rightarrow\infty$,
and we also have $\frac{\sum_{i=1}^{n}X_{i}}{\sqrt{n\log\log n}}\rightarrow 0$
in probability as $n\rightarrow\infty$, but the convergence does not hold a.s. The law of the iterated logarithm says
that $\limsup_{n\rightarrow\infty}\frac{\sum_{i=1}^{n}X_{i}}{\sqrt{n\log\log n}}=\sqrt{2}$ a.s. A functional version
of the law of the iterated logarithm is called Strassen's invariance principle.

It turns out that we also have a Strassen's invariance principle for nonlinear Hawkes processes under Assumption \ref{Assumption}.
\begin{theorem}\label{LIL}
Under Assumption \ref{Assumption}, let $N$ be the stationary and ergodic nonlinear Hawkes process with dynamics \eqref{dynamics}.
Let $X_{n}:=N[n-1,n]-\mu$, $S_{n}:=\sum_{i=1}^{n}X_{i}$, $s_{n}^{2}:=\mathbb{E}[S_{n}^{2}]$, $g(t)=\sup\{n: s_{n}^{2}\leq t\}$, and
for $t\in[0,1]$, let $\eta_{n}(t)$ be the usual linear interpolation, i.e.
\begin{equation}
\eta_{n}(t)=\frac{S_{k}+(s_{n}^{2}t-s_{k}^{2})(s_{k+1}^{2}-s_{k}^{2})^{-1}X_{k+1}}
{\sqrt{2s_{n}^{2}\log\log s_{n}^{2}}},\quad s_{k}^{2}\leq s_{n}^{2}t\leq s_{k+1}^{2},k=0,1,\ldots,n-1.
\end{equation}
Then, $g(e)<\infty$, $\{\eta_{n},n>g(e)\}$ is relatively compact in $C[0,1]$, the set of continuous functions
on $[0,1]$ equipped with uniform topology, and the set of limit points is
the set of absolutely continuous functions $f(\cdot)$ on $[0,1]$ such that $f(0)=0$ and $\int_{0}^{1}f'(t)^{2}dt\leq 1$.
\end{theorem}

\section{Proofs}

This section is devoted to the proof of Theorem \ref{mainthm}. We use a standard central limit theorem, i.e. Theorem \ref{BTheorem}. 
In our proof, we need the fact that $\mathbb{E}[N[0,1]^{2}]<\infty$, which is proved in Lemma \ref{secondmoment}.
Lemma \ref{secondmoment} is proved by proving a stronger result first, i.e. Lemma \ref{midstep}.
We will also prove Lemma \ref{positivesigma} to guarantee that $\sigma>0$ so that the central limit theorem is not degenerate.

Let us first quote the two necessary central limit theorems from Billingsley \cite{Billingsley}. In both Theorem \ref{BTheorem}
and Theorem \ref{BTheoremII}, the filtrations are the natural ones, i.e. given a stochastic process $(X_{n})_{n\in\mathbb{Z}}$,
$\mathcal{F}^{a}_{b}:=\sigma(X_{n},a\leq n\leq b)$, for $-\infty\leq a\leq b\leq\infty$.

\begin{theorem}[Page 197 \cite{Billingsley}]\label{BTheorem}
Suppose $X_{n}$, $n\in\mathbb{Z}$, is an ergodic stationary sequence such that $\mathbb{E}[X_{n}]=0$ and
\begin{equation}
\sum_{n\geq 1}\Vert\mathbb{E}[X_{0}|\mathcal{F}^{-\infty}_{-n}]\Vert_{2}<\infty,
\end{equation}
where $\Vert Y\Vert_{2}=(\mathbb{E}[Y^{2}])^{1/2}$. Let $S_{n}=X_{1}+\cdots+X_{n}$. 
Then $S_{[n\cdot]}/\sqrt{n}\rightarrow\sigma B(\cdot)$ weakly, where the weak convergence is on $D[0,1]$ equipped with the Skorohod topology and
$\sigma^{2}=\mathbb{E}[X_{0}^{2}]+2\sum_{n=1}^{\infty}\mathbb{E}[X_{0}X_{n}]$. The series converges absolutely.
\end{theorem}

\begin{theorem}[Page 196 \cite{Billingsley}]\label{BTheoremII}
Suppose $X_{n}$, $n\in\mathbb{Z}$, is an erogdic stationary sequence of square integrable martingale differences,
i.e. $\sigma^{2}=\mathbb{E}[X_{n}^{2}]<\infty$, and let $\mathbb{E}[X_{n}|\mathcal{F}^{-\infty}_{n-1}]=0$. Let
$S_{n}=X_{1}+\cdots+X_{n}$. Then $S_{[n\cdot]}/\sqrt{n}\rightarrow\sigma B(\cdot)$ weakly, where
the weak convergence is on $D[0,1]$ equipped with the Skorohod topology.
\end{theorem}

Now, we are ready to prove our main result.

\begin{proof}[Proof of Theorem \ref{mainthm}]
Since in the stationary regime, $\mathbb{E}[N[n,n+1]]=\mathbb{E}[N[0,1]]$ for any $n\in\mathbb{Z}$ and let us denote
$\mathbb{E}[N[0,1]]=\mu$. In order to apply Theorem \ref{BTheorem}, let us first prove that
\begin{equation}\label{finitesum}
\sum_{n=1}^{\infty}\left\{\mathbb{E}\left[\left(\mathbb{E}[N(n,n+1]-\mu|\mathcal{F}^{-\infty}_{0}]\right)^{2}\right]\right\}^{1/2}<\infty.
\end{equation}
Let $\mathbb{E}^{\omega^{-}_{1}}[N(n,n+1]]$ and $\mathbb{E}^{\omega^{-}_{2}}[N(n,n+1]]$ be two independent copies of
$\mathbb{E}[N(n,n+1]|\mathcal{F}^{-\infty}_{0}]$. It is easy to check that
\begin{align}
&\frac{1}{2}\mathbb{E}\left\{\left[\mathbb{E}^{\omega^{-}_{1}}[N(n,n+1]]-\mathbb{E}^{\omega^{-}_{2}}[N(n,n+1]]\right]^{2}\right\}
\\
&=\frac{1}{2}\mathbb{E}\left[\mathbb{E}^{\omega^{-}_{1}}[N(n,n+1]]^{2}\right]
+\frac{1}{2}\mathbb{E}\left[\mathbb{E}^{\omega^{-}_{2}}[N(n,n+1]]^{2}\right]\nonumber
\\
&\phantom{=\frac{1}{2}\mathbb{E}\left[\mathbb{E}^{\omega^{-}_{1}}[N(n,n+1]]^{2}\right]}
-\mathbb{E}\left[\mathbb{E}^{\omega^{-}_{1}}[N(n,n+1]]\mathbb{E}^{\omega^{-}_{2}}[N(n,n+1]]\right]\nonumber
\\
&=\mathbb{E}\left[\mathbb{E}[N(n,n+1]|\mathcal{F}^{-\infty}_{0}]^{2}\right]-\mu^{2}\nonumber
\\
&=\mathbb{E}\left[(\mathbb{E}[N(n,n+1]-\mu|\mathcal{F}^{-\infty}_{0}])^{2}\right].\nonumber
\end{align}
Therefore, we have
\begin{align}
&\mathbb{E}\left[(\mathbb{E}[N(n,n+1]-\mu|\mathcal{F}^{-\infty}_{0}])^{2}\right]
\\
&=\frac{1}{2}\mathbb{E}\left\{\left[\mathbb{E}^{\omega^{-}_{1}}[N(n,n+1]]-\mathbb{E}^{\omega^{-}_{2}}[N(n,n+1]]\right]^{2}\right\}\nonumber
\\
&\leq\mathbb{E}\left\{\left[\mathbb{E}^{\omega^{-}_{1}}[N(n,n+1]]-\mathbb{E}^{\emptyset}[N(n,n+1]]\right]^{2}\right\}\nonumber
\\
&\phantom{\leq\mathbb{E}\mathbb{E}^{\omega^{-}_{1}}[N(n,n+1]]}
+\mathbb{E}\left\{\left[\mathbb{E}^{\omega^{-}_{2}}[N(n,n+1]]-\mathbb{E}^{\emptyset}[N(n,n+1]]\right]^{2}\right\}\nonumber
\\
&=2\mathbb{E}\left\{\left[\mathbb{E}^{\omega^{-}_{1}}[N(n,n+1]]-\mathbb{E}^{\emptyset}[N(n,n+1]]\right]^{2}\right\},\nonumber
\end{align}
where $\mathbb{E}^{\emptyset}[N(n,n+1]]$ denotes the expectation of the number of points in $(n,n+1]$ for the 
Hawkes process with the same dynamics \eqref{dynamics} and empty history, i.e. $N(-\infty,0]=0$.

Next, let us estimate $\mathbb{E}^{\omega^{-}_{1}}[N(n,n+1]]-\mathbb{E}^{\emptyset}[N(n,n+1]]$. 
$\mathbb{E}^{\omega^{-}_{1}}[N(n,n+1]]$
is the expectation of the number of points in $(n,n+1]$ for the Hawkes process with intensity
$\lambda_{t}=\lambda\left(\sum_{\tau: \tau\in\omega^{-}_{1}\cup\omega[0,t)}h(t-\tau)\right)$.
It is well defined for a.e. $\omega^{-}_{1}$ under $\mathbb{P}$ because, under Assumption \ref{Assumption},
\begin{equation}
\mathbb{E}[\lambda_{t}]\leq\lambda(0)+\alpha\mathbb{E}\left[\int_{-\infty}^{t}h(t-s)N(ds)\right]
=\lambda(0)+\alpha\Vert h\Vert_{L^{1}}\mathbb{E}[N[0,1]]<\infty,
\end{equation}
which implies that $\lambda_{t}<\infty$ $\mathbb{P}$-a.s.

It is clear that $\mathbb{E}^{\omega^{-}_{1}}[N(n,n+1]]\geq\mathbb{E}^{\emptyset}[N(n,n+1]]$ almost surely, 
so we can use a coupling method to estimate the difference. 
We will follow the ideas in Br\'{e}maud and Massouli\'{e} \cite{Bremaud} using the Poisson embedding method.
Consider $(\Omega,\mathcal{F},\mathcal{P})$, the canonical space of a point process on $\mathbb{R}^{+}\times\mathbb{R}^{+}$ 
in which $\overline{N}$ is Poisson with intensity $1$ under the probability measure $\mathcal{P}$. 
Then the Hawkes process $N^{0}$ with empty past history and intensity
$\lambda^{0}_{t}$ satisfies the following.
\begin{equation}
\begin{cases}
\lambda^{0}_{t}=\lambda\left(\int_{(0,t)}h(t-s)N^{0}(ds)\right)& t\in\mathbb{R}^{+},
\\
N^{0}(C)=\int_{C}\overline{N}(dt\times[0,\lambda^{0}_{t}])& C\in\mathcal{B}(\mathbb{R}^{+}).
\end{cases}
\end{equation}
For $n\geq 1$, let us define recursively $\lambda^{n}_{t}$, $D_{n}$ and $N^{n}$ as follows.
\begin{equation}\label{canonical}
\begin{cases}
\lambda^{n}_{t}=\lambda\left(\int_{(0,t)}h(t-s)N^{n-1}(ds)+\sum_{\tau\in\omega^{-}_{1}}h(t-\tau)\right)& t\in\mathbb{R}^{+},
\\
D_{n}(C)=\int_{C}\overline{N}(dt\times[\lambda^{n-1}_{t},\lambda^{n}_{t}])& C\in\mathcal{B}(\mathbb{R}^{+}),
\\
N^{n}(C)=N^{n-1}(C)+D_{n}(C) & C\in\mathcal{B}(\mathbb{R}^{+}).
\end{cases}
\end{equation}
Following the arguments as in Br\'{e}maud and Massouli\'{e} \cite{Bremaud}, we know that each $\lambda^{n}_{t}$
is an $\mathcal{F}^{\overline{N}}_{t}$-intensity of $N^{n}$, where $\mathcal{F}^{\overline{N}}_{t}$ is the $\sigma$-algebra
generated by $\overline{N}$ up to time $t$. 
By our Assumption \ref{Assumption}, $\lambda(\cdot)$ is
increasing, and it is clear that $\lambda^{n}(t)$ and $N^{n}(C)$ increase in $n$ for all $t\in\mathbb{R}^{+}$
and $C\in\mathcal{B}(\mathbb{R}^{+})$. Thus, $D_{n}$ is well defined and also that as $n\rightarrow\infty$, the limiting
processes $\lambda_{t}$ and $N$ exist. $N$ counts the number of points
of $\overline{N}$ below the curve $t\mapsto\lambda_{t}$ and admits $\lambda_{t}$ as an $\mathcal{F}^{\overline{N}}_{t}$-intensity.
By the monotonicity properties of $\lambda^{n}_{t}$ and $N^{n}$, we have
\begin{align}
&\lambda^{n}_{t}\leq\lambda\left(\int_{(0,t)}h(t-s)N(ds)+\sum_{\tau\in\omega^{-}_{1}}h(t-\tau)\right),
\\
&\lambda_{t}\geq\lambda\left(\int_{(0,t)}h(t-s)N^{n}(ds)+\sum_{\tau\in\omega^{-}_{1}}h(t-\tau)\right).
\end{align}
Letting $n\rightarrow\infty$ (it is valid since we
assume that $\lambda(\cdot)$ is Lipschitz and thus continuous), 
we conclude that $N$, $\lambda_{t}$ satisfies the dynamics \eqref{dynamics}.
Therefore, with intensity $\lambda_{t}$, $N=N^{0}+\sum_{i=1}^{\infty}D_{i}$ is the Hawkes process with past
history $\omega^{-}_{1}$.

We can then estimate the difference by noticing that
\begin{equation}
\mathbb{E}^{\omega^{-}_{1}}[N(n,n+1]]-\mathbb{E}^{\emptyset}[N(n,n+1]]=\sum_{i=1}^{\infty}\mathbb{E}^{\mathcal{P}}[D_{i}(n,n+1]].
\end{equation}
Here $\mathbb{E}^{\mathcal{P}}$ means the expectation with respect to $\mathcal{P}$, the probability measure on 
the canonical space that we defined earlier.

We have
\begin{align}\label{estimation}
&\mathbb{E}^{\mathcal{P}}[D_{1}(n,n+1]]
\\
&=\mathbb{E}^{\mathcal{P}}\left[\int_{n}^{n+1}(\lambda^{1}(t)-\lambda^{0}(t))dt\right]\nonumber
\\
&=\mathbb{E}^{\mathcal{P}}\left[\int_{n}^{n+1}\lambda\left(\sum_{\tau<t,\tau\in N^{0}\cup\omega^{-}_{1}}h(t-\tau)\right)
-\lambda\left(\sum_{\tau<t,\tau\in N^{0}\cup\emptyset}h(t-\tau)\right)dt\right]\nonumber
\\
&\leq\alpha\int_{n}^{n+1}\sum_{\tau\in\omega_{1}^{-}}h(t-\tau)dt,\nonumber
\end{align}
where the first equality in \eqref{estimation} is due to the construction of $D_{1}$ in \eqref{canonical},
the second equality in \eqref{estimation} is due to the definitions of $\lambda^{1}$ and $\lambda^{0}$ in \eqref{canonical}
and finally the inequality in \eqref{estimation} is due to the fact that
$\lambda(\cdot)$ is $\alpha$-Lipschitz. Similarly, 
\begin{align}
\mathbb{E}^{\mathcal{P}}[D_{2}(n,n+1]]&\leq\mathbb{E}^{\omega_{1}^{-}}\left[\alpha\int_{n}^{n+1}\sum_{\tau\in D_{1},\tau<t}h(t-\tau)dt\right]
\\
&\leq\sum_{\tau\in\omega_{1}^{-}}\alpha^{2}\int_{n}^{n+1}\int_{0}^{t}h(t-s)h(s-\tau)dsdt.\nonumber
\end{align}
Iteratively, we have, for any $k\in\mathbb{N}$,
\begin{align}
\mathbb{E}^{\mathcal{P}}[D_{k}(n,n+1]]
\leq\sum_{\tau\in\omega_{1}^{-}}
\alpha^{k}&\int_{n}^{n+1}\int_{0}^{t_{k}}\cdots\int_{0}^{t_{2}}
h(t_{k}-t_{k-1})h(t_{k-1}-t_{k-2})\nonumber
\\
&\cdots h(t_{2}-t_{1})h(t_{1}-\tau)dt_{1}\cdots dt_{k}
=:\sum_{\tau\in\omega_{1}^{-}}K_{k}(n,\tau).\nonumber
\end{align}
Now let $K(n,\tau):=\sum_{k=1}^{\infty}K_{k}(n,\tau)$. Then,
\begin{align}
&\mathbb{E}\left\{\left[\mathbb{E}^{\omega^{-}_{1}}[N(n,n+1]]-\mathbb{E}^{\emptyset}[N(n,n+1]]\right]^{2}\right\}
\\
&\leq\mathbb{E}\left[\left(\sum_{\tau\in\omega_{1}^{-}}K(n,\tau)\right)^{2}\right]\nonumber
\\
&\leq\mathbb{E}\left[\sum_{i,j\leq 0}K(n,i)K(n,j)N[i,i+1]N[j,j+1]\right]\nonumber
\\
&=\sum_{i,j\leq 0}K(n,i)K(n,j)\mathbb{E}[N[i,i+1]N[j,j+1]]\nonumber
\\
&\leq\sum_{i,j\leq 0}K(n,i)K(n,j)\frac{1}{2}\left\{\mathbb{E}[N[i,i+1]^{2}]+\mathbb{E}[N[j,j+1]^{2}]\right\}\nonumber
\\
&=\mathbb{E}[N[0,1]^{2}]\left(\sum_{i\leq 0}K(n,i)\right)^{2}.\nonumber
\end{align}
Here, $\mathbb{E}[N[0,1]^{2}]<\infty$ by Lemma \ref{secondmoment}. Therefore, we have
\begin{align}
&\sum_{n=1}^{\infty}\left\{\mathbb{E}\left[\left(\mathbb{E}[N(n,n+1]-\mu|\mathcal{F}^{-\infty}_{0}]\right)^{2}\right]\right\}^{1/2}
\\
&\leq\sqrt{2\mathbb{E}[N[0,1]^{2}]}\sum_{n=1}^{\infty}\sum_{i=-\infty}^{0}K(n,i)\nonumber
\\
&\leq\sqrt{2\mathbb{E}[N[0,1]^{2}]}\sum_{k=1}^{\infty}
\alpha^{k}\int_{0}^{\infty}\int_{0}^{t_{k}}\cdots\int_{0}^{t_{2}}\int_{-\infty}^{0}\nonumber
\\
&h(t_{k}-t_{k-1})h(t_{k-1}-t_{k-2})\cdots h(t_{2}-t_{1})h(t_{1}-s)dsdt_{1}\cdots dt_{k}.\nonumber
\end{align}
Let $H(t):=\int_{t}^{\infty}h(s)ds$. It is easy to check that $\int_{0}^{\infty}H(t)dt=\int_{0}^{\infty}th(t)dt<\infty$ by Assumption \ref{Assumption}. 
We have
\begin{align}
&\alpha^{k}\int_{0}^{\infty}\int_{0}^{t_{k}}\cdots\int_{0}^{t_{2}}\int_{-\infty}^{0}
\\
& h(t_{k}-t_{k-1})h(t_{k-1}-t_{k-2})\cdots h(t_{2}-t_{1})h(t_{1}-s)dsdt_{1}\cdots dt_{k}\nonumber
\\
&=\alpha^{k}\int_{0}^{\infty}\int_{0}^{t_{k}}\cdots\int_{0}^{t_{2}}
h(t_{k}-t_{k-1})h(t_{k-1}-t_{k-2})\cdots h(t_{2}-t_{1})H(t_{1})dt_{1}\cdots dt_{k}\nonumber
\\
&=\alpha^{k}\int_{0}^{\infty}\cdots\int_{t_{k-2}}^{\infty}\int_{t_{k-1}}^{\infty}h(t_{k}-t_{k-1})dt_{k}h(t_{k-1}-t_{k-2})dt_{k-1}
\cdots H(t_{1})dt_{1}\nonumber
\\
&=\alpha^{k}\Vert h\Vert_{L^{1}}^{k-1}\int_{0}^{\infty}H(t_{1})dt_{1}=\alpha^{k}\Vert h\Vert_{L^{1}}^{k-1}\int_{0}^{\infty}th(t)dt.\nonumber
\end{align}
Since $\alpha\Vert h\Vert_{L^{1}}<1$, we conclude that
\begin{align}
&\sum_{n=1}^{\infty}\left\{\mathbb{E}\left[\left(\mathbb{E}[N(n,n+1]-\mu|\mathcal{F}^{-\infty}_{0}]\right)^{2}\right]\right\}^{1/2}
\\
&\leq\sum_{k=1}^{\infty}\sqrt{2\mathbb{E}[N[0,1]^{2}]}\alpha^{k}\Vert h\Vert_{L^{1}}^{k-1}\int_{0}^{\infty}th(t)dt\nonumber
\\
&=\sqrt{2\mathbb{E}[N[0,1]^{2}]}\cdot\frac{\alpha}{1-\alpha\Vert h\Vert_{L^{1}}}\cdot\int_{0}^{\infty}th(t)dt<\infty.\nonumber
\end{align}
Hence, by Theorem \ref{BTheorem}, we have
\begin{equation}
\frac{N_{[\cdot t]}-\mu[\cdot t]}{\sqrt{t}}\rightarrow\sigma B(\cdot)\quad\text{as $t\rightarrow\infty$,}
\end{equation}
where
\begin{equation}
\sigma^{2}=\mathbb{E}[(N[0,1]-\mu)^{2}]+2\sum_{j=1}^{\infty}\mathbb{E}[(N[0,1]-\mu)(N[j,j+1]-\mu)]<\infty.\label{sigmadefn}
\end{equation}
By Lemma \ref{positivesigma}, $\sigma>0$. Now, finally, for any $\epsilon>0$, for $t$ sufficiently large,
\begin{align}
&\mathbb{P}\left(\sup_{0\leq s\leq 1}\left|\frac{N_{[st]}-\mu[st]}{\sqrt{t}}
-\frac{N_{st}-\mu st}{\sqrt{t}}\right|>\epsilon\right)
\\
&=\mathbb{P}\left(\sup_{0\leq s\leq 1}\left|(N_{[st]}-N_{st})+\mu(st-[st])\right|>\epsilon\sqrt{t}\right)\nonumber
\\
&\leq\mathbb{P}\left(\sup_{0\leq s\leq 1}\left|N_{[st]}-N_{st}\right|+\mu>\epsilon\sqrt{t}\right)\nonumber
\\
&\leq\mathbb{P}\left(\max_{0\leq k\leq [t],k\in\mathbb{Z}}N[k,k+1]>\epsilon\sqrt{t}-\mu\right)\nonumber
\\
&\leq([t]+1)\mathbb{P}(N[0,1]>\epsilon\sqrt{t}-\mu)\nonumber
\\
&\leq\frac{[t]+1}{(\epsilon\sqrt{t}-\mu)^{2}}\int_{N[0,1]>\epsilon\sqrt{t}-\mu}N[0,1]^{2}d\mathbb{P}\rightarrow 0,\nonumber
\end{align}
as $t\rightarrow\infty$ by Lemma \ref{secondmoment}. Hence, we conclude that $\frac{N_{\cdot t}-\cdot\mu t}{\sqrt{t}}\rightarrow\sigma B(\cdot)$
as $t\rightarrow\infty$.
\end{proof}

The following Lemma \ref{midstep} is used to prove Lemma \ref{secondmoment}.

\begin{lemma}\label{midstep}
There exists some $\theta>0$ such that 
$\sup_{t\geq 0}\mathbb{E}^{\emptyset}\left[e^{\int_{0}^{t}\theta h(t-s)N(ds)}\right]<\infty$.
\end{lemma}

\begin{proof}
Notice first that for any bounded deterministic function $f(\cdot)$,
\begin{equation}
\exp\left\{\int_{0}^{t}f(s)N(ds)-\int_{0}^{t}(e^{f(s)}-1)\lambda(s)ds\right\}
\end{equation}
is a martingale. Therefore, using the Lipschitz assumption of $\lambda(\cdot)$, 
i.e. $\lambda(z)\leq\lambda(0)+\alpha z$ and applying H\"{o}lder's inequality, for $\frac{1}{p}+\frac{1}{q}=1$, we have
\begin{align}
&\mathbb{E}^{\emptyset}\left[e^{\int_{0}^{t}\theta h(t-s)N(ds)}\right]
\\
&=\mathbb{E}^{\emptyset}
\left[e^{\int_{0}^{t}\theta h(t-s)N(ds)-\frac{1}{p}\int_{0}^{t}(e^{p\theta h(t-s)}-1)\lambda(s)ds
+\frac{1}{p}\int_{0}^{t}(e^{p\theta h(t-s)}-1)\lambda(s)ds}\right]\nonumber
\\
&\leq\mathbb{E}^{\emptyset}\left[e^{\frac{q}{p}\int_{0}^{t}(e^{p\theta h(t-s)}-1)\lambda(s)ds}\right]^{\frac{1}{q}}\nonumber
\\
&\leq\mathbb{E}^{\emptyset}\left[e^{\frac{q}{p}\int_{0}^{t}(e^{p\theta h(t-s)}-1)(\lambda(0)+\alpha\int_{0}^{s}h(s-u)N(du))ds}
\right]^{\frac{1}{q}}\nonumber
\\
&\leq\mathbb{E}^{\emptyset}\left[e^{\int_{0}^{t}\frac{q}{p}(e^{p\theta h(t-s)}-1)\alpha\int_{0}^{s}h(s-u)N(du)ds}\right]^{\frac{1}{q}}
\cdot e^{\frac{1}{p}\int_{0}^{\infty}(e^{p\theta h(s)}-1)\lambda(0)ds}.\nonumber
\end{align}
Let $C(t)=\int_{0}^{t}\frac{q}{p}(e^{p\theta h(t-s)}-1)\alpha ds$. Then, for any $t\in[0,T]$,
\begin{align}\label{Jensen}
&\mathbb{E}^{\emptyset}\left[e^{\int_{0}^{t}\frac{q}{p}(e^{p\theta h(t-s)}-1)\alpha\int_{0}^{s}h(s-u)N(du)ds}\right]
\\
&=\mathbb{E}^{\emptyset}\left[e^{\frac{1}{C(t)}\int_{0}^{t}\frac{q}{p}(e^{p\theta h(t-s)}-1)\alpha C(t)\int_{0}^{s}h(s-u)N(du)ds}\right]
\nonumber
\\
&\leq\mathbb{E}^{\emptyset}\left[\frac{1}{C(t)}\int_{0}^{t}\frac{q}{p}(e^{p\theta h(t-s)}-1)\alpha
e^{C(t)\int_{0}^{s}h(s-u)N(du)}ds\right]\nonumber
\\
&\leq\sup_{0\leq s\leq T}\mathbb{E}^{\emptyset}\left[e^{C(\infty)\int_{0}^{s}h(s-u)N(du)}\right],\nonumber
\end{align}
where in the first inequality in \eqref{Jensen}, we used the Jensen's inequality since $x\mapsto e^{x}$ is convex
and $\frac{1}{C(t)}\int_{0}^{t}\frac{q}{p}(e^{p\theta h(t-s)}-1)\alpha ds=1$, and in the second inequality
in \eqref{Jensen}, we used the fact that $C(t)\leq C(\infty)$ 
and again $\frac{1}{C(t)}\int_{0}^{t}\frac{q}{p}(e^{p\theta h(t-s)}-1)\alpha ds=1$.
Now choose $q>1$ so small that $q\alpha\Vert h\Vert_{L^{1}}<1$. Once $p$ and $q$ are fixed, choose $\theta>0$ 
so small that
\begin{equation}
C(\infty)=\int_{0}^{\infty}\frac{q}{p}(e^{p\theta h(s)}-1)\alpha ds<\theta.
\end{equation}
This implies that for any $t\in[0,T]$,
\begin{equation}
\mathbb{E}^{\emptyset}\left[e^{\int_{0}^{t}\theta h(t-s)N(ds)}\right]
\leq\sup_{0\leq s\leq T}\mathbb{E}^{\emptyset}\left[e^{\theta\int_{0}^{s}h(s-u)N(du)}\right]^{\frac{1}{q}}
\cdot e^{\frac{1}{p}\int_{0}^{\infty}(e^{p\theta h(s)}-1)\lambda(0)ds}.
\end{equation}
Hence, we conclude that for any $T>0$,
\begin{equation}
\sup_{0\leq t\leq T}\mathbb{E}^{\emptyset}\left[e^{\theta\int_{0}^{t}h(t-s)N(ds)}\right]
\leq e^{\int_{0}^{\infty}(e^{p\theta h(s)}-1)\lambda(0)ds}<\infty.
\end{equation}
\end{proof}

\begin{lemma}\label{secondmoment}
There exists some $\theta>0$ such that $\mathbb{E}[e^{\theta N[0,1]}]<\infty$. Hence $\mathbb{E}[N[0,1]^{2}]<\infty$.
\end{lemma}

\begin{proof}
By Assumption \ref{Assumption}, $h(\cdot)$ is positive and decreasing. Thus, $\delta=\inf_{t\in[0,1]}h(t)>0$. Hence,
\begin{equation}
\mathbb{E}^{\emptyset}[e^{\theta N[t-1,t]}]\leq\mathbb{E}^{\emptyset}[e^{\frac{\theta}{\delta}\int_{0}^{t}h(t-s)N(ds)}].
\end{equation}
By Lemma \ref{midstep}, we can choose $\theta>0$ so small that
\begin{equation}
\limsup_{t\rightarrow\infty}\mathbb{E}^{\emptyset}[e^{\theta N[t-1,t]}]<\infty.
\end{equation}
Finally, $\mathbb{E}[e^{\theta N[0,1]}]\leq\liminf_{t\rightarrow\infty}\mathbb{E}^{\emptyset}[e^{\theta N[t-1,t]}]<\infty$.
\end{proof}

It is intuitively clear that $\sigma>0$. But still we need a proof.

\begin{lemma}\label{positivesigma}
$\sigma>0$, where $\sigma$ is defined in \eqref{sigmadefn}.
\end{lemma}

\begin{proof}
Let $\eta_{n}=\sum_{j=n}^{\infty}\mathbb{E}[N(j,j+1]-\mu|\mathcal{F}^{-\infty}_{n+1}]$, where $\mu=\mathbb{E}[N[0,1]]$.
$\eta_{n}$ is well defined because we proved \eqref{finitesum}. To see this, notice that
\begin{align}
\Vert\eta_{n}\Vert_{2}&=\bigg\Vert\sum_{j=n}^{\infty}\mathbb{E}[N(j,j+1]-\mu|\mathcal{F}^{-\infty}_{n+1}]\bigg\Vert_{2}
\\
&\leq\sum_{j=n}^{\infty}\Vert\mathbb{E}[N(j,j+1]-\mu|\mathcal{F}^{-\infty}_{n+1}]\Vert_{2}<\infty,\nonumber
\end{align}
by \eqref{finitesum}. Also, it is easy to check that
\begin{align}
&\mathbb{E}[\eta_{n+1}-\eta_{n}+N(n,n+1]-\mu|\mathcal{F}^{-\infty}_{n+1}]
\\
&=\mathbb{E}\left[\sum_{j=n+1}^{\infty}\mathbb{E}[N(j,j+1]-\mu|\mathcal{F}^{-\infty}_{n+2}]\bigg|\mathcal{F}^{-\infty}_{n+1}\right]\nonumber
\\
&-\mathbb{E}\left[\sum_{j=n}^{\infty}\mathbb{E}[N(j,j+1]-\mu|\mathcal{F}^{-\infty}_{n+1}]\bigg|\mathcal{F}^{-\infty}_{n+1}\right]+N(n,n+1]-\mu\nonumber
\\
&=\sum_{j=n+1}^{\infty}\mathbb{E}[N(j,j+1]-\mu|\mathcal{F}^{-\infty}_{n+1}]
-\sum_{j=n+1}^{\infty}\mathbb{E}[N(j,j+1]-\mu|\mathcal{F}^{-\infty}_{n+1}]\nonumber
\\
&-N(n,n+1]+\mu+N(n,n+1]-\mu=0.\nonumber
\end{align}
Let $Y_{n}=\eta_{n-1}-\eta_{n-2}+N(n-2,n-1]-\mu$. This is an ergodic, stationary sequence such that 
$\mathbb{E}[Y_{n}|\mathcal{F}^{-\infty}_{n-1}]=0$. By \eqref{finitesum}, $\mathbb{E}[Y_{n}^{2}]<\infty$ and 
by Theorem \ref{BTheoremII}, $S'_{[n\cdot]}/\sqrt{n}\rightarrow\sigma'B(\cdot)$, where $S'_{n}=\sum_{j=1}^{n}Y_{j}$.
It is clear that $\sigma=\sigma'<\infty$ 
since for any $\epsilon>0$,
\begin{align}
&\mathbb{P}\left(\max_{1\leq k\leq [n],k\in\mathbb{Z}}\frac{1}{\sqrt{n}}\sum_{j=1}^{k}(\eta_{j-1}-\eta_{j-2})>\epsilon\right)
\\
&=\mathbb{P}\left(\max_{1\leq k\leq [n],k\in\mathbb{Z}}(\eta_{k-1}-\eta_{-1})>\epsilon\sqrt{n}\right)\nonumber
\\
&\leq\mathbb{P}\left(\left\{\max_{1\leq k\leq [n],k\in\mathbb{Z}}|\eta_{k-1}|>\frac{\epsilon\sqrt{n}}{2}\right\}
\bigcup\left\{|\eta_{-1}|>\frac{\epsilon\sqrt{n}}{2}\right\}\right)\nonumber
\\
&\leq\sum_{k=1}^{[n]}\mathbb{P}\left(|\eta_{k-1}|>\frac{\epsilon\sqrt{n}}{2}\right)
+\mathbb{P}\left(|\eta_{-1}|>\frac{\epsilon\sqrt{n}}{2}\right)\nonumber
\\
&=([n]+1)\mathbb{P}\left(|\eta_{-1}|>\frac{\epsilon\sqrt{n}}{2}\right)\nonumber
\\
&\leq\frac{4([n]+1)}{\epsilon^{2}n}\int_{|\eta_{-1}|>\frac{\epsilon\sqrt{n}}{2}}|\eta_{-1}|^{2}d\mathbb{P}\rightarrow 0,\nonumber
\end{align}
as $n\rightarrow\infty$, where we used the stationarity of $\mathbb{P}$, Chebychev's inequality and \eqref{finitesum}.

Now, it becomes clear that
\begin{align}
\sigma^{2}&=(\sigma')^{2}=\mathbb{E}[Y_{1}^{2}]
\\
&=\mathbb{E}\left(\eta_{0}-\eta_{-1}+N(-1,0]-\mu\right)^{2}\nonumber
\\
&=\mathbb{E}\left(\sum_{j=0}^{\infty}\mathbb{E}[N(j,j+1]-\mu|\mathcal{F}^{-\infty}_{1}]
-\sum_{j=0}^{\infty}\mathbb{E}[N(j,j+1]-\mu|\mathcal{F}^{-\infty}_{0}]\right)^{2}.\nonumber
\end{align}
Consider $D=\{\omega:\omega^{-}\neq\emptyset, \omega(0,1]=\emptyset\}$. Notice that
$\mathbb{P}(\omega^{-}=\emptyset)=0$. By Jensen's inequality and Assumption \ref{Assumption}, we have
\begin{align}
\mathbb{P}(D)&=\int\mathbb{P}^{\omega^{-}}(N(0,1]=0)\mathbb{P}(d\omega^{-})
\\
&=\mathbb{E}\left[e^{-\int_{0}^{1}\lambda(\sum_{\tau\in\omega^{-}}h(t-\tau))dt}\right]\nonumber
\\
&\geq\exp\left\{-\mathbb{E}\int_{0}^{1}\lambda\left(\sum_{\tau\in\omega^{-}}h(t-\tau)\right)dt\right\}\nonumber
\\
&\geq\exp\left\{-\lambda(0)-\alpha\mathbb{E}\int_{0}^{1}\sum_{\tau\in\omega^{-}}h(t-\tau)dt\right\}\nonumber
\\
&\geq\exp\left\{-\lambda(0)-\alpha\mathbb{E}[N[0,1]]\cdot\Vert h\Vert_{L^{1}}\right\}>0.\nonumber
\end{align}
It is clear that given the event $D$, 
\begin{equation}
\sum_{j=0}^{\infty}\mathbb{E}[N(j,j+1]-\mu|\mathcal{F}^{-\infty}_{1}]
<\sum_{j=0}^{\infty}\mathbb{E}[N(j,j+1]-\mu|\mathcal{F}^{-\infty}_{0}]. 
\end{equation}
Therefore,
\begin{equation}
\mathbb{P}\left(\sum_{j=0}^{\infty}\mathbb{E}[N(j,j+1]-\mu|\mathcal{F}^{-\infty}_{1}]
\neq\sum_{j=0}^{\infty}\mathbb{E}[N(j,j+1]-\mu|\mathcal{F}^{-\infty}_{0}]\right)>0,
\end{equation}
which implies that $\sigma>0$.
\end{proof}

\begin{proof}[Proof of Theorem \ref{LIL}]
By Heyde and Scott \cite{Heyde}, the Strassen's invariance principle holds if we have \eqref{finitesum} and $\sigma>0$.
\end{proof}

%% file: chap3.tex
\chapter{Process-Level Large Deviations for Nonlinear Hawkes Processes\label{chap:three}}

\section{Main Results}

In this chaper, we prove a process-level, i.e. level-3 large deviation principle
for nonlinear Hawkes processes. As a corollary, a level-1 large deviation principle is obtained by a contraction principle.

Let us recall that $N$ is a nonlinear Hawkes process with intensity
\begin{equation}
\lambda_{t}:=\lambda\left(\int_{(-\infty,t)}h(t-s)N(ds)\right).
\end{equation}

Throughout this chapter, we assume that
\begin{itemize}
\item 
The exciting function $h(t)$ is positive, continuous and decreasing for $t\geq 0$ and $h(t)=0$ for any $t<0$. 
We also assume that $\int_{0}^{\infty}h(t)dt<\infty$.

\item 
The rate function $\lambda(\cdot):[0,\infty)\rightarrow\mathbb{R}^{+}$ is increasing and $\lim_{z\rightarrow\infty}\frac{\lambda(z)}{z}=0$. 
We also assume that $\lambda(\cdot)$ is Lipschitz with constant $\alpha>0$, i.e. $|\lambda(x)-\lambda(y)|\leq\alpha|x-y|$ for any $x,y\geq 0$.
\end{itemize}

Let $\Omega$ be the set of countable, locally finite subsets of $\mathbb{R}$ and for any $\omega\in\Omega$ 
and $A\subseteq\mathbb{R}$, write $\omega(A):=\omega\cap A$. For any $t\in\mathbb{R}$, we write $\omega(t)=\omega(\{t\})$. 
Let $N(A)=\#|\omega\cap A|$ denote the number of points in the set $A$ for any $A\subset\mathbb{R}$. 
We also use the notation $N_{t}$ to denote $N[0,t]$, the number of points up to time $t$, starting from time $0$. 
We define the shift operator $\theta_{t}$ by $\theta_{t}(\omega)(s)=\omega(t+s)$. We equip the sample space $\Omega$ 
with the topology in which the convergence $\omega_{n}\rightarrow\omega$ as $n\rightarrow\infty$ is defined by
\begin{equation}
\sum_{\tau\in\omega_{n}}f(\tau)\rightarrow\sum_{\tau\in\omega}f(\tau),
\end{equation}
for any continuous $f$ with compact support.

This topology is equivalent to the vague topology for random measures, for which, 
see for example Grandell \cite{Grandell}. One can equip the space of 
locally finite random measures with the vague topology. The subspace of integer valued random measures is 
then the space of point processes. A simple point processes is a point process without multiple jumps. 
The space of point processes is closed. But the space of simple point processes is not closed. 

Denote $\mathcal{F}^{s}_{t}=\sigma(\omega[s,t])$ for any $s<t$, i.e. the $\sigma$-algebra generated by 
all the possible configurations of points in the interval $[s,t]$. Denote $\mathcal{M}(\Omega)$ the space of 
probability measures on $\Omega$. We also define $\mathcal{M}_{S}(\Omega)$ as the space of simple point processes 
that are invariant with respect to $\theta_{t}$ with bounded first moment, i.e. for any $Q\in\mathcal{M}_{S}(\Omega)$,
$\mathbb{E}^{Q}[N[0,1]]<\infty$.
Define $\mathcal{M}_{E}(\Omega)$ as the set of ergodic simple point processes in $\mathcal{M}_{S}(\Omega)$. 
We define the topology of $\mathcal{M}_{S}(\Omega)$ as follows.
For a sequence $Q_{n}$ in $\mathcal{M}_{S}(\Omega)$ and $Q\in\mathcal{M}_{S}(\Omega)$, we say
$Q_{n}\rightarrow Q$ as $n\rightarrow\infty$ if and only if
\begin{equation}
\int fdQ_{n}\rightarrow\int fdQ,
\end{equation}
as $n\rightarrow\infty$ for any continuous and bounded $f$ and
\begin{equation}
\int N[0,1](\omega)Q_{n}(d\omega)\rightarrow\int N[0,1](\omega)Q(d\omega),
\end{equation}
as $n\rightarrow\infty$. In other words, the topology is the weak topology strengthened by the convergence of the first moment of $N[0,1]$.
For any $Q_{1}$, $Q_{2}$ in $\mathcal{M}_{S}(\Omega)$, one can define the metric $d(\cdot,\cdot)$ by
\begin{equation}
d(Q_{1},Q_{2})=d_{p}(Q_{1},Q_{2})+\left|\mathbb{E}^{Q_{1}}[N[0,1]]-\mathbb{E}^{Q_{2}}[N[0,1]]\right|,
\end{equation}
where $d_{p}(\cdot,\cdot)$ is the usual Prokhorov metric.
Because this is an unusual topology, the compactness is different
from that in the usual weak topology; later, when we prove the exponential tightness, we need to take some extra care.
See Lemma \ref{tightness} and (iii) of Lemma \ref{finalestimate}.

We denote by $C(\Omega)$ the set of real-valued continous functions on $\Omega$. We similarly define $C(\Omega\times\mathbb{R})$.
We also denote by $\mathcal{B}(\mathcal{F}^{-\infty}_{t})$ the set of all bounded $\mathcal{F}^{-\infty}_{t}$
progressively measurable and $\mathcal{F}^{-\infty}_{t}$ predictable functions.

Before we proceed, recall that a sequence $(P_{n})_{n\in\mathbb{N}}$ of probability measures on a topological space $X$ 
satisfies the large deviation principle (LDP) with rate function $I:X\rightarrow\mathbb{R}$ if $I$ is non-negative, 
lower semicontinuous and for any measurable set $A$,
\begin{equation}
-\inf_{x\in A^{o}}I(x)\leq\liminf_{n\rightarrow\infty}\frac{1}{n}\log P_{n}(A)
\leq\limsup_{n\rightarrow\infty}\frac{1}{n}\log P_{n}(A)\leq-\inf_{x\in\overline{A}}I(x).
\end{equation}
Here, $A^{o}$ is the interior of $A$ and $\overline{A}$ is its closure. See Dembo and Zeitouni \cite{Dembo} or Varadhan \cite{Varadhan} 
for general background regarding large deviations and their applications. 
Also Varadhan \cite{VaradhanII} has an excellent survey article on this subject.

In the pioneering work by Donsker and Varadhan \cite{Donsker}, they obtained a level-3 large deviation result for 
certain stationary Markov processes. 

We would like to prove the large deviation principle for nonlinear Hawkes processes by proving a process-level, also known as 
level-3 large deviation principle first. We can then use the contraction principle to obtain the level-1 large deviation principle for $(N_{t}/t\in\cdot)$.

Let us define the empirical measure for the process as
\begin{equation}
R_{t,\omega}(A)=\frac{1}{t}\int_{0}^{t}\chi_{A}(\theta_{s}\omega_{t})ds,
\end{equation}
for any $A$, where $\omega_{t}(s)=\omega(s)$ for $0\leq s\leq t$ and $\omega_{t}(s+t)=\omega_{t}(s)$ for any $s$. 
Donsker and Varadhan \cite{Donsker} proved that in the case when $\Omega$ is a space of c\`{a}dl\`{a}g functions $\omega(\cdot)$ 
on $-\infty<t<\infty$ endowed with Skorohod topology and taking values in a Polish space $X$, under certain conditions, 
$P^{0,x}(R_{t,\omega}\in\cdot)$ satisfies a large deviation principle, where $P^{0,x}$ is a Markov process on $\Omega^{0}_{\infty}$ 
with initial value $x\in X$. The rate function $H(Q)$ is some entropy function.

Let $h(\alpha,\beta)_{\Sigma}$ be the relative entropy of $\alpha$ with respect to $\beta$ restricted to the $\sigma$-algebra $\Sigma$. 
For any $Q\in\mathcal{M}_{S}(\Omega)$, let $Q^{\omega^{-}}$ be the regular conditional probability distribution of $Q$. 
Similarly we define $P^{\omega^{-}}$.

Let us define the entropy function $H(Q)$ as
\begin{equation}
H(Q)=
\mathbb{E}^{Q}[h(Q^{\omega^{-}},P^{\omega^{-}})_{\mathcal{F}^{0}_{1}}].
\end{equation}
Notice that $P^{\omega^{-}}$ describes the Hawkes process conditional on the past history $\omega^{-}$. 
It has rate $\lambda(\sum_{\tau\in\omega[0,s)\cup\omega^{-}}h(s-\tau))$ at time $0\leq s\leq 1$, 
which is well defined for almost every $\omega^{-}$ under $Q$ if $\mathbb{E}^{Q}[N[0,1]]<\infty$ 
since $\mathbb{E}^{Q}[\sum_{\tau\in\omega^{-}}h(-\tau)]=\Vert h\Vert_{L^{1}}\mathbb{E}^{Q}[N[0,1]]<\infty$ 
implies $\sum_{\tau\in\omega^{-}}h(s-\tau)\leq\sum_{\tau\in\omega^{-}}h(-\tau)<\infty$ for all $0\leq s\leq 1$.

When $H(Q)<\infty$, $h(Q^{\omega^{-}},P^{\omega^{-}})<\infty$ for a.e. $\omega^{-}$ under $Q$, which implies that 
$Q^{\omega^{-}}\ll P^{\omega^{-}}$ on $\mathcal{F}^{0}_{1}$. By the theory of absolute continuity of point processes, 
see for example Chapter 19 of Lipster and Shiryaev \cite{Lipster} or Chapter 13 of Daley and Vere-Jones \cite{Daley}, 
the compensator of $Q^{\omega^{-}}$ is absolutely continuous, i.e. it has some density $\hat{\lambda}$ say, such that by the Girsanov formula,
\begin{align}
H(Q)&=\int_{\Omega^{-}}\int\left[\int_{0}^{1}\left(\lambda-\hat{\lambda}\right)ds+\int_{0}^{1}
\log(\hat{\lambda}/\lambda)dN_{s}\right]dQ^{\omega^{-}}Q(d\omega^{-})\label{HFunction}
\\
&=\int_{\Omega}\left[\int_{0}^{1}\lambda(\omega,s)-\hat{\lambda}(\omega,s)
+\log\left(\frac{\hat{\lambda}(\omega,s)}{\lambda(\omega,s)}\right)\hat{\lambda}ds\right]Q(d\omega)\nonumber,
\end{align}
where $\lambda=\lambda\left(\sum_{\tau\in\omega[0,s)\cup\omega^{-}}h(s-\tau)\right)$. Both $\lambda$ and $\hat{\lambda}$ 
are $\mathcal{F}^{-\infty}_{s}$-predictable for $0\leq s\leq 1$. For the equality in \eqref{HFunction}, we 
used the fact that $N_{t}-\int_{0}^{t}\hat{\lambda}(\omega,s)ds$ is a martingale under $Q$ and
for any $f(\omega,s)$ which is bounded, $\mathcal{F}^{-\infty}_{s}$ progressively measurable and predictable, we have
\begin{equation}
\int_{\Omega}\int_{0}^{1}f(\omega,s)dN_{s}Q(d\omega)=\int_{\Omega}\int_{0}^{1}f(\omega,s)\hat{\lambda}(\omega,s)dsQ(d\omega).
\end{equation}
We will use the above fact repeatedly in this chapter.

The following theorem is the main result of this chapter.

\begin{theorem}
For any open set $G\subset\mathcal{M}_{S}(\Omega)$,
\begin{equation}
\liminf_{t\rightarrow\infty}\frac{1}{t}\log P\left(R_{t,\omega}\in G\right)\geq -\inf_{Q\in G}H(Q),
\end{equation}
and for any closed set $C\subset\mathcal{M}_{S}(\Omega)$,
\begin{equation}
\limsup_{t\rightarrow\infty}\frac{1}{t}\log P\left(R_{t,\omega}\in C\right)\leq -\inf_{Q\in C}H(Q).
\end{equation}
\end{theorem}

We will prove the lower bound in Section \ref{lowerbound}, the upper bound in Section \ref{upperbound}, and the 
superexponential estimates that are needed in the proof of the upper bound in Section \ref{superexpestimates}.

Once we establish the level-3 large deviation result, we can obtain the large deviation principle for $(N_{t}/t\in\cdot)$ 
directly by using the contraction principle.

\begin{theorem}
$(N_{t}/t\in\cdot)$ satisfies a large deviation principle with the rate function $I(\cdot)$ given by
\begin{equation}
I(x)=\inf_{Q\in\mathcal{M}_{S}(\Omega),\mathbb{E}^{Q}[N[0,1]]=x}H(Q).
\end{equation}
\end{theorem}

\begin{proof}
Since $Q\mapsto\mathbb{E}^{Q}[N[0,1]]$ is continuous, $\int_{\Omega}N[0,1]dR_{t,\omega}$ satisfies a large deviation principle 
with the rate function $I(\cdot)$ by the contraction principle. (For a discussion on contraction principle, 
see for example Varadhan \cite{Varadhan}.)
\begin{align}
\int_{\Omega}N[0,1]dR_{t,\omega}&=\frac{1}{t}\int_{0}^{t}N[0,1](\theta_{s}\omega_{t})ds
\\
&=\frac{1}{t}\int_{0}^{t-1}N[s,s+1](\omega)ds+\frac{1}{t}\int_{t-1}^{t}N[s,s+1](\omega_{t})ds.\nonumber
\end{align}
Notice that
\begin{equation}
0\leq\frac{1}{t}\int_{t-1}^{t}N[s,s+1](\omega_{t})ds\leq\frac{1}{t}(N[t-1,t](\omega)+N[0,1](\omega)),
\end{equation}
and
\begin{equation}
\frac{1}{t}\int_{0}^{t-1}N[s,s+1](\omega)ds=\frac{1}{t}\left[\int_{t-1}^{t}N_{s}(\omega)ds-\int_{0}^{1}N_{s}(\omega)ds\right]
\leq\frac{N_{t}}{t},
\end{equation}
and 
\begin{equation}
\frac{1}{t}\int_{0}^{t-1}N[s,s+1](\omega)ds\geq\frac{N_{t-1}-N_{1}}{t}=\frac{N_{t}}{t}-\frac{N[t-1,t]+N_{1}}{t}. 
\end{equation}
Hence,
\begin{equation}
\frac{N_{t}}{t}-\frac{N[t-1,t]+N_{1}}{t}\leq\int_{\Omega}N[0,1]dR_{t,\omega}\leq\frac{N_{t}}{t}+\frac{N[t-1,t]+N_{1}}{t}.
\end{equation}
For the lower bound, for any open ball $B_{\epsilon}(x)$ centered at $x$ with radius $\epsilon>0$,
\begin{align}
P\left(\frac{N_{t}}{t}\in B_{\epsilon}(x)\right)&\geq P\left(\int_{\Omega}N[0,1]dR_{t,\omega}\in B_{\epsilon/2}(x)\right)
\\
&-P\left(\frac{N[t-1,t]}{t}\geq\frac{\epsilon}{4}\right)-P\left(\frac{N_{1}}{t}\geq\frac{\epsilon}{4}\right).\nonumber
\end{align}
For the upper bound, for any closed set $C$ and $C^{\epsilon}=\bigcup_{x\in C}\overline{B_{\epsilon}(x)}$,
\begin{align}
P\left(\frac{N_{t}}{t}\in C\right)&\leq P\left(\int_{\Omega}N[0,1]dR_{t,\omega}\in C^{\epsilon}\right)
\\
&+P\left(\frac{N[t-1,t]}{t}\geq\frac{\epsilon}{4}\right)+P\left(\frac{N_{1}}{t}\geq\frac{\epsilon}{4}\right).\nonumber
\end{align}
Finally, by Lemma \ref{indfinite}, we have the following superexponential estimates
\begin{equation}
\limsup_{t\rightarrow\infty}\frac{1}{t}\log P\left(\frac{N[t-1,t]}{t}\geq\frac{\epsilon}{4}\right)
=\limsup_{t\rightarrow\infty}\frac{1}{t}\log P\left(\frac{N_{1}}{t}\geq\frac{\epsilon}{4}\right)=-\infty.
\end{equation}
Hence, for the lower bound, we have
\begin{equation}
\liminf_{t\rightarrow\infty}\frac{1}{t}\log P\left(\frac{N_{t}}{t}\in B_{\epsilon}(x)\right)\geq -I(x),
\end{equation}
and for the upper bound, we have
\begin{equation}
\limsup_{t\rightarrow\infty}\frac{1}{t}\log P\left(\frac{N_{t}}{t}\in C\right)\leq-\inf_{x\in C^{\epsilon}}I(x),
\end{equation}
which holds for any $\epsilon>0$. Letting $\epsilon\downarrow 0$, we get the desired result.
\end{proof}

\section{Lower Bound}\label{lowerbound}

\begin{lemma}\label{positivity}
For any $\lambda,\hat{\lambda}\geq 0$, $\lambda-\hat{\lambda}+\hat{\lambda}\log(\hat{\lambda}/\lambda)\geq 0$.
\end{lemma}

\begin{proof}
Write $\lambda-\hat{\lambda}+\hat{\lambda}\log(\hat{\lambda}/\lambda)
=\hat{\lambda}\left[(\lambda/\hat{\lambda})-1-\log(\lambda/\hat{\lambda})\right]$. 
Thus, it is sufficient to show that $F(x)=x-1-\log x\geq 0$ for any $x\geq 0$. 
Note that $F(0)=F(\infty)=0$ and $F'(x)=1-\frac{1}{x}<0$ when $0<x<1$ and $F'(x)>0$ 
when $x>1$ and finally $F(1)=0$. Hence $F(x)\geq 0$ for any $x\geq 0$.
\end{proof}

\begin{lemma}\label{finitemean}
Assume $H(Q)<\infty$. Then,
\begin{equation}
\mathbb{E}^{Q}[N[0,1]]\leq C_{1}+C_{2}H(Q),
\end{equation}
where $C_{1},C_{2}>0$ are some constants independent of $Q$.
\end{lemma}

\begin{proof}
If $H(Q)<\infty$, then $h(Q^{\omega^{-}},P^{\omega^{-}})_{\mathcal{F}^{0}_{1}}<\infty$ for a.e. $\omega^{-}$ under $Q$, 
which implies that $Q^{\omega^{-}}\ll P^{\omega^{-}}$ and thus $\hat{A}_{t}\ll A_{t}$, where $\hat{A}_{t}$ and $A_{t}$ 
are the compensators of $N_{t}$ under $Q^{\omega^{-}}$ and $P^{\omega^{-}}$ respectively. 
(For the theory of absolute continuity of point processes and Girsanov formula, 
see for example Lipster and Shiryaev \cite{Lipster} or Daley and Vere-Jones \cite{Daley}.) 
Since $A_{t}=\int_{0}^{t}\lambda(\omega,s)ds$, we have $\hat{A}_{t}=\int_{0}^{t}\hat{\lambda}(\omega,s)ds$ for some $\hat{\lambda}$. 
By the Girsanov formula,
\begin{equation}
H(Q)=\mathbb{E}^{Q}\left[\int_{0}^{1}\lambda-\hat{\lambda}+\log\left(\hat{\lambda}/\lambda\right)\hat{\lambda}ds\right].
\end{equation}
Notice that $\mathbb{E}^{Q}[N[0,1]]=\int\int_{0}^{1}\hat{\lambda}dsdQ$.
\begin{align}
\int\int_{0}^{1}\lambda dsdQ&\leq\epsilon\int\int_{0}^{1}\sum_{\tau<s}h(s-\tau)dsdQ+C_{\epsilon}
\\
&\leq\epsilon\int h(0)N[0,1]dQ+\epsilon\int\sum_{\tau<0}h(-\tau)dQ+C_{\epsilon}\nonumber
\\
&=\epsilon(h(0)+\Vert h\Vert_{L^{1}})\mathbb{E}^{Q}[N[0,1]]+C_{\epsilon}\nonumber
\\
&=\epsilon(h(0)+\Vert h\Vert_{L^{1}})\int\int_{0}^{1}\hat{\lambda}dsdQ+C_{\epsilon}.\nonumber
\end{align}
Therefore, we have
\begin{equation}
\int\int_{0}^{1}\hat{\lambda}\cdot 1_{\hat{\lambda}<K\lambda}dsdQ\leq K\epsilon(h(0)+\Vert h\Vert_{L^{1}})\int\int_{0}^{1}\hat{\lambda}dsdQ+KC_{\epsilon}.
\end{equation}
On the other hand, by Lemma \ref{positivity},
\begin{align}
H(Q)&\geq\int\int_{0}^{1}\left[\lambda-\hat{\lambda}+\hat{\lambda}\log(\hat{\lambda}/\lambda)\right]\cdot 1_{\hat{\lambda}\geq K\lambda}dsdQ
\\
&\geq(\log K-1)\int\int_{0}^{1}\hat{\lambda}\cdot 1_{\hat{\lambda}\geq K\lambda}dsdQ.\nonumber
\end{align}
Thus,
\begin{equation}
\int\int_{0}^{1}\hat{\lambda}dsdQ\leq K\epsilon(h(0)+\Vert h\Vert_{L^{1}})\int\int_{0}^{1}\hat{\lambda}dsdQ+KC_{\epsilon}+\frac{H(Q)}{\log K-1}.
\end{equation}
Choosing $K>e$ and $\epsilon<\frac{1}{K(h(0)+\Vert h\Vert_{L^{1}})}$, we get
\begin{equation}
\mathbb{E}^{Q}[N[0,1]]\leq\frac{KC_{\epsilon}}{1-K\epsilon(h(0)+\Vert h\Vert_{L^{1}})}+\frac{H(Q)}{(\log K-1)K\epsilon(h(0)+\Vert h\Vert_{L^{1}})}.
\end{equation}
\end{proof}

\begin{lemma}\label{variational}
We have the following alternative expression for $H(Q)$.
\begin{equation}
H(Q)=\sup_{f(\omega,s)\in\mathcal{B}(\mathcal{F}^{-\infty}_{s})\cap C(\Omega\times\mathbb{R}),0\leq s\leq 1}
\mathbb{E}^{Q}\left[\int_{0}^{1}\lambda(1-e^{f})ds+\int_{0}^{1}fdN_{s}\right].
\end{equation}
\end{lemma}

\begin{proof}
$\mathbb{E}^{Q}[N[0,1]]<\infty$ implies that $\mathbb{E}^{Q^{\omega^{-}}}[N[0,1]]<\infty$ for almost every $\omega^{-}$ under $Q$, 
also $\sum_{\tau\in\omega^{-}}h(-\tau)<\infty$ since $\mathbb{E}^{Q}[\sum_{\tau\in\omega^{-}}h(-\tau)]
=\Vert h\Vert_{L^{1}}\mathbb{E}^{Q}[N[0,1]]<\infty$. Thus,
\begin{align}
\mathbb{E}^{P^{\omega^{-}}}[N[0,1]]&=\mathbb{E}^{P^{\omega^{-}}}\left[\int_{0}^{1}\lambda\left(\sum_{\tau\in\omega[0,s)\cup\omega^{-}}h(s-\tau)\right)ds\right]
\\
&\leq C_{\epsilon}+\epsilon h(0)\mathbb{E}^{P^{\omega^{-}}}[N[0,1]]+\epsilon\sum_{\tau\in\omega^{-}}h(-\tau)<\infty,\nonumber
\end{align}
so $\mathbb{E}^{P^{\omega^{-}}}[N[0,1]]<\infty$ by choice of $\epsilon<\frac{1}{h(0)}$.

By the theory of absolute continuity of point processes, see for example Chapter 13 of Daley and Vere-Jones \cite{Daley}, 
if $\mathbb{E}^{Q^{\omega^{-}}}[N[0,1]],\mathbb{E}^{P^{\omega^{-}}}[N[0,1]]<\infty$, $Q^{\omega^{-}}\ll P^{\omega^{-}}$ 
if and only if $\hat{A}_{t}\ll A_{t}$, where $\hat{A}_{t}$ and $A_{t}=\int_{0}^{t}\lambda(\omega^{-},\omega,s)ds$ 
are the compensators of $N_{t}$ under $Q^{\omega^{-}}$ and $P^{\omega^{-}}$ respectively. 
If that's the case, we can write $\hat{A}_{t}=\int_{0}^{t}\hat{\lambda}(\omega^{-},\omega,s)ds$ for some $\hat{\lambda}$ and there is Girsanov formula
\begin{equation}
\log\frac{dQ^{\omega^{-}}}{dP^{\omega^{-}}}\bigg|_{\mathcal{F}^{0}_{1}}=\int_{0}^{1}(\lambda-\hat{\lambda})ds
+\int_{0}^{1}\log\left(\hat{\lambda}/\lambda\right)dN_{s},
\end{equation}
which implies that
\begin{equation}
H(Q)=\mathbb{E}^{Q}\left[\int_{0}^{1}\lambda-\hat{\lambda}+\log\left(\hat{\lambda}/\lambda\right)\hat{\lambda}ds\right].
\end{equation}
For any $f$, $\hat{\lambda}f+(1-e^{f})\lambda\leq\hat{\lambda}\log(\hat{\lambda}/\lambda)+\lambda-\hat{\lambda}$ 
and the equality is achieved when $f=\log(\hat{\lambda}/\lambda)$. Thus, clearly, we have
\begin{equation}
\sup_{f(\omega,s)\in\mathcal{B}(\mathcal{F}^{-\infty}_{s})\cap C(\Omega\times\mathbb{R}),0\leq s\leq 1}\mathbb{E}^{Q}
\left[\int_{0}^{1}\lambda(1-e^{f})ds+\int_{0}^{1}fdN_{s}\right]\leq H(Q).
\end{equation}
On the other hand, we can always find a sequence $f_{n}$ convergent to $\log(\hat{\lambda}/\lambda)$ and by Fatou's lemma, we get the 
opposite inequality.

Now, assume that we do not have $Q^{\omega^{-}}\ll P^{\omega^{-}}$ for a.e. $\omega^{-}$ under $Q$. That implies that $H(Q)=\infty$.
We want to show that
\begin{equation}
\sup_{f(\omega,s)\in\mathcal{B}(\mathcal{F}^{-\infty}_{s})\cap C(\Omega\times\mathbb{R}),0\leq s\leq 1}
\mathbb{E}^{Q}\left[\int_{0}^{1}\lambda(1-e^{f})ds+\int_{0}^{1}fdN_{s}\right]=\infty.
\end{equation}
Let us assume that
\begin{equation}
\sup_{f(\omega,s)\in\mathcal{B}(\mathcal{F}^{-\infty}_{s})\cap C(\Omega\times\mathbb{R}),0\leq s\leq 1}
\mathbb{E}^{Q}\left[\int_{0}^{1}\lambda(1-e^{f})ds+\int_{0}^{1}fdN_{s}\right]<\infty.
\end{equation}
We want to prove that $H(Q)<\infty$.

Let $P^{\omega^{-}}_{\epsilon}$ be the point process on $[0,1]$ with compensator $A_{t}+\epsilon\hat{A}_{t}$. 
Clearly $\hat{A}_{t}\ll A_{t}+\epsilon\hat{A}_{t}$ and $Q^{\omega^{-}}\ll P^{\omega^{-}}_{\epsilon}$.

For any $f$,
\begin{align}
&\mathbb{E}^{Q}\left[\int_{0}^{1}(1-e^{f})d(A_{s}+\epsilon\hat{A}_{s})+fd\hat{A}_{s}\right]
\\
&=\mathbb{E}^{Q}\left[\int_{0}^{1}(1-e^{f})\chi_{f<0}d(A_{s}+\epsilon\hat{A}_{s})+f\chi_{f<0}d\hat{A}_{s}\right]\nonumber
\\
&+\mathbb{E}^{Q}\left[\int_{0}^{1}(1-e^{f})\chi_{f\geq 0}d(A_{s}+\epsilon\hat{A}_{s})+f\chi_{f\geq 0}d\hat{A}_{s}\right]\nonumber
\\
&\leq\mathbb{E}^{Q}\left[\int_{0}^{1}d(A_{s}+\epsilon\hat{A}_{s})\right]+\mathbb{E}^{Q}\left[\int_{0}^{1}(1-e^{f})
\chi_{f\geq 0}dA_{s}+f\chi_{f\geq 0}d\hat{A}_{s}\right]\nonumber
\\
&=\mathbb{E}^{Q}\left[\int_{0}^{1}d(A_{s}+\epsilon\hat{A}_{s})\right]
+\mathbb{E}^{Q}\left[\int_{0}^{1}(1-e^{f\chi_{f\geq 0}})dA_{s}+f\chi_{f\geq 0}d\hat{A}_{s}\right]\nonumber
\\
&\leq C_{\delta}+\delta(h(0)+\Vert h\Vert_{L^{1}})\mathbb{E}^{Q}[N[0,1]]\nonumber
\\
&+\sup_{f(\omega,s)\in\mathcal{B}(\mathcal{F}^{-\infty}_{s})\cap C(\Omega\times\mathbb{R}),0\leq s\leq 1}
\mathbb{E}^{Q}\left[\int_{0}^{1}\lambda(1-e^{f})ds+\int_{0}^{1}fdN_{s}\right]<\infty.\nonumber
\end{align}
Therefore,
\begin{align}
\infty&>\liminf_{\epsilon\downarrow 0}\sup_{f(\omega,s)\in\mathcal{B}(\mathcal{F}^{-\infty}_{s})
\cap C(\Omega\times\mathbb{R}),0\leq s\leq 1}\mathbb{E}^{Q}\left[\int_{0}^{1}(1-e^{f})d(A_{s}+\epsilon\hat{A}_{s})+fd\hat{A}_{s}\right]
\\
&=\liminf_{\epsilon\downarrow 0}\sup_{f(\omega,s)\in\mathcal{B}(\mathcal{F}^{-\infty}_{s})
\cap C(\Omega\times\mathbb{R}),0\leq s\leq 1}\nonumber
\\
&\qquad\qquad\qquad\qquad\qquad\mathbb{E}^{Q}\left[\int_{0}^{1}\left(1-e^{f}
+f\cdot\frac{d\hat{A}_{s}}{d(A_{s}+\epsilon\hat{A}_{s})}\right)d(A_{s}+\epsilon\hat{A}_{s})\right]\nonumber
\\
&=\liminf_{\epsilon\downarrow 0}\mathbb{E}^{Q}[h(Q^{\omega^{-}},P^{\omega^{-}}_{\epsilon})_{\mathcal{F}^{0}_{1}}]\nonumber
\\
&=\mathbb{E}^{Q}[h(Q^{\omega^{-}},P^{\omega^{-}})_{\mathcal{F}^{0}_{1}}]=H(Q),\nonumber
\end{align}
by lower semicontinuity of the relative entropy $h(\cdot,\cdot)$, Fatou's lemma, and the fact that 
$P^{\omega^{-}}_{\epsilon}\rightarrow P^{\omega^{-}}$ weakly as $\epsilon\downarrow 0$. Hence $H(Q)<\infty$.
\end{proof}

\begin{lemma}\label{lscconvex}
$H(Q)$ is lower semicontinuous and convex in $Q$.
\end{lemma}

\begin{proof}
By Lemma \ref{variational}, we can rewrite $H(Q)$ as
\begin{align}
H(Q)&=\sup_{f(\omega,s)\in\mathcal{B}(\mathcal{F}^{-\infty}_{s})\cap C(\Omega\times\mathbb{R}),0\leq s\leq 1}
\mathbb{E}^{Q}\left[\int_{0}^{1}\lambda(1-e^{f})+\hat{\lambda}fds\right]
\\
&=\sup_{f(\omega,s)\in\mathcal{B}(\mathcal{F}^{-\infty}_{s})\cap C(\Omega\times\mathbb{R}),0\leq s\leq 1}
\mathbb{E}^{Q}\left[\int_{0}^{1}\lambda(1-e^{f})ds+\int_{0}^{1}fdN_{s}\right].\nonumber
\end{align}
If $Q_{n}\rightarrow Q$, then $\mathbb{E}^{Q_{n}}[N[0,1]]\rightarrow\mathbb{E}^{Q}[N[0,1]]$ and $Q_{n}\rightarrow Q$ weakly. 
Since $f(\omega,s)\in C(\Omega\times\mathbb{R})\cap\mathcal{B}(\mathcal{F}^{-\infty}_{s})$, $\int_{0}^{1}f(\omega,s)dN_{s}$ is continuous 
on $\Omega$, and since $f$ is uniformly bounded, $\int_{0}^{1}f(\omega,s)dN_{s}\leq\Vert f\Vert_{L^{\infty}}N[0,1]$. Hence,
\begin{equation}
\mathbb{E}^{Q_{n}}\left[\int_{0}^{1}f(\omega,s)dN_{s}\right]\rightarrow\mathbb{E}^{Q}\left[\int_{0}^{1}f(\omega,s)dN_{s}\right].
\end{equation}
Let $\lambda^{M}=\lambda\left(\sum_{\tau<s}h^{M}(s-\tau)\right)$, where $h^{M}(s)=h(s)\chi_{s\leq M}$. 
Then, $\lambda^{M}(\omega,s)\in C(\Omega\times\mathbb{R})$ and thus $\int_{0}^{1}\lambda^{M}(1-e^{f(\omega,s)})ds\in C(\Omega)$. 
Also, $\int_{0}^{1}\lambda^{M}(1-e^{f(\omega,s)})ds\leq K(1+e^{\Vert f\Vert_{L^{\infty}}})N[-M,1]$, where $K>0$ is some constant. 
Therefore,
\begin{equation}
\mathbb{E}^{Q_{n}}\left[\int_{0}^{1}\lambda^{M}(1-e^{f(\omega,s)})ds\right]
\rightarrow\mathbb{E}^{Q}\left[\int_{0}^{1}\lambda^{M}(1-e^{f(\omega,s)})ds\right]
\end{equation} 
as $n\rightarrow\infty$. Next, notice that
\begin{align}
&\left|\mathbb{E}^{Q}\left[\int_{0}^{1}\lambda^{M}(1-e^{f(\omega,s)})ds\right]-\mathbb{E}^{Q}\left[\int_{0}^{1}\lambda(1-e^{f(\omega,s)})ds\right]\right|
\\
&\leq\mathbb{E}^{Q}(1+e^{\Vert f\Vert_{L^{\infty}}})\alpha\mathbb{E}^{Q}[N[0,1]]\int_{M}^{\infty}h(s)ds\rightarrow 0\nonumber
\end{align}
as $M\rightarrow\infty$. Similarly, we have
\begin{equation}
\limsup_{M\rightarrow\infty}\limsup_{n\rightarrow\infty}\left|\mathbb{E}^{Q_{n}}
\left[\int_{0}^{1}\lambda^{M}(1-e^{f(\omega,s)})ds\right]-\mathbb{E}^{Q_{n}}\left[\int_{0}^{1}\lambda(1-e^{f(\omega,s)})ds\right]\right|=0.
\end{equation}
Hence,
\begin{equation}
\mathbb{E}^{Q_{n}}\left[\int_{0}^{1}\lambda(\omega,s)(1-e^{f(\omega,s)})ds\right]
\rightarrow\mathbb{E}^{Q}\left[\int_{0}^{1}\lambda(\omega,s)(1-e^{f(\omega,s)})ds\right].
\end{equation}
The supremum is taken over a linear functional of $Q$, which is continuous in $Q$, 
therefore the supremum over these linear functionals will be lower semicontinuous. 
Similarly, since in the variational formula expression 
of $H(Q)$ in Lemma \ref{variational}, the supremum is taken over a linear functional of $Q$, $H(Q)$ is convex in $Q$.
\end{proof}

\begin{lemma}
$H(Q)$ is linear in $Q$.
\end{lemma}

\begin{proof}
It is in general true that the process-level entropy function $H(Q)$ is linear in $Q$. Following the arguments in Donsker and Varadhan \cite{Donsker},
there exists a subset $\Omega_{0}\subset\Omega$ which is $\mathcal{F}^{-\infty}_{0}$ measurable and a $\mathcal{F}^{-\infty}_{0}$ 
measurable map $\hat{Q}:\Omega_{0}\rightarrow\mathcal{M}_{E}(\Omega)$ such that $Q(\Omega_{0})=1$ for all $Q\in\mathcal{M}_{S}(\Omega)$
and $Q(\omega:\hat{Q}=Q)=1$ for all $Q\in\mathcal{M}_{E}(\Omega)$. Therefore, there exists a universal version,
say $\hat{Q}^{\omega^{-}}$ independent of $Q$ such that $\int\hat{Q}^{\omega^{-}}Q(d\omega^{-})=Q$. Since that is
true for all $Q\in\mathcal{M}_{E}(\Omega)$, it also holds for $Q\in\mathcal{M}_{S}(\Omega)$. Hence, 
\begin{equation}
H(Q)=\mathbb{E}^{Q}\left[h(Q^{\omega^{-}},P^{\omega^{-}})_{\mathcal{F}^{0}_{1}}\right]
=\mathbb{E}^{Q}\left[h(\hat{Q}^{\omega^{-}},P^{\omega^{-}})_{\mathcal{F}^{0}_{1}}\right],
\end{equation}
i.e. $H(Q)$ is linear in $Q$.
\end{proof}

In this chapter, we are proving the large deviation principle
for Hawkes processes started with empty history, i.e. with probability measure $P^{\emptyset}$. But when
time elapses, the Hawkes process generates points and that create a new history. We need
to understand how the history created affects the future. What we want to prove is some uniform
estimates to the effect that if the past history is well controlled, then the new history will also be well controlled.
This is essentially what the following Lemma \ref{mainlower} says. Consider the configuration of points
starting from time $0$ up to time $t$. We shift it by $t$ and denote that by $w_{t}$ such that $w_{t}\in\Omega^{-}$, where $\Omega^{-}$ is
$\Omega$ restricted to $\mathbb{R}^{-}$.
These notations will be used in Lemma \ref{mainlower}.

\begin{remark}
At the very beginning of the chapter, we defined $\omega_{t}$. It should not be confused with $w_{t}$ in this section.
\end{remark}

\begin{lemma}\label{mainlower}
For any $Q\in\mathcal{M}_{E}(\Omega)$ such that $H(Q)<\infty$ and any open neighborhood $N$ of $Q$, 
there exists some $K^{-}_{\ell}$ such that $\emptyset\in K^{-}_{\ell}$ and $Q(K^{-}_{\ell})\rightarrow 1$ as $\ell\rightarrow\infty$ and
\begin{equation}
\liminf_{t\rightarrow\infty}\frac{1}{t}\inf_{w_{0}\in K^{-}_{\ell}}\log P^{w_{0}}(R_{t,\omega}\in N,w_{t}\in K^{-}_{\ell})\geq-H(Q).
\end{equation}
\end{lemma}

\begin{proof}
Let us abuse the notations a bit by defining
\begin{equation}
\lambda(\omega^{-})=\lambda\left(\sum_{\tau\in\omega^{-},\tau\in\omega[0,s)}h(s-\tau)\right).
\end{equation}

For any $t>0$, since $\lambda(\cdot)\geq c>0$ and $\lambda(\cdot)$ is Lipschitz with constant $\alpha$, we have
\begin{align}
\log\frac{dP^{\omega^{-}}}{dP^{w_{0}}}\bigg|_{\mathcal{F}^{0}_{t}}
&=\int_{0}^{t}\lambda(w_{0})-\lambda(\omega^{-})ds+\int_{0}^{t}\log\left(\frac{\lambda(\omega^{-})}{\lambda(w_{0})}\right)dN_{s}
\\
&\leq\int_{0}^{t}|\lambda(w_{0})-\lambda(\omega^{-})|ds
+\int_{0}^{t}\log\left(1+\frac{|\lambda(w_{0})-\lambda(\omega^{-})|}{\lambda(w_{0})}\right)dN_{s}\nonumber
\\
&\leq\int_{0}^{t}\alpha\sum_{\tau\in\omega^{-}\cup w_{0}}h(s-\tau)ds+\int_{0}^{t}\frac{\alpha}{c}\sum_{\tau\in\omega^{-}\cup w_{0}}h(s-\tau)dN_{s}.\nonumber
\end{align}
Define
\begin{equation}
K^{-}_{\ell}=\left\{\omega:N[-t,0](\omega)\leq\ell(1+t), \forall t>0\right\}.
\end{equation}
By the maximal ergodic theorem,
\begin{align}
Q((K^{-}_{\ell})^{c})&=Q\left(\sup_{t>0}\frac{N[-t,0]}{t+1}>\ell\right)
\\
&\leq Q\left(\sup_{t>0}\frac{N[-([t]+1),0]}{[t]+1}>\ell\right)\nonumber
\\
&=Q\left(\sup_{n\geq 1, n\in\mathbb{N}}\frac{N[-n,0]}{n}>\ell\right)\nonumber
\\
&\leq\frac{\mathbb{E}^{Q}[N[0,1]]}{\ell}\rightarrow 0
\end{align}
as $\ell\rightarrow\infty$. Thus $Q(K^{-}_{\ell})\rightarrow 1$ as $\ell\rightarrow\infty$.

Fix any $s>0$ and $\omega^{-}\in K^{-}_{\ell}$. Since $h$ is decreasing, $h'\leq 0$, integration by parts shows that
\begin{align}
\sum_{\tau\in\omega^{-}}h(s-\tau)&=\int_{0}^{\infty}h(s+\sigma)dN[-\sigma,0]
\\
&=-\int_{0}^{\infty}N[-\sigma,0]h'(s+\sigma)d\sigma\nonumber
\\
&\leq-\int_{0}^{\infty}\ell(1+\sigma)h'(s+\sigma)d\sigma\nonumber
\\
&=\ell h(s)+\ell\int_{0}^{\infty}h(s+\sigma)d\sigma\nonumber
\\
&=\ell h(s)+\ell H(s),\nonumber
\end{align}
where $H(t)=\int_{t}^{\infty}h(s)ds$.

Therefore, uniformly for $\omega_{-},w_{0}\in K^{-}_{\ell}$,
\begin{equation}
\int_{0}^{t}\alpha\sum_{\tau\in\omega^{-}\cup w_{0}}h(s-\tau)ds\leq 2\ell\alpha\Vert h\Vert_{L^{1}}+2\ell\alpha u(t),
\end{equation}
where $u(t)=\int_{0}^{t}H(s)ds$ and
\begin{equation}
\int_{0}^{t}\frac{\alpha}{c}\sum_{\tau\in\omega^{-}\cup w_{0}}h(s-\tau)dN_{s}\leq\frac{2\ell\alpha}{c}\int_{0}^{t}(h(s)+H(s))dN_{s}.
\end{equation}
Define
\begin{equation}
K^{+}_{\ell,t}=\left\{\omega:\frac{2\ell\alpha}{c}\int_{0}^{t}(h(s)+H(s))dN_{s}\leq\ell^{2}(\Vert h\Vert_{L^{1}}+u(t))\right\}.
\end{equation}
Then, uniformly in $t>0$,
\begin{equation}
Q((K^{+}_{\ell,t})^{c})\leq\frac{2\alpha\mathbb{E}^{Q}[N[0,1]]}{c\cdot\ell}\rightarrow 0,
\end{equation}
as $\ell\rightarrow\infty$. Thus $\inf_{t>0}Q(K^{+}_{\ell,t})\rightarrow 1$ as $\ell\rightarrow\infty$.

Hence, uniformly for $\omega_{-},w_{0}\in K^{-}_{\ell}$ and $\omega\in K^{+}_{\ell,t}$,
\begin{align}
\log\frac{dP^{\omega^{-}}}{dP^{w_{0}}}\bigg|_{\mathcal{F}^{0}_{t}}&\leq 2\ell\alpha\Vert h\Vert_{L^{1}}+2\ell\alpha u(t)+\ell^{2}(\Vert h\Vert_{L^{1}}+u(t))
\\
&=C_{1}(\ell)+C_{2}(\ell)u(t),
\end{align}
where $C_{1}(\ell)=2\ell\alpha\Vert h\Vert_{L^{1}}+\ell^{2}\Vert h\Vert_{L^{1}}$ and $C_{2}(\ell)=2\ell\alpha+\ell^{2}$.

Observe that
\begin{equation}
\limsup_{t\rightarrow\infty}\frac{u(t)}{t}=\limsup_{t\rightarrow\infty}\frac{1}{t}\int_{0}^{t}H(s)ds=0.
\end{equation}

Let $D_{t}=\{R_{t,\omega}\in N,w_{t}\in K^{-}_{\ell}\}$. 

Uniformly for $w_{0}\in K^{-}_{\ell,t}$,
\begin{align}
&P^{w_{0}}(D_{t})
\\
&\geq e^{-t(H(Q)+\epsilon)-C_{1}(\ell)-C_{2}(\ell)u(t)}\nonumber
\\
&\cdot Q\left[D_{t}\cap\left\{\frac{1}{t}\log\frac{dP^{\omega^{-}}}{dQ^{\omega^{-}}}\bigg|_{\mathcal{F}^{0}_{t}}
\leq H(Q)+\epsilon\right\}\cap\left\{\log\frac{dP^{\omega^{-}}}{dP^{w_{0}}}\bigg|_{\mathcal{F}^{0}_{t}}
\leq C_{1}(\ell)+C_{2}(\ell)u(t)\right\}\right]\nonumber
\\
&\geq e^{-t(H(Q)+\epsilon)-C_{1}(\ell)-C_{2}(\ell)u(t)}\nonumber
\\
&\qquad\qquad\qquad\qquad\cdot 
Q\left[D_{t}\cap\left\{\frac{1}{t}\log\frac{dP^{\omega^{-}}}{dQ^{\omega^{-}}}\bigg|_{\mathcal{F}^{0}_{t}}
\leq H(Q)+\epsilon\right\}\cap\{K^{+}_{\ell,t}\cap K^{-}_{\ell}\}\right].\nonumber
\end{align}

Since $Q\in\mathcal{M}_{E}(\Omega)$, by ergodic theorem,
\begin{equation}
\lim_{t\rightarrow\infty}Q(R_{t,\omega}\in N)=1,
\end{equation}
and since $\psi(\omega,t)=\log\frac{dQ^{\omega}}{dP^{\omega}}\big|_{\mathcal{F}^{0}_{t}}$ satisfies,
\begin{equation}
\psi(\omega,t+s)=\psi(\omega,t)+\psi(\theta_{t}\omega,s),\quad\mathbb{E}^{Q}[\psi(\omega,t)]=tH(Q),
\end{equation}
for almost every $\omega^{-}$ under $Q$,
\begin{equation}
\lim_{t\rightarrow\infty}\frac{1}{t}\log\frac{dP^{\omega^{-}}}{dQ^{\omega^{-}}}\bigg|_{\mathcal{F}^{0}_{t}}=H(Q).
\end{equation}
$Q$ is stationary, so $Q(w_{t}\in K^{-}_{\ell})\geq Q(K^{-}_{\ell})\rightarrow 1$ as $\ell\rightarrow\infty$. 
Also, $Q(K^{+}_{\ell,t})\geq\inf_{t>0}Q(K^{+}_{\ell,t})\rightarrow 1$ as $\ell\rightarrow\infty$. 
Remember that $\limsup_{t\rightarrow\infty}\frac{u(t)}{t}=0$. By choosing $\ell$ big enough, we conclude that
\begin{equation}
\liminf_{t\rightarrow\infty}\frac{1}{t}\inf_{w_{0}\in K^{-}_{\ell}}\log P^{w_{0}}(R_{t,\omega}\in N,w_{t}\in K^{-}_{\ell})\geq-H(Q)-\epsilon.
\end{equation}
Since it holds for any $\epsilon>0$, we get the desired result.
\end{proof}

\begin{theorem}[Lower Bound]
For any open set $G$, 
\begin{equation}
\liminf_{t\rightarrow\infty}\frac{1}{t}\log P(R_{t,\omega}\in G)\geq-\inf_{Q\in G}H(Q).
\end{equation}
\end{theorem}

\begin{proof}
It is sufficent to prove that for any $Q\in\mathcal{M}_{S}(\Omega)$, $H(Q)<\infty$, for any neighborhood $N$ of $Q$, 
$\liminf_{t\rightarrow\infty}\frac{1}{t}\log P(R_{t,\omega}\in N)\geq -H(Q)$. 
Since for every invariant measure $P\in\mathcal{M}_{S}$, 
there exists a probability measure $\mu_{P}$ on the space $\mathcal{M}_{E}$ of ergodic measures 
such that $P=\int_{\mathcal{M}_{E}}Q\mu_{P}(dQ)$, for any $Q\in\mathcal{M}_{S}(\Omega)$ such that $H(Q)<\infty$, 
without loss of generality, we can assume that $Q=\sum_{j=1}^{\ell}\alpha_{j}Q_{j}$, 
where $\alpha_{j}\geq 0$, $1\leq j\leq\ell$ and $\sum_{j=1}^{\ell}\alpha_{j}=1$. 
By linearity of $H(\cdot)$, $H(Q)=\sum_{j=1}^{\ell}\alpha_{j}H(Q_{j})$. 
Divide the interval $[0,t]$ into subintervals of length $\alpha_{j}t$, 
let $t_{j}$, $1\leq j\leq\ell$ be the right hand endpoints of these subintervals, and let $t_{0}=0$. 
For each $Q_{j}$, take $K^{-}_{M}$ as in Lemma \ref{mainlower}. 
We have $\min_{1\leq j\leq\ell}Q_{j}(K^{-}_{M})\rightarrow 1$, as $M\rightarrow\infty$. 
Choose neighborhoods $N_{j}$ of $Q_{j}$, $1\leq j\leq\ell$ such that $\bigcup_{j=1}^{\ell}\alpha_{j}N_{j}\subseteq N$.  We have
\begin{align}
P^{\emptyset}(R_{t,\omega}\in N)&\geq P^{\emptyset}(R_{t_{1},\omega}\in N_{1},w_{t_{1}}\in K^{-}_{M})
\\
&\cdot\prod_{j=2}^{\ell}\inf_{w_{0}\in K^{-}_{t_{j-1}-t_{j-2}}}P^{w_{0}}(R_{t_{j}-t_{j-1},\omega}\in N_{j},w_{t_{j}-t_{j-1}}\in K^{-}_{M}).\nonumber
\end{align}
Now, applying Lemma \ref{mainlower} and the linearity of $H(\cdot)$,
\begin{equation}
\liminf_{t\rightarrow\infty}\frac{1}{t}\log P^{\emptyset}(R_{t,\omega}\in N)\geq-\sum_{j=1}^{\ell}\alpha_{j}H(Q_{j})=-H(Q).
\end{equation}
\end{proof}

\section{Upper Bound}\label{upperbound}

\begin{remark}
By following the argument in Donsker and Varadhan \cite{Donsker}, if $\omega^{-}\mapsto P^{\omega^{-}}$ is weakly continuous, then
\begin{equation}
\limsup_{t\rightarrow\infty}\frac{1}{t}\log P(R_{t,\omega}\in A)\leq-\inf_{Q\in A}H(Q),
\end{equation}
for any compact $A$. If the Hawkes process has finite range of memory, i.e. $h(\cdot)$ has compact support, and if it is continuous, 
then, for any $a<b$, if $\omega^{-}_{n}\rightarrow\omega^{-}$, we have
\begin{align}
&\left|\int_{a}^{b}\lambda(\omega^{-}_{n},\omega,s)ds-\int_{a}^{b}\lambda(\omega^{-},\omega,s)ds\right|
\\
&\leq\alpha\int_{a}^{b}\left|\sum_{\tau\in\omega^{-}_{n}}h(s-\tau)-\sum_{\tau\in\omega^{-}}h(s-\tau)\right|ds\rightarrow\infty,\nonumber
\end{align}
as $n\rightarrow\infty$, which implies that $P^{\omega^{-}_{n}}\rightarrow P^{\omega^{-}}$.

If the Hawkes process does not have finite range of memory, then we should use the specific features of the Hawkes process to obtain the upper bound.
\end{remark}

Before we proceed, let us prove an easy but very useful lemma that we will use repeatedly in the proofs of the estimates in this chapter.

\begin{lemma}
Let $f(\omega,s)$ be $\mathcal{F}^{-\infty}_{s}$ progressively measurable and predictable. Then,
\begin{equation}
\mathbb{E}\left[e^{\int_{0}^{t}f(\omega,s)dN_{s}}\right]\leq\mathbb{E}\left[e^{\int_{0}^{t}(e^{2f(\omega,s)}-1)\lambda(\omega,s)ds}\right]^{1/2}.
\end{equation} 
\end{lemma}

\begin{proof}
Since $\exp\left\{\int_{0}^{t}2f(\omega,s)dN_{s}-\int_{0}^{t}(e^{2f(\omega,s)}-1)\lambda(\omega,s)ds\right\}$ is a martingale, 
by Cauchy-Schwarz inequality,
\begin{align}
\mathbb{E}\left[e^{\int_{0}^{t}f(\omega,s)dN_{s}}\right]
&=\mathbb{E}\left[e^{\frac{1}{2}\int_{0}^{t}2f(\omega,s)dN_{s}-\frac{1}{2}\int_{0}^{t}(e^{2f(\omega,s)}-1)\lambda(\omega,s)ds
+\frac{1}{2}\int_{0}^{t}(e^{2f(\omega,s)}-1)\lambda(\omega,s)ds}\right]
\\
&\leq\mathbb{E}\left[e^{\int_{0}^{t}(e^{2f(\omega,s)}-1)\lambda(\omega,s)ds}\right]^{1/2}.\nonumber
\end{align}
\end{proof}

Define $\mathcal{C}_{T}$
\begin{align}
\mathcal{C}_{T}&=\bigg\{F(\omega):=\int_{0}^{T}f(\omega,s)dN_{s}-\int_{0}^{T}(e^{f(\omega,s)}-1)\lambda(\omega,s)ds, 
\\
&\qquad\qquad\qquad\qquad\qquad\qquad\qquad f(\omega,s)\in\mathcal{B}(\mathcal{F}^{0}_{s})\cap C(\Omega\times\mathbb{R})\bigg\}.\nonumber
\end{align}
Here $\lambda(\omega,s)$ is $\mathcal{F}^{-\infty}_{s}$ progressively measurable and predictable, and 
$f(\omega,s)\in\mathcal{B}(\mathcal{F}^{0}_{s})\cap C(\Omega\times\mathbb{R})$ means that 
$f$ is $\mathcal{F}^{0}_{s}$ progressively measurable, predictable and also bounded and continuous.

\begin{lemma}\label{expectation}
For any $T>0$ and $F\in\mathcal{C}_{T}$, we have, for any $t>0$,
\begin{equation}
\mathbb{E}^{P^{\emptyset}}\left[e^{\frac{1}{T}\int_{0}^{t}F(\theta_{s}\omega)ds}\right]\leq 1.
\end{equation}
\end{lemma}

\begin{proof}
For any $t>0$, writing $\psi(s)=\sum_{k:s+kT\leq t}F(\theta_{s+kT}\omega)$,
\begin{align}
\mathbb{E}^{P^{\emptyset}}\left[e^{\frac{1}{T}\int_{0}^{t}F(\theta_{s}\omega)ds}\right]
&=\mathbb{E}^{P^{\emptyset}}\left[e^{\frac{1}{T}\int_{0}^{T}\psi(s)ds}\right]
\\
&\leq\frac{1}{T}\int_{0}^{T}\mathbb{E}^{P^{\emptyset}}\left[e^{\psi(s)}\right]ds=1,\nonumber
\end{align}
by Jensen's inequality and the fact that $\mathbb{E}^{P^{\emptyset}}\left[e^{\psi(s)}\right]=1$ 
by iteratively conditioning since $\mathbb{E}^{P^{\omega^{-}}}\left[e^{F(\omega)}\right]=1$ for any $\omega^{-}$.
\end{proof}

\begin{remark}
Under $P^{\emptyset}$, the $\mathcal{F}^{-\infty}_{s}$ progressively measurable rate function $\lambda$ 
is well defined since it only creates a history between time $0$ and time $t$. 
Similary, in the proof in Lemma \ref{expectation}, $\mathbb{E}^{P^{\omega^{-}}}\left[e^{F(\omega)}\right]=1$ for any $\omega^{-}$ 
should be interpreted as the expectation is $1$ given any history created between time $0$ and $t$, which is well defined.
\end{remark}

Next, we need to compare $\frac{1}{T}\int_{0}^{t}F(\theta_{s}\omega_{t})ds$ and $\frac{1}{T}\int_{0}^{t}F(\theta_{s}\omega)ds$.

\begin{lemma}\label{difference}
For any $q>0$, $T>0$ and $F\in\mathcal{C}_{T}$,
\begin{equation}
\limsup_{t\rightarrow\infty}\frac{1}{t}\log\mathbb{E}^{P^{\emptyset}}
\left[\exp\left\{q\left|\frac{1}{T}\int_{0}^{t}F(\theta_{s}\omega_{t})ds-\frac{1}{T}\int_{0}^{t}F(\theta_{s}\omega)ds\right|\right\}\right]=0.
\end{equation}
\end{lemma}

\begin{proof}
\begin{align}
&\left|\frac{1}{T}\int_{0}^{t}F(\theta_{s}\omega_{t})ds-\frac{1}{T}\int_{0}^{t}F(\theta_{s}\omega)ds\right|
\\
&\leq\left|\frac{1}{T}\int_{0}^{t}\int_{0}^{T}f(u,\theta_{s}\omega)dN_{u}ds-\frac{1}{T}\int_{0}^{t}
\int_{0}^{T}f(u,\theta_{s}\omega_{t})dN_{u}ds\right|\nonumber
\\
&\qquad\qquad\qquad+\bigg|\frac{1}{T}\int_{0}^{t}\int_{0}^{T}(e^{f(u,\theta_{s}\omega)}-1)\lambda(\theta_{s}\omega,u)duds\nonumber
\\
&\qquad\qquad\qquad\qquad\qquad\qquad
-\frac{1}{T}\int_{0}^{t}\int_{0}^{T}(e^{f(u,\theta_{s}\omega_{t})}-1)\lambda(\theta_{s}\omega_{t},u)duds\bigg|.\nonumber
\end{align}
It is easy to see that $\int_{0}^{T}f(u,\theta_{s}\omega)dN_{u}ds$ is $\mathcal{F}^{s}_{s+T}$-measurable and
\begin{equation}
\int_{0}^{T}f(u,\theta_{s}\omega)dN_{u}ds=\int_{0}^{T}f(u,\theta_{s}\omega_{t})dN_{u}ds
\end{equation}
for any $0\leq s\leq t-T$. Hence,
\begin{align}
&\left|\frac{1}{T}\int_{0}^{t}\int_{0}^{T}f(u,\theta_{s}\omega)dN_{u}ds-\frac{1}{T}\int_{0}^{t}\int_{0}^{T}f(u,\theta_{s}\omega_{t})dN_{u}ds\right|
\\
&\leq\frac{1}{T}\int_{t-T}^{t}\int_{0}^{T}|f(u,\theta_{s}\omega)|dN_{u}ds+\frac{1}{T}\int_{t-T}^{t}\int_{0}^{T}|f(u,\theta_{s}\omega_{t})|dN_{u}ds\nonumber
\\
&\leq\frac{\Vert f\Vert_{L^{\infty}}}{T}\int_{t-T}^{t}N[s,s+T](\omega)ds+\frac{\Vert f\Vert_{L^{\infty}}}{T}\int_{t-T}^{t}N[s,s+T](\omega_{t})ds\nonumber
\\
&\leq\frac{\Vert f\Vert_{L^{\infty}}}{T}\left[N[t-T,t+T](\omega)+N[t-T,t+T](\omega_{t})\right]\nonumber
\\
&=\frac{\Vert f\Vert_{L^{\infty}}}{T}\left[N[t-T,t+T](\omega)+N[t-T,T](\omega)+N[0,T](\omega)\right].\nonumber
\end{align}
By H\"{o}lder's inequality and Lemma \ref{indfinite}, we have
\begin{align}
&\limsup_{t\rightarrow\infty}\frac{1}{t}\log\mathbb{E}^{P^{\emptyset}}
\left[e^{\left|\frac{1}{T}\int_{0}^{t}\int_{0}^{T}f(u,\theta_{s}\omega)dN_{u}ds
-\frac{1}{T}\int_{0}^{t}\int_{0}^{T}f(u,\theta_{s}\omega_{t})dN_{u}ds\right|}\right]
\\
&\leq\limsup_{t\rightarrow\infty}\frac{1}{t}\log\mathbb{E}^{P^{\emptyset}}
\left[e^{\frac{\Vert f\Vert_{L^{\infty}}}{T}\left[N[t-T,t+T](\omega)+N[t-T,T](\omega)+N[0,T](\omega)\right]}\right]=0.\nonumber
\end{align}
Furthermore,
\begin{align}
&\left|\frac{1}{T}\int_{0}^{t}\int_{0}^{T}(e^{f(u,\theta_{s}\omega)}-1)\lambda(\theta_{s}\omega,u)duds
-\frac{1}{T}\int_{0}^{t}\int_{0}^{T}(e^{f(u,\theta_{s}\omega_{t})}-1)\lambda(\theta_{s}\omega_{t},u)duds\right|
\\
&\leq\frac{1}{T}\int_{0}^{t}\int_{0}^{T}\left|e^{f(u,\theta_{s}\omega)}-e^{f(u,\theta_{s}\omega_{t})}\right|\lambda(\theta_{s}\omega,u)duds\nonumber
\\
&+\frac{1}{T}\int_{0}^{t}\int_{0}^{T}(e^{f(u,\theta_{s}\omega_{t})}-1)\left|\lambda(\theta_{s}\omega_{t},u)-\lambda(\theta_{s}\omega,u)\right|duds.\nonumber
\end{align}
For the first term
\begin{align}
&\frac{1}{T}\int_{0}^{t}\int_{0}^{T}\left|e^{f(u,\theta_{s}\omega)}-e^{f(u,\theta_{s}\omega_{t})}\right|\lambda(\theta_{s}\omega,u)duds
\\
&=\frac{1}{T}\int_{t-T}^{t}\int_{0}^{T}\left|e^{f(u,\theta_{s}\omega)}-e^{f(u,\theta_{s}\omega_{t})}\right|\lambda(\theta_{s}\omega,u)duds\nonumber
\\
&\leq\frac{2 e^{\Vert f\Vert_{L^{\infty}}}}{T}\int_{t-T}^{t}\int_{0}^{T}\lambda(\theta_{s}\omega,u)duds\nonumber
\\
&=\frac{2 e^{\Vert f\Vert_{L^{\infty}}}}{T}\int_{t-T}^{t}\int_{0}^{T}\lambda\left(\sum_{\tau\in\omega[0,u+s)}h(u+s-\tau)\right)duds\nonumber
\\
&\leq 2 e^{\Vert f\Vert_{L^{\infty}}}TC_{\epsilon}
+\frac{2 e^{\Vert f\Vert_{L^{\infty}}}}{T}\epsilon\int_{t-T}^{t}\int_{0}^{T}\sum_{\tau\in\omega[0,u+s)}h(u+s-\tau)duds\nonumber
\\
&\leq 2 e^{\Vert f\Vert_{L^{\infty}}}TC_{\epsilon}+\frac{2 e^{\Vert f\Vert_{L^{\infty}}}}{T}\epsilon\int_{t-T}^{t}\int_{0}^{T}N[0,u+s]h(0)duds\nonumber
\\
&\leq 2 e^{\Vert f\Vert_{L^{\infty}}}TC_{\epsilon}+2 e^{\Vert f\Vert_{L^{\infty}}}\epsilon\int_{t-T}^{t}N[0,s+T]h(0)ds\nonumber
\\
&\leq 2 e^{\Vert f\Vert_{L^{\infty}}}TC_{\epsilon}+2 e^{\Vert f\Vert_{L^{\infty}}}\epsilon T(N[0,t]+N[t,t+T])h(0).\nonumber
\end{align}
Therefore,
\begin{equation}
\limsup_{t\rightarrow\infty}\frac{1}{t}\log\mathbb{E}^{P^{\emptyset}}
\left[e^{\frac{1}{T}\int_{0}^{t}\int_{0}^{T}\left|e^{f(u,\theta_{s}\omega)}
-e^{f(u,\theta_{s}\omega_{t})}\right|\lambda(\theta_{s}\omega,u)duds}\right]\leq c(\epsilon),
\end{equation}
where $c(\epsilon)\rightarrow 0$ as $\epsilon\rightarrow 0$; in other words, it vanishes.

For the second term,
\begin{align}
&\frac{1}{T}\int_{0}^{t}\int_{0}^{T}(e^{f(u,\theta_{s}\omega_{t})}-1)\left|\lambda(\theta_{s}\omega_{t},u)-\lambda(\theta_{s}\omega,u)\right|duds
\\
&\leq\frac{e^{\Vert f\Vert_{L^{\infty}}}+1}{T}\nonumber
\\
&\qquad\cdot\int_{0}^{t}\int_{0}^{T}\alpha\left|\sum_{\tau\in\omega_{t}[0,u+s)
\cup(\omega_{t})^{-}}h(u+s-\tau)-\sum_{\tau\in\omega[0,u+s)}h(u+s-\tau)\right|duds\nonumber
\\
&\leq\frac{e^{\Vert f\Vert_{L^{\infty}}}+1}{T}\int_{t-T}^{t}\int_{0}^{T}\alpha\sum_{\tau\in\omega_{t}[0,u+s)}h(u+s-\tau)duds\nonumber
\\
&+\frac{e^{\Vert f\Vert_{L^{\infty}}}+1}{T}\int_{t-T}^{t}\int_{0}^{T}\alpha\sum_{\tau\in\omega[0,u+s)}h(u+s-\tau)duds\nonumber
\\
&+\frac{e^{\Vert f\Vert_{L^{\infty}}}+1}{T}\int_{0}^{t}\int_{0}^{T}\alpha\sum_{\tau\in(\omega_{t})^{-}}h(u+s-\tau)duds\nonumber
\end{align}
Assume that $h(\cdot)$ is decreasing and $\lim_{z\rightarrow\infty}\frac{\lambda(z)}{z}=0$. 
By applying Jensen's inequality twice, we can estimate the second term above,
\begin{align}
&\mathbb{E}^{P^{\emptyset}}\left[e^{\frac{e^{\Vert f\Vert_{L^{\infty}}}+1}{T}\alpha\int_{t-T}^{t}\int_{0}^{T}\int_{0}^{u+s}h(u+s-v)dN_{v}duds}\right]
\\
&\leq\frac{1}{T}\int_{t-T}^{t}\mathbb{E}^{P^{\emptyset}}\left[e^{(e^{\Vert f\Vert_{L^{\infty}}}+1)
\alpha\int_{0}^{T}\int_{0}^{u+s}h(u+s-v)dN_{v}du}\right]ds\nonumber
\\
&\leq\frac{1}{T^{2}}\int_{t-T}^{t}\int_{0}^{T}\mathbb{E}^{P^{\emptyset}}\left[e^{(e^{\Vert f\Vert_{L^{\infty}}}+1)
\alpha T\int_{0}^{u+s}h(u+s-v)dN_{v}}\right]duds\nonumber
\\
&\leq\frac{1}{T^{2}}\int_{t-T}^{t}\int_{0}^{T}\mathbb{E}^{P^{\emptyset}}\left[e^{\int_{0}^{u+s}C(\alpha,T,h)\lambda(v)dv}\right]^{1/2}duds\nonumber
\\
&\leq\frac{e^{C(\alpha,T,h)C_{\epsilon}}}{T^{2}}\int_{t-T}^{t}
\int_{0}^{T}\mathbb{E}^{P^{\emptyset}}\left[e^{\epsilon C(\alpha,T,h)N[0,u+s]h(0)}\right]^{1/2}duds\nonumber
\\
&\leq e^{C(\alpha,T,h)C_{\epsilon}}\mathbb{E}^{P^{\emptyset}}\left[e^{\epsilon C(\alpha,T,h)N[0,t+T]h(0)}\right]^{1/2}.\nonumber
\end{align}
where $C(\alpha,T,h)=\exp(2(e^{\Vert f\Vert_{L^{\infty}}}+1)\alpha T h(0))-1$. Thus,
\begin{equation}
\limsup_{t\rightarrow\infty}\frac{1}{t}\log\mathbb{E}^{P^{\emptyset}}
\left[e^{\frac{e^{\Vert f\Vert_{L^{\infty}}}+1}{T}\alpha\int_{t-T}^{t}\int_{0}^{T}\int_{0}^{u+s}h(u+s-v)dN_{v}duds}\right]=0.
\end{equation}
Similarly, we can estimate the first term.

For the third term, by Jensen's inequality, we have
\begin{align}
&\mathbb{E}^{P^{\emptyset}}\left[e^{\frac{e^{\Vert f\Vert_{L^{\infty}}}+1}{T}
\int_{0}^{t}\int_{0}^{T}\alpha\sum_{\tau\in(\omega_{t})^{-}}h(u+s-\tau)duds}\right]
\\
&\leq\frac{1}{T}\int_{0}^{T}\mathbb{E}^{P^{\emptyset}}
\left[e^{\alpha(\exp(\Vert f\Vert_{L^{\infty}})+1)\int_{0}^{t}\sum_{\tau\in(\omega_{t})^{-}}h(u+s-\tau)ds}\right]du\nonumber
\\
&\leq\mathbb{E}^{P^{\emptyset}}\left[e^{\alpha(\exp(\Vert f\Vert_{L^{\infty}})+1)\int_{0}^{t}\sum_{\tau\in(\omega_{t})^{-}}h(s-\tau)ds}\right]\nonumber
\\
&=\mathbb{E}^{P^{\emptyset}}\left[e^{\alpha(\exp(\Vert f\Vert_{L^{\infty}})+1)\int_{0}^{t}\int_{0}^{t}\sum_{k=0}^{\infty}h(s+kt+t-u)dN_{u}ds}\right]\nonumber
\end{align}
Since $h(\cdot)$ is decreasing, $\int_{kt}^{(k+1)t}h(s)ds\geq th((k+1)t)$. Thus
\begin{equation}
\sum_{k=0}^{\infty}h(s+kt+t-u)\leq h(s+t-u)+\frac{1}{t}\int_{s+t-u}^{\infty}h(v)dv.
\end{equation}
Let $C(\alpha,f)=\alpha(\exp(\Vert f\Vert_{L^{\infty}})+1)$ and $H(t)=\int_{t}^{\infty}h(s)ds$. Then,
\begin{align}
&\mathbb{E}^{P^{\emptyset}}\left[e^{\alpha(\exp(\Vert f\Vert_{L^{\infty}})+1)\int_{0}^{t}\int_{0}^{t}\sum_{k=0}^{\infty}h(s+kt+t-u)dN_{u}ds}\right]
\\
&\leq\mathbb{E}^{P^{\emptyset}}
\left[e^{C(\alpha,f)\int_{0}^{t}\int_{0}^{t}\frac{1}{t}H(s+t-u)dN_{u}ds+C(\alpha,f)\int_{0}^{t}
\left[\int_{0}^{t}h(s+t-u)ds\right]dN_{u}}\right]\nonumber
\\
&=\mathbb{E}^{P^{\emptyset}}\left[e^{\int_{0}^{t}\left[\frac{C(\alpha,f)}{t}\int_{0}^{t}H(s+t-u)ds\right]dN_{u}
+\int_{0}^{t}\left[\int_{0}^{t}C(\alpha,f)h(s+t-u)ds\right]dN_{u}}\right].\nonumber
\end{align}
Notice that
\begin{equation}
\mathbb{E}^{P^{\emptyset}}\left[e^{\int_{0}^{t}\left[\frac{C(\alpha,f)}{t}\int_{0}^{t}H(s+t-u)ds\right]dN_{u}}\right]
\leq\mathbb{E}^{P^{\emptyset}}\left[e^{\left[\frac{C(\alpha,f)}{t}\int_{0}^{t}H(s)ds\right]N_{t}}\right],
\end{equation}
where $\frac{C(\alpha,f)}{t}\int_{0}^{t}H(s)ds\rightarrow 0$ as $t\rightarrow\infty$, which implies that
\begin{equation}
\limsup_{t\rightarrow\infty}\frac{1}{t}\log\mathbb{E}^{P^{\emptyset}}\left[e^{\int_{0}^{t}
\left[\frac{C(\alpha,f)}{t}\int_{0}^{t}H(s+u)ds\right]dN_{u}}\right]=0.
\end{equation}
Moreover,
\begin{align}
&\mathbb{E}^{P^{\emptyset}}\left[e^{\int_{0}^{t}\left[\int_{0}^{t}C(\alpha,f)h(s+t-u)ds\right]dN_{u}}\right]
\\
&\leq\mathbb{E}^{P^{\emptyset}}\left[e^{\int_{0}^{t}(e^{2\int_{0}^{t}C(\alpha,f)h(s+t-u)ds}-1)\lambda(u)du}\right]^{1/2}\nonumber
\\
&\leq e^{\frac{1}{2}C_{\epsilon}\int_{0}^{t}(e^{2\int_{0}^{t}C(\alpha,f)h(s+t-u)ds}-1)du}
\mathbb{E}^{P^{\emptyset}}\left[e^{\int_{0}^{t}(e^{\int_{0}^{t}2C(\alpha,f)h(s+t-u)ds}-1)\epsilon\sum_{\tau<u}h(u-\tau)du}\right]^{1/2}\nonumber
\\
&\leq
e^{\frac{1}{2}C_{\epsilon}\int_{0}^{t}(e^{2C(\alpha,f)H(t-u)}-1)du}\mathbb{E}^{P^{\emptyset}}
\left[e^{\int_{0}^{t}(e^{2C(\alpha,f)\Vert h\Vert_{L^{1}}}-1)\epsilon\sum_{\tau<u}h(u-\tau)du}\right]^{1/2}\nonumber
\\
&\leq
e^{\frac{1}{2}C_{\epsilon}\int_{0}^{t}(e^{2C(\alpha,f)H(u)}-1)du}\mathbb{E}^{P^{\emptyset}}
\left[e^{\epsilon(e^{2C(\alpha,f)\Vert h\Vert_{L^{1}}}-1)\Vert h\Vert_{L^{1}}N_{t}}\right]^{1/2}\nonumber
\end{align}
Notice that it holds for any $\epsilon>0$ and that $\frac{1}{t}\int_{0}^{t}(e^{2C(\alpha,f)H(u)}-1)du\rightarrow 0$ as $t\rightarrow\infty$, 
which implies
\begin{equation}
\limsup_{t\rightarrow\infty}\frac{1}{t}\log\mathbb{E}^{P^{\emptyset}}\left[e^{\int_{0}^{t}\left[\int_{0}^{t}C(\alpha,f)h(s+t-u)ds\right]dN_{u}}\right]=0.
\end{equation}

Putting all these things together and applying H\"{o}lder's inequality several times, we find that for any $q>0$, $T>0$ and $F\in\mathcal{C}_{T}$,
\begin{equation}
\limsup_{t\rightarrow\infty}\frac{1}{t}\log\mathbb{E}^{P^{\emptyset}}
\left[\exp\left\{q\left|\frac{1}{T}\int_{0}^{t}F(\theta_{s}\omega_{t})ds-\frac{1}{T}\int_{0}^{t}F(\theta_{s}\omega)ds\right|\right\}\right]=0.
\end{equation}
\end{proof}

\begin{lemma}\label{finitetoinfinite}
\begin{equation}
\lim_{T\rightarrow\infty}\frac{1}{T}\sup_{F\in\mathcal{C}_{T}}\int_{\Omega}F(\omega)Q(d\omega)\geq H(Q).
\end{equation}
\end{lemma}

\begin{proof}
Assume $H(Q)<\infty$. For any $\epsilon>0$, there exists some $f_{\epsilon}$ such that
\begin{equation}
\mathbb{E}^{Q}\left[\int_{0}^{1}f_{\epsilon}dN_{s}-\int_{0}^{1}(e^{f_{\epsilon}}-1)\lambda ds\right]\geq H(Q)-\epsilon.
\end{equation}
We can find a sequence $f_{T}\in\mathcal{B}\left(\mathcal{F}^{-(T-1)}_{s}\right)
\cap C(\Omega\times\mathbb{R})\rightarrow f_{\epsilon}$ as $T\rightarrow\infty$. By Fatou's lemma,
\begin{align}
&\liminf_{T\rightarrow\infty}\frac{1}{T}\sup_{F\in\mathcal{C}_{T}}\int_{\Omega}F(\omega)Q(d\omega)
\\
&\geq\liminf_{T\rightarrow\infty}\mathbb{E}^{Q}\left[\int_{0}^{1}f_{T}dN_{s}-\int_{0}^{1}(e^{f_{T}}-1)\lambda ds\right]\geq H(Q)-\epsilon.\nonumber
\end{align}
If $H(Q)=\infty$, then, for any $M>0$,
there exists some $f_{M}$ such that
\begin{equation}
\mathbb{E}^{Q}\left[\int_{0}^{1}f_{M}dN_{s}-\int_{0}^{1}(e^{f_{M}}-1)\lambda ds\right]\geq M.
\end{equation}
Repeat the same argument as in the case that $H(Q)<\infty$.
\end{proof}

\begin{lemma}\label{uppercompact}
For any compact set $A$, 
\begin{equation}
\limsup_{t\rightarrow\infty}\frac{1}{t}\log P(R_{t,\omega}\in A)\leq-\inf_{Q\in A}H(Q).
\end{equation}
\end{lemma}

\begin{proof}
Notice that
\begin{equation}
\mathbb{E}^{P^{\emptyset}}\left[e^{N[0,t]}\right]\leq\mathbb{E}^{P^{\emptyset}}\left[e^{(e^{2}-1)\int_{0}^{t}\lambda(s)ds}\right]^{1/2}
\leq\mathbb{E}^{P^{\emptyset}}\left[e^{(e^{2}-1)\epsilon h(0)N[0,t]+C_{\epsilon}(e^{2}-1)}\right]^{1/2}.
\end{equation}
By choosing $\epsilon>0$ small enough, we have $\mathbb{E}^{P^{\emptyset}}[e^{N[0,t]}]\leq e^{Ct}$ for some constant $C>0$. Therefore
\begin{equation}
\limsup_{\ell\rightarrow\infty}\limsup_{t\rightarrow\infty}\frac{1}{t}\log P^{\emptyset}\left(N[0,t]>\ell t\right)=-\infty,
\end{equation}
which implies (by comparing $\int_{\Omega}N[0,1]dR_{t,\omega}$ and $N[0,t]/t$ and the superexponential estimates in Lemma \ref{indfinite})
\begin{equation}
\limsup_{\ell\rightarrow\infty}\limsup_{t\rightarrow\infty}\frac{1}{t}\log P^{\emptyset}\left(\int_{\Omega}N[0,1]dR_{t,\omega}>\ell\right)=-\infty.
\end{equation}
Therefore, we need only to consider compact sets $A$ such that for any $Q\in A$, $\mathbb{E}^{Q}[N[0,1]]<\infty$.

Now for any $A$ compact consisting of $Q$ with $\mathbb{E}^{Q}[N[0,1]]<\infty$ and for any $F\in\mathcal{C}_{T}$ and 
for any $p,q>1$, $\frac{1}{p}+\frac{1}{q}=1$, by H\"{o}lder's inequality, Chebychev's inequality, and Lemma \ref{expectation},
\begin{align}
&P^{\emptyset}(R_{t,\omega}\in A)
\\
&\leq\mathbb{E}^{P^{\emptyset}}\left[e^{\frac{1}{pT}\int_{0}^{t}F(\theta_{s}\omega_{t})ds}\right]
\cdot\exp\left\{-\frac{t}{pT}\inf_{Q\in A}\int_{\Omega}F(\omega)Q(d\omega)\right\}\nonumber
\\
&\leq\mathbb{E}^{P^{\emptyset}}\left[e^{\frac{1}{T}\int_{0}^{t}F(\theta_{s}\omega)ds}\right]^{1/p}
\mathbb{E}^{P^{\emptyset}}\left[e^{\frac{q}{pT}\left|\int_{0}^{t}F(\theta_{s}\omega_{t})ds
-\int_{0}^{t}F(\theta_{s}\omega)ds\right|}\right]^{1/q}\nonumber
\\
&\phantom{\mathbb{E}^{P^{\emptyset}}\left[e^{\frac{1}{T}\int_{0}^{t}F(\theta_{s}\omega)ds}\right]^{1/p}\mathbb{E}^{P^{\emptyset}}}
\cdot\exp\left\{-\frac{t}{pT}\inf_{Q\in A}\int_{\Omega}F(\omega)Q(d\omega)\right\}\nonumber
\\
&\leq\mathbb{E}^{P^{\emptyset}}\left[e^{\frac{q}{pT}\left|\int_{0}^{t}F(\theta_{s}\omega_{t})ds
-\int_{0}^{t}F(\theta_{s}\omega)ds\right|}\right]^{1/q}\cdot\exp\left\{-\frac{t}{pT}\inf_{Q\in A}\int_{\Omega}F(\omega)Q(d\omega)\right\}\nonumber
\end{align}
By Lemma \ref{difference},
\begin{equation}
\limsup_{t\rightarrow\infty}\frac{1}{t}\log P^{\emptyset}(R_{t,\omega}\in A)\leq-\frac{1}{p}\inf_{Q\in A}\frac{1}{T}\int_{\Omega}F(\omega)Q(d\omega).
\end{equation}
Since it holds for any $p>1$, we get
\begin{equation}
\limsup_{t\rightarrow\infty}\frac{1}{t}\log P^{\emptyset}(R_{t,\omega}\in A)\leq-\inf_{Q\in A}\frac{1}{T}\int_{\Omega}F(\omega)Q(d\omega).
\end{equation}

For any compact $A$, given $Q\in A$ and $\epsilon>0$, by Lemma \ref{finitetoinfinite}, there exists $T_{Q}>0$ and 
$F_{Q}\in\mathcal{C}_{T_{Q}}$ such that $\frac{1}{T_{Q}}\int_{\Omega}F_{Q}(\omega)Q(d\omega)\geq\inf_{A\in Q}H(Q)-\frac{1}{2}\epsilon$. 
Since the linear integral is a continuous functional of $Q$ (see the proof of Lemma \ref{lscconvex}), 
there exists a neighborhood $G_{Q}$ of $Q$ such that $\frac{1}{T_{Q}}\int_{\Omega}F_{Q}(\omega)Q(d\omega)\geq\inf_{A\in Q}H(Q)-\epsilon$ 
for all $Q\in G_{Q}$. Since $A$ is compact, there exists $G_{Q_{1}},\ldots,G_{Q_{\ell}}$ such that $A\subset\bigcup_{j=1}^{\ell}G_{Q_{j}}$. 
Hence
\begin{equation}
\inf_{1\leq j\leq\ell}\sup_{T>0}\sup_{F\in\mathcal{C}_{T}}\inf_{Q\in G_{j}}\frac{1}{T}\int_{\Omega}F(\omega)Q(d\omega)\geq\inf_{Q\in A}H(Q)-\epsilon.
\end{equation}
Note that for any $A$ and $B$,
\begin{align}
&\limsup_{t\rightarrow\infty}\frac{1}{t}\log P(R_{t,\omega}\in A\cup B)
\\
&\leq\max\left\{\limsup_{t\rightarrow\infty}\frac{1}{t}\log 
P(R_{t,\omega}\in A),\limsup_{t\rightarrow\infty}\frac{1}{t}\log P(R_{t,\omega}\in B)\right\}.\nonumber
\end{align}
Thus, for $A\subset\bigcup_{j=1}^{\ell}G_{j}$,
\begin{equation}
\limsup_{t\rightarrow\infty}\frac{1}{t}\log P(R_{t,\omega}\in A)\leq-\inf_{1\leq j\leq\ell}
\sup_{T>0}\sup_{F\in\mathcal{C}_{T}}\inf_{Q\in G_{j}}\frac{1}{T}\int F(\omega)Q(d\omega),
\end{equation}
whence $\limsup_{t\rightarrow\infty}\frac{1}{t}\log P(R_{t,\omega}\in A)\leq-\inf_{Q\in A}H(Q)$ for any compact $A$.
\end{proof}

\begin{theorem}[Upper Bound]
For any closed set $C$, 
\begin{equation}
\limsup_{t\rightarrow\infty}\frac{1}{t}\log P(R_{t,\omega}\in C)\leq-\inf_{Q\in C}H(Q).
\end{equation}
\end{theorem}

\begin{proof}
For any closed set $C$ and compact $\mathcal{A}^{n}$ which is defined in Lemma \ref{tightness}, we have
\begin{align}
&\limsup_{t\rightarrow\infty}\frac{1}{t}\log P(R_{t,\omega}\in C)
\\
&\leq\max\left\{\limsup_{t\rightarrow\infty}\frac{1}{t}\log P(R_{t,\omega}\in C\cap\mathcal{A}^{n}),
\limsup_{t\rightarrow\infty}\frac{1}{t}\log P(R_{t,\omega}\in(\mathcal{A}^{n})^{c})\right\}.\nonumber
\end{align}
Since $C\cap\mathcal{A}_{n}$ is compact, Lemma \ref{uppercompact} implies
\begin{equation}
\limsup_{t\rightarrow\infty}\frac{1}{t}\log P(R_{t,\omega}\in C\cap\mathcal{A}^{n})\leq-\inf_{Q\in C\cap\mathcal{A}^{n}}H(Q)\leq-\inf_{Q\in C}H(Q).
\end{equation}
Furthermore, by Lemma \ref{finalestimate}, 
\begin{align}
&\limsup_{t\rightarrow\infty}\frac{1}{t}\log P(R_{t,\omega}\in(\mathcal{A}^{n})^{c})
\\
&=\limsup_{t\rightarrow\infty}\frac{1}{t}\log P\left(R_{t,\omega}\in\bigcup_{j=n}^{\infty}\mathcal{A}_{\frac{1}{j},j,j}^{c}\right)\nonumber
\\
&\leq\max_{j\geq n}\max\bigg\{\limsup_{t\rightarrow\infty}\frac{1}{t}
\log P\left(\frac{1}{t}\int_{0}^{t}\chi_{N[0,1]\geq j}(\theta_{s}\omega_{t})ds\geq\varepsilon(j)\right),\nonumber
\\
&\limsup_{t\rightarrow\infty}\frac{1}{t}\log P\left(\frac{1}{t}\int_{0}^{t}\chi_{N[0,1/j]\geq 2}(\theta_{s}\omega_{t})ds\geq (1/j)g(1/j)\right),\nonumber
\\
&\limsup_{t\rightarrow\infty}\frac{1}{t}\log 
P\left(\frac{1}{t}\int_{0}^{t}N[0,1]\chi_{N[0,1]\geq\ell}(\theta_{s}\omega_{t})ds\geq m(\ell)\right)\bigg\}\rightarrow-\infty\nonumber
\end{align}
as $n\rightarrow\infty$. Hence, 
\begin{equation}
\limsup_{t\rightarrow\infty}\frac{1}{t}\log P(R_{t,\omega}\in C)\leq-\inf_{Q\in C}H(Q).
\end{equation}
\end{proof}

\section{Superexponential Estimates}\label{superexpestimates}

In order to get the full large deviation principle, we need the upper bound inequality valid for any closed set instead of for any compact set, 
which requires some superexponential estimates.

\begin{lemma}\label{reverse}
For any $q>0$,
\begin{equation}
\limsup_{t\rightarrow\infty}\frac{1}{t}\log\mathbb{E}^{P^{\emptyset}}\left[e^{q\int_{0}^{t}h(t-s)dN_{s}}\right]=0.
\end{equation}
\end{lemma}

\begin{proof}
\begin{align}
\mathbb{E}^{P^{\emptyset}}\left[e^{q\int_{0}^{t}h(t-s)dN_{s}}\right]
&\leq\mathbb{E}^{P^{\emptyset}}\left[e^{\int_{0}^{t}(e^{2qh(t-s)}-1)\lambda(\sum_{0<\tau<s}h(s-\tau))ds}\right]^{1/2}
\\
&\leq\mathbb{E}^{P^{\emptyset}}\left[e^{(C_{\epsilon}+h(0)\epsilon N_{t})\int_{0}^{t}(e^{2qh(t-s)}-1)ds}\right]^{1/2}.\nonumber
\end{align}
Note that $\int_{0}^{t}(e^{2qh(t-s)}-1)ds=\int_{0}^{t}(e^{2qh(s)}-1)ds\in L_{1}$ since $h\in L^{1}$. Therefore,
\begin{equation}
\limsup_{t\rightarrow\infty}\frac{1}{t}\log\mathbb{E}^{P^{\emptyset}}\left[e^{q\int_{0}^{t}h(t-s)dN_{s}}\right]\leq c(\epsilon),
\end{equation}
where $c(\epsilon)\rightarrow 0$ as $\epsilon\rightarrow 0$. Since it holds for any $\epsilon$, we get the desired result.
\end{proof}

\begin{lemma}\label{indfinite}
For any $q>0$ and $T>0$,
\begin{equation}
\limsup_{t\rightarrow\infty}\frac{1}{t}\log\mathbb{E}^{P^{\emptyset}}\left[e^{qN[t,t+T]}\right]=0.
\end{equation}
Therefore, for any $\epsilon>0$,
\begin{equation}
\limsup_{t\rightarrow\infty}\frac{1}{t}\log P^{\emptyset}\left(N[t,t+T]\geq\epsilon t\right)=-\infty.
\end{equation}
\end{lemma}

\begin{proof}
By H\"{o}lder's inequality,
\begin{align}
\mathbb{E}^{P^{\emptyset}}\left[e^{qN[t,t+T]}\right]&\leq\mathbb{E}^{P^{\emptyset}}
\left[e^{(e^{2q}-1)\int_{t}^{t+T}\lambda(\sum_{0<\tau<s}h(s-\tau))ds}\right]^{1/2}
\\
&\leq e^{\frac{1}{2}(e^{2q}-1)C_{\epsilon}T}\cdot\mathbb{E}^{P^{\emptyset}}
\left[e^{\epsilon(e^{2q}-1)h(0)N[t,t+T]+\epsilon(e^{2q}-1)\int_{0}^{t}h(t-s)dN_{s}}\right]^{1/2}\nonumber
\\
&\leq e^{\frac{1}{2}(e^{2q}-1)C_{\epsilon}T}\cdot\mathbb{E}^{P^{\emptyset}}
\left[e^{2\epsilon(e^{2q}-1)h(0)N[t,t+T]}\right]^{1/4}\mathbb{E}^{P^{\emptyset}}\nonumber
\\
&\phantom{\leq e^{\frac{1}{2}(e^{2q}-1)C_{\epsilon}T}\cdot\mathbb{E}^{P^{\emptyset}}}
\cdot\left[e^{2\epsilon(e^{2q}-1)\int_{0}^{t}h(t-s)dN_{s}}\right]^{1/4}.\nonumber
\end{align}
Choose $\epsilon<q[2(e^{2q}-1)h(0)]^{-1}$. Then
\begin{equation}
\mathbb{E}^{P^{\emptyset}}\left[e^{qN[t,t+T]}\right]^{3/4}\leq e^{\frac{1}{2}(e^{2q}-1)C_{\epsilon}T}
\cdot\mathbb{E}^{P^{\emptyset}}\left[e^{2\epsilon(e^{2q}-1)\int_{0}^{t}h(t-s)dN_{s}}\right]^{1/4}.
\end{equation}
Lemma \ref{reverse} completes the proof.
\end{proof}

\begin{lemma}\label{midestimate}
We have the following superexponential estimates.

(i) For any $\epsilon>0$,
\begin{equation}
\limsup_{\delta\rightarrow 0}\limsup_{t\rightarrow\infty}\frac{1}{t}\log P\left(\frac{1}{\delta t}
\int_{0}^{t}\chi_{N[0,\delta]\geq 2}(\theta_{s}\omega)ds\geq\epsilon\right)=-\infty.
\end{equation}

(ii) For any $\epsilon>0$,
\begin{equation}
\limsup_{M\rightarrow\infty}\limsup_{t\rightarrow\infty}\frac{1}{t}
\log P\left(\frac{1}{t}\int_{0}^{t}\chi_{N[0,1]\geq M}(\theta_{s}\omega)ds\geq\epsilon\right)=-\infty.
\end{equation}

(iii) For any $\epsilon>0$,
\begin{equation}
\limsup_{\ell\rightarrow\infty}\limsup_{t\rightarrow\infty}\frac{1}{t}
\log P\left(\frac{1}{t}\int_{0}^{t}N[0,1]\chi_{N[0,1]\geq\ell}(\theta_{s}\omega)ds\geq \epsilon\right)=-\infty.
\end{equation}
\end{lemma}

\begin{proof}
(i) Define
\begin{equation}
N_{\ell'}[0,t]=\int_{0}^{t}\chi_{\lambda(s)<\ell'}dN_{s},\quad\hat{N}_{\ell'}[0,t]=\int_{0}^{t}\chi_{\lambda(s)\geq\ell'}dN_{s}.
\end{equation}
Then $N[0,t]=N_{\ell'}[0,t]+\hat{N}_{\ell'}[0,t]$ and $N_{\ell'}[0,t]$ has compensator 
$\int_{0}^{t}\lambda(s)\chi_{\lambda(s)<\ell'}ds$ and $\hat{N}_{\ell'}[0,t]$ has 
compensator $\int_{0}^{t}\lambda(s)\chi_{\lambda(s)\geq\ell'}ds$. Notice that
\begin{equation}
\chi_{N[0,\delta]\geq 2}\leq\chi_{N_{\ell'}[0,\delta]\geq 2}+\chi_{\hat{N}_{\ell'}[0,\delta]\geq 1}.
\end{equation}
It is clear that $N_{\ell'}$ is dominated by the usual Poisson process with rate $\ell'$. By Lemma \ref{constrate},
\begin{equation}
\limsup_{\delta\rightarrow 0}\limsup_{t\rightarrow\infty}\frac{1}{t}
\log P\left(\frac{1}{\delta t}\int_{0}^{t}\chi_{N_{\ell'}[0,\delta]\geq 2}(\theta_{s}\omega)ds\geq\frac{\epsilon}{2}\right)=-\infty.
\end{equation}
On the other hand,
\begin{align}
\frac{1}{\delta}\int_{0}^{t}\chi_{\hat{N}_{\ell'}[0,\delta]\geq 1}(\theta_{s}\omega)ds
&=\frac{1}{\delta}\int_{0}^{t}\chi_{\hat{N}_{\ell'}[s,s+\delta]\geq 1}(\omega)ds
\\
&\leq\frac{1}{\delta}\int_{0}^{t}\hat{N}_{\ell'}[s,s+\delta]ds\nonumber
\\
&=\frac{1}{\delta}\int_{\delta}^{t+\delta}\hat{N}_{\ell'}[0,s]ds-\frac{1}{\delta}\int_{0}^{t}\hat{N}_{\ell'}[0,s]ds\nonumber
\\
&\leq\hat{N}_{\ell'}[0,t]+N[t,t+\delta].\nonumber
\end{align}
By Lemma \ref{indfinite}, we have
\begin{equation}
\limsup_{t\rightarrow\infty}\frac{1}{t}\log P\left(\frac{1}{t}N[t,t+\delta]\geq\frac{\epsilon}{4}\right)=-\infty,
\end{equation}
for any $\delta>0$. Hence
\begin{equation}
\limsup_{\delta\rightarrow 0}\limsup_{t\rightarrow\infty}\frac{1}{t}\log P\left(\frac{1}{t}N[t,t+\delta]\geq\frac{\epsilon}{4}\right)=-\infty.
\end{equation}
Finally, for some positive $h(\ell')$ to be chosen later, 
\begin{align}
P\left(\frac{1}{t}\hat{N}_{\ell'}[0,t]\geq\frac{\epsilon}{4}\right)&\leq\mathbb{E}\left[e^{h(\ell')\hat{N}_{\ell'}[0,t]}\right]e^{-t h(\ell')\epsilon/4}
\\
&\leq\mathbb{E}\left[e^{(e^{2h(\ell')}-1)\int_{0}^{t}\lambda(s)\chi_{\lambda(s)\geq\ell'}ds}\right]^{1/2}e^{-t h(\ell')\epsilon/4}.\nonumber
\end{align}
Let $f(z)=\frac{z}{\lambda(z)}$. Then $f(z)\rightarrow\infty$ as $z\rightarrow\infty$. Let $Z_{s}=\sum_{\tau\in\omega[0,s]}h(s-\tau)$. 
Then, by the definition of $\lambda(s)$ and abusing the notation a little bit, we see that $\lambda(s)=\lambda(Z_{s})$. 
Since $\lambda(\cdot)$ is increasing, its inverse function $\lambda^{-1}$ exists and $\lambda^{-1}(\ell')\rightarrow\infty$ 
as $\ell'\rightarrow\infty$. We have
\begin{align}
\mathbb{E}\left[e^{(e^{2h(\ell')}-1)\int_{0}^{t}\lambda(s)\chi_{\lambda(s)\geq\ell'}ds}\right]^{1/2}
&\leq\mathbb{E}\left[e^{(e^{2h(\ell')}-1)\int_{0}^{t}\lambda(Z_{s})\chi_{Z_{s}\geq\lambda^{-1}(\ell')}ds}\right]^{1/2}
\\
&\leq\mathbb{E}\left[e^{(e^{2h(\ell')}-1)\int_{0}^{t}\lambda(Z_{s})\frac{f(Z_{s})}{\inf_{z\geq\ell'}f(\lambda^{-1}(z))}ds}\right]^{1/2}.
\nonumber
\end{align}
It is clear that $\lim_{\ell'\rightarrow\infty}\inf_{z\geq\ell'}f(\lambda^{-1}(z))=\infty$. Choose
\begin{equation}
h(\ell')=\frac{1}{2}\log\left[\inf_{z\geq\ell'}f(\lambda^{-1}(z))+1\right].
\end{equation}
Then $h(\ell')\rightarrow\infty$ as $\ell'\rightarrow\infty$ and
\begin{align}
\mathbb{E}\left[e^{(e^{2h(\ell')}-1)\int_{0}^{t}\lambda(Z_{s})\frac{f(Z_{s})}{\inf_{z\geq\ell'}f(\lambda^{-1}(z))}ds}\right]^{1/2}
&=\mathbb{E}\left[e^{\int_{0}^{t}Z_{s}ds}\right]^{1/2}
\\
&=\mathbb{E}\left[e^{\int_{0}^{t}\sum_{\tau\in\omega[0,s]}h(s-\tau)ds}\right]^{1/2}\nonumber
\\
&\leq\mathbb{E}\left[e^{\Vert h\Vert_{L^{1}}N_{t}}\right]^{1/2}.\nonumber
\end{align}
Hence,
\begin{equation}
\limsup_{\ell'\rightarrow\infty}\limsup_{\delta\rightarrow 0}\limsup_{t\rightarrow\infty}\frac{1}{t}
\log P\left(\frac{1}{t}\hat{N}_{\ell'}[0,t]\geq\frac{\epsilon}{4}\right)=-\infty.
\end{equation}

(ii) It is easy to see that (iii) implies (ii).

(iii) Observe first that
\begin{align}
N[s,s+1]\chi_{N[s,s+1]\geq\ell}&\leq N_{\ell'}[s,s+1]\chi_{N_{\ell'}[s,s+1]\geq\frac{\ell}{2}}
+\hat{N}_{\ell'}[s,s+1]\chi_{\hat{N}_{\ell'}[s,s+1]\geq\frac{\ell}{2}}
\\
&+\frac{\ell}{2}\chi_{N_{\ell'}[s,s+1]\geq\frac{\ell}{2}}+\frac{\ell}{2}\chi_{\hat{N}_{\ell'}[s,s+1]\geq\frac{\ell}{2}}.\nonumber
\end{align}
For the first term, notice that $N_{\ell'}$ is dominated by a usual Poisson process with rate $\ell'$. 
Thus, by Lemma \ref{constrateIII},
\begin{equation}
\limsup_{\ell\rightarrow\infty}\limsup_{t\rightarrow\infty}\frac{1}{t}
\log P\left(\frac{1}{t}\int_{0}^{t}N_{\ell'}[s,s+1]\chi_{N_{\ell'}[s,s+1]\geq\frac{\ell}{2}}(\omega)ds\geq\frac{\epsilon}{4}\right)=-\infty.
\end{equation}
For the second term, $\hat{N}_{\ell'}[s,s+1]\chi_{\hat{N}_{\ell'}[s,s+1]\geq\frac{\ell}{2}}\leq\hat{N}_{\ell'}[s,s+1]$ and
\begin{equation}
\int_{0}^{t}\hat{N}_{\ell'}[s,s+1]ds\leq\hat{N}_{\ell'}[0,t]+N[t,t+1].
\end{equation}
By Lemma \ref{indfinite},
\begin{equation}
\limsup_{t\rightarrow\infty}\frac{1}{t}\log P\left(\frac{1}{t}N[t,t+1]\geq\frac{\epsilon}{8}\right)=-\infty,
\end{equation}
and by the same argument as in (i),
\begin{equation}
\limsup_{\ell'\rightarrow\infty}\limsup_{\ell\rightarrow\infty}\limsup_{t\rightarrow\infty}\frac{1}{t}
\log P\left(\frac{1}{t}\hat{N}_{\ell'}[0,t]\geq\frac{\epsilon}{8}\right)=-\infty.
\end{equation}
For the third term, notice that
\begin{equation}
\int_{0}^{t}\frac{\ell}{2}\chi_{N_{\ell'}[s,s+1]\geq\frac{\ell}{2}}ds\leq\int_{0}^{t}N_{\ell'}[s,s+1]
\chi_{N_{\ell'}[s,s+1]\geq\frac{\ell}{2}}(\omega)ds.
\end{equation}
So we can get the same superexponential estimate as before. Finally, for the fourth term,
\begin{equation}
\int_{0}^{t}\frac{\ell}{2}\chi_{\hat{N}_{\ell'}[s,s+1]\geq\frac{\ell}{2}}ds\leq\int_{0}^{t}\hat{N}_{\ell'}[s,s+1](\omega)ds.
\end{equation}
We can get the same superexponential estimate as before.
\end{proof}

\begin{lemma}\label{constrate}
Assume $N_{t}$ is a Poisson process with constant rate $\lambda$. Then for any $\epsilon>0$,
\begin{equation}
\limsup_{\delta\rightarrow 0}\limsup_{t\rightarrow\infty}\frac{1}{t}\log\mathbb{P}\left(\frac{1}{\delta t}
\int_{0}^{t}\chi_{N[s,s+\delta]\geq 2}(\omega)ds\geq\epsilon\right)=-\infty.
\end{equation}
\end{lemma}

\begin{proof}
Let $f(\delta,\omega)=\frac{1}{h(\delta)}\chi_{N[0,\delta]\geq 2}(\omega)$, where $h(\delta)$ is to be chosen later. 
By Jensen's inequality and stationarity and independence of increments of the Poisson process,
\begin{align}
\mathbb{E}\left[e^{\int_{0}^{t}\frac{1}{\delta}f(\delta,\theta_{s}\omega)ds}\right]
&\leq\mathbb{E}\left[e^{\frac{1}{\delta}\int_{0}^{\delta}\sum_{j=0}^{[t/\delta]}f(\delta,\theta_{s+j\delta}\omega)ds}\right]
\\
&\leq\mathbb{E}\left[\frac{1}{\delta}\int_{0}^{\delta}e^{\sum_{j=0}^{[t/\delta]}f(\delta,\theta_{s+j\delta}\omega)}ds\right]\nonumber
\\
&=\mathbb{E}\left[e^{\sum_{j=0}^{[t/\delta]}f(\delta,\theta_{j\delta}\omega)}\right]\nonumber
\\
&=\mathbb{E}\left[e^{f(\delta,\omega)}\right]^{[t/\delta]+1}\nonumber
\\
&=\left\{e^{1/h(\delta)}(1-e^{-\lambda\delta}-\lambda\delta e^{-\lambda\delta})+e^{-\lambda\delta}
+\lambda\delta e^{-\lambda\delta}\right\}^{[t/\delta]+1}\nonumber
\\
&\leq (M'e^{1/h(\delta)}\lambda^{2}\delta^{2}+1)^{[t/\delta]+1},\nonumber
\end{align}
for some $M'>0$. Choose $h(\delta)=\frac{1}{\log(1/\delta)}$. Then,
\begin{equation}
\mathbb{E}\left[e^{\int_{0}^{t}\frac{1}{\delta}f(\delta,\theta_{s}\omega)ds}\right]\leq(M'\delta+1)^{[t/\delta]+1}\leq e^{Mt},
\end{equation}
for some $M>0$. Therefore, by Chebychev's inequality,
\begin{equation}
\limsup_{t\rightarrow\infty}\frac{1}{t}\log\mathbb{P}\left(\frac{1}{\delta h(\delta)t}\int_{0}^{t}
\chi_{N[s,s+\delta]\geq 2}(\omega)ds\geq\frac{\epsilon}{h(\delta)}\right)
\leq M-\frac{\epsilon}{h(\delta)},
\end{equation}
which holds for any $\delta>0$. Letting $\delta\rightarrow 0$, we get the desired result.
\end{proof}

\begin{lemma}\label{constrateIII}
Assume $N_{t}$ is a Poisson process with constant rate $\lambda$. Then for any $\epsilon>0$,
\begin{equation}
\limsup_{\ell\rightarrow\infty}\limsup_{t\rightarrow\infty}\frac{1}{t}
\log P\left(\frac{1}{t}\int_{0}^{t}N[0,1]\chi_{N[0,1]\geq\ell}(\theta_{s}\omega)ds\geq \epsilon\right)=-\infty.
\end{equation}
\end{lemma}

\begin{proof}
Let $h(\ell)$ be some function of $\ell$ to be chosen later. Following the same argument as in the proof of Lemma \ref{constrate}, we have
\begin{align}
&\mathbb{P}\left(h(\ell)\int_{0}^{t}N[0,1]\chi_{N[0,1]\geq\ell}(\theta_{s}\omega)ds\geq\epsilon h(\ell)t\right)
\\
&\leq\mathbb{E}\left[e^{h(\ell)\int_{0}^{t}N[0,1]\chi_{N[0,1]\geq\ell}(\theta_{s}\omega)ds}\right]e^{-\epsilon h(\ell)t}\nonumber
\\
&\leq\mathbb{E}\left[e^{h(\ell)N[0,1]\chi_{N[0,1]\geq\ell}}\right]^{[t]+1}e^{-\epsilon h(\ell)t}\nonumber
\\
&=\left\{\mathbb{P}(N[0,1]<\ell)+\sum_{k=\ell}^{\infty}e^{h(\ell)k}e^{-\lambda}\frac{\lambda^{k}}{k!}\right\}^{[t]+1}e^{-\epsilon h(\ell)t}\nonumber
\\
&\leq\left\{1+C_{1}\sum_{k=\ell}^{\infty}e^{h(\ell)k+\log(\lambda)k-\log(k)k}\right\}^{[t]+1}e^{-\epsilon h(\ell)t}\nonumber
\\
&\leq\left\{1+C_{2}e^{h(\ell)\ell+\log(\lambda)\ell-\log(\ell)\ell}\right\}^{[t]+1}e^{-\epsilon h(\ell)t}.\nonumber
\end{align}
Choosing $h(\ell)=(\log(\ell))^{1/2}$ will do the work.
\end{proof}

The following Lemma \ref{finalestimate} provides us the superexponential estimates that we need. These superexponential
estimates have basically been done in Lemma \ref{midestimate}. The difference is that in the statement in Lemma \ref{midestimate},
we used $\omega$ and in Lemma \ref{finalestimate} it is changed to $\omega_{t}$ which is what we needed.
Lemma \ref{finalestimate} has three statements. Part (i) says if you start with a sequence of simple point processes,
the limiting point process may not be simple, but this has probability that is superexponentially small. Part (ii) is the
usual superexponential we would expect if $\mathcal{M}_{S}(\Omega)$ were equipped with weak topology. But since
we are using a strengthened weak topology with the convergence of first moment as well, we will also need Part (iii).

\begin{lemma}\label{finalestimate}
We have the following superexponential estimates.

(i) For some $g(\delta)\rightarrow 0$ as $\delta\rightarrow 0$,
\begin{equation}
\limsup_{\delta\rightarrow 0}\limsup_{t\rightarrow\infty}\frac{1}{t}\log P\left(\frac{1}{\delta t}
\int_{0}^{t}\chi_{N[0,\delta]\geq 2}(\theta_{s}\omega_{t})ds\geq g(\delta)\right)=-\infty.
\end{equation}

(ii) For some $\varepsilon(M)\rightarrow 0$ as $M\rightarrow\infty$,
\begin{equation}
\limsup_{M\rightarrow\infty}\limsup_{t\rightarrow\infty}\frac{1}{t}\log P\left(\frac{1}{t}\int_{0}^{t}
\chi_{N[0,1]\geq M}(\theta_{s}\omega_{t})ds\geq\varepsilon(M)\right)=-\infty.
\end{equation}

(iii) For some $m(\ell)\rightarrow 0$ as $\ell\rightarrow\infty$,
\begin{equation}
\limsup_{\ell\rightarrow\infty}\limsup_{t\rightarrow\infty}\frac{1}{t}\log P\left(\frac{1}{t}
\int_{0}^{t}N[0,1]\chi_{N[0,1]\geq\ell}(\theta_{s}\omega_{t})ds\geq m(\ell)\right)=-\infty.
\end{equation}
\end{lemma}

\begin{proof}
We can replace the $\epsilon$ in the statement of Lemma \ref{midestimate} by $g(\delta)$, $\varepsilon(M)$ 
and $m(\ell)$ by a standard analysis argument. Here, we can also replace the $\omega$ in Lemma \ref{midestimate} by $\omega_{t}$ since
\begin{equation}
\left|\int_{0}^{t}\chi_{N[0,\delta]\geq 2}(\theta_{s}\omega_{t})ds-\int_{0}^{t}\chi_{N[0,\delta]\geq 2}(\theta_{s}\omega)ds\right|\leq 2\delta,
\end{equation}
\begin{equation}
\left|\int_{0}^{t}\chi_{N[0,1]\geq M}(\theta_{s}\omega_{t})ds-\int_{0}^{t}\chi_{N[0,1]\geq M}(\theta_{s}\omega)ds\right|\leq 2,
\end{equation}
and
\begin{align}
&\left|\int_{0}^{t}N[0,1]\chi_{N[0,1]\geq\ell}(\theta_{s}\omega_{t})ds-\int_{0}^{t}N[0,1]\chi_{N[0,1]\geq\ell}(\theta_{s}\omega)ds\right|
\\
&\leq\int_{t-1}^{t}N[s,s+1](\omega)ds+\int_{t-1}^{t}N[s,s+1](\omega_{t})ds\nonumber
\\
&\leq N[t-1,t+1](\omega)+N[t-1,t+1](\omega_{t})\nonumber
\\
&=N[t-1,t+1](\omega)+N[t-1,t](\omega)+N[0,1](\omega).\nonumber
\end{align}
By Lemma \ref{indfinite}, we have the superexponential estimate, for any $\epsilon>0$,
\begin{equation}
\limsup_{t\rightarrow\infty}\frac{1}{t}\log P\left(\frac{1}{t}\left\{N[t-1,t+1](\omega)+N[t-1,t](\omega)+N[0,1](\omega)\right\}\geq\epsilon\right)=-\infty.
\end{equation}
\end{proof}

\begin{lemma}\label{tightness}
For any $\delta, M>0,\ell>0$, define
\begin{align}
&\mathcal{A}_{\delta}=\left\{Q\in\mathcal{M}_{S}(\Omega): Q(N[0,\delta]\geq 2)\leq\delta g(\delta)\right\},
\\
&\mathcal{A}_{M}=\left\{Q\in\mathcal{M}_{S}(\Omega): Q(N[0,1]\geq M)\leq\varepsilon(M)\right\},\nonumber
\\
&\mathcal{A}_{\ell}=\left\{Q\in\mathcal{M}_{S}(\Omega):\int_{N[0,1]\geq\ell}N[0,1]dQ\leq m(\ell)\right\},\nonumber
\end{align}
where $\varepsilon(M)\rightarrow 0$ as $M\rightarrow\infty$, $m(\ell)\rightarrow 0$ as $\ell\rightarrow\infty$ 
and $g(\delta)\rightarrow 0$ as $\delta\rightarrow 0$. 
Let $\mathcal{A}_{\delta,M,\ell}=\mathcal{A}_{\delta}\cap\mathcal{A}_{M}\cap\mathcal{A}_{\ell}$ and
\begin{equation}
\mathcal{A}^{n}=\bigcap_{j=n}^{\infty}\mathcal{A}_{\frac{1}{j},j,j}.
\end{equation}
Then $\mathcal{A}^{n}$ is compact.
\end{lemma}

\begin{proof}
Observe that for $\beta>0$, the sets
\begin{equation}
K_{\beta}=\bigcap_{k=1}^{\infty}\left\{\omega:\{N[-k,-(k-1)](\omega)\leq\beta\ell_{k}\}\cap\{N[k-1,k](\omega)\leq\beta\ell_{k}\}\right\}
\end{equation}
are relatively compact in $\Omega$. Let $\overline{K_{\beta}}$ be the closure of $K_{\beta}$, 
which is then compact.

For any $Q\in\mathcal{A}^{n}$, $Q(N[0,1]\geq M)\leq\epsilon(M)$ for any $M\geq n$. 
We can choose $\beta$ big enough and an increasing sequence $\ell_{k}$ such that $\beta\ell_{1}\geq n$ 
and $\infty>\sum_{k=1}^{\infty}\epsilon(\beta\ell_{k})\rightarrow 0$ as $\beta\rightarrow\infty$, uniformly for $Q\in\mathcal{A}^{n}$,
\begin{align}
Q\left(\overline{K_{\beta}}^{c}\right)&\leq Q(K_{\beta}^{c})
\\
&=Q\left(\bigcup_{k=1}^{\infty}\{N[-k,-(k-1)](\omega)>\beta\ell_{k}\}\cap\{N[k-1,k](\omega)>\beta\ell_{k}\}\right)\nonumber
\\
&\leq\sum_{k=1}^{\infty}\left\{Q(N[-(k-1),-k]>\beta\ell)+Q(N[k-1,k]>\beta\ell_{k})\right\}\nonumber
\\
&=2\sum_{k=1}^{\infty}Q(N[0,1]>\beta\ell_{k})\nonumber
\\
&\leq 2\sum_{k=1}^{\infty}\epsilon(\beta\ell_{k})\rightarrow 0\nonumber
\end{align}
as $\beta\rightarrow\infty$. Therefore, $\mathcal{A}^{n}$ is tight in the weak topology and by 
Prokhorov theorem $\mathcal{A}^{n}$ is precompact in the weak topology. 
In other words, for any sequence in $\mathcal{A}^{n}$, there exists a subsequence, say $Q_{n}$ 
such that $Q_{n}\rightarrow Q$ weakly as $n\rightarrow\infty$ for some $Q$. By the definition of $\mathcal{A}^{n}$, $Q_{n}$ 
are uniformly integrable, which implies that $\int N[0,1]dQ_{n}\rightarrow\int N[0,1]dQ$ 
as $n\rightarrow\infty$. It is also easy to see that $\mathcal{A}^{n}$ is closed by checking that 
each $\mathcal{A}_{\frac{1}{j},j,j}$ is closed. That implies that $Q\in\mathcal{A}^{n}$. 
Finally, we need to check that $Q$ is a simple point process. Let $I_{j,\delta}=[(j-1)\delta,j\delta]$. 
We have for any $Q\in\mathcal{A}^{n}$,
\begin{align}
Q\left(\exists t: N[t-,t]\geq 2\right)&=Q\left(\bigcup_{k=1}^{\infty}\left\{\exists t\in[-k,k]: N[t-,t]\geq 2\right\}\right)
\\
&=Q\left(\bigcup_{k=1}^{\infty}\bigcap_{\delta>0}\bigcup_{j=-[k/\delta]+1}^{[k/\delta]}
\left\{\omega:\#\{\omega\cup I_{j,\delta}\}\geq 2\right\}\right)\nonumber
\\
&\leq\sum_{k=1}^{\infty}\inf_{\delta=\frac{1}{m},m\geq n}\sum_{j=-[k/\delta]+1}^{[k/\delta]}Q(\#\{\omega\cup I_{j,\delta}\}\geq 2)\nonumber
\\
&\leq\sum_{k=1}^{\infty}\inf_{\delta=\frac{1}{m},m\geq n}\{2[k/\delta]\delta g(\delta)\}\nonumber
\\
&=0.\nonumber
\end{align}
Hence, $\mathcal{A}^{n}$ is precompact in our topology.  Since $\mathcal{A}^{n}$ is closed, it is compact.
\end{proof}

\section{Concluding Remarks}\label{conclusion}

In this chapter, we obtained a process-level large deviation principle for a wide class of simple point processes,
i.e. nonlinear Hawkes processes. Indeed, the methods and ideas should apply to other simple point processes as well
and we should expect to get the same expression for the rate function $H(Q)$. For $H(Q)<\infty$, it should be of the form
\begin{equation}
H(Q)=\int_{\Omega}\int_{0}^{1}\lambda(\omega,s)-\hat{\lambda}(\omega,s)+\log\left(\frac{\hat{\lambda}(\omega,s)}{\lambda(\omega,s)}\right)
\hat{\lambda}(\omega,s)dsQ(d\omega),
\end{equation}
where $\lambda(\omega,s)$ is the intensity of the underlying simple point process. Now, it would be interesting
to ask for what conditions for a simple point process would guarantee the process-level large deviation principle 
that we obtained in this chapter? First, we have to assume that $\lambda(\omega,t)$ is predictable and progressively measurable.
Second, in our proof of the upper bound in this chapter, the key assumption we used about
nonlinear Hawkes process was that $\lim_{z\rightarrow\infty}\frac{\lambda(z)}{z}=0$. That is crucial to guarantee
the superexponential estimates we needed for the upper bound. If for a simple point process, we have $\lambda(\omega,t)\leq F(N(t,\omega))$
for some sublinear function $F(\cdot)$, we would expect the superexponential estimates still to work for the upper bound.
Third, it is not enough to have $\lambda(\omega,t)\leq F(N(t,\omega))$ for sublinear $F(\cdot)$ to get the full large deviation
principle. The reason is that in the proof of lower bound, in particular, in Lemma \ref{mainlower}, we need to use the fact
that any memory in  $\lambda(\omega,t)$ has memory will decay to zero over time. For nonlinear Hawkes processes,
this is guaranteed by the assumption that $\int_{0}^{\infty}h(t)dt<\infty$, which is crucial in the proof of Lemma \ref{mainlower}.
Indeed for any simple point process $P$, if you want to define $P^{\omega^{-}}$, the probability measure conditional 
on the past history $\omega^{-}$, to make sense of it, you have to have some regularities to ensure that the memory
of the history will decay to zero eventually over time.
From this perspective, nonlinear Hawkes processes form a rich and ideal class for which the process-level large deviation principle holds. 

%% file: chap4.tex
\chapter{Large Deviations for Markovian Nonlinear Hawkes Processes\label{chap:four}}

In Chapter \ref{chap:three}, we studied the large deviations for $(N_{t}/t\in\cdot)$ by proving first a process-level, i.e. level-3
large deviation principle and then applying the contraction principle. In this chapter, we will obtain an alternative
expression for the rate function of the large deviation principle of $(N_{t}/t\in\cdot)$ when $h(\cdot)$ is exponential
or sums of exponentials. The main idea is that when $h(\cdot)$ is exponential or sums of exponentials, the system
is Markovian and we can use Feynman-Kac formula to obtain an upper bound and some tilting method to get a lower bound.
The assumption $\lim_{z\rightarrow\infty}\frac{\lambda(z)}{z}=0$ will provide us the compactness in order
to apply a minmax theorem to match the lower bound and the upper bound.

\section{An Ergodic Lemma}\label{ErgodicSection}

In this section, we prove an ergodic theorem for a class of Markovian processes with jumps more general than the Markovian nonlinear Hawkes processes. 

Let $Z_{i}(t):=\sum_{\tau_{j}<t}a_{i}e^{-b_{i}(t-\tau_{j})}$, $1\leq i\leq d$, where
$b_{i}>0$, $a_{i}\neq 0$ (might be negative), and $\tau_{j}$'s are the arrivals
of the simple point process with intensity $\lambda(Z_{1}(t),\cdots,Z_{d}(t))$ at time $t$, where
$\lambda:\mathcal{Z}\rightarrow\mathbb{R}^{+}$ and $\mathcal{Z}:=\mathbb{R}^{\epsilon_{1}}\times\cdots\mathbb{R}^{\epsilon_{d}}$
is the domain for $(Z_{1}(t),\ldots,Z_{d}(t))$, 
where $\mathbb{R}^{\epsilon_{i}}:=\mathbb{R}^{+}$ or $\mathbb{R}^{-}$ depending on whether $\epsilon_{i}=+1$ or $-1$,
where $\epsilon_{i}=+1$ if $a_{i}>0$ and $\epsilon_{i}=-1$ otherwise. If we assume the exciting function to be $h(t)=\sum_{i=1}^{d}a_{i}e^{-b_{i}t}$,
then a Markovian nonlinear Hawkes process is a simple point process with intensity of the form $\lambda(\sum_{i=1}^{d}Z_{i}(t))$. 

The generator $\mathcal{A}$ for $(Z_{1}(t),\ldots,Z_{d}(t))$ is given by
\begin{equation}
\mathcal{A}f=-\sum_{i=1}^{d}b_{i}z_{i}\frac{\partial f}{\partial z_{i}}+\lambda(z_{1},\ldots,z_{d})[f(z_{1}+a_{1},\ldots,z_{d}+a_{d})-f(z_{1},\ldots,z_{d})].
\end{equation}

We want to prove the existence and uniqueness of the invariant probability measure for $(Z_{1}(t),\ldots,Z_{d}(t))$.
Here the invariance is in time.

The lecture notes \cite{Hairer} by Martin Hairer gives the criterion for the existence and uniqueness of the invariant probability measure for Markov processes.

Suppose we have a jump diffusion process with generator $\mathcal{L}$. 
If we can find $u$ such that $u\geq 0$, $\mathcal{L}u\leq C_{1}-C_{2}u$ for some constants $C_{1},C_{2}>0$, 
then, there exists an invariant probability measure. We thereby have the following lemma.

\begin{lemma}\label{ergodiclemma}
Consider $h(t)=\sum_{i=1}^{d}a_{i}e^{-b_{i}t}>0$. Let $\epsilon_{i}=+1$ if $a_{i}>0$ and $\epsilon_{i}=-1$ if $a_{i}<0$. 
Assume $\lambda(z_{1},\ldots,z_{n})\leq\sum_{i=1}^{d}\alpha_{i}|z_{i}|+\beta$, where $\beta>0$ and $\alpha_{i}>0$, $1\leq i\leq d$, 
satisfies $\sum_{i=1}^{d}\frac{|a_{i}|}{b_{i}}\alpha_{i}<1$. Then, there exists a unique invariant probability measure for $(Z_{1}(t),\ldots,Z_{d}(t))$. 
\end{lemma}

\begin{proof}
The lecture notes \cite{Hairer} by Martin Hairer gives the criterion for the existence of an invariant probability measure for Markov processes.
Suppose we have a jump diffusion process with generator $\mathcal{L}$. 
If we can find $u$ such that $u\geq 0$, $\mathcal{L}u\leq C_{1}-C_{2}u$ for some constants $C_{1},C_{2}>0$, 
then, there exists an invariant probability measure. 

Try $u(z_{1},\ldots,z_{d})=\sum_{i=1}^{d}\epsilon_{i}c_{i}z_{i}\geq 0$, where $c_{i}>0$, $1\leq i\leq d$. Then,
\begin{align}
\mathcal{A}u&=-\sum_{i=1}^{d}b_{i}\epsilon_{i}c_{i}z_{i}+\lambda(z_{1},\ldots,z_{d})\sum_{i=1}^{d}a_{i}\epsilon_{i}c_{i}
\\
&\leq-\sum_{i=1}^{d}b_{i}c_{i}|z_{i}|+\sum_{i=1}^{d}\alpha_{i}|z_{i}|\sum_{i=1}^{d}|a_{i}|c_{i}+\beta\sum_{i=1}^{d}|a_{i}|c_{i}\nonumber.
\end{align}
Taking $c_{i}=\frac{\alpha_{i}}{b_{i}}>0$, we get
\begin{align}
\mathcal{A}u&\leq-\left(1-\sum_{i=1}^{d}\frac{|a_{i}|\alpha_{i}}{b_{i}}\right)
\sum_{i=1}^{d}\alpha_{i}|z_{i}|+\beta\sum_{i=1}^{d}\frac{|a_{i}|\alpha_{i}}{b_{i}}
\\
&\leq-\min_{1\leq i\leq d}b_{i}\cdot\left(1-\sum_{i=1}^{d}\frac{|a_{i}|\alpha_{i}}{b_{i}}\right)u+\beta\sum_{i=1}^{d}\frac{|a_{i}|\alpha_{i}}{b_{i}}.\nonumber
\end{align}

Next, we will prove the uniqueness of the invariant probability measure. 
It is sufficient to prove that for any $x,y\in\mathcal{Z}_{d}$, there exist times $T_{1},T_{2}>0$ such that
$\mathcal{P}^{x}(T_{1},\cdot)$ and $\mathcal{P}^{y}(T_{2},\cdot)$ are not mutually singular. 
Here $\mathcal{P}^{x}(T,\cdot):=\mathbb{P}(Z^{x}_{T}\in\cdot)$,
where $Z^{x}_{T}$ is $Z_{T}$ starting at $Z_{0}=x$, i.e. $Z^{x}_{T}=xe^{-bT}+\sum_{\tau_{j}<T}ae^{-b(T-\tau_{j})}$.
To see this, let us prove by contradiction. If there were two distinct invariant probability measures $\mu_{1}$
and $\mu_{2}$, then there exist two disjoints sets $E_{1}$ and $E_{2}$ such that $\mu_{1}:E_{1}\rightarrow E_{1}$
and $\mu_{2}:E_{2}\rightarrow E_{2}$, see for example Varadhan \cite{VaradhanIV}. Now, we can choose $x_{1}\in E_{1}$ and $x_{2}\in E_{2}$. 
so that $\mathcal{P}^{x_{1}}(T_{1},\cdot)$
and $\mathcal{P}^{x_{2}}(T_{2},\cdot)$ are supported on $E_{1}$ and $E_{2}$ respectively for any $T_{1},T_{2}>0$, which implies
that $\mathcal{P}^{x_{1}}(T_{1},\cdot)$ and $\mathcal{P}^{x_{2}}(T_{2},\cdot)$ are mutually singular.
This leads to a contradiction.

Consider the simplest case $h(t)=ae^{-bt}$. 
Let us assume that $x>y>0$. Conditioning on the event that $Z_{t}^{x}$ and $Z_{t}^{y}$ have exactly one jump during the time interval $(0,T)$ respectively, 
the laws of $\mathcal{P}^{x}(T,\cdot)$ and $\mathcal{P}^{y}(T,\cdot)$ have positive densities on the sets 
\begin{equation}
\left((a+x)e^{-bT},xe^{-bT}+a\right)\quad\text{and}\quad\left((a+y)e^{-bT},ye^{-bT}+a\right)
\end{equation}
respectively. Choosing $T>\frac{1}{b}\log(\frac{x-y+a}{a})$, we have
\begin{equation}
\left((a+x)e^{-bT},xe^{-bT}+a\right)\bigcap\left((a+y)e^{-bT},ye^{-bT}+a\right)\neq\emptyset,
\end{equation}
which implies that $\mathcal{P}^{x}(T,\cdot)$ and $\mathcal{P}^{y}(T,\cdot)$ are not mutually singular. 

Similarly, one can show the uniqueness of the invariant probability measure for the multidimensional case. Indeed, it is easy to see
that for any $x,y\in\mathcal{Z}_{d}$, $Z_{T_{1}}^{x}$ and $Z_{T_{2}}^{y}$ hit a common point for some $T_{1}$
and $T_{2}$ after possibly different number of jumps. Here $Z_{t}^{x}:=(Z_{t}^{x_{1}},\ldots,Z_{t}^{x_{d}})\in\mathcal{Z}_{d}$
and $Z_{t}^{y}:=(Z_{t}^{y_{1}},\ldots,Z_{t}^{y_{d}})\in\mathcal{Z}_{d}$, 
where $Z_{t}^{x_{i}}=x_{i}e^{-b_{i}t}+\sum_{\tau_{j}<t}a_{i}e^{-b_{i}(t-\tau_{j})}$, $1\leq i\leq d$. 
Since $\mathcal{P}^{x}(T_{1},\cdot)$ and $\mathcal{P}^{y}(T_{2},\cdot)$ have probability densities,
$\mathcal{P}^{x}(T_{1},\cdot)$ and $\mathcal{P}^{y}(T_{2},\cdot)$ are not mutually singular for some $T_{1}$ and $T_{2}$.
\end{proof}

\section{Large Deviations for Markovian Nonlinear Hawkes Processes with Exponential Exciting Function}

We assume first that $h(t)=ae^{-bt}$, where $a,b>0$, i.e. the process $Z_{t}$ jumps upwards an amount $a$ at each point 
and decays exponentially between points with rate $b$. In this case, $Z_{t}$ is Markovian. 

Notice first that $Z_{0}=0$ and
\begin{equation}
dZ_{t}=-bZ_{t}dt+adN_{t},
\end{equation}
which implies that $N_{t}=\frac{1}{a}Z_{t}+\frac{b}{a}\int_{0}^{t}Z_{s}ds$.

We prove first the existence of the limit of the logarithmic moment generating function of $N_{t}$.

\begin{theorem}\label{mainlimit}
Assume that $\lim_{z\rightarrow\infty}\frac{\lambda(z)}{z}=0$ and that $\lambda(\cdot)$
is continuous and bounded below by some positive constant. Then,
\begin{equation}
\lim_{t\rightarrow\infty}\frac{1}{t}\log\mathbb{E}[e^{\theta N_{t}}]=\Gamma(\theta),
\end{equation}
where 
\begin{equation}
\Gamma(\theta)=\sup_{(\hat{\lambda},\hat{\pi})\in\mathcal{Q}_{e}}\left\{\int\frac{\theta b}{a}z\hat{\pi}(dz)
+\int(\hat{\lambda}-\lambda)\hat{\pi}(dz)-\int\left(\log(\hat{\lambda}/\lambda)\right)\hat{\lambda}\hat{\pi}(dz)\right\},
\end{equation}
where $\mathcal{Q}_{e}$ is defined as
\begin{equation}
\mathcal{Q}_{e}=\left\{(\hat{\lambda},\hat{\pi})\in\mathcal{Q}: \text{$\hat{\mathcal{A}}$ has unique invariant probability measure $\hat{\pi}$}\right\},
\end{equation}
where
\begin{equation}
\mathcal{Q}=\left\{(\hat{\lambda},\hat{\pi}):\hat{\pi}\in\mathcal{M}(\mathbb{R}^{+}),\int z\hat{\pi}(dz)<\infty,
\hat{\lambda}\in L^{1}(\hat{\pi}), \hat{\lambda}>0\right\},
\end{equation}
where $\mathcal{M}(\mathbb{R}^{+})$ denotes the space of probability measures on $\mathbb{R}^{+}$
and for any $\hat{\lambda}$ such that $(\hat{\lambda},\hat{\pi})\in\mathcal{Q}$, we define the generator $\hat{\mathcal{A}}$ as
\begin{equation}
\hat{\mathcal{A}}f(z)=-bz\frac{\partial f}{\partial z}+\hat{\lambda}(z)[f(z+a)-f(z)].
\end{equation}
for any $f:\mathbb{R}^{+}\rightarrow\mathbb{R}$ that is $C^{1}$, i.e. continuously differentiable.
\end{theorem}

\begin{proof}
By Lemma \ref{thetafinite},$\mathbb{E}[e^{\theta N_{t}}]<\infty$ for any $\theta\in\mathbb{R}$, also
\begin{equation}
\mathbb{E}[e^{\theta N_{t}}]=\mathbb{E}\left[e^{\frac{\theta}{a}\left(Z_{t}+b\int_{0}^{t}Z_{s}ds\right)}\right].\label{upper1}
\end{equation}
Define the set
\begin{equation}
\mathcal{U}_{\theta}=\left\{u\in C^{1}(\mathbb{R}^{+},\mathbb{R}^{+}): u(z)=e^{f(z)}, \text{ where }f\in\mathcal{F}\right\},
\end{equation}
where
\begin{align}
\mathcal{F}&=\bigg\{f: f(z)=Kz+g(z)+L, K>\frac{\theta}{a}, K,L\in\mathbb{R},
\\
&\qquad\qquad\qquad\qquad\qquad\qquad\qquad\text{ $g$ is $C_{1}$ with compact support}\bigg\}.\nonumber
\end{align}
Now for any $u\in\mathcal{U}_{\theta}$, define
\begin{equation}
M:=\sup_{z\geq 0}\frac{\mathcal{A}u(z)+\frac{\theta b}{a}zu(z)}{u(z)}.
\end{equation}
By Dynkin's formula if $M<\infty$, for $V(z):=\frac{\theta b}{a}z$, we have 
\begin{align}
\mathbb{E}\left[u(Z_{t})e^{\int_{0}^{t}V(Z_{s})ds}\right]&=u(Z_{0})
+\int_{0}^{t}\mathbb{E}\left[(\mathcal{A}u(Z_{s})+V(Z_{s})u(Z_{s}))e^{\int_{0}^{s}V(Z_{v})dv}\right]ds
\\
&\leq u(Z_{0})+M\int_{0}^{t}\mathbb{E}\left[u(Z_{s})e^{\int_{0}^{s}V(Z_{v})dv}\right]ds,\nonumber
\end{align}
which implies by Gronwall's lemma that
\begin{equation}
\mathbb{E}\left[u(Z_{t})e^{\int_{0}^{t}V(Z_{s})ds}\right]\leq u(Z_{0})e^{Mt}=u(0)e^{Mt}.\label{upper2}
\end{equation}
Observe that by the definition of $\mathcal{U}_{\theta}$, for any $u\in\mathcal{U}_{\theta}$,
we have $u(z)\geq c_{1}e^{\frac{\theta}{a}z}$ for some constant $c_{1}>0$ and therefore by \eqref{upper1} and \eqref{upper2},
\begin{equation}
\mathbb{E}\left[e^{\theta N_{t}}\right]\leq\frac{1}{c_{1}}
\mathbb{E}\left[u(Z_{t})e^{\int_{0}^{t}\frac{\theta b}{a}Z_{s}ds}\right]\leq\frac{1}{c_{1}}u(0)e^{Mt}.
\end{equation}
Hence,
\begin{equation}
\limsup_{t\rightarrow\infty}\frac{1}{t}\log\mathbb{E}\left[e^{\theta N_{t}}\right]\leq
M=\sup_{z\geq 0}\frac{\mathcal{A}u(z)+\frac{\theta b}{a}zu(z)}{u(z)},
\end{equation}
which is still true even if $M=\infty$. Since this holds for any $u\in\mathcal{U}_{\theta}$, 
\begin{equation}
\limsup_{t\rightarrow\infty}\frac{1}{t}\log\mathbb{E}\left[e^{\theta N_{t}}\right]\leq
\inf_{u\in\mathcal{U}_{\theta}}\sup_{z\geq 0}\frac{\mathcal{A}u(z)+\frac{\theta b}{a}zu(z)}{u(z)}.
\end{equation}

Define the tilted probability measure $\hat{\mathbb{P}}$ by
\begin{equation}
\frac{d\hat{\mathbb{P}}}{d\mathbb{P}}\bigg|_{\mathcal{F}_{t}}
=\exp\left\{\int_{0}^{t}(\lambda(Z_{s})-\hat{\lambda}(Z_{s}))ds+\int_{0}^{t}\log\left(\frac{\hat{\lambda}(Z_{s})}{\lambda(Z_{s})}\right)dN_{s}\right\}.
\label{tiltedP}
\end{equation}
Notice that $\hat{\mathbb{P}}$ defined in \eqref{tiltedP} is indeed a probability measure by Girsanov formula. (For the theory of absolute continuity
for point processes and their Girsanov formulas, we refer to Lipster and Shiryaev \cite{Lipster}.)

Now by Jensen's inequality, 
\begin{align}
&\liminf_{t\rightarrow\infty}\frac{1}{t}\log\mathbb{E}[e^{\theta N_{t}}]
\\
&=\liminf_{t\rightarrow\infty}\frac{1}{t}\log\hat{\mathbb{E}}\left[\exp\left\{\theta N_{t}-\log\frac{d\hat{\mathbb{P}}}{d\mathbb{P}}
\bigg|_{\mathcal{F}_{t}}\right\}\right]\nonumber
\\
&\geq\liminf_{t\rightarrow\infty}\hat{\mathbb{E}}\left[\frac{1}{t}\theta N_{t}-\frac{1}{t}\log\frac{d\hat{\mathbb{P}}}{d\mathbb{P}}
\bigg|_{\mathcal{F}_{t}}\right]\nonumber
\\
&=\liminf_{t\rightarrow\infty}\hat{\mathbb{E}}\left[\frac{1}{t}\theta N_{t}-\frac{1}{t}\int_{0}^{t}(\lambda(Z_{s})
-\hat{\lambda}(Z_{s}))ds-\int_{0}^{t}\log\left(\frac{\hat{\lambda}(Z_{s})}{\lambda(Z_{s})}\right)dN_{s}\right].\nonumber
\end{align}
Since $N_{t}-\int_{0}^{t}\hat{\lambda}(Z_{s})ds$ is a martingale under $\hat{\mathbb{P}}$, we have
\begin{equation}
\hat{\mathbb{E}}\left[\int_{0}^{t}\log\left(\frac{\hat{\lambda}(Z_{s})}{\lambda(Z_{s})}\right)(dN_{s}-\hat{\lambda}(Z_{s})ds)\right]=0.
\end{equation}
Therefore, by the ergodic theorem, (for a reference, see Chapter 16.4 of Koralov and Sinai \cite{Koralov}),
for any $(\hat{\lambda},\hat{\pi})\in\mathcal{Q}_{e}$,
\begin{align}
&\liminf_{t\rightarrow\infty}\frac{1}{t}\log\mathbb{E}[e^{\theta N_{t}}]
\\
&\geq\liminf_{t\rightarrow\infty}\hat{\mathbb{E}}\left[\frac{1}{t}\theta N_{t}-\frac{1}{t}\int_{0}^{t}(\lambda(Z_{s})
-\hat{\lambda}(Z_{s}))ds-\int_{0}^{t}\log\left(\frac{\hat{\lambda}(Z_{s})}{\lambda(Z_{s})}\right)\hat{\lambda}(Z_{s})ds\right]\nonumber
\\
&=\int\frac{\theta b}{a}z\hat{\pi}(dz)+\int(\hat{\lambda}-\lambda)\hat{\pi}(dz)-\int\left(\log(\hat{\lambda})
-\log(\lambda)\right)\hat{\lambda}\hat{\pi}(dz).\nonumber
\end{align}
Hence,
\begin{align}
&\liminf_{t\rightarrow\infty}\frac{1}{t}\log\mathbb{E}[e^{\theta N_{t}}]
\\
&\geq\sup_{(\hat{\lambda},\hat{\pi})\in\mathcal{Q}_{e}}\left\{\int\frac{\theta b}{a}z\hat{\pi}+\int(\hat{\lambda}-\lambda)\hat{\pi}
-\int\left(\log(\hat{\lambda})-\log(\lambda)\right)\hat{\lambda}\hat{\pi}\right\}.\nonumber
\end{align}

Recall that
\begin{align}
\mathcal{F}&=\bigg\{f: f(z)=Kz+g(z)+L, K>\frac{\theta}{a}, K,L\in\mathbb{R},
\\
&\qquad\qquad\qquad\qquad\qquad\qquad\qquad\text{ $g$ is $C_{1}$ with compact support}\bigg\}.\nonumber
\end{align}
We claim that
\begin{equation}
\inf_{f\in\mathcal{F}}\left\{\int\hat{\mathcal{A}}f(z)\hat{\pi}(dz)\right\}=
\begin{cases}
0 &\text{if $(\hat{\lambda},\hat{\pi})\in\mathcal{Q}_{e}$,}
\\
-\infty &\text{if $(\hat{\lambda},\hat{\pi})\in\mathcal{Q}\backslash\mathcal{Q}_{e}$.}
\end{cases}
\end{equation}
It is easy to see that for $(\hat{\lambda},\hat{\pi})\in\mathcal{Q}_{e}$, and $g$ being $C_{1}$ with compact support, $\int\mathcal{A}g\hat{\pi}=0$.
Next, we can find a sequence $f_{n}(z)\rightarrow z$ pointwise under the bound $|f_{n}(z)|\leq\alpha z+\beta$, for some $\alpha,\beta>0$, where
$f_{n}(z)$ is $C_{1}$ with compact support. But by our definition of $\mathcal{Q}$, $\int z\hat{\pi}<\infty$. So by the dominated convergence theorem, 
$\int\hat{\mathcal{A}}z\hat{\pi}=0$.
The nontrivial part is to prove that if for any $g\in\mathcal{G}=\{g(z)+L, g\text{ is $C_{1}$ with compact support}\}$ 
such that $\int\hat{\mathcal{A}}g\hat{\pi}=0$, then $(\hat{\lambda},\hat{\pi})\in\mathcal{Q}_{e}$. We can easily check the conditions 
in Echevrr\'{i}a \cite{Echeverria}. (For instance, $\mathcal{G}$ is dense in $C(\mathbb{R}^{+})$, the set of continuous and bounded functions
on $\mathbb{R}^{+}$ with limit that exists at infinity and $\hat{\mathcal{A}}$ satisfies the minimum principle, 
i.e. $\hat{\mathcal{A}}f(z_{0})\geq 0$ for any $f(z_{0})=\inf_{z\in\mathbb{R}^{+}}f(z)$. This is because at minimum, the first derivative of $f$
vanishes and $\hat{\lambda}(z_{0})(f(z_{0}+a)-f(z_{0}))\geq 0$. The other conditions in Echeverr\'{i}a \cite{Echeverria} can also
be easily verified.) Thus, Echevrr\'{i}a \cite{Echeverria} implies that $\hat{\pi}$ is an 
invariant measure. Now, our proof in Lemma \ref{ergodiclemma} shows that $\hat{\pi}$ has to be unique as well. 
Therefore, $(\hat{\lambda},\hat{\pi})\in\mathcal{Q}_{e}$. This implies that if $(\hat{\lambda},\hat{\pi})\in\mathcal{Q}\backslash\mathcal{Q}_{e}$,
there exists some $g\in\mathcal{G}$, such that $\int\hat{\mathcal{A}}g\hat{\pi}\neq 0$. Now, any constant multiplier of $g$ still belongs
to $\mathcal{G}$ and thus $\inf_{g\in\mathcal{G}}\int\hat{\mathcal{A}}g\hat{\pi}=-\infty$ and 
hence $\inf_{f\in\mathcal{F}}\int\hat{\mathcal{A}}f\hat{\pi}=-\infty$
if $(\hat{\lambda},\hat{\pi})\in\mathcal{Q}\backslash\mathcal{Q}_{e}$.

Therefore, 
\begin{align}
\liminf_{t\rightarrow\infty}\frac{1}{t}\log\mathbb{E}[e^{\theta N_{t}}]
&\geq\sup_{(\hat{\lambda},\hat{\pi})\in\mathcal{Q}}\inf_{f\in\mathcal{F}}\left\{\int\frac{\theta b}{a}z\hat{\pi}-\hat{H}(\hat{\lambda},\hat{\pi})
+\int\hat{\mathcal{A}}f\hat{\pi}\right\}
\\
&\geq\sup_{(\hat{\lambda}\hat{\pi},\hat{\pi})\in\mathcal{R}}\inf_{f\in\mathcal{F}}\left\{\int\frac{\theta b}{a}z\hat{\pi}
-\hat{H}(\hat{\lambda},\hat{\pi})+\int\hat{\mathcal{A}}f\hat{\pi}\right\},\label{switch}
\end{align}
where $\mathcal{R}=\{(\hat{\lambda}\hat{\pi},\hat{\pi}):(\hat{\lambda},\hat{\pi})\in\mathcal{Q}\}$ and
\begin{equation}
\hat{H}(\hat{\lambda},\hat{\pi})=\int\left[(\lambda-\hat{\lambda})+\log\left(\hat{\lambda}/\lambda\right)\hat{\lambda}\right]\hat{\pi}.
\end{equation}
Define
\begin{align}
F(\hat{\lambda}\hat{\pi},\hat{\pi},f)&=\int\frac{\theta b}{a}z\hat{\pi}-\hat{H}(\hat{\lambda},\hat{\pi})+\int\hat{\mathcal{A}}f\hat{\pi}
\\
&=\int\frac{\theta b}{a}z\hat{\pi}-\hat{H}(\hat{\lambda},\hat{\pi})-\int bz\frac{\partial f}{\partial z}\hat{\pi}
+\int(f(z+a)-f(z))\hat{\lambda}\hat{\pi}.\nonumber
\end{align}
Notice that $F$ is linear in $f$ and hence convex in $f$ and also
\begin{equation}
\hat{H}(\hat{\lambda},\hat{\pi})=\sup_{f\in C_{b}(\mathbb{R}^{+})}\left\{\int\left[\hat{\lambda}f+\lambda(1-e^{f})\right]\hat{\pi}\right\},
\end{equation}
where $C_{b}(\mathbb{R}^{+})$ denotes the set of bounded functions on $\mathbb{R}^{+}$. Inside the bracket above, 
it is linear in both $\hat{\pi}$ and $\hat{\lambda}\hat{\pi}$. Hence $\hat{H}$ is weakly lower semicontinuous and convex in 
$(\hat{\lambda}\hat{\pi},\hat{\pi})$. Therefore, $F$ is concave in $(\hat{\lambda}\hat{\pi},\hat{\pi})$. Furthermore, for any $f=Kz+g+L\in\mathcal{F}$,
\begin{align}
F(\hat{\lambda}\hat{\pi},\hat{\pi},f)&=\int\left(\frac{\theta}{a}-K\right)bz\hat{\pi}-\hat{H}(\hat{\lambda},\hat{\pi})
-\int bz\frac{\partial g}{\partial z}\hat{\pi}
\\
&+\int(g(z+a)-g(z))\hat{\lambda}\hat{\pi}+Ka\int\hat{\lambda}\hat{\pi}.\nonumber
\end{align}
If $\lambda_{n}\pi_{n}\rightarrow\gamma_{\infty}$ and $\pi_{n}\rightarrow\pi_{\infty}$ weakly, then, since $g$ is $C_{1}$ with compact support, we have
\begin{align}
&-\int bz\frac{\partial g}{\partial z}\pi_{n}+\int(g(z+a)-g(z))\lambda_{n}\pi_{n}+Ka\int\lambda_{n}\pi_{n}
\\
&\rightarrow-\int bz\frac{\partial g}{\partial z}\pi_{\infty}+\int(g(z+a)-g(z))\gamma_{\infty}+Ka\int\gamma_{\infty},\nonumber
\end{align}
as $n\rightarrow\infty$. Moreover, in general, if $P_{n}\rightarrow P$ weakly, then, for any $f$ which is upper semicontinuous and 
bounded from above, we have $\limsup_{n}\int fdP_{n}\leq\int fdP$. Since $\left(\frac{\theta}{a}-K\right)bz$ is continuous and nonpositive 
on $\mathbb{R}^{+}$, we have
\begin{equation}
\limsup_{n\rightarrow\infty}\int\left(\frac{\theta}{a}-K\right)bz\pi_{n}\leq\int\left(\frac{\theta}{a}-K\right)bz\pi_{\infty}.
\end{equation}
Hence, we conclude that $F$ is upper semicontinuous in the weak topology.

In order to switch the supremum and infimum in \eqref{switch}, since we have already proved that $F$ is concave, 
upper semicontinuous in $(\hat{\lambda}\hat{\pi},\hat{\pi})$ and convex in $f$, it is sufficient to prove the compactness of $\mathcal{R}$ 
to apply Ky Fan's minmax theorem (see Fan \cite{Fan}). Indeed, Jo\'{o} developed some level set method 
and proved that it is sufficient to show the compactness of the level set (see Jo\'{o} \cite{Joo} and Frenk and Kassay \cite{Frenk}). 
In other words, it suffices to prove that, for any $C\in\mathbb{R}$ and $f\in\mathcal{F}$, the level set
\begin{equation}
\left\{(\hat{\lambda}\hat{\pi},\hat{\pi})\in\mathcal{R}: \hat{H}+\int bz\frac{\partial f}{\partial z}\hat{\pi}
-\frac{\theta b}{a}z\hat{\pi}-\hat{\lambda}[f(z+a)-f(z)]\hat{\pi}\leq C\right\}\label{pairinlevelset}
\end{equation}
is compact.

Fix any $f=Kz+g+L\in\mathcal{F}$, where $K>\frac{\theta}{a}$ and $g$ is $C_{1}$ with compact support and $L$ is some constant, uniformly for any
pair $(\hat{\lambda}\hat{\pi},\hat{\pi})$ that is in the level set of \eqref{pairinlevelset}, there exists some $C_{1},C_{2}>0$ such that
\begin{align}
C_{1}&\geq\hat{H}+\left(K-\frac{\theta}{a}\right)b\int z\hat{\pi}-C_{2}\int\hat{\lambda}\hat{\pi}
\\
&\geq\int_{\hat{\lambda}\geq cz+\ell}\left[\lambda-\hat{\lambda}+\hat{\lambda}\log(\hat{\lambda}/\lambda)\right]\hat{\pi}
+\left(K-\frac{\theta}{a}\right)b\int z\hat{\pi}\nonumber
\\
&-C_{2}\int_{\hat{\lambda}\geq cz+\ell}\hat{\lambda}\hat{\pi}-C_{2}\int_{\hat{\lambda}<cz+\ell}\hat{\lambda}\hat{\pi}\nonumber
\\
&\geq\left[\min_{z\geq 0}\log\frac{cz+\ell}{\lambda(z)}-1-C_{2}\right]\int_{\hat{\lambda}\geq cz+\ell}\hat{\lambda}\hat{\pi}
+\left[-c\cdot C_{2}+\left(K-\frac{\theta}{a}\right)b\right]\int z\hat{\pi}-\ell C_{2}.\nonumber
\end{align}
We choose $0<c<\left(K-\frac{\theta}{a}\right)\frac{b}{C_{2}}$ and $\ell$ large enough so that 
$\min_{z\geq 0}\log\frac{cz+\ell}{\lambda(z)}-1-C_{2}>0$, where we used the fact that $\lim_{z\rightarrow\infty}\frac{\lambda(z)}{z}=0$
and $\min_{z}\lambda(z)>0$. 
Hence, 
\begin{equation}
\int z\hat{\pi}\leq C_{3},\quad\int_{\hat{\lambda}\geq cz+\ell}\hat{\lambda}\hat{\pi}\leq C_{4},
\end{equation}
where
\begin{equation}
C_{3}=\frac{C_{1}+\ell C_{2}}{-c\cdot C_{2}+\left(K-\frac{\theta}{a}\right)b},
\quad C_{4}=\frac{C_{1}+\ell C_{2}}{\min_{z\geq 0}\log\frac{cz+\ell}{\lambda(z)}-1-C_{2}}.
\end{equation}
Therefore, we have
\begin{equation}
\int\hat{\lambda}\hat{\pi}=\int_{\hat{\lambda}\geq cz+\ell}\hat{\lambda}\hat{\pi}+\int_{\hat{\lambda}<cz+\ell}\hat{\lambda}\hat{\pi}
\leq C_{4}+c\cdot C_{3}+\ell,
\end{equation}
and hence
\begin{equation}
\hat{H}(\hat{\lambda},\hat{\pi})\leq C_{1}+C_{2}\left[C_{4}+c\cdot C_{3}+\ell\right]<\infty.
\end{equation}

Therefore, for any $(\lambda_{n}\pi_{n},\pi_{n})\in\mathcal{R}$, we get
\begin{equation}
\lim_{\ell\rightarrow\infty}\sup_{n}\int_{z\geq\ell}\pi_{n}\leq\lim_{\ell\rightarrow\infty}\sup_{n}\frac{1}{\ell}\int z\pi_{n}
\leq\lim_{\ell\rightarrow\infty}\frac{C_{3}}{\ell}=0,
\end{equation}
which implies the tightness of $\pi_{n}$. By Prokhorov's Theorem, there exists a subsequence of $\pi_{n}$ 
which converges weakly to $\pi_{\infty}$. We also want to show that there exists some $\gamma_{\infty}$ such that 
$\lambda_{n}\pi_{n}\rightarrow\gamma_{\infty}$ weakly (passing to a subsequence if necessary). 
It is enough to show that

(i) $\sup_{n}\int\lambda_{n}\pi_{n}<\infty$.

(ii) $\lim_{\ell\rightarrow\infty}\sup_{n}\int_{z\geq\ell}\lambda_{n}\pi_{n}=0$.

(i) and (ii) will give us tightness of $\lambda_{n}\pi_{n}$ and hence implies the weak convergence for a subsequence.

Now, let us prove statements (i) and (ii).

To prove (i), notice that
\begin{equation}
\sup_{n}\int\lambda_{n}\pi_{n}=\sup_{n}\int\frac{b}{a}z\pi_{n}\leq\frac{b}{a}[C_{4}+c\cdot C_{3}+\ell]<\infty.
\end{equation}

To prove (ii), notice that $(\lambda-\lambda_{n})+\lambda_{n}\log(\lambda_{n}/\lambda)\geq 0$. That is because
$x-1-\log x\geq 0$ for any $x>0$ and hence
\begin{equation}
\lambda-\hat{\lambda}+\hat{\lambda}\log(\hat{\lambda}/\lambda)=\hat{\lambda}\left[(\lambda/\hat{\lambda})-1-\log(\lambda/\hat{\lambda})\right]\geq 0.
\end{equation}
Notice that
\begin{align}
\lim_{\ell\rightarrow\infty}\sup_{n}\int_{z\geq\ell}\lambda_{n}\pi_{n}&\leq\lim_{\ell\rightarrow\infty}
\sup_{n}\int_{\lambda_{n}<\sqrt{\lambda z},z\geq\ell}\lambda_{n}\pi_{n}
\\
&+\lim_{\ell\rightarrow\infty}\sup_{n}\int_{\lambda_{n}\geq\sqrt{\lambda z},z\geq\ell}\lambda_{n}\pi_{n}.\nonumber
\end{align}
For the first term, since $\sup_{n}\int z\pi_{n}<\infty$ and $\lim_{z\rightarrow\infty}\frac{\lambda(z)}{z}=0$,
\begin{equation}
\lim_{\ell\rightarrow\infty}\sup_{n}\int_{\lambda_{n}<\sqrt{\lambda z},z\geq\ell}\lambda_{n}\pi_{n}
\leq\lim_{\ell\rightarrow\infty}\sup_{n}\int_{z\geq\ell}\sqrt{\lambda z}\pi_{n}=0.
\end{equation}
For the second term, since $\limsup_{z\rightarrow\infty}\frac{\lambda(z)}{z}=0$,
\begin{align}
&\lim_{\ell\rightarrow\infty}\sup_{n}\int_{\lambda_{n}\geq\sqrt{\lambda z},z\geq\ell}\lambda_{n}\pi_{n}
\\
&\leq\lim_{\ell\rightarrow\infty}\sup_{n}\hat{H}(\lambda_{n},\pi_{n})\sup_{\lambda_{n}\geq\sqrt{\lambda  z}, z\geq\ell}
\frac{\lambda_{n}}{\lambda-\lambda_{n}+\lambda_{n}\log(\lambda_{n}/\lambda)}=0.\nonumber
\end{align}

Therefore, passing to some subsequence if necessary, we have $\lambda_{n}\pi_{n}\rightarrow\gamma_{\infty}$ and $\pi_{n}\rightarrow\pi_{\infty}$ weakly. 
Since we proved that $F$ is upper semicontinuous in the weak topology, the level set is compact in the weak topology. 
Therefore, we can switch the supremum and infimum in \eqref{switch} and get
\begin{align}
&\liminf_{t\rightarrow\infty}\frac{1}{t}\log\mathbb{E}\left[e^{\theta N_{t}}\right]
\\
&\geq\inf_{f\in\mathcal{F}}\sup_{\hat{\pi}:\int z\hat{\pi}<\infty}\sup_{\hat{\lambda}\in L^{1}(\hat{\pi})}
\left\{\int\frac{\theta b}{a}z\hat{\pi}+(\hat{\lambda}-\lambda)\hat{\pi}-\log(\hat{\lambda}/\lambda)\hat{\lambda}\hat{\pi}+\hat{\mathcal{A}}f\hat{\pi}\right\}
\\
&=\inf_{f\in\mathcal{F}}\sup_{\hat{\pi}:\int z\hat{\pi}<\infty}
\int\left[\frac{\theta bz}{a}+\lambda(z)(e^{f(z+a)-f(z)}-1)-bz\frac{\partial f}{\partial z}\right]\hat{\pi}(dz)\label{optimallambda}
\\
&=\inf_{f\in\mathcal{F}}\sup_{z\geq 0}\left[\frac{\theta bz}{a}+\lambda(z)(e^{f(z+a)-f(z)}-1)-bz\frac{\partial f}{\partial z}\right]\label{optimalpi}
\\
&=\inf_{f\in\mathcal{F}}\sup_{z\geq 0}\left[\frac{\theta bze^{f(z)}}{ae^{f(z)}}
+\frac{\lambda(z)}{e^{f(z)}}(e^{f(z+a)}-e^{f(z)})-\frac{bz}{e^{f(z)}}\frac{\partial e^{f(z)}}{\partial z}\right]
\\
&\geq\inf_{u\in\mathcal{U}_{\theta}}\sup_{z\geq 0}\left\{\frac{\mathcal{A}u}{u}+\frac{\theta b}{a}z\right\}.\label{fandu}
\end{align}
We need some justifications. Define $G(\hat{\lambda})=\hat{\lambda}-\log(\hat{\lambda}/\lambda)\hat{\lambda}+\hat{\mathcal{A}}f$. 
The supremum of $G(\hat{\lambda})$ is achieved when $\frac{\partial G}{\partial\hat{\lambda}}=0$ 
which implies $\hat{\lambda}=\lambda e^{f(z+a)-f(z)}$. Notice that for $f\in\mathcal{F}$, 
the optimal $\hat{\lambda}=\lambda e^{f(z+a)-f(z)}$ satisfies $\int\hat{\lambda}\hat{\pi}<\infty$ 
since $\int\lambda\hat{\pi}<\infty$ and $\int z\hat{\pi}<\infty$. This gives us \eqref{optimallambda}.
Next, let us explain \eqref{optimalpi}. For any probability measure $\hat{\pi}$, 
\begin{align}
&\int\left[\frac{\theta bz}{a}+\lambda(z)(e^{f(z+a)-f(z)}-1)-bz\frac{\partial f}{\partial z}\right]\hat{\pi}(dz)
\\
&\leq\sup_{z\geq 0}\left[\frac{\theta bz}{a}+\lambda(z)(e^{f(z+a)-f(z)}-1)-bz\frac{\partial f}{\partial z}\right],
\nonumber
\end{align}
which implies the right hand side of \eqref{optimallambda} is less or equal to the right hand side of \eqref{optimalpi}. To prove 
the other direction.
For any $f=Kz+g+L\in\mathcal{F}$, we have
\begin{align}
&\frac{\theta bz}{a}+\lambda(z)(e^{f(z+a)-f(z)}-1)-bz\frac{\partial f}{\partial z}
\\
&=\left(\frac{\theta b}{a}-Kb\right)z+\lambda(z)(e^{Ka+g(z+a)-g(z)}-1)-bz\frac{\partial g}{\partial z},\nonumber
\end{align}
which is continuous in $z$ and also bounded on $z\in[0,\infty)$ since $g$ is $C^{1}$ with compact support
and $K>\frac{\theta}{a}$ and $\lim_{z\rightarrow\infty}\frac{\lambda(z)}{z}=0$. Hence
there exists some $z^{\ast}\geq 0$ such that
\begin{align}
&\frac{\theta bz}{a}+\lambda(z)(e^{f(z+a)-f(z)}-1)-bz\frac{\partial f}{\partial z}
\\
&=\frac{\theta bz^{\ast}}{a}+\lambda(z^{\ast})(e^{f(z^{\ast}+a)-f(z^{\ast})}-1)-bz^{\ast}\frac{\partial f}{\partial z}\bigg|_{z=z^{\ast}}.\nonumber
\end{align}
Take a sequence of probability measures $\hat{\pi}_{n}$ such that it has probability density function
$n$ if $z\in[z^{\ast}-\frac{1}{2n},z^{\ast}+\frac{1}{2n}]$ and $0$ otherwise. Then, for every $n$, $\int z\hat{\pi}_{n}(dz)<\infty$.
Therefore, we have
\begin{align}
&\lim_{n\rightarrow\infty}\int\left[\frac{\theta bz}{a}+\lambda(z)(e^{f(z+a)-f(z)}-1)-bz\frac{\partial f}{\partial z}\right]\hat{\pi}_{n}(dz)
\\
&=\lim_{n\rightarrow\infty}n\int_{z^{\ast}-\frac{1}{2n}}^{z^{\ast}+\frac{1}{2n}}
\left[\frac{\theta bz}{a}+\lambda(z)(e^{f(z+a)-f(z)}-1)-bz\frac{\partial f}{\partial z}\right]dz\nonumber
\\
&=\frac{\theta bz^{\ast}}{a}+\lambda(z^{\ast})(e^{f(z^{\ast}+a)-f(z^{\ast})}-1)-bz^{\ast}\frac{\partial f}{\partial z}\bigg|_{z=z^{\ast}}
\nonumber
\\
&=\sup_{z\geq 0}\left[\frac{\theta bz}{a}+\lambda(z)(e^{f(z+a)-f(z)}-1)-bz\frac{\partial f}{\partial z}\right].\nonumber
\end{align}
We conclude that the right hand side of \eqref{optimallambda} is greater or equal to
the right hand side of \eqref{optimalpi}.

Notice that for any $f=Kz+g+L\in\mathcal{F}$,
\begin{align}
&\frac{\theta bz}{a}+\lambda(z)(e^{f(z+a)-f(z)}-1)-bz\frac{\partial f}{\partial z}
\\
&=\frac{b(\theta-Ka)}{a}z+\lambda(z)(e^{Ka+g(z+a)-g(z)}-1)-bz\frac{\partial g}{\partial z},\nonumber
\end{align}
whose supremum is achieved at some finite $z^{\ast}>0$ since $\lim_{z\rightarrow\infty}\frac{\lambda(z)}{z}=0$, 
$K>\frac{\theta}{a}$ and $g\in C^{1}$ with compact support. Hence $\int z\hat{\pi}<\infty$ is satisified for the optimal $\hat{\pi}$. 
This gives us \eqref{optimalpi}. Finally, for any $f\in\mathcal{F}$, $u=e^{f}\in\mathcal{U}_{\theta}$, which implies \eqref{fandu}. 
\end{proof}

\begin{lemma}\label{thetafinite}
Assume $\lim_{z\rightarrow\infty}\frac{\lambda(z)}{z}=0$, we have $\mathbb{E}[e^{\theta N_{t}}]<\infty$
for any $\theta\in\mathbb{R}$.
\end{lemma}

\begin{proof}
Observe that for any $\gamma\in\mathbb{R}$,
\begin{equation}
\exp\left\{\gamma N_{t}-\int_{0}^{t}(e^{\gamma}-1)\lambda(Z_{s})ds\right\}
\end{equation}
is a martinagle. Since $\lim_{z\rightarrow\infty}\frac{\lambda(z)}{z}=0$, for any $\epsilon>0$,
there exists a constant $C_{\epsilon}>0$ such that $\lambda(z)\leq C_{\epsilon}+\epsilon z$ for any $z\geq 0$. 
Also, 
\begin{align}
\int_{0}^{t}Z_{s}ds&=\int_{0}^{t}\int_{0}^{s}h(s-u)N(du)ds
\\
&=\int_{0}^{t}\left[\int_{u}^{t}h(s-u)ds\right]N(du)\nonumber
\\
&\leq\int_{0}^{t}\left[\int_{u}^{\infty}h(s-u)ds\right]N(du)=\Vert h\Vert_{L^{1}}N_{t}.\nonumber
\end{align}
Therefore, for any $\gamma>0$,
\begin{align}
1&=\mathbb{E}\left[e^{\gamma N_{t}-\int_{0}^{t}(e^{\gamma}-1)\lambda(Z_{s})ds}\right]
\\
&\geq\mathbb{E}\left[e^{\gamma N_{t}-(e^{\gamma}-1)\int_{0}^{t}(C_{\epsilon}+\epsilon Z_{s})ds}\right]\nonumber
\\
&\geq\mathbb{E}\left[e^{\gamma N_{t}-(e^{\gamma}-1)C_{\epsilon}t-(e^{\gamma}-1)\epsilon\Vert h\Vert_{L^{1}}N_{t}}\right].\nonumber
\end{align}
For any $\theta>0$, choose $\gamma>\theta$ and $\epsilon$ small enough so that $\gamma-(e^{\gamma}-1)\epsilon\Vert h\Vert_{L^{1}}\geq\theta$.
Then,
\begin{equation}
\mathbb{E}\left[e^{\theta N_{t}}\right]
\leq e^{(e^{\gamma}-1)C_{\epsilon}t}<\infty.
\end{equation}
\end{proof}

Now, we are ready to prove the large deviations result.

\begin{theorem}
Assume $\lim_{z\rightarrow\infty}\frac{\lambda(z)}{z}=0$ and that $\lambda(\cdot)$ is continuous and bounded below by some positive constant. 
Then, $(\frac{N_{t}}{t}\in\cdot)$ satisfies the large deviation principle with the rate function $I(\cdot)$ as the Fenchel-Legendre 
transform of $\Gamma(\cdot)$,
\begin{equation}
I(x)=\sup_{\theta\in\mathbb{R}}\left\{\theta x-\Gamma(\theta)\right\}.
\end{equation}
\end{theorem}

\begin{proof}
If $\limsup_{z\rightarrow\infty}\frac{\lambda(z)}{z}=0$, then the forthcoming Lemma \ref{Gammafinite} implies that $\Gamma(\theta)<\infty$ for any $\theta$. 
Thus, by G\"{a}rtner-Ellis Theorem, we have the upper bound. For G\"{a}rtner-Ellis Theorem and a general theory of large deviations, 
see for example \cite{Dembo}. To prove the lower bound, it suffices to show that for any $x>0$, $\epsilon>0$, we have
\begin{equation}
\liminf_{t\rightarrow\infty}\frac{1}{t}\log\mathbb{P}\left(\frac{N_{t}}{t}\in B_{\epsilon}(x)\right)\geq-\sup_{\theta}\{\theta x-\Gamma(\theta)\},
\end{equation}
where $B_{\epsilon}(x)$ denotes the open ball centered at $x$ with radius $\epsilon$. Let $\hat{\mathbb{P}}$ denote the 
tilted probability measure with rate $\hat{\lambda}$ defined in Theorem \ref{mainlimit}. By Jensen's inequality,
\begin{align}
&\frac{1}{t}\log\mathbb{P}\left(\frac{N_{t}}{t}\in B_{\epsilon}(x)\right)
\\
&=\frac{1}{t}\log\int_{\frac{N_{t}}{t}\in B_{\epsilon}(x)}\frac{d\mathbb{P}}{d\hat{\mathbb{P}}}d\hat{\mathbb{P}}\nonumber
\\
&=\frac{1}{t}\log\hat{\mathbb{P}}\left(\frac{N_{t}}{t}\in B_{\epsilon}(x)\right)
+\frac{1}{t}\log\left[\frac{1}{\hat{\mathbb{P}}\left(\frac{N_{t}}{t}\in B_{\epsilon}(x)\right)}
\int_{\frac{N_{t}}{t}\in B_{\epsilon}(x)}\frac{d\mathbb{P}}{d\hat{\mathbb{P}}}d\hat{\mathbb{P}}\right]\nonumber
\\
&\geq\frac{1}{t}\log\hat{\mathbb{P}}\left(\frac{N_{t}}{t}\in B_{\epsilon}(x)\right)-\frac{1}{\hat{\mathbb{P}}\left(\frac{N_{t}}{t}\in B_{\epsilon}(x)\right)}
\cdot\frac{1}{t}\hat{\mathbb{E}}\left[1_{\frac{N_{t}}{t}\in B_{\epsilon}(x)}\log\frac{d\hat{\mathbb{P}}}{d\mathbb{P}}\right].\nonumber
\end{align}
By the ergodic theorem, 
\begin{equation}
\liminf_{t\rightarrow\infty}\frac{1}{t}\log\mathbb{P}\left(\frac{N_{t}}{t}\in B_{\epsilon}(x)\right)
\geq -\Lambda(x),
\end{equation}
where
\begin{equation}\label{Lambdaofx}
\Lambda(x)=\inf_{(\hat{\lambda},\hat{\pi})\in\mathcal{Q}_{e}^{x}}\left\{\int(\lambda-\hat{\lambda})\hat{\pi}
+\int\log(\hat{\lambda}/\lambda)\hat{\lambda}\hat{\pi}\right\},
\end{equation}
and
\begin{equation}
\mathcal{Q}_{e}^{x}=\left\{(\hat{\lambda},\hat{\pi})\in\mathcal{Q}_{e}:\int\hat{\lambda}(z)\hat{\pi}(dz)=x\right\}.
\end{equation}
Notice that
\begin{align}
\Gamma(\theta)&=\sup_{(\hat{\lambda},\hat{\pi})\in\mathcal{Q}_{e}}\left\{\int\theta\hat{\lambda}\hat{\pi}+\int(\hat{\lambda}-\lambda)\hat{\pi}
-\int\log(\hat{\lambda}/\lambda)\hat{\lambda}\hat{\pi}\right\}
\\
&=\sup_{x}\sup_{(\hat{\lambda},\hat{\pi})\in\mathcal{Q}_{e}^{x}}\left\{\int\theta\hat{\lambda}\hat{\pi}+\int(\hat{\lambda}-\lambda)\hat{\pi}
-\int\log(\hat{\lambda}/\lambda)\hat{\lambda}\hat{\pi}\right\}\nonumber
\\
&=\sup_{x}\sup_{(\hat{\lambda},\hat{\pi})\in\mathcal{Q}_{e}^{x}}\left\{\int\frac{\theta b}{a}z\hat{\pi}(dz)+\int(\hat{\lambda}-\lambda)\hat{\pi}
-\int\log(\hat{\lambda}/\lambda)\hat{\lambda}\hat{\pi}\right\}\nonumber
\\
&=\sup_{x}\{\theta x-\Lambda(x)\}.\nonumber
\end{align}
We prove in Lemma \ref{convexity} that $\Lambda(x)$ is convex in $x$, identify it as the convex conjugate of $\Gamma(\theta)$ 
and thus conclude the proof.
\end{proof}

\begin{lemma}\label{convexity}
$\Lambda(x)$ in \eqref{Lambdaofx} is convex in $x$.
\end{lemma}

\begin{proof}
Define
\begin{equation}
\hat{H}(\hat{\lambda},\hat{\pi})=\int(\lambda-\hat{\lambda})\hat{\pi}+\int\log(\hat{\lambda}/\lambda)\hat{\lambda}\hat{\pi}.
\end{equation}
Then,
\begin{equation}
\Lambda(x)=\inf_{(\hat{\lambda},\hat{\pi})\in\mathcal{Q}_{e}^{x}}\hat{H}(\hat{\lambda},\hat{\pi}).
\end{equation}
We want to prove that $\Lambda(\alpha x_{1}+\beta x_{2})\leq\alpha\Lambda(x_{1})+\beta\Lambda(x_{2})$ 
for any $\alpha,\beta\geq 0$ with $\alpha+\beta=1$. For any $\epsilon>0$, 
we can choose $(\hat{\lambda}_{k},\hat{\pi}_{k})\in\mathcal{Q}_{e}^{x_{k}}$ such that 
$\hat{H}(\hat{\lambda}_{k},\hat{\pi}_{k})\leq\Lambda(x_{k})+\epsilon/2$, for $k=1,2$. Set
\begin{equation}
\hat{\pi}_{3}=\alpha\hat{\pi}_{1}+\beta\hat{\pi}_{2},
\quad\hat{\lambda}_{3}=\frac{d(\alpha\hat{\pi}_{1})}{d(\alpha\hat{\pi}_{1}+\beta\hat{\pi}_{2})}\hat{\lambda}_{1}
+\frac{d(\beta\hat{\pi}_{2})}{d(\alpha\hat{\pi}_{1}+\beta\hat{\pi}_{2})}\hat{\lambda}_{2}.
\end{equation}
Then for any test function $f$,
\begin{equation}
\int\hat{\mathcal{A}}_{3}f\hat{\pi}_{3}=\alpha\int\hat{\mathcal{A}}_{1}f\hat{\pi}_{1}+\beta\int\hat{\mathcal{A}}_{2}f\hat{\pi}_{2}=0,
\end{equation}
which implies $(\hat{\lambda}_{3},\hat{\pi}_{3})\in\mathcal{Q}_{e}$. Furthermore,
\begin{equation}
\int\hat{\lambda}_{3}\hat{\pi}_{3}=\alpha\int\hat{\lambda}_{1}\hat{\pi}_{1}+\beta\int\hat{\lambda}_{2}\hat{\pi}_{2}=\alpha x_{1}+\beta x_{2}.
\end{equation}
Therefore, $(\hat{\lambda}_{3},\hat{\pi}_{3})\in\mathcal{Q}_{e}^{\alpha x_{1}+\beta x_{2}}$. 
Finally, since $x\log x$ is a convex function and if we apply Jensen's inequality, we get
\begin{align}
\hat{H}(\hat{\lambda}_{3},\hat{\pi}_{3})
&=\int\left[(\lambda-\hat{\lambda}_{3}-\hat{\lambda}_{3}\log\lambda)+\hat{\lambda}_{3}\log\hat{\lambda}_{3}\right]\hat{\pi}_{3}
\\
&\leq\int\left[(\lambda-\hat{\lambda}_{3}-\hat{\lambda}_{3}\log\lambda)
+\alpha\frac{d\hat{\pi}_{1}}{d\hat{\pi}_{3}}\hat{\lambda}_{1}\log\hat{\lambda}_{1}
+\beta\frac{d\hat{\pi}_{2}}{d\hat{\pi}_{3}}\hat{\lambda}_{2}\log\hat{\lambda}_{2}\right]\hat{\pi}_{3}\nonumber
\\
&=\alpha\hat{H}(\hat{\lambda}_{1},\hat{\pi}_{1})+\beta\hat{H}(\hat{\lambda}_{2},\hat{\pi}_{2})\nonumber.
\end{align}
Therefore, 
\begin{equation}
\Lambda(\alpha x_{1}+\beta x_{2})\leq\hat{H}(\hat{\lambda}_{3},\hat{\pi}_{3})
\leq\alpha\hat{H}(\hat{\lambda}_{1},\hat{\pi}_{1})+\beta \hat{H}(\hat{\lambda}_{2},\hat{\pi}_{2})\leq\alpha\Lambda(x_{1})+\beta\Lambda(x_{2})+\epsilon.
\end{equation}
\end{proof}

\begin{lemma}\label{Gammafinite}
If $\limsup_{z\rightarrow\infty}\frac{\lambda(z)}{bz}<\frac{1}{a}$, then for any
\begin{equation}
\theta<\log\left(\frac{b}{a\limsup_{z\rightarrow\infty}\frac{\lambda(z)}{z}}\right)-1+\frac{a}{b}\cdot\limsup_{z\rightarrow\infty}\frac{\lambda(z)}{z},
\end{equation}
we have $\Gamma(\theta)<\infty$. If $\limsup_{z\rightarrow\infty}\frac{\lambda(z)}{z}=0$, then $\Gamma(\theta)<\infty$ for any $\theta\in\mathbb{R}$.
\end{lemma}

\begin{proof}
For $K\geq\frac{\theta}{a}$, we have $e^{Kz}\in\mathcal{U}_{\theta}$ and
\begin{align}
\Gamma(\theta)&\leq\inf_{g\in\mathcal{U}_{\theta}}\sup_{z\geq 0}\frac{\mathcal{A}g(z)
+\frac{\theta b}{a}zg(z)}{g(z)}\leq\sup_{z\geq 0}\left\{\frac{\mathcal{A}e^{Kz}}{e^{Kz}}+\frac{\theta b}{a}z\right\}
\\
&=\sup_{z\geq 0}\left\{-\left(bK-\frac{\theta b}{a}\right)z+\lambda(z)(e^{Ka}-1)\right\}\nonumber.
\end{align}
Define the function
\begin{equation}
F(K)=-K+\limsup_{z\rightarrow\infty}\frac{\lambda(z)}{bz}\cdot(e^{Ka}-1).
\end{equation}
Then $F(0)=0$, $F$ is convex and $F(K)\rightarrow\infty$ as $K\rightarrow\infty$ and its minimum is attained at
\begin{equation}
K^{\ast}=\frac{1}{a}\log\left(\frac{b}{a\limsup_{z\rightarrow\infty}\frac{\lambda(z)}{z}}\right)>0,
\end{equation}
and $F(K^{\ast})<0$. Therefore, $\Gamma(\theta)<\infty$ for any
\begin{align}
\theta&<-a\min_{K>0}\left\{-K+\limsup_{z\rightarrow\infty}\frac{\lambda(z)}{bz}\cdot(e^{Ka}-1)\right\}
\\
&=\log\left(\frac{b}{a\limsup_{z\rightarrow\infty}\frac{\lambda(z)}{z}}\right)-1
+\frac{a}{b}\cdot\limsup_{z\rightarrow\infty}\frac{\lambda(z)}{z}<K^{\ast}a.\nonumber
\end{align}
If $\limsup_{z\rightarrow\infty}\frac{\lambda(z)}{z}=0$, trying $e^{Kz}\in\mathcal{U}_{\theta}$ 
for any $K>\frac{\theta}{a}$, we have $\Gamma(\theta)<\infty$ for any $\theta$.
\end{proof}

\section{Large Deviations for Markovian Nonlinear Hawkes Processes with Sum of Exponentials Exciting Function}

In this section, we consider the Markovian nonlinear Hawkes processes with sum of exponentials exciting functions, 
i.e. $h(t)=\sum_{i=1}^{d}a_{i}e^{-b_{i}t}$.
Let
\begin{equation}
Z_{i}(t)=\sum_{\tau_{j}<t}a_{i}e^{-b_{i}(t-\tau_{j})},\quad 1\leq i\leq d,
\end{equation}
and $Z_{t}=\sum_{i=1}^{d}Z_{i}(t)=\sum_{\tau_{j}<t}h(t-\tau_{j})$, where $\tau_{j}$'s are the arrivals of the Hawkes process
with intensity $\lambda(Z_{t})=\lambda(Z_{1}(t)+\cdots+Z_{d}(t))$ at time $t$. Observe that this is a special
case of the Markovian processes with jumps studied in Section \ref{ErgodicSection} with $\lambda(Z_{1}(t),Z_{2}(t),\cdots,Z_{d}(t))$
taking the form $\lambda(\sum_{i=1}^{d}Z_{i}(t))$. 
It is easy to see that $(Z_{1},\ldots,Z_{d})$ is Markovian with generator
\begin{equation}
\mathcal{A}f=-\sum_{i=1}^{d}b_{i}z_{i}\frac{\partial f}{\partial z_{i}}
+\lambda\left(\sum_{i=1}^{d}z_{i}\right)\cdot[f(z_{1}+a_{1},\ldots,z_{d}+a_{d})-f(z_{1},\ldots,z_{d})].
\end{equation}
Here $b_{i}>0$ for any $1\leq i\leq d$ and $a_{i}$ can be negative. But we restrict ourselves
to the set of $b_{i}$'s and $a_{i}$'s so that $h(t)=\sum_{i=1}^{d}a_{i}e^{-b_{i}t}>0$ for any $t\geq 0$
for the rest of this paper. 
In particular, $h(0)=\sum_{i=1}^{d}a_{i}>0$. If $a_{i}>0$, then $Z_{i}(t)\geq 0$ almost surely; if $a_{i}<0$, then $Z_{i}(t)\leq 0$ almost surely.

\begin{theorem}
Assume $\lim_{z\rightarrow\infty}\frac{\lambda(z)}{z}=0$, $\lambda(\cdot)$ is continuous and bounded below
by a positive constant. Then,
\begin{equation}
\lim_{t\rightarrow\infty}\frac{1}{t}\log\mathbb{E}[e^{\theta N_{t}}]
=\inf_{u\in\mathcal{U}_{\theta}}\sup_{(z_{1},\ldots,z_{d})\in\mathcal{Z}}\left\{\frac{\mathcal{A}u}{u}+\frac{\theta}{\sum_{i=1}^{d}a_{i}}
\sum_{i=1}^{d}b_{i}z_{i}\right\},
\end{equation}
where $\mathcal{Z}=\{(z_{1},\ldots,z_{d}): a_{i}z_{i}\geq 0, 1\leq i\leq d\}$ and
\begin{equation}
\mathcal{U}_{\theta}=\left\{u\in C_{1}(\mathbb{R}^{d},\mathbb{R}^{+}),u=e^{f}, f\in\mathcal{F}\right\},
\end{equation}
where 
\begin{equation}
\mathcal{F}=\left\{f=g+\frac{\theta\sum_{i=1}^{d}z_{i}}{\sum_{i=1}^{d}a_{i}}+L, L\in\mathbb{R}, g\in\mathcal{G}\right\},
\end{equation}
where
\begin{equation}
\mathcal{G}=\left\{\sum_{i=1}^{d}K\epsilon_{i}z_{i}+g, K>0, g\text{ is $C_{1}$ with compact support}\right\}.
\end{equation}
\end{theorem}

\begin{proof}
Notice that
\begin{equation}
dZ_{i}(t)=-b_{i}Z_{i}(t)dt+a_{i}dN_{t},\quad 1\leq i\leq d.
\end{equation}
Hence, $a_{i}N_{t}=Z_{i}(t)-Z_{i}(0)+\int_{0}^{t}b_{i}Z_{i}(s)ds$ and
\begin{equation}
\mathbb{E}[e^{\theta N_{t}}]=\mathbb{E}\left[\exp\left\{\frac{\theta\sum_{i=1}^{d}Z_{i}(t)-Z_{i}(0)}{\sum_{i=1}^{d}a_{i}}
+\frac{\theta}{\sum_{i=1}^{d}a_{i}}\int_{0}^{t}\sum_{i=1}^{d}b_{i}Z_{i}(s)ds\right\}\right].
\end{equation}
Following the same arguments in the proof of Theorem \ref{mainlimit}, we obtain the upper bound
\begin{equation}
\limsup_{t\rightarrow\infty}\frac{1}{t}\log\mathbb{E}[e^{\theta N_{t}}]\leq\inf_{u\in\mathcal{U}_{\theta}}
\sup_{(z_{1},\ldots,z_{d})\in\mathcal{Z}}\left\{\frac{\mathcal{A}u}{u}+\frac{\theta}{\sum_{i=1}^{d}a_{i}}\sum_{i=1}^{d}b_{i}z_{i}\right\}.
\end{equation}
As before, we can obtain the lower bound
\begin{align}
&\liminf_{t\rightarrow\infty}\frac{1}{t}\log\mathbb{E}[e^{\theta N_{t}}]
\\
&\geq\sup_{(\hat{\lambda},\hat{\pi})\in\mathcal{Q}_{e}}\int\left[\theta\hat{\lambda}-\lambda+\hat{\lambda}
-\hat{\lambda}\log\left(\hat{\lambda}/\lambda\right)\right]\hat{\pi}(dz_{1},\ldots,dz_{d})\nonumber
\\
&\geq\sup_{(\hat{\lambda},\hat{\pi})\in\mathcal{Q}}\inf_{g\in\mathcal{G}}\int\left[\theta\hat{\lambda}-\lambda
+\hat{\lambda}-\hat{\lambda}\log\left(\hat{\lambda}/\lambda\right)+\hat{\mathcal{A}}g\right]\hat{\pi}\nonumber
\\
&=\sup_{(\hat{\lambda},\hat{\pi})\in\mathcal{Q}}\inf_{f\in\mathcal{F}}
\int\left[\frac{\theta\sum_{i=1}^{d}b_{i}z_{i}}{\sum_{i=1}^{d}a_{i}}-\lambda+\hat{\lambda}-\hat{\lambda}\log\left(\hat{\lambda}/\lambda\right)
+\hat{\mathcal{A}}f\right]\hat{\pi}.\nonumber
\end{align}
The equality in the last line above holds by taking $f=g+L+\frac{\theta\sum_{i=1}^{d}z_{i}}{\sum_{i=1}^{d}a_{i}}\in\mathcal{F}$ for $g\in\mathcal{G}$, where
\begin{equation}
\mathcal{G}=\left\{\sum_{i=1}^{d}K\epsilon_{i}z_{i}+g, K>0, g\text{ is $C_{1}$ with compact support}\right\}.
\end{equation}
Here, $\epsilon_{i}=a_{i}/|a_{i}|$, $1\leq i\leq d$. Define
\begin{equation}
F(\hat{\lambda}\hat{\pi},\hat{\pi},f)=\int\left[\frac{\theta\sum_{i=1}^{d}b_{i}z_{i}}{\sum_{i=1}^{d}a_{i}}+\hat{\mathcal{A}}f\right]\hat{\pi}
-\hat{H}(\hat{\lambda},\hat{\pi}).
\end{equation}
$F$ is linear in $f$ and hence convex in $f$. Also $\hat{H}$ is weakly lower semicontinuous and convex in $(\hat{\lambda}\hat{\pi},\hat{\pi})$. 
Therefore, $F$ is concave in $(\hat{\lambda}\hat{\pi},\hat{\pi})$. 
Furthermore, for any $f=\frac{\theta\sum_{i=1}^{d}z_{i}}{\sum_{i=1}^{d}a_{i}}+\sum_{i=1}^{d}K\epsilon_{i}z_{i}+g+L\in\mathcal{F}$,
\begin{equation}
F(\hat{\lambda}\hat{\pi},\hat{\pi},f)=\int\left[\theta+\sum_{i=1}^{d}K\epsilon_{i}a_{i}\right]\hat{\lambda}\hat{\pi}
-\int\sum_{i=1}^{d}K\epsilon_{i}b_{i}z_{i}\hat{\pi}-\hat{H}(\hat{\lambda},\hat{\pi})+\int\hat{\mathcal{A}}g\hat{\pi}.
\end{equation}
If $\lambda_{n}\pi_{n}\rightarrow\gamma_{\infty}$ and $\pi_{n}\rightarrow\pi_{\infty}$ weakly, 
then, since $g$ is $C_{1}$ with compact support, we have
\begin{equation}
\int\left[\theta+\sum_{i=1}^{d}K\epsilon_{i}a_{i}\right]\lambda_{n}\pi_{n}+\int\hat{\mathcal{A}}g\pi_{n}\rightarrow
\int\left[\theta+\sum_{i=1}^{d}K\epsilon_{i}a_{i}\right]\gamma_{\infty}+\int\hat{\mathcal{A}}g\pi_{\infty}.
\end{equation}
Since $-\sum_{i=1}^{d}K\epsilon_{i}b_{i}z_{i}$ is continuous and nonpositive on $\mathcal{Z}$, we have
\begin{equation}
\limsup_{n\rightarrow\infty}\int\left[-\sum_{i=1}^{d}K\epsilon_{i}b_{i}z_{i}\right]\pi_{n}
\leq\int\left[-\sum_{i=1}^{d}K\epsilon_{i}b_{i}z_{i}\right]\pi_{\infty}.
\end{equation}
Hence, we conclude that $F$ is upper semicontinuous in the weak topology.

In order to apply the minmax theorem, we want to prove the compactness in the weak topology of the level set 
\begin{equation}
\left\{(\hat{\lambda}\hat{\pi},\hat{\pi}):\int\left[-\frac{\theta\sum_{i=1}^{d}b_{i}z_{i}}{\sum_{i=1}^{d}a_{i}}
-\hat{\mathcal{A}}f\right]\hat{\pi}+\hat{H}(\hat{\lambda},\hat{\pi})\leq C\right\}.
\end{equation}
For any $f=\frac{\theta\sum_{i=1}^{d}z_{i}}{\sum_{i=1}^{d}a_{i}}+\sum_{i=1}^{d}K\epsilon_{i}z_{i}+g+L\in\mathcal{F}$,
where $g$ is $C_{1}$ with compact support etc., there exist some $C_{1}, C_{2}>0$ such that
\begin{align}
C_{1}&\geq\hat{H}+\sum_{i=1}^{d}Kb_{i}\epsilon_{i}\int z_{i}\hat{\pi}-C_{2}\int\hat{\lambda}\hat{\pi}
\\
&\geq\int_{\hat{\lambda}\geq\sum_{i=1}^{d}c_{i}z_{i}+\ell}\left[\lambda-\hat{\lambda}
+\hat{\lambda}\log(\hat{\lambda}/\lambda)\right]\hat{\pi}+\sum_{i=1}^{d}Kb_{i}\epsilon_{i}\int z_{i}\hat{\pi}\nonumber
\\
&-C_{2}\int_{\hat{\lambda}\geq\sum_{i=1}^{d}c_{i}z_{i}+\ell}\hat{\lambda}\hat{\pi}-C_{2}\int_{\hat{\lambda}
<\sum_{i=1}^{d}c_{i}z_{i}+\ell}\hat{\lambda}\hat{\pi}\nonumber
\\
&\geq\left[\min_{(z_{1},\ldots,z_{d})\in\mathcal{Z}}\log\frac{c_{1}z_{1}+\cdots+c_{d}z_{d}+\ell}{\lambda(z_{1}+\cdots+z_{d})}-1-C_{2}\right]
\int_{\hat{\lambda}\geq\sum_{i=1}^{d}c_{i}z_{i}+\ell}\hat{\lambda}\hat{\pi}\nonumber
\\
&+\sum_{i=1}^{d}[-c_{i}\cdot C_{2}+Kb_{i}\epsilon_{i}]\int z_{i}\hat{\pi}-\ell C_{2}.\nonumber
\end{align}
If $a_{i}>0$, then $\epsilon_{i}>0$, pick up $c_{i}>0$ such that $-c_{i}\cdot C_{2}+Kb_{i}\epsilon_{i}>0$. 
If $a_{i}<0$, then $\epsilon_{i}<0$, pick up $c_{i}$ such that $-c_{i}\cdot C_{2}+Kb_{i}\epsilon_{i}<0$. 
Finally, choose $\ell$ big enough such that the big bracket above is positive. Then
\begin{equation}
\int|z_{i}|\hat{\pi}\leq C_{3},\quad \int_{\hat{\lambda}\geq\sum_{i=1}^{d}c_{i}z_{i}+\ell}\hat{\lambda}\hat{\pi}\leq C_{4}.
\end{equation}
Hence, $\int\hat{\lambda}\hat{\pi}\leq C_{5}$ and $\hat{H}\leq C_{6}$. 
We can use the similar method as in the proof of Theorem \ref{mainlimit} to show that
\begin{equation}
\lim_{\ell\rightarrow\infty}\sup_{n}\int_{|z_{i}|>\ell}\lambda_{n}\pi_{n}=0,\quad 1\leq i\leq d.
\end{equation}
For any $(\lambda_{n}\pi_{n},\pi_{n})\in\mathcal{R}$, we can find a subsequence that converges in the weak topology by Prokhorov's Theorem.
Therefore,
\begin{align}
&\liminf_{t\rightarrow\infty}\frac{1}{t}\log\mathbb{E}[e^{\theta N_{t}}]
\\
&\geq\sup_{(\hat{\lambda},\hat{\pi})\in\mathcal{Q}}\inf_{f\in\mathcal{F}}
\int\left[\frac{\theta\sum_{i=1}^{d}b_{i}z_{i}}{\sum_{i=1}^{d}a_{i}}-\lambda+\hat{\lambda}-\hat{\lambda}
\log\left(\hat{\lambda}/\lambda\right)+\hat{\mathcal{A}}f\right]\hat{\pi}\nonumber
\\
&=\inf_{f\in\mathcal{F}}\sup_{\hat{\pi}}\sup_{\hat{\lambda}}\int\left[\frac{\theta\sum_{i=1}^{d}b_{i}z_{i}}{\sum_{i=1}^{d}a_{i}}-\lambda
+\hat{\lambda}-\hat{\lambda}\log\left(\hat{\lambda}/\lambda\right)+\hat{\mathcal{A}}f\right]\hat{\pi}\nonumber
\\
&=\inf_{f\in\mathcal{F}}\sup_{(z_{1},\ldots,z_{d})\in\mathcal{Z}}\frac{\theta\sum_{i=1}^{d}b_{i}z_{i}}{\sum_{i=1}^{d}a_{i}}
+\lambda(e^{f(z_{1}+a_{1},\ldots,z_{d}+a_{d})-f(z_{1},\ldots,z_{d})}-1)-\sum_{i=1}^{d}b_{i}z_{i}\frac{\partial f}{\partial z_{i}}\nonumber
\\
&\geq\inf_{u\in\mathcal{U}_{\theta}}\sup_{(z_{1},\ldots,z_{d})\in\mathcal{Z}}\left\{\frac{\mathcal{A}u}{u}
+\frac{\theta}{\sum_{i=1}^{d}a_{i}}\sum_{i=1}^{d}b_{i}z_{i}\right\}.\nonumber
\end{align}
That is because optimizing over $\hat{\lambda}$, we get $\hat{\lambda}=\lambda e^{f(z_{1}+a_{1},\ldots,z_{d}+a_{d})-f(z_{1},\ldots,z_{d})}$ 
and finally for each $f\in\mathcal{F}$, $u=e^{f}\in\mathcal{U}_{\theta}$.
\end{proof}

\begin{theorem}
Assume $\lim_{z\rightarrow\infty}\frac{\lambda(z)}{z}=0$, $\lambda(\cdot)$ is continuous and bounded below by some positive constant. 
Then, $(\frac{N_{t}}{t}\in\cdot)$ satisfies the large deviation principle with the rate function $I(\cdot)$ as the 
Fenchel-Legendre transform of $\Gamma(\cdot)$,
\begin{equation}
I(x)=\sup_{\theta\in\mathbb{R}}\left\{\theta x-\Gamma(\theta)\right\},
\end{equation}
where
\begin{equation}
\Gamma(\theta)=\sup_{(\hat{\lambda},\hat{\pi})\in\mathcal{Q}_{e}}\int\left[\theta\hat{\lambda}-\lambda+\hat{\lambda}
-\hat{\lambda}\log\left(\hat{\lambda}/\lambda\right)\right]\hat{\pi}.
\end{equation}
\end{theorem}

\begin{proof}
The proof is the same as in the case of exponential $h(\cdot)$.
\end{proof}

\section{Large Deviations for a Special Class of Nonlinear Hawkes Processes: An Approximation Approach}

We already proved in Chapter \ref{chap:three} a large deviation principle of $(N_{t}/t\in\cdot)$ for nonlinear
Hawkes process by proving a level-3 large deviation first and then applying the contraction principle.
In this section, we point out that there is an alternative approach, i.e. for general exciting function $h(\cdot)$,
we can use sums of exponential functions to approximate $h(\cdot)$ and use the large deviations 
for the case when $h(\cdot)$ is a sum of exponentials to obtain the large deviations for general $h(\cdot)$.
The advantage of approximating the general case by the case when $h$ is a sum of exponentials is that 
the rate function for the large deviations when $h$ is a sum of exponentials
can be evaluated by an optimization problem, which should be computable by some numerical scheme.

Before we proceed, let us first prove that $h$ can be approximated by a sum of exponentials in both $L_{1}$
and $L_{\infty}$ norms.

\begin{lemma}\label{littleh}
If $h(t)>0$, $\int_{0}^{\infty}h(t)dt<\infty$, $h(\infty)=0$, and $h$ is continuous, 
then $h$ can be approximated by a sum of exponentials both in $L^{1}$ and $L^{\infty}$ norms.
\end{lemma}

\begin{proof}
The Stone-Weierstrass theorem says that if $X$ is a compact Hausdorff space and suppose $A$ is a subspace of $C(X)$ 
with the following properties. (i) If $f,g\in A$, then $f\times g\in A$. (ii) $1\in A$. 
(iii) If $x,y\in X$ then we can find an $f\in A$ such that $f(x)\neq f(y)$. Then $A$ is dense in $C(X)$ in $L^{\infty}$ norm. 
Consider $X=\mathbb{R}^{+}\cup\{\infty\}=[0,\infty]$ and $C[0,\infty]$ consists of continuous functions vanishing at $\infty$ and the constant function $1$.

By Stone-Weierstrass theorem, the linear combination of $1$, $e^{-t}$, $e^{-2t}$ etc. is dense in $C[0,\infty]$. 
In other words, for any continuous function $h$ on $C[0,\infty]$, we have
\begin{equation}
\sup_{t\geq 0}\left|h(t)-\sum_{j=0}^{n}a_{j}e^{-jt}\right|\leq\epsilon.
\end{equation}
In fact, since $h(\infty)=0$, we get $|a_{0}|\leq\epsilon$. Thus
\begin{equation}
\sup_{t\geq 0}\left|h(t)-\sum_{j=1}^{n}a_{j}e^{-jt}\right|\leq 2\epsilon.
\end{equation}
However, $\sum_{j=1}^{n}a_{j}e^{-jt}$ may not be positive. We can approximate $\sqrt{h(t)}$ first by a sum of exponentials 
and then approximate $h(t)$ by the square of that sum of exponentials, which is again a sum of exponentials but positive this time.

Indeed, we can approximate $h(t)$ by the sum of exponentials in $L^{1}$ norm as well. 
Suppose $\Vert h-h_{n}\Vert_{L^{\infty}}\rightarrow 0$, where $h_{n}$ is a sum of exponentials. 
Then, by dominated convergence theorem, for any $\delta>0$, $\int|h-h_{n}|e^{-\delta t}dt\rightarrow 0$ as $n\rightarrow\infty$. 
Thus, we can find a sequence $\delta_{n}>0$ such that $\delta_{n}\rightarrow 0$ as $n\rightarrow\infty$ and $\int|h-h_{n}|e^{-\delta_{n}t}dt\rightarrow 0$. 
By dominated convergence theorem again, $\int h(1-e^{-\delta_{n}t})dt\rightarrow 0$.
Hence, we have $\int|h-h_{n}e^{-\delta_{n}t}|dt\rightarrow 0$ as $n\rightarrow\infty$, where $h_{n}e^{-\delta_{n}t}$ is a sum of exponentials.

We will show that $h_{n}e^{-\delta_{n}t}$ converges to $h$ in $L^{\infty}$ as well. 
\begin{equation}
\Vert h-h_{n}e^{-\delta_{n}t}\Vert_{L^{\infty}}\leq\Vert h-h_{n}\Vert_{L^{\infty}}+\Vert h_{n}-h_{n}e^{-\delta_{n}t}\Vert_{L^{\infty}}.
\end{equation}
Notice that $(1-e^{-\delta_{n}t})h_{n}\leq(1-e^{-\delta_{n}t})(h(t)+\epsilon)$. Since $h(\infty)=0$, 
there exists some $M>0$, such that for $t>M$, $h(t)\leq\epsilon$ so that $(1-e^{-\delta_{n}t})h_{n}\leq 2\epsilon$ for $t>M$. 
For $t\leq M$, $(1-e^{-\delta_{n}t})h_{n}\leq(1-e^{-\delta_{n}M})(\Vert h\Vert_{L^{\infty}}+\epsilon)$ which is small if $\delta_{n}$ is small.
\end{proof}

We have the following results.

\begin{theorem}\label{specialthm}
Assume that $\lambda(\cdot)\geq c$ for some $c>0$, $\lim_{z\rightarrow\infty}\frac{\lambda(z)}{z}=0$ 
and $\lambda(\cdot)^{\alpha}$ is Lipschitz with constant $L_{\alpha}$ for any $\alpha\geq 1$. 
We have $(N_{t}/t\in\cdot)$ satisfies the large deviation principle with the rate function 
\begin{equation}
I(x)=\sup_{\theta\in\mathbb{R}}\{\theta x-\Gamma(\theta)\}.
\end{equation}
\end{theorem}

\begin{remark}
The class of nonlinear Hawkes process with general exciting function $h$ for which 
we proved the large deviation principle here is unfortunately a bit too special. 
It works for the rate function like $\lambda(z)=[\log(c+z)]^{\beta}$ for example 
but does not work for $\lambda(\cdot)$ that has sublinear power law growth. 
\end{remark}

We end this chapter with the proof of Theorem \ref{specialthm}.

Let $P_{n}$ denote the probability measure under which $N_{t}$ follows the Hawkes process with 
exciting function $h_{n}=\sum_{i=1}^{n}a_{i}e^{-b_{i}t}$ such that $h_{n}\rightarrow h$ as $n\rightarrow\infty$ 
in both $L^{1}$ and $L^{\infty}$ norms. We can find such a sequence $h_{n}$ by Lemma \ref{littleh}.
Let us define
\begin{equation}
\Gamma_{n}(\theta)=\lim_{t\rightarrow\infty}\frac{1}{t}\log\mathbb{E}^{P_{n}}\left[e^{\theta N_{t}}\right].
\end{equation}

We need the following lemmas to prove Theorem \ref{specialthm}.

\begin{lemma}\label{Lipschitz}
For any $K>0$ and $\theta_{1},\theta_{2}\in[-K,K]$, there exists some constant $C(K)$ such that for any $n$,
\begin{equation}
|\Gamma_{n}(\theta_{1})-\Gamma_{n}(\theta_{2})|\leq C(K)|\theta_{1}-\theta_{2}|.
\end{equation}
\end{lemma}

\begin{proof}
Without loss of generality, take $\theta_{2}>\theta_{1}$. Then
\begin{align}
\Gamma_{n}(\theta_{1})&\leq\Gamma_{n}(\theta_{2})
\\
&=\sup_{(\hat{\lambda},\hat{\pi})\in\mathcal{Q}^{\ast}_{e}}\int (\theta_{2}-\theta_{1})\hat{\lambda}\hat{\pi}
+\theta_{1}\hat{\lambda}\hat{\pi}-\hat{H}(\hat{\lambda},\hat{\pi})\nonumber
\\
&\leq\sup_{(\hat{\lambda},\hat{\pi})\in\mathcal{Q}^{\ast}_{e}}\int (\theta_{2}-\theta_{1})\hat{\lambda}\hat{\pi}+\Gamma_{n}(\theta_{1}),\nonumber
\end{align}
where
\begin{equation} 
\mathcal{Q}_{e}^{\ast}=\left\{(\hat{\lambda},\hat{\pi})\in\mathcal{Q}_{e}:\int\theta_{1}\hat{\lambda}\hat{\pi}
-\hat{H}(\hat{\lambda},\hat{\pi})\geq\Gamma_{n}(\theta_{1})-1\right\}.
\end{equation}
The key is to prove that $\sup_{(\hat{\lambda},\hat{\pi})\in\mathcal{Q}_{e}^{\ast}}\int\hat{\lambda}\hat{\pi}\leq C(K)$ 
for some positive constant $C(K)$ depending only on $K$. Define $u=u(z_{1},\ldots,z_{n})=e^{\sum_{i=1}^{n}c_{i}z_{i}}$ where
\begin{equation}
c_{i}=\frac{3K}{\sum_{i=1}^{n}\frac{a_{i}}{b_{i}}}\cdot\frac{1}{b_{i}},\quad 1\leq i\leq n.
\end{equation}
Define $V=-\frac{\mathcal{A}u}{u}$ such that
\begin{equation}
V(z_{1},\ldots,z_{n})=\frac{3K}{\sum_{i=1}^{n}\frac{a_{i}}{b_{i}}}\sum_{i=1}^{n}z_{i}-\lambda(z_{1}+\cdots+z_{n})(e^{3K}-1).
\end{equation}
Notice that $\int\hat{\mathcal{A}}f\hat{\pi}=0$ for any test function $f$ with certain regularities. 
If we try $f=\frac{z_{i}}{b_{i}}$, $1\leq i\leq n$, we get
\begin{equation}
-\int z_{i}\hat{\pi}+\frac{a_{i}}{b_{i}}\int\hat{\lambda}\hat{\pi}=0,\quad 1\leq i\leq n.
\end{equation}
Summing over $1\leq i\leq n$, we get
\begin{equation}
\int\hat{\lambda}\hat{\pi}=\frac{1}{\sum_{i=1}^{n}\frac{a_{i}}{b_{i}}}\int\sum_{i=1}^{n}z_{i}\hat{\pi}.
\end{equation}
Notice that $\sum_{i=1}^{n}\frac{a_{i}}{b_{i}}=\Vert h_{n}\Vert_{L^{1}}$ which is approximately $\Vert h\Vert_{L^{1}}$ 
when $n$ is large. Since $\limsup_{z\rightarrow\infty}\frac{\lambda(z)}{z}=0$ and $\sum_{i=1}^{n}z_{i}\geq 0$, we have
\begin{equation}
\theta_{1}\int\hat{\lambda}\hat{\pi}\leq K\int\hat{\lambda}\hat{\pi}=\frac{K}{\sum_{i=1}^{n}\frac{a_{i}}{b_{i}}}
\int\sum_{i=1}^{n}z_{i}\hat{\pi}\leq\frac{1}{2}\int V\hat{\pi}+C_{1/2}(K),
\end{equation}
where $C_{1/2}(K)$ is some positive constant depending only on $K$. 

We claim that $\int V(z)\hat{\pi}\leq\hat{H}(\hat{\pi})$ for any $\hat{\pi}\in\mathcal{Q}_{e}^{\ast}$. 
Let us prove it. By the ergodic theorem and Jensen's inequality,
\begin{equation}
\int V(z)\hat{\pi}=\lim_{t\rightarrow\infty}\mathbb{E}^{\hat{\pi}}\left[\frac{1}{t}\int_{0}^{t}V(Z_{s})ds\right]
\leq\limsup_{t\rightarrow\infty}\frac{1}{t}\log\mathbb{E}^{\pi}\left[e^{\int_{0}^{t}V(Z_{s})ds}\right]+\hat{H}(\hat{\pi}).
\end{equation}
Next, we will show that $u\geq 1$. That is equivalent to proving $\sum_{i=1}^{n}\frac{z_{i}}{b_{i}}\geq 0$. Consider the process
\begin{equation}
Y_{t}=\sum_{i=1}^{n}\frac{Z_{i}(t)}{b_{i}}=\sum_{\tau_{j}<t}\sum_{i=1}^{n}\frac{a_{i}}{b_{i}}e^{-b_{i}(t-\tau_{j})}=\sum_{\tau_{j}<t}g(t-\tau_{j}),
\end{equation}
where $g(t)=\sum_{i=1}^{n}\frac{a_{i}}{b_{i}}e^{-b_{i}t}$. Notice that $g(t)=\int_{t}^{\infty}h(s)ds>0$. 
Therefore, $Y_{t}\geq 0$ almost surely and $\sum_{i=1}^{n}\frac{Z_{i}(t)}{b_{i}}\geq 0$. 
Since $\frac{\mathcal{A}u}{u}+V=0$ and $u\geq 1$, by Feynman-Kac formula and Dynkin's formula, 
\begin{align}
\mathbb{E}^{\pi}\left[e^{\int_{0}^{t}V(Z_{s})ds}\right]&\leq\mathbb{E}^{\pi}\left[u(Z_{t})e^{\int_{0}^{t}V(Z_{s})ds}\right]
\\
&=u(Z_{0})+\int_{0}^{t}\mathbb{E}^{\pi}\left[(\mathcal{A}u(Z_{s})+V(Z_{s})u(Z_{s}))e^{\int_{0}^{s}V(Z_{u})du}\right]ds\nonumber
\\
&=u(Z_{0}),\nonumber
\end{align}
and therefore $\int V(z)\hat{\pi}\leq\hat{H}(\hat{\pi})$ for any $\hat{\pi}\in\mathcal{Q}_{e}^{\ast}$. Hence,
\begin{equation}
\theta_{1}\int\hat{\lambda}\hat{\pi}\leq\frac{1}{2}\int V(z)+C_{1/2}(K)\leq\frac{1}{2}\hat{H}+C_{1/2}(K).
\end{equation}
Notice that
\begin{equation}
-\infty<\Gamma_{n}(\theta_{1})-1\leq\theta_{1}\int\hat{\lambda}\hat{\pi}-\hat{H}\leq\Gamma_{n}(\theta_{1})<\infty.
\end{equation}
Hence, 
\begin{equation}
\Gamma_{n}(\theta_{1})-1+\frac{1}{2}\hat{H}\leq\theta_{1}\int\hat{\lambda}\hat{\pi}-\frac{1}{2}\hat{H}\leq C_{1/2}(K),
\end{equation}
which implies $\hat{H}\leq 2(C_{1/2}(K)-\Gamma_{n}(\theta_{1})+1)$ and so also,
\begin{equation}
\int\hat{\lambda}\hat{\pi}\leq\frac{1}{2K}\int V\hat{\pi}+\frac{1}{K}C_{1/2}(K)\leq\frac{1}{K}(C_{1/2}(K)-\Gamma_{n}(\theta_{1})+1)+\frac{1}{K}C_{1/2}(K).
\end{equation}
Finally, notice that since $h_{n}\rightarrow h$ in both $L^{1}$ and $L^{\infty}$ norms, we can find a function $g$
such that  $\sup_{n}h_{n}\leq g$ and $\Vert g\Vert_{L^{1}}<\infty$.
and thus
\begin{equation}
\Gamma_{n}(\theta_{1})\geq\Gamma_{n}(-K)\geq\Gamma_{g}(-K),
\end{equation}
where $\Gamma_{g}$ denotes the case when the rate function is still $\lambda(\cdot)$ but the exciting function is $g(\cdot)$ instead of $h_{n}(\cdot)$.
Notice that here $\Vert g\Vert_{L^{1}}<\infty$ but may not be less than $1$. It is still well defined because of 
the assumption $\lim_{z\rightarrow\infty}\frac{\lambda(z)}{z}=0$. Indeed, we can find
$\lambda(z)=\nu_{\epsilon}+\epsilon z$ that dominates the original $\lambda(\cdot)$ for $\nu_{\epsilon}>0$ big enough 
and $\epsilon>0$ small enough so that $\epsilon\Vert g\Vert_{L^{1}}<1$. 
Now, we have $\Gamma_{g}(-K)\geq\Gamma^{\nu_{\epsilon}}_{\epsilon g}(-K)$ which is finite,
where $\Gamma^{\nu_{\epsilon}}_{\epsilon g}(-K)$ corresponds to the case when $\lambda(z)=\nu_{\epsilon}+\epsilon z$. Hence,
\begin{equation}
\sup_{(\hat{\lambda},\hat{\pi})\in\mathcal{Q}_{e}^{\ast}}\int\hat{\lambda}\hat{\pi}\leq C(K),
\end{equation}
for some $C(K)>0$ depending only on $K$.
\end{proof}

\begin{lemma}\label{Cauchy}
Assume that $\lambda(\cdot)\geq c$ for some $c>0$, $\lim_{z\rightarrow\infty}\frac{\lambda(z)}{z}=0$ 
and $\lambda(\cdot)^{\alpha}$ is Lipschitz with constant $L_{\alpha}$ for any $\alpha\geq 1$. 
Then for any $K>0$, $\Gamma_{n}(\theta)$ is Cauchy with $\theta$ uniformly in $[-K,K]$.
\end{lemma}

\begin{proof}
Let us write $H_{n}(t)=\sum_{\tau_{j}<t}h_{n}(t-\tau_{j})$. Observe first, that for any $q$,
\begin{equation}
\exp\left\{q\int_{0}^{t}\log\left(\frac{\lambda(H_{m}(s))}{\lambda(H_{n}(s))}\right)dN_{s}
-\int_{0}^{t}\left(\frac{\lambda(H_{m}(s))^{q}}{\lambda(H_{n}(s))^{q-1}}-\lambda(H_{n}(s))\right)ds\right\}
\end{equation}
is a martingale under $P_{n}$. By H\"{o}lder's inequality, for any $p,q>1$ with $\frac{1}{p}+\frac{1}{q}=1$, 
\begin{align}
\mathbb{E}^{P_{m}}[e^{\theta N_{t}}]&=\mathbb{E}^{P_{n}}\left[e^{\theta N_{t}}\frac{dP_{m}}{dP_{n}}\right]
\\
&=\mathbb{E}^{P_{n}}\left[e^{\theta N_{t}-\int_{0}^{t}(\lambda(H_{m}(s))-\lambda(H_{n}(s)))ds-\int_{0}^{t}
\log\left(\frac{\lambda(H_{n}(s))}{\lambda(H_{m}(s))}\right)dN_{s}}\right]\nonumber
\\
&\leq\mathbb{E}^{P_{n}}\left[e^{p\theta N_{t}-p\int_{0}^{t}(\lambda(H_{m}(s))-\lambda(H_{n}(s)))ds}\right]^{1/p}
\mathbb{E}^{P_{n}}\left[e^{q\int_{0}^{t}\log\left(\frac{\lambda(H_{m}(s))}{\lambda(H_{n}(s))}\right)dN_{s}}\right]^{1/q}.\nonumber
\end{align}
By the Cauchy-Schwarz inequality, 
\begin{align}
\mathbb{E}^{P_{n}}\left[e^{q\int_{0}^{t}\log\left(\frac{\lambda(H_{m}(s))}{\lambda(H_{n}(s))}\right)dN_{s}}\right]^{1/q}
&\leq\mathbb{E}^{P_{n}}\left[e^{\int_{0}^{t}\left(\frac{\lambda(H_{m}(s))^{2q}}{\lambda(H_{n}(s))^{2q-1}}-\lambda(H_{n}(s))\right)ds}\right]^{\frac{1}{2q}}
\\
&\leq\mathbb{E}^{P_{n}}\left[e^{\frac{1}{c^{2q-1}}L_{2q}\int_{0}^{t}\sum_{\tau<s}|h_{m}(s-\tau)-h_{n}(s-\tau)|ds}\right]^{\frac{1}{2q}}\nonumber
\\
&\leq\mathbb{E}^{P_{n}}\left[e^{\frac{1}{c^{2q-1}}L_{2q}\Vert h_{m}-h_{n}\Vert_{L^{1}}N_{t}}\right]^{\frac{1}{2q}}.\nonumber
\end{align}
We also have
\begin{equation}
\mathbb{E}^{P_{n}}\left[e^{p\theta N_{t}-p\int_{0}^{t}(\lambda(H_{m}(s))-\lambda(H_{n}(s)))ds}\right]^{1/p}
\leq\mathbb{E}^{P_{n}}\left[e^{p\theta N_{t}+pL_{1}\Vert h_{m}-h_{n}\Vert_{L^{1}}N_{t}}\right]^{1/p}.
\end{equation}
Therefore, by Lemma \ref{Lipschitz} and the fact $\Gamma_{n}(0)=0$ for any $n$, we have
\begin{align}
&\Gamma_{m}(\theta)-\Gamma_{n}(\theta)
\\
&\leq\frac{1}{p}\Gamma_{n}\left(p\theta+pL_{1}\epsilon_{m,n}\right)
+\frac{1}{2q}\Gamma_{n}\left(\frac{L_{2q}\epsilon_{m,n}}{c^{2q-1}}\right)-\Gamma_{n}(\theta)\nonumber
\\
&\leq C(K)L_{1}\epsilon_{m,n}+\frac{C(K)}{2q}\cdot\frac{L_{2q}\epsilon_{m,n}}{c^{2q-1}}
+\frac{1}{p}\Gamma_{n}(p\theta)-\frac{1}{p}\Gamma_{n}(\theta)+\left(1-\frac{1}{p}\right)|\Gamma_{n}(\theta)|,\nonumber
\\
&\leq C(K)L_{1}\epsilon_{m,n}+\frac{C(K)}{2q}\cdot\frac{L_{2q}\epsilon_{m,n}}{c^{2q-1}}+\frac{C(K)(p-1)K}{p}+\left(1-\frac{1}{p}\right)C(K)K,\nonumber
\end{align}
where $\epsilon_{m,n}=\Vert h_{m}-h_{n}\Vert_{L^{1}}$. Hence, 
\begin{equation}
\limsup_{m,n\rightarrow\infty}\{\Gamma_{m}(\theta)-\Gamma_{n}(\theta)\}\leq 2\left(1-\frac{1}{p}\right)C(K)K,
\end{equation}
which is true for any $p>1$. Letting $p\downarrow 1$, we get the desired result.
\end{proof}

\begin{remark}
If $\lambda(\cdot)\geq c>0$ and $\lim_{z\rightarrow\infty}\frac{\lambda(z)}{z^{\alpha}}=0$ for any $\alpha>0$, 
then, $\lambda(\cdot)^{\sigma}$ is Lipschitz for any $\sigma\geq 1$. For instance, $\lambda(z)=[\log(z+c)]^{\beta}$ satisfies the conditions
if $\beta>0$ and $c>1$.
\end{remark}

\begin{theorem}
Assume that $\lambda(\cdot)\geq c$ for some $c>0$, $\lim_{z\rightarrow\infty}\frac{\lambda(z)}{z}=0$ and $\lambda(\cdot)^{\alpha}$ is
Lipschitz with constant $L_{\alpha}$ for any $\alpha\geq 1$.
\begin{equation}
\lim_{t\rightarrow\infty}\frac{1}{t}\log\mathbb{E}[e^{\theta N_{t}}]=\Gamma(\theta)=\lim_{n\rightarrow\infty}\Gamma_{n}(\theta),
\end{equation}
for any $\theta\in\mathbb{R}$.
\end{theorem}

\begin{proof}
By Lemma \ref{Cauchy}, $\Gamma_{n}(\theta)$ tends to $\Gamma(\theta)$ uniformly on any compact set $[-K,K]$. Since $\Gamma_{n}(\theta)$ is Lipschitz
by Lemma \ref{Lipschitz}, it is continuous and the limit $\Gamma$ is also continuous. Let $\epsilon_{n}=\Vert h_{n}-h\Vert_{L^{1}}\leq\epsilon$. 
As in the proof of Lemma \ref{Cauchy}, for any $\theta\in[-K,K]$, $p,q>1$, $\frac{1}{p}+\frac{1}{q}=1$, we get
\begin{align}
&\limsup_{t\rightarrow\infty}\frac{1}{t}\log\mathbb{E}[e^{\theta N_{t}}]
\\
&\leq\Gamma_{n}(\theta)+C(K)L_{1}\epsilon_{n}+\frac{C(K)}{2q}\cdot\frac{L_{2q}\epsilon_{n}}{c^{2q-1}}+2\left(1-\frac{1}{p}\right)C(K)K.\nonumber
\end{align}
Letting $n\rightarrow\infty$ first and then $p\downarrow 1$, we get $\limsup_{t\rightarrow\infty}\frac{1}{t}\log\mathbb{E}[e^{\theta N_{t}}]
\leq\Gamma(\theta)$. Similarly, for any $p',q'>1$ with $\frac{1}{p'}+\frac{1}{q'}=1$,
\begin{align}
\Gamma_{n}(\theta)&\leq\liminf_{t\rightarrow\infty}\frac{1}{pt}\log\mathbb{E}[e^{(p\theta+pL_{1}\epsilon_{n})N_{t}}]
+\liminf_{t\rightarrow\infty}\frac{1}{2qt}\log\mathbb{E}\left[e^{\frac{L_{2q}\epsilon_{n}}{c^{2q-1}}N_{t}}\right]
\\
&\leq\liminf_{t\rightarrow\infty}\frac{1}{pp't}\log\mathbb{E}[e^{pp'\theta N_{t}}]
+\liminf_{t\rightarrow\infty}\frac{1}{pq't}\log\mathbb{E}[e^{q'pL_{1}\epsilon_{n}N_{t}}]\nonumber
\\
&+\liminf_{t\rightarrow\infty}\frac{1}{2qt}\log\mathbb{E}\left[e^{\frac{L_{2q}\epsilon_{n}}{c^{2q-1}}N_{t}}\right].\nonumber
\end{align}
Since we can dominate $\lambda(\cdot)$ by the linear function $\lambda(z)=\nu+z$ in which case the limit of logarithmic 
moment generating function $\Gamma_{\nu}(\theta)$ is continuous in $\theta$, we may let $n\rightarrow\infty$ to obtain
\begin{equation}
\Gamma(\theta)\leq\liminf_{t\rightarrow\infty}\frac{1}{pp't}\log\mathbb{E}[e^{pp'\theta N_{t}}].
\end{equation}
This holds for any $\theta$ and thus
\begin{equation}
\liminf_{t\rightarrow\infty}\frac{1}{t}\log\mathbb{E}[e^{\theta N_{t}}]\geq pp'\Gamma\left(\frac{\theta}{pp'}\right).
\end{equation}
Letting $p,p'\downarrow 1$ and using the continuity of $\Gamma(\cdot)$, we get the desired result.
\end{proof}

Finally, let us prove Theorem \ref{specialthm}.

\begin{proof}[Proof of Theorem \ref{specialthm}]
For the upper bound, apply the G\"{a}rtner-Ellis Theorem. Let us prove the lower bound. 
Let $B_{\epsilon}(x)$ denote the open ball centered at $x$ with radius $\epsilon>0$. By H\"{o}lder's inequality, 
for any $p,q>1$ with $\frac{1}{p}+\frac{1}{q}=1$, 
\begin{equation}
P_{n}\left(\frac{N_{t}}{t}\in B_{\epsilon}(x)\right)
\leq\bigg\Vert\frac{dP_{n}}{d\mathbb{P}}\bigg\Vert_{L^{p}(\mathbb{P})}\mathbb{P}\left(\frac{N_{t}}{t}\in B_{\epsilon}(x)\right)^{1/q}.
\end{equation}
Therefore, letting $t\rightarrow\infty$, we have
\begin{align}
&\sup_{\theta\in\mathbb{R}}\{\theta x-\Gamma_{n}(\theta)\}=\lim_{t\rightarrow\infty}\frac{1}{t}\log P_{n}\left(\frac{N_{t}}{t}\in B_{\epsilon}(x)\right)
\\
&\leq\frac{1}{pp'}\Gamma(pp'L_{1}\epsilon_{n})+\frac{1}{2pq'}\Gamma\left(\frac{L_{2pq'}\epsilon_{n}}{c^{2pq'-1}}\right)
+\frac{1}{q}\liminf_{t\rightarrow\infty}\frac{1}{t}\log\mathbb{P}\left(\frac{N_{t}}{t}\in B_{\epsilon}(x)\right),\nonumber
\end{align}
where $\epsilon_{n}=\Vert h_{n}-h\Vert_{L^{1}}$. Hence, letting $n\rightarrow\infty$, see that
\begin{equation}
\frac{1}{q}\liminf_{t\rightarrow\infty}\frac{1}{t}\log\mathbb{P}\left(\frac{N_{t}}{t}\in B_{\epsilon}(x)\right)\geq\limsup_{n\rightarrow\infty}
\sup_{\theta\in\mathbb{R}}\{\theta x-\Gamma_{n}(\theta)\}.
\end{equation}
Since $\Gamma_{n}(\theta)\rightarrow\Gamma(\theta)$ uniformly on any compact set $K$, 
\begin{equation}
\sup_{\theta\in K}\{\theta x-\Gamma_{n}(\theta)\}\rightarrow\sup_{\theta\in K}\{\theta x-\Gamma(\theta)\},
\end{equation}
as $n\rightarrow\infty$ for any such set $K$. Notice that $\lambda(\cdot)\geq c>0$ and recall that the limit for 
the logarithmic moment generating function with parameter $\theta$ for a Poisson process with constant rate $c$ is $(e^{\theta}-1)c$. Hence
\begin{equation}
\liminf_{\theta\rightarrow +\infty}\frac{\Gamma_{n}(\theta)}{\theta}\geq\liminf_{\theta\rightarrow +\infty}\frac{(e^{\theta}-1)c}{\theta}=+\infty,
\end{equation}
which implies that $\sup_{\theta\in\mathbb{R}}\{\theta x-\Gamma_{n}(\theta)\}\rightarrow\sup_{\theta\in\mathbb{R}}\{\theta x-\Gamma(\theta)\}$. 
Therefore,
\begin{equation}
\frac{1}{q}\liminf_{t\rightarrow\infty}\frac{1}{t}\log\mathbb{P}\left(\frac{N_{t}}{t}\in B_{\epsilon}(x)\right)
\geq\sup_{\theta\in\mathbb{R}}\{\theta x-\Gamma(\theta)\}.
\end{equation}
Letting $q\downarrow 1$, we get the desired result.
\end{proof}

%% file: chap5.tex
\chapter{Asymptotics for Nonlinear Hawkes Processes\label{chap:five}}

In the existing literature of on nonlinear Hawkes processes, the usual assumption is that $\lambda(\cdot)$ is $\alpha$-Lipschitz,
$h(\cdot)$ is integrable and $\alpha\Vert h\Vert_{L^{1}}<1$. But how about other regimes? How do the asymptotics vary
in different regimes? This is the question we would try to answer in this chapter.

We divide the nonlinear Hawkes process into the following regimes.

\begin{enumerate}
\item 
$\lim_{z\rightarrow\infty}\frac{\lambda(z)}{z}=0$.
This is the sublinear regime. In this regime, if we assume that $\lambda(\cdot)$ is $\alpha$-Lipschitz, $\Vert h\Vert_{L^{1}}<\infty$
and $\alpha\Vert h\Vert_{L^{1}}<1$,
then there exists a unique stationary version of the nonlinear Hawkes process. The central limit theorem and large deviations
for this regime are proved in Zhu \cite{ZhuIII}, \cite{ZhuI} and \cite{ZhuII}.
On the contrary, if we assume that $\Vert h\Vert_{L^{1}}=\infty$, then, there is no stationary version.
Figure \ref{IntensityPlotOne} illustrates $\lambda_{t}$ in this case. We will obtain the time asymptotics for $\lambda_{t}$
in Section \ref{sublinear}.

\item
$\lim_{z\rightarrow\infty}\frac{\lambda(z)}{z}=1$ and $\Vert h\Vert_{L^{1}}<1$.
This is the sub-critical regime. In this regime, if we assume that $\lambda(\cdot)$ is $\alpha$-Lipschitz
and $\alpha\Vert h\Vert_{L^{1}}<1$, then
there exists a unique stationary version of the nonlinear Hawkes process, see Br\'{e}maud and Massouli\'{e} \cite{Bremaud}.
The central limit theorem is proved in Zhu \cite{ZhuIII}. Figure \ref{IntensityPlotThree} illustrates $\lambda_{t}$
in this case. We will summarize some known results
about the limit theorems in Section \ref{sub}.

\item
$\lim_{z\rightarrow\infty}\frac{\lambda(z)}{z}=1$ and $\Vert h\Vert_{L^{1}}=1$.
This is the critical regime. This regime is very subtle. We will show in Section \ref{critical} that in
some cases, there exists a stationary version of the Hawkes process. In some other cases, it does not exist. 
In particular, when $\lambda(z)=\nu+z$
and $\int_{0}^{\infty}th(t)dt<\infty$, 
we will prove that $\frac{N_{tT}}{T^{2}}\rightarrow\int_{0}^{t}\eta_{s}ds$, 
where $\eta_{s}$ is a squared Bessel process.
$N[T, T+\frac{t}{T}]$ will converge to a P\'{o}lya process as $T\rightarrow\infty$. 
Figure \ref{IntensityPlotFour} illustrates the behavior of $\lambda_{t}$ in this case.
When $h(\cdot)$ has heavy tails, i.e. $\int_{0}^{\infty}th(t)dt=\infty$, we will prove that the time
asymptotic behavior is different from the light tail case.

\item
$\lim_{z\rightarrow\infty}\frac{\lambda(z)}{z}=1$ and $\Vert h\Vert_{L^{1}}>1$.
This is the super-critical regime. We will prove in Section \ref{super} that $\lambda_{t}$ grows exponentially in $t$ in this regime,
which is consistent with what we can see in Figure \ref{IntensityPlotFive}.

\item
$\sum_{n=0}^{\infty}\frac{1}{\lambda(n)}<\infty$.
This is the explosive regime. In Section \ref{explosive}, we will first provide a criterion for the explosion and non-explosion
for nonlinear Hawkes process. Then, we will study the asymptotic behavior of the explosion time. 
Figure \ref{IntensityPlotSix}
illustrates the explosion of a finite time.
\end{enumerate}

Notice that if $\Vert h\Vert_{L^{1}}=\infty$ and $\lim_{z\rightarrow\infty}\frac{\lambda(z)}{z}=\alpha>0$, then one is in
the super-critical regime and we will see that $\lambda_{t}$
grows exponentially; this is discussed Section \ref{super}. If $\Vert h\Vert_{L^{1}}=\infty$ and $\sum_{n=0}^{\infty}\frac{1}{\lambda(n)}<\infty$, 
then one is in the explosive regime to be discussed in Section \ref{explosive}.

We will launch a systematic study of the time asymptoics for Hawkes process
in different regimes. We will study the sublinear regime, sub-critical regime, critical regime and super-critical regime
in Sections \ref{sublinear}, \ref{sub}, \ref{critical}, \ref{super} respectively. Finally, in Section \ref{explosive}, we will
provide a criterion for explosion and non-explosion for Hawkes process and obtain some asymptotics for the explosion time.

\begin{figure}[htb]
\begin{center}
\includegraphics[scale=0.70]{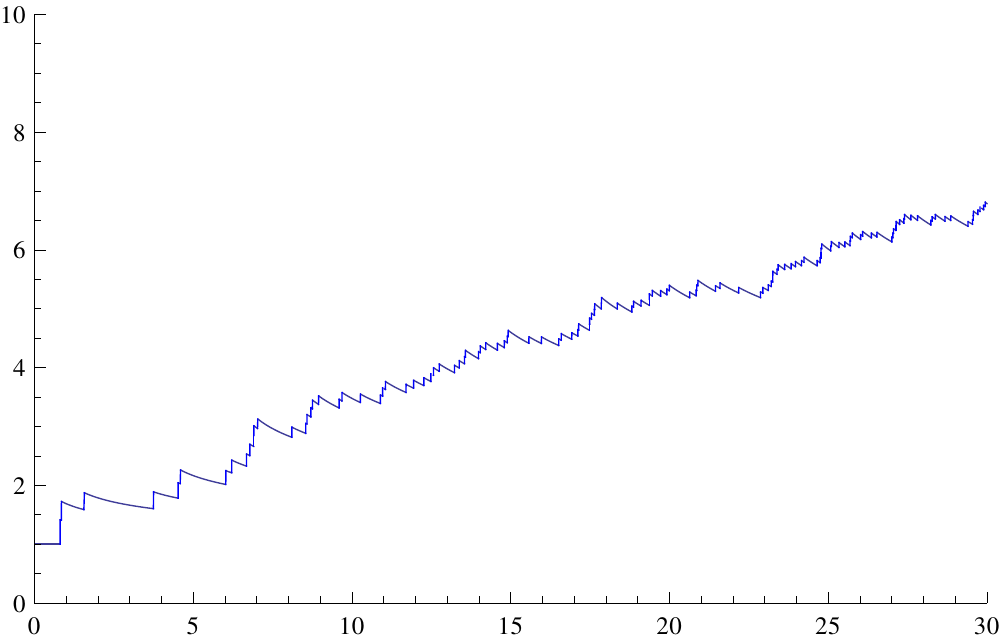}
\caption{Plot of intensity $\lambda_{t}$ for a realization of Hawkes process.
Here $h(t)=(t+1)^{-\frac{1}{2}}$ and $\lambda(z)=(1+z)^{\frac{1}{2}}$.}
\label{IntensityPlotOne}
\end{center}
\end{figure}

\begin{figure}[htb]
\begin{center}
\includegraphics[scale=0.70]{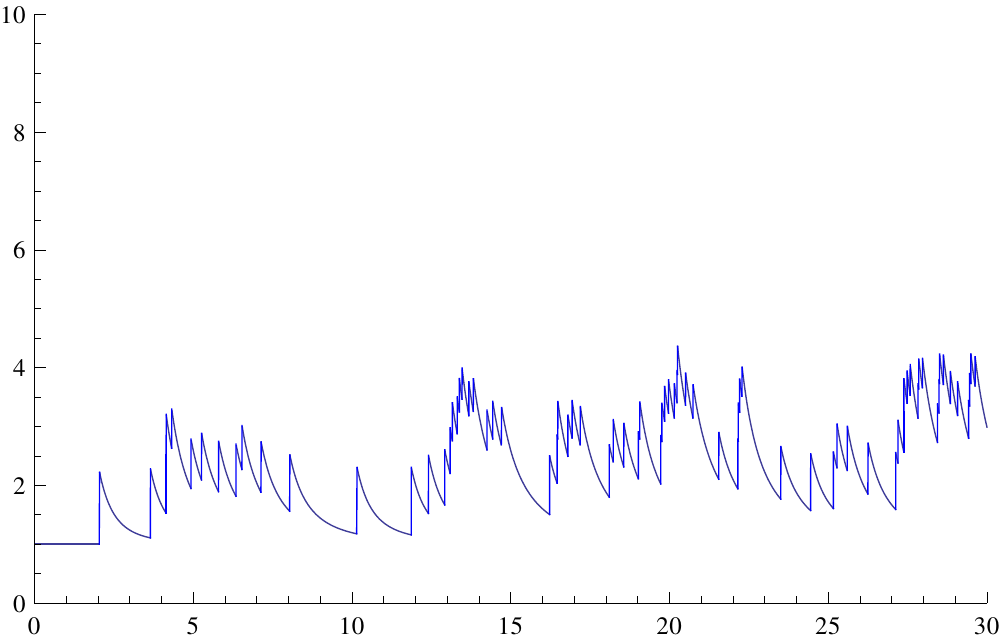}
\caption{Plot of intensity $\lambda_{t}$ for a realization of Hawkes process.
Here $h(t)=\frac{4}{(t+1)^{3}}$ and $\lambda(z)=(1+z)^{\frac{1}{2}}$. In this case, $\Vert h\Vert_{L^{1}}<\infty$ and
$\lambda(\cdot)$ is sublinear and Lipschitz. It will converge to the unique stationary version of the Hawkes process.}
\label{IntensityPlotTwo}
\end{center}
\end{figure}

\begin{figure}[htb]
\begin{center}
\includegraphics[scale=0.70]{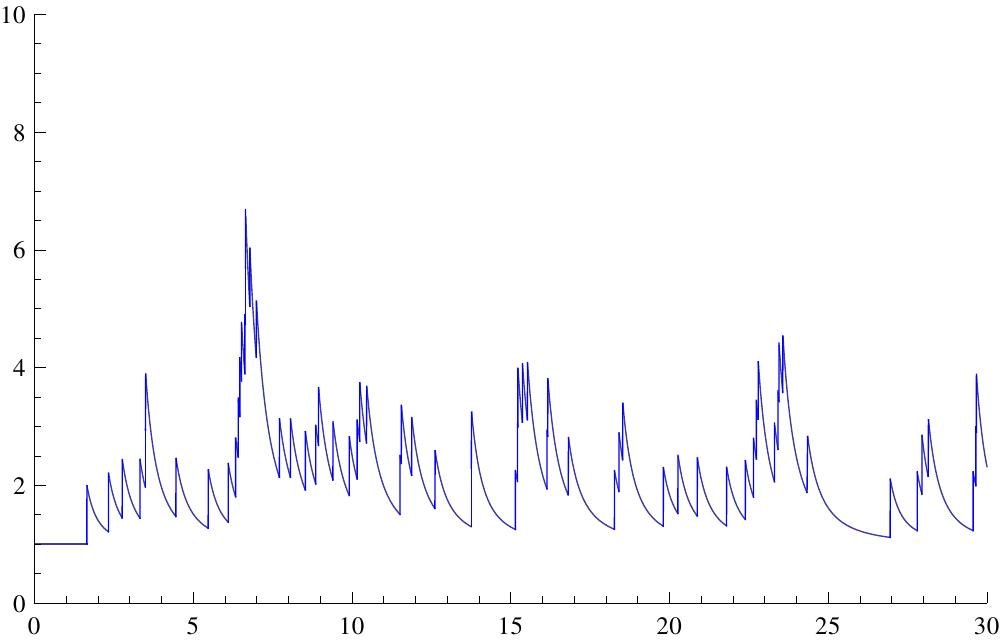}
\caption{Plot of intensity $\lambda_{t}$ for a realization of Hawkes process.
Here $h(t)=\frac{1}{(t+1)^{3}}$ and $\lambda(z)=1+z$. In this case, $\Vert h\Vert_{L^{1}}=\frac{1}{2}<1$. It is
in the sub-critical regime. This is a classical Hawkes process and it will converge to the unique stationary version
of the Hawkes process.}
\label{IntensityPlotThree}
\end{center}
\end{figure}

\begin{figure}[htb]
\begin{center}
\includegraphics[scale=0.70]{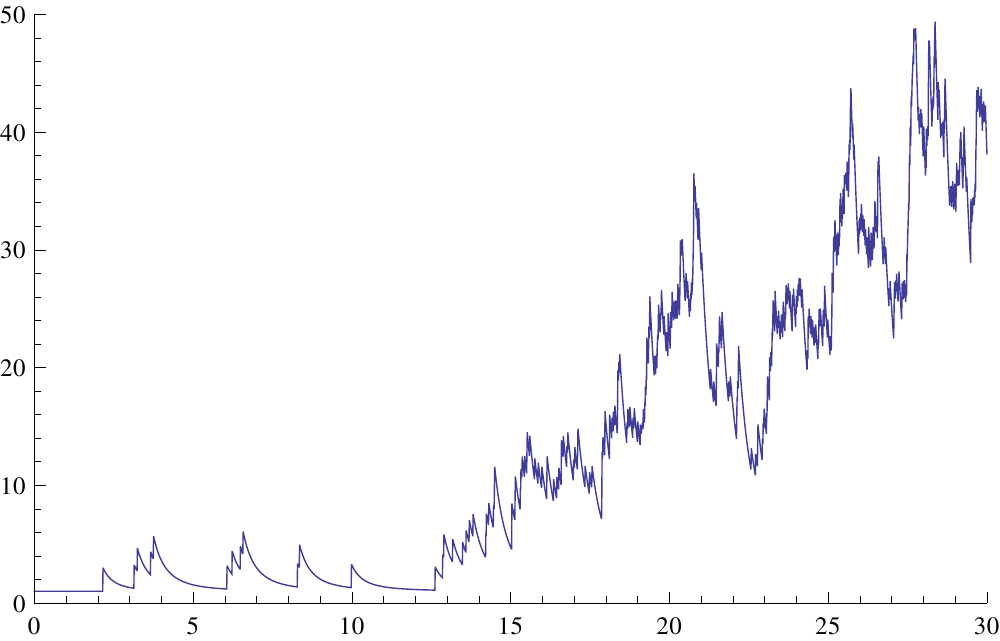}
\caption{Plot of intensity $\lambda_{t}$ for a realization of Hawkes process.
Here $h(t)=\frac{2}{(t+1)^{3}}$ and $\lambda(z)=1+z$. In this case, $\Vert h\Vert_{L^{1}}=1$, $\int_{0}^{\infty}th(t)dt<\infty$
and $\lambda(\cdot)$ is linear. It is therefore in the critical regime.
From the graph, we can see that $\lambda_{t}$ grows linearly in $t$, which will be proved in this chapter. Indeed,
we will prove that $\frac{N_{\cdot T}}{T^{2}}$ converges to $\int_{0}^{\cdot}\eta_{s}ds$ as $T\rightarrow\infty$, where
$\eta_{s}$ is a squared Bessel process.}
\label{IntensityPlotFour}
\end{center}
\end{figure}

\begin{figure}[htb]
\begin{center}
\includegraphics[scale=0.70]{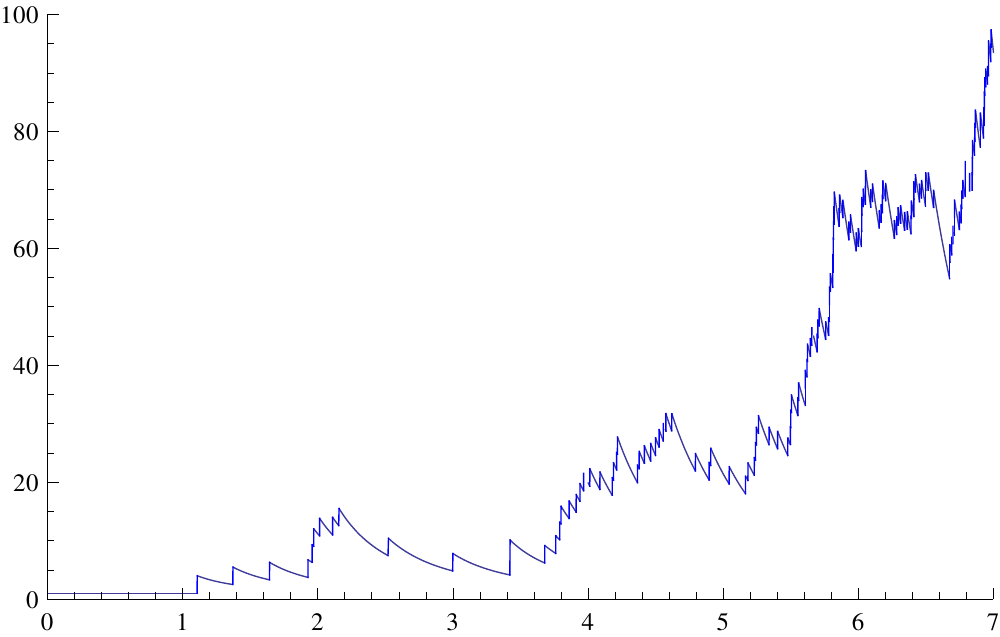}
\caption{Plot of intensity $\lambda_{t}$ for a realization of Hawkes process.
Here $h(t)=\frac{3}{(t+1)^{3}}$ and $\lambda(z)=1+z$. In this case, $\Vert h\Vert_{L^{1}}=\frac{3}{2}>1$ and it is
in the super-critical regime. We expect that $\lambda_{t}$ would grow exponentially in this case.}
\label{IntensityPlotFive}
\end{center}
\end{figure}

\begin{figure}[htb]
\begin{center}
\includegraphics[scale=0.70]{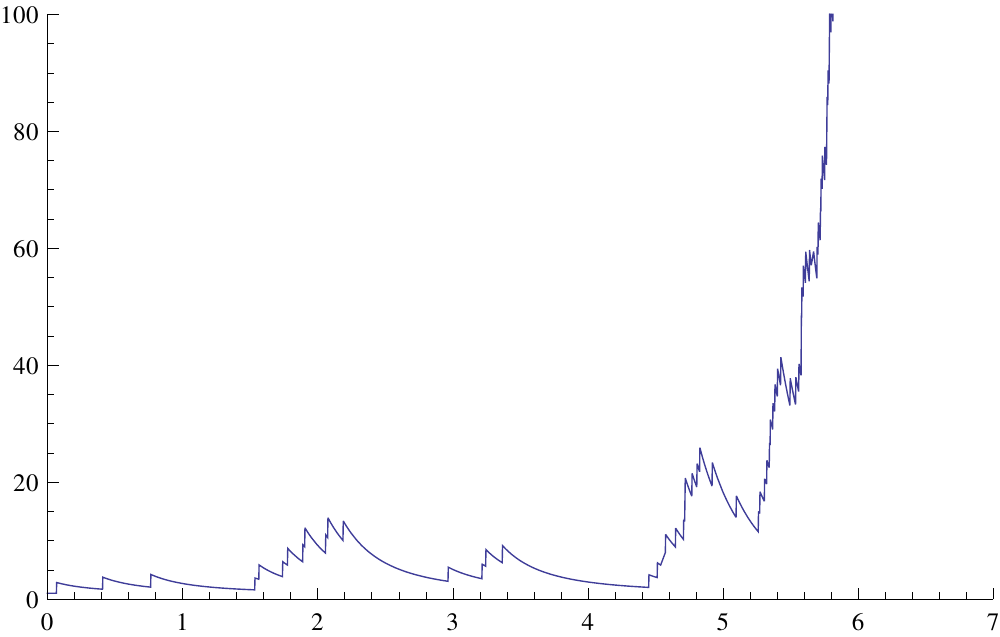}
\caption{Plot of intensity $\lambda_{t}$ for a realization of Hawkes process.
Here $h(t)=\frac{1}{(t+1)^{3}}$ and $\lambda(z)=(1+z)^{\frac{3}{2}}$. This is in the explosive regime. The plot
is a little bit cheating because it is impossible to ``plot'' explosion. Nevertheless, you can think it as an illustration.
It ``appears'' that the process explodes near time $t=6$.}
\label{IntensityPlotSix}
\end{center}
\end{figure}

\section{Sublinear Regime}\label{sublinear}

In this section, we are interested in the sublinear case
$\lim_{z\rightarrow\infty}\frac{\lambda(z)}{z}=0$.
If $\Vert h\Vert_{L^{1}}<\infty$ and $\lambda(\cdot)$ is $\alpha$-Lipschitz
and $\alpha\Vert h\Vert_{L^{1}}<1$, then,
as Br\'{e}maud and Massouli\'{e} \cite{Bremaud} proved, there exists
a unique stationary Hawkes process. Recently, Karabash \cite{KarabashII} relaxed the Lipschitz condition and proved
the stability result for a wider class of $\lambda(\cdot)$.
Let $\mathbb{P}$ and $\mathbb{E}$ denote the probability measure and expectation for stationary Hawkes process.
Then, by ergodic theorem, we have the law of large numbers,
\begin{equation}
\frac{N_{t}}{t}\rightarrow\mu=\mathbb{E}[N[0,1]],\quad\text{as $t\rightarrow\infty$}.
\end{equation}
The central limit theorem and large deviations have already been discussed in Chapter \ref{chap:two}, Chapter \ref{chap:three}
and Chapter \ref{chap:four}.

If $\Vert h\Vert_{L^{1}}=\infty$, then there is no stationary version
of Hawkes process and $\lambda_{t}$ tends to $\infty$ as $t\rightarrow\infty$. This is the
case we are going to study for the rest of the subsection. 
We are interested the time asymptotic behavior of the nonlinear Hawkes process in this regime.

Let us first make a simple observation. Assume that $\lambda(z)\uparrow\infty$ as $z\rightarrow\infty$. Then,
assuming $\Vert h\Vert_{L^{1}}=\infty$, we have $\lambda_{t}\rightarrow\infty$ as $t\rightarrow\infty$ a.s. This can
be seen by noticing that $\int_{0}^{t}h(t-s)N(ds)\rightarrow\infty$ a.s. if $\Vert h\Vert_{L^{1}}=\infty$, where $N_{t}$ follows
from a standard Poisson process with constant rate $\lambda(0)$.

Let us prove a special case first.
\begin{proposition}
Assume that $h(\cdot)\equiv 1$ and $\lambda(z)=\gamma(\nu+z)^{\beta}$, where $\gamma,\nu>0$ and $0<\beta<1$.
Then,
\begin{equation}
\frac{\lambda_{t}}{t^{\frac{\beta}{1-\beta}}}\rightarrow\gamma^{\frac{1}{1-\beta}}(1-\beta)^{\frac{\beta}{1-\beta}},
\end{equation}
in probability as $t\rightarrow\infty$.
\end{proposition}

\begin{proof}
For $\alpha>0$, 
\begin{align}
d\lambda_{t}^{\frac{1}{\alpha}}
&=\left[\lambda(\nu+N_{t}+1)^{\frac{1}{\alpha}}-\lambda(\nu+N_{t})^{\frac{1}{\alpha}}\right]dN_{t}
\\
&=\left[\gamma^{\frac{1}{\alpha}}(\nu+N_{t}+1)^{\frac{\beta}{\alpha}}-\gamma^{\frac{1}{\alpha}}(\nu+N_{t})^{\frac{\beta}{\alpha}}\right]dN_{t}\nonumber
\\
&=\left[\left(\lambda_{t}^{\frac{1}{\beta}}+\gamma^{\frac{1}{\beta}}\right)^{\frac{\beta}{\alpha}}-\lambda_{t}^{\frac{1}{\alpha}}\right]dN_{t}.\nonumber
\end{align}
Let $\alpha=\frac{\beta}{1-\beta}$. We have
\begin{equation}
\lambda_{t}^{\frac{1-\beta}{\beta}}=\int_{0}^{t}\left[(\lambda_{s}^{\frac{1}{\beta}}+\gamma^{\frac{1}{\beta}})^{1-\beta}
-(\lambda_{s}^{\frac{1}{\beta}})^{1-\beta}\right]\lambda_{s}ds+\int_{0}^{t}\left[(\lambda_{s}^{\frac{1}{\beta}}+\gamma^{\frac{1}{\beta}})^{1-\beta}
-(\lambda_{s}^{\frac{1}{\beta}})^{1-\beta}\right]dM_{s}.
\end{equation}
Since $\lambda_{t}\rightarrow\infty$ a.s. as $t\rightarrow\infty$, by the bounded convergence theorem, 
\begin{equation}
\frac{1}{t}\int_{0}^{t}\mathbb{E}\left\{\left[(\lambda_{s}^{\frac{1}{\beta}}+\gamma^{\frac{1}{\beta}})^{1-\beta}
-(\lambda_{s}^{\frac{1}{\beta}})^{1-\beta}\right]\lambda_{s}\right\}ds\rightarrow(1-\beta)\gamma^{\frac{1}{\beta}},
\end{equation}
as $t\rightarrow\infty$. It is not difficult to see that $\frac{1}{t}\int_{0}^{t}[(\lambda_{s}^{\frac{1}{\beta}}+\gamma^{\frac{1}{\beta}})^{1-\beta}
-(\lambda_{s}^{\frac{1}{\beta}})^{1-\beta}]dM_{s}\rightarrow 0$ in probability as $t\rightarrow\infty$.
Hence, $\frac{\lambda_{t}^{\frac{1-\beta}{\beta}}}{t}\rightarrow(1-\beta)\gamma^{\frac{1}{\beta}}$ in probability
as $t\rightarrow\infty$.
\end{proof}

\begin{remark}
Assume that $h(t)=(t+1)^{\delta}$, $\delta>-1$ and $\lambda(z)=\gamma(\nu+z)^{\beta}$, where $\gamma,\nu>0$ and $0<\beta<1$.
We conjecture that 
\begin{equation}
\frac{\lambda_{t}}{t^{\alpha}}\rightarrow\gamma^{\frac{1}{1-\beta}}B(\delta,\alpha)^{\frac{\beta}{1-\beta}},
\end{equation}
as $t\rightarrow\infty$ a.s., where $\alpha=\frac{(1+\delta)\beta}{1-\beta}$ and $B(\delta,\alpha)=\int_{0}^{1}u^{\delta}(1-u)^{\alpha}du$.
\end{remark}

\section{Sub-Critical Regime}\label{sub}

In this section, we review some known results about the limit theorems in the sub-critical regime.
We say the Hawkes process is in the sub-critical regime if $\lim_{z\rightarrow\infty}\frac{\lambda(z)}{z}=1$ and
$\Vert h\Vert_{L^{1}}<1$. If we further assume that $\lambda(\cdot)$ is $\alpha$-Lipschitz and $\alpha\Vert h\Vert_{L^{1}}<1$, then
Br\'{e}maud and Massouli\'{e} \cite{Bremaud} proved that there exists a unique stationary Hawkes process.
In this regime, we also have the law of large numbers and the central limit theorem just as in Section \ref{sublinear}.
For the case when $\lambda(\cdot)$ is nonlinear, we refer to the review in Section \ref{sublinear} for the
law of large numbers and central limit theorem.

In particular, when $\lambda(z)=\nu+z$ and $\nu>0$, we have explict expressions for
the law of large numbers, central limit theorem and large deviation principle. They are well known in the literature.

The ergodic theorem implies the following law of large numbers,
\begin{equation}
\frac{N_{t}}{t}\rightarrow\frac{\nu}{1-\Vert h\Vert_{L^{1}}},\quad\text{as $t\rightarrow\infty$ a.s.}
\end{equation}
Bordenave and Torrisi \cite{Bordenave} proved a large deviation principle for $(\frac{N_{t}}{t}\in\cdot)$ with the rate function
\begin{equation}
I(x)=
\begin{cases}
x\log\left(\frac{x}{\nu+x\Vert h\Vert_{L^{1}}}\right)-x+x\Vert h\Vert_{L^{1}}+\nu &\text{if $x\in[0,\infty)$}
\\
+\infty &\text{otherwise}
\end{cases}.
\end{equation}
Bacry et al. \cite{Bacry} proved a functional central limit theorem, stating that
\begin{equation}
\frac{N_{\cdot t}-\cdot\mu t}{\sqrt{t}}\rightarrow\sigma B(\cdot),\quad\text{as $t\rightarrow\infty$,}
\end{equation}
on $D[0,1]$ with Skorokhod topology, where
\begin{equation}
\mu=\frac{\nu}{1-\Vert h\Vert_{L^{1}}}\quad\text{and}\quad\sigma^{2}=\frac{\nu}{(1-\Vert h\Vert_{L^{1}})^{3}}.
\end{equation}

When $\lambda(\cdot)$ is nonlinear and sub-critical, the central limit theorem has been obtained in Chapter \ref{chap:two}.

\section{Critical Regime}\label{critical}

In this section, we are interested in the critical regime, i.e. $\lim_{z\rightarrow\infty}\frac{\lambda(z)}{z}=1$ and $\Vert h\Vert_{L^{1}}=1$. 
This regime is
very subtle. In some cases, there exists a stationary version of Hawkes process whilst in some cases there does not.
For example, Br\'{e}maud and Massouli\'{e} \cite{BremaudII} proved that 
\begin{proposition}[Br\'{e}maud and Massouli\'{e}]
Assume $\lambda(z)=z$, $\Vert h\Vert_{L^{1}}=1$ and
\begin{equation}
\sup_{t\geq 0}t^{1+\alpha}h(t)\leq R,\quad\lim_{t\rightarrow\infty}t^{1+\alpha}h(t)=r,
\end{equation}
for some finite constants $r,R>0$ and $0<\alpha<\frac{1}{2}$. Then, there exists a non-trivial stationary Hawkes process with finite intensity.
\end{proposition}

Br\'{e}maud and Massouli\'{e} considered only the linear Hawkes process in their paper \cite{BremaudII}. 
If you allow nonlinear rate function,
you get a much richer class of Hawkes processes and in some cases, there still exists a stationary Hawkes process. It is much easier
to work with the exponential case, i.e. when $h(t)=ae^{-at}$ and $\Vert h\Vert_{L^{1}}=1$.

The lecture notes by Hairer \cite{Hairer} provides a sufficient condition for which there exists an invariant probability measure. 
Let $\mathcal{L}$ be the generator of a Markov process. 
If there exists $V\geq 1$, continuous, with precompact sublevel sets and some function $\phi:\mathbb{R}^{+}\rightarrow\mathbb{R}^{+}$ strictly concave, 
increasing, with $\phi(0)=0$, and $\phi(x)\rightarrow\infty$ as $x\rightarrow\infty$ and $\mathcal{L}V\leq K-\phi(V)$ for some $K>0$, 
then there exists an invariant probability measure.

\begin{proposition}\label{criticalstationary}
Assume $h(t)=ae^{-at}$, $a>0$ and $\lambda(z)=z-\psi(z)+\nu$, where $\psi(z)$ is positive, increasing, strictly concave 
and $\psi(z)\rightarrow\infty$ and $\frac{\psi(z)}{z}\rightarrow 0$ as $z\rightarrow\infty$.
If also $\lambda(z)$ is strictly positive. Then there exists an invarint probability measure.
\end{proposition}

\begin{proof}
Let $V(z)=z+1$ and $\phi(V)=a(\psi(V)-\psi(0))$. Then $\phi:\mathbb{R}^{+}\rightarrow\mathbb{R}^{+}$ is increasing 
and strictly concave, $\phi(z)\rightarrow\infty$, and $\phi(0)=0$. Recall that the generator is given by
\begin{equation}
\mathcal{A}f(z)=-az\frac{\partial f}{\partial z}+\lambda(z)[f(z+a)-f(z)].
\end{equation}
Hence, we have
\begin{equation}
\mathcal{A}V+\phi(V)=-\psi(z)a+a\psi(z+1)-a\psi(0)+a\nu\leq a\psi(1)-2a\psi(0)+a\nu.
\end{equation}
\end{proof}

We can generalize our result to the much wider class of $h(\cdot)$ when $h(\cdot)$ is a sum of exponentials: 
$h(t)=\sum_{i=1}^{d}a_{i}e^{-b_{i}t}$, where $b_{i}>0$ and $a_{i}>0$, $1\leq i\leq d$. 
Write $Z_{i}(t)=\sum_{\tau<t}a_{i}e^{-b_{i}(t-\tau)}$. Then $Z_{t}=\sum_{i=1}^{d}Z_{i}(t)$ 
and $(Z_{1}(t),\ldots,Z_{d}(t))$ is Markovian with the generator
\begin{equation}
\mathcal{A}f=-\sum_{i=1}^{d}b_{i}z_{i}\frac{\partial f}{\partial z_{i}}
+\lambda\left(\sum_{i=1}^{d}z_{i}\right)\cdot\left[f(z_{1}+a_{1},\ldots,z_{d}+a_{d})-f(z_{1},\ldots,z_{d})\right].
\end{equation}
We have the following result.

\begin{proposition}\label{criticalstationaryII}
Assume $h(t)=\sum_{i=1}^{d}a_{i}e^{-b_{i}t}$, $b_{i}>0$ and $a_{i}>0$, $1\leq i\leq d$
and $\Vert h\Vert_{L^{1}}=\sum_{i=1}^{d}\frac{a_{i}}{b_{i}}=1$.
Also assume that $\lambda(z)=z-\psi(z)+\nu$, where $\psi(z)$ is positive, increasing, strictly concave 
and $\psi(z)\rightarrow\infty$ and $\frac{\psi(z)}{z}\rightarrow 0$ as $z\rightarrow\infty$ 
and $\lambda(z)$ is strictly positive. Then, there exists an invariant probability measure.
\end{proposition}

\begin{proof}
Let $V=\sum_{i=1}^{d}\frac{z_{i}}{b_{i}}+1$ and $\phi(V)=\psi(\min_{1\leq i\leq d}b_{i}V)-\psi(0)$. 
Then $\phi:\mathbb{R}^{+}\rightarrow\mathbb{R}^{+}$ is increasing 
and strictly concave, $\phi(z)\rightarrow\infty$ as $z\rightarrow\infty$ and $\phi(0)=0$.
Using the concavity and monotonicity of $\psi(\cdot)$, we have
\begin{align}
&\mathcal{A}V+\phi(V)
\\
&=-\psi\left(\sum_{i=1}^{d}z_{i}\right)+
\psi\left(\min_{1\leq i\leq d}b_{i}\sum_{i=1}^{d}b_{i}\frac{z_{i}}{b_{i}}+\min_{1\leq i\leq d}b_{i}\right)-\psi(0)+\nu\nonumber
\\
&\leq-\psi\left(\min_{1\leq i\leq d}b_{i}\sum_{i=1}^{d}\frac{z_{i}}{b_{i}}\right)+
\psi\left(\min_{1\leq i\leq d}b_{i}\sum_{i=1}^{d}b_{i}\frac{z_{i}}{b_{i}}+\min_{1\leq i\leq d}b_{i}\right)-\psi(0)+\nu\nonumber
\\
&\leq\psi\left(\min_{1\leq i\leq d}b_{i}\right)-2\psi(0)+\nu.\nonumber
\end{align}
\end{proof}

\begin{remark}
The following $\psi(z)$ satisfies the assumptions in Proposition \ref{criticalstationaryII} for sufficiently large $\nu>0$.

(i) $\psi(z)=(c_{1}+c_{2}z)^{\alpha}$, where $c_{1},c_{2}>0$ and $0<\alpha<1$.

(ii) $\psi(z)=\log(c_{3}+z)$, where $c_{3}>1$.
\end{remark}

\begin{remark}
Let $\mu$ be the invariant probability measure for $(Z_{1}(t),\ldots,Z_{d}(t))$ in Prosposition \ref{criticalstationaryII}. Then, we have
$\int\psi\left(\min_{1\leq i\leq d}b_{i}\sum_{i=1}^{d}\frac{z_{i}}{b_{i}}+1\right)\mu(dz)<\infty$.
\end{remark}

Indeed, when $h(\cdot)$ may not be exponential or a sum of exponentials, we have the following result.
\begin{theorem}\label{criticalstationaryIII}
Assume $\lambda(z)=\nu+z-\psi(z)$, where $\psi(\cdot):\mathbb{R}^{+}\rightarrow\mathbb{R}^{+}$ satisfies $\lim_{z\rightarrow\infty}\psi(z)=\infty$ 
and $\lim_{z\rightarrow\infty}\frac{\psi(z)}{z}=0$
and also $\lambda(z)$ is increasing. Also assume that $\Vert h\Vert_{L^{1}}=1$. Then there exists a stationary Hawkes process satisfying
the dynamics \eqref{dynamics}. 
\end{theorem}

\begin{proof}
The proof uses Poisson embedding and follows the ideas in Br\'{e}maud and Massouli\'{e} \cite{Bremaud}.
Consider the canonical space of a point process on $\mathbb{R}^{2}$ in which $\overline{N}$ is Poisson with intensity $1$.
Let $\lambda^{0}_{t}=Z^{0}_{t}=0$, $t\in\mathbb{R}$ and let $N^{0}$ be the point process counting the points of $\overline{N}$
below the curve $t\mapsto\lambda^{0}_{t}$, i.e. $N^{0}=\emptyset$. Define recursively the processes $\lambda^{n}_{t}$, $Z^{n}_{t}$
and $N^{n}$, $n\geq 0$ as follows.
\begin{align}
&\lambda^{n+1}_{t}=\lambda\left(\int_{-\infty}^{t}h(t-s)N^{n}(ds)\right),
\quad Z^{n+1}_{t}=\int_{-\infty}^{t}h(t-s)N^{n}(ds),\quad t\in\mathbb{R},
\\
& N^{n+1}(C)=\int_{C}\overline{N}(dt\times[0,\lambda^{n+1}_{t}]),\quad C\in\mathcal{B}(\mathbb{R}).\nonumber
\end{align}
By our construction, $\lambda^{n}_{t}$ is an $\mathcal{F}^{\overline{N}}_{t}$-intensity of $N^{n}$ (see Br\'{e}maud and Massouli\'{e} \cite{Bremaud}). 
Since $\lambda(\cdot)$
is increasing, the processes $\lambda^{n}_{t}$, $Z^{n}_{t}$ and $N^{n}$ are increasing in $n$ 
Thus, the limit processes $\lambda_{t}$, $Z_{t}$, $N$ exist. Since $\lambda^{n}_{t}$, $Z^{n}_{t}$ are stationary in $t$
and increasing in $n$, we have
\begin{equation}
\mathbb{E}\lambda^{n+1}_{0}=\nu+\mathbb{E}[\lambda^{n}_{0}]\int_{0}^{\infty}h(t)dt-\mathbb{E}\psi(Z^{n+1}_{0})
\leq\nu+\mathbb{E}\lambda^{n+1}_{0}-\mathbb{E}\psi(Z^{n+1}_{0}).
\end{equation}
Therefore, by Fatou's lemma, $\mathbb{E}[\psi(Z_{0})]\leq\nu<\infty$. Thus, $\psi(Z_{t})$ is finite a.s. Since
$\lim_{z\rightarrow\infty}\psi(z)=\infty$, $Z_{t}$ is finite a.s. and thus $\lambda_{t}$ is finite a.s.
$N$, which counts the number of points of $\overline{N}$ below the curve $t\mapsto\lambda_{t}$, admits $\lambda_{t}$ as
an $\mathcal{F}^{\overline{N}}_{t}$-intensity. The monotonicity implies
\begin{equation}
\lambda^{n}_{t}\leq\lambda\left(\int_{-\infty}^{t}h(t-s)N(ds)\right),
\quad\lambda_{t}\geq\lambda\left(\int_{-\infty}^{t}h(t-s)N^{n}(ds)\right).
\end{equation}
Letting $n\rightarrow\infty$, we complete the proof.
\end{proof}

\begin{remark}
The following $\psi(z)$ satisfies the assumptions in Theorem \ref{criticalstationaryIII}.

(i) $\psi(z)=(c_{1}+c_{2}z)^{\alpha}$, where $c_{1},c_{2}>0$, $0<\alpha<1$, $\nu>c_{1}^{\alpha}$ and $\alpha c_{1}^{\alpha-1}c_{2}<1$.

(ii) $\psi(z)=\log(c_{3}+z)$, where $1<c_{3}<e^{\nu}$.
\end{remark}

Next, let us consider the critical linear case, i.e. $\lambda(z)=\nu+z$, $\nu>0$ and $\Vert h\Vert_{L^{1}}$=1. We also assume
that $m:=\int_{0}^{\infty}th(t)dt<\infty$. There is no stationary Hawkes process in this regime and in the rest of this subsection, 
we will try to understand its time asymptotics.

First, let us prove a lemma concerning the expectations of $\lambda_{t}$ and $N_{t}$.

\begin{lemma}\label{momentoflambda}
Assume $\lambda(z)=\nu+z$, $\nu>0$ and $\Vert h\Vert_{L^{1}}$=1 and $m=\int_{0}^{\infty}th(t)dt<\infty$. We have
\begin{equation}
\lim_{t\rightarrow\infty}\frac{\mathbb{E}[\lambda_{t}]}{t}=\frac{\nu}{m},\quad
\lim_{t\rightarrow\infty}\frac{\mathbb{E}[N_{t}]}{t^{2}}=\frac{\nu}{2m}.
\end{equation}
\end{lemma}

\begin{proof}
Since
\begin{equation}
\lambda_{t}=\nu+\int_{0}^{t}h(t-s)dN_{s},
\end{equation}
taking $f(t)=\mathbb{E}[\lambda_{t}]$, we get
\begin{equation}
f(t)=\nu+\int_{0}^{t}h(t-s)f(s)ds=\nu+\int_{0}^{t}h(s)f(t-s)ds.
\end{equation}
Taking the Laplace transform on both sides of the equation, it is easy to see that 
the Laplace transform $\hat{f}$ of $f$ is given by
\begin{equation}
\hat{f}(\sigma)=\frac{\nu}{\sigma(1-\hat{h}(\sigma))}\sim\frac{\nu}{m}\frac{1}{\sigma^{2}},
\quad\text{as $\sigma\downarrow 0$,}
\end{equation}
since $\hat{h}(0)=1$ by $\Vert h\Vert_{L^{1}}=1$ and $\frac{1-\hat{h}(\sigma)}{\sigma}\sim-\hat{h}'(0)=m$. 
By a Tauberian theorem, 
(see Chapter XIII of Feller \cite{Feller}), we get $\frac{f(t)}{t}\rightarrow\frac{\nu}{m}$ as $t\rightarrow\infty$.
Using the simple fact that $\mathbb{E}[N_{t}]=\int_{0}^{t}f(s)ds$, we complete the proof.
\end{proof}

\begin{theorem}\label{SquareBesselThm}
Assume $\lambda(z)=\nu+z$, $\nu>0$ and $\Vert h\Vert_{L^{1}}=1$, $m=\int_{0}^{\infty}th(t)dt<\infty$
and $h(\cdot)$ Lipschitz. We have the following asymptotics.

(i) As $T\rightarrow\infty$, on $D[0,1]$,
\begin{equation}
\frac{N_{tT}}{T^{2}}\rightarrow\int_{0}^{t}\eta_{s}ds,
\end{equation}
where $\eta_{t}$ is a squared Bessel process, i.e.
\begin{equation}
d\eta_{t}=\frac{\nu}{m}dt+\frac{1}{m}\sqrt{\eta_{t}}dB_{t},\quad\eta_{0}=0.
\end{equation}

(ii) $\lim_{T\rightarrow\infty}N\left[T,T+\frac{t}{T}\right]=P(t)$,
where $P(t)$ is a P\'{o}lya process with parameters $\frac{1}{2m^{2}}$ and $2\nu m$. 
\end{theorem}

\begin{remark}
The fact that a squared Bessel process arises in the limit of a critical linear Hawkes process is not a surprise.
It is well known that a critical branching process after certain scalings will converge to a squared
Bessel process in the limit. This was discovered by Wei and Winnicki \cite{Wei}.
\end{remark}

\begin{remark}
Before we proceed to the proof of Theorem \ref{SquareBesselThm}, let us recall
that a P\'{o}lya process with parameters $\alpha$ and $\beta$ is a point process defined as the following.
Generate a positive random variable $\xi$, with Gamma distribution of parameters $\alpha$ (shape) and $\beta$ (scale).
Conditional on $\xi$, $P(t)$ is a Poisson process with intensity $\xi$. The marginal distribution of $P(t)$ is negative
binomial and unlike the usual Poisson process, P\'{o}lya process has dependent increments. The covariance
of the increments can be computed explicitly as Cov$(P(t+\delta t)-P(t),P(t))=t\cdot\delta t\cdot\alpha\beta^{2}$.
Peng and Kou \cite{Peng} used P\'{o}lya process to model clustering effects in the credit markets.
\end{remark}

\begin{proof}[Proof of Theorem \ref{SquareBesselThm}]
(i) Let $H(t):=\int_{t}^{\infty}h(s)ds$. Then, we have $H(0)=1$ and $\int_{0}^{\infty}H(t)dt=\int_{0}^{\infty}th(t)dt=m$.
Let $M_{t}:=N_{t}-\int_{0}^{t}\lambda_{s}ds$.
Let us integrate $\lambda_{s}=\int_{0}^{s}h(s-u)N(du)+\nu$ over $0\leq s\leq tT$. We get
\begin{equation}
\int_{0}^{tT}\lambda_{s}ds=\int_{0}^{tT}\int_{0}^{s}h(s-u)dM_{u}ds+\int_{0}^{tT}\int_{0}^{s}h(s-u)\lambda_{u}duds+\nu tT.
\end{equation}
Rearranging the equation and dividing by $T$, we get
\begin{equation}
\frac{1}{T}\left[\int_{0}^{tT}\lambda_{s}ds-\int_{0}^{tT}\int_{0}^{s}h(s-u)\lambda_{u}duds\right]
=\frac{1}{T}\int_{0}^{tT}\int_{0}^{s}h(s-u)dM_{u}ds+\nu t.
\end{equation}
Fubini's theorem implies that
\begin{equation}
\frac{1}{T}\left[\int_{0}^{tT}\lambda_{u}du-\int_{0}^{tT}\left(\int_{0}^{tT-u}h(s)ds\right)\lambda_{u}du\right]
=\frac{1}{T}\int_{0}^{tT}\left(\int_{0}^{tT-u}h(s)ds\right)dM_{u}+\nu t.
\end{equation}
By the definition of $H(\cdot)$, this is equivalent to
\begin{equation}\label{convolution}
\int_{0}^{t}TH(tT-uT)\frac{\lambda_{uT}}{T}du=\frac{M_{tT}}{T}+\nu t
+\frac{1}{T}\int_{0}^{t}TH(tT-uT)d\left(\frac{M_{uT}}{T}\right).
\end{equation}
$\frac{M_{tT}}{T}$ is a martingale and the tightness can be easily established. 
Furthermore, we have
\begin{equation}
\sup_{T>0}\mathbb{E}\left[\left(\frac{M_{tT}}{T}\right)^{2}\right]
=\sup_{T>0}\frac{1}{T^{2}}\mathbb{E}\left[\int_{0}^{tT}\lambda_{s}ds\right]<\infty,
\end{equation}
since $\mathbb{E}[\lambda_{t}]\leq Ct$ for some $C>0$ by Lemma \ref{momentoflambda}.
This implies that the limit of $\frac{M_{tT}}{T}$ is also a martinagle.

Moreover, $\frac{N_{tT}}{T^{2}}$ and $\int_{0}^{t}\frac{\lambda_{sT}}{T}ds$ are both tight.
To see this, since $N_{t}$ and $\lambda_{t}$ are nonnegative, 
we can think of $\left(d\left(\frac{N_{tT}}{T^{2}}\right),0\leq t\leq 1\right)$ 
and $\left(\frac{\lambda_{tT}}{T}dt,0\leq t\leq 1\right)$
as two measures. But by Lemma \ref{momentoflambda}, we know that there exist some positive
contant $C>0$, such that
\begin{equation}
\mathbb{E}\left[\frac{N_{T}}{T^{2}}\right]\leq C\quad\text{and}\quad
\mathbb{E}\left[\int_{0}^{1}\frac{\lambda_{sT}}{T}ds\right]\leq C,
\end{equation}
uniformly in $T>0$. Therefore,
$\left(d\left(\frac{N_{tT}}{T^{2}}\right),0\leq t\leq 1\right)$ and $\left(\frac{\lambda_{tT}}{T}dt,0\leq t\leq 1\right)$
are tight in the weak topology. Hence, their distribution functions $\frac{N_{tT}}{T^{2}}$
and $\int_{0}^{t}\frac{\lambda_{sT}}{T}ds$ are tight in $D[0,1]$ equipped with the Skorohod topology.
Let us say that $\frac{M_{tT}}{T}\rightarrow\beta_{t}$, $\frac{N_{tT}}{T^{2}}\rightarrow\psi_{t}$
and $\int_{0}^{t}\frac{\lambda_{sT}}{T}ds\rightarrow\phi_{t}$ as $T\rightarrow\infty$.
Since the jumps of $\frac{N_{tT}}{T^{2}}$ are uniformly bounded by $\frac{1}{T^{2}}$ which goes
to zero as $T\rightarrow\infty$, we conclude that $\psi_{t}$ is continuous. Similarly,
$\beta_{t}$ and $\phi_{t}$ are continuous. Moreover, the difference
\begin{equation}
\frac{N_{tT}}{T^{2}}-\int_{0}^{t}\frac{\lambda_{sT}}{T}ds=\frac{M_{tT}}{T^{2}},
\end{equation}
is a martingale and by Doob's martingale inequality, for any $\epsilon>0$,
\begin{equation}
\mathbb{P}\left(\sup_{0\leq t\leq 1}\left|\frac{M_{tT}}{T^{2}}\right|\geq\epsilon\right)
\leq\frac{4}{T^{4}}\mathbb{E}\left[\int_{0}^{tT}\lambda_{s}ds\right]\rightarrow 0,
\end{equation}
as $T\rightarrow\infty$. Therefore, $\psi_{t}=\phi_{t}$.
Let us denote $TH(\cdot T)$ by $H_{T}$, $\frac{M_{\cdot T}}{T}$ by $M_{T}$ and
$\int_{0}^{\cdot}\frac{\lambda_{sT}}{T}ds$ by $\Lambda_{T}$.
For any smooth function $K(\cdot)$ supported on $\mathbb{R}^{+}$, taking the convolutions of the both sides
of \eqref{convolution}, we get
\begin{equation}
K\ast H_{T}\ast\Lambda_{T}=K\ast M_{T}+K\ast(\nu\cdot)+\frac{1}{T}K\ast H_{T}\ast M_{T}.
\end{equation}
Letting $T\rightarrow\infty$, using the fact that $\int_{0}^{\infty}H(t)dt=\int_{0}^{\infty}th(t)dt=m$,
we get
\begin{equation}
m\int_{0}^{t}K(t-s)d\phi_{s}=\int_{0}^{t}K(t-s)(\beta_{s}+\nu s)ds.
\end{equation}
Since this is true for any $K$, we get $\frac{d\phi_{t}}{dt}=\frac{\beta_{t}}{m}+\frac{\nu}{t}$.
Finally,
\begin{equation}
\left(\frac{M_{tT}}{T}\right)^{2}-\int_{0}^{t}\frac{\lambda_{sT}}{T}ds
\end{equation}
is a martingale and if we let $T\rightarrow\infty$, we conclude that $\beta_{t}^{2}-\phi_{t}$ is a martingale.
Let $\eta_{t}:=\frac{d\phi_{t}}{dt}$. We have proved that $\frac{N_{tT}}{T^{2}}\rightarrow\int_{0}^{t}\eta_{s}ds$
weakly on $D[0,1]$ equipped with Skorohod topology and $\eta_{t}$ is a squared Bessel process,
\begin{equation}
d\eta_{t}=\frac{\nu}{m}dt+\frac{1}{m}\sqrt{\eta_{t}}dB_{t},\quad\eta_{0}=0.
\end{equation}

(ii) $N[T,T+\frac{t}{T}]$ has the compensator $\int_{T}^{T+\frac{t}{T}}\lambda_{s}ds$.
Observe that $\int_{T}^{T+\frac{t}{T}}\lambda_{s}ds=T^{2}\int_{1}^{1+\frac{t}{T^{2}}}\frac{\lambda_{sT}}{T}ds
\rightarrow\eta_{1}t$ as $T\rightarrow\infty$, where $\eta_{1}$ has a Gamma distribution with 
shape $\frac{1}{2m^{2}}$ and scale $2\nu m$.
\end{proof}

Now let us consider the case
when $h(\cdot)$ has heavy tail, i.e. $\int_{0}^{\infty}th(t)dt=\infty$. 
Let us first prove the following lemma.

\begin{lemma}\label{HeavyExpectation}
Assume that
\begin{equation}
1-\int_{0}^{t}h(s)ds=\int_{t}^{\infty}h(s)ds\sim t^{-\alpha},\quad 0<\alpha<1.
\end{equation}
Then, 
\begin{equation}
\lim_{t\rightarrow\infty}\frac{\mathbb{E}[\lambda_{t}]}{t^{\alpha}}=\nu\cdot\frac{\sin\pi\alpha}{\pi\alpha},
\quad
\lim_{t\rightarrow\infty}\frac{\mathbb{E}[N_{t}]}{t^{1+\alpha}}=\frac{\nu}{\Gamma(1-\alpha)\Gamma(2+\alpha)}.
\end{equation}
\end{lemma}

\begin{proof}
The Tauberian theorem of Chapter XIII of Feller \cite{Feller} says that
\begin{equation}
1-\hat{h}(\sigma)\sim\Gamma(1-\alpha)\sigma^{\alpha},\quad\sigma\rightarrow 0^{+}.
\end{equation}
Let $\mathbb{E}[\lambda_{t}]=f(t)$. This implies that
\begin{equation}
\hat{f}(\sigma)=\frac{\nu}{\sigma(1-\hat{h}(\sigma))}\sim\frac{\nu\sigma^{-1-\alpha}}{\Gamma(1-\alpha)},
\quad\sigma\rightarrow 0^{+},
\end{equation}
which again by a Tauberian theorem (Theorem 2 of Chapter XIII.5 of Feller \cite{Feller}) implies 
\begin{equation}
\int_{0}^{t}f(s)ds\sim\frac{\nu}{\Gamma(1-\alpha)\Gamma(2+\alpha)}t^{1+\alpha},\quad t\rightarrow\infty.
\end{equation}
Hence, 
\begin{equation}
\mathbb{E}[N_{t}]=\int_{0}^{t}\mathbb{E}[\lambda_{s}]ds
=\int_{0}^{t}f(s)ds\sim\frac{\nu}{\Gamma(1-\alpha)\Gamma(2+\alpha)}t^{1+\alpha},\quad t\rightarrow\infty.
\end{equation}
Since $\mathbb{E}[\lambda_{t}]=\nu+\int_{0}^{t}h(t-s)d\mathbb{E}[N_{s}]$, it is easy to check that
\begin{equation}
\mathbb{E}[\lambda_{t}]=f(t)\sim\frac{\nu}{\Gamma(1-\alpha)\Gamma(1+\alpha)}t^{\alpha}
=\nu\cdot\frac{\sin\pi\alpha}{\pi\alpha}\cdot t^{\alpha},\quad t\rightarrow\infty.
\end{equation}
\end{proof}

We obtain the following law of large numbers.
\begin{theorem}
Assume that $\int_{t}^{\infty}h(s)ds\sim\frac{1}{t^{\alpha}}$, $0<\alpha<1$. Then,
\begin{equation}
\frac{N_{t}}{t^{1+\alpha}}\rightarrow\frac{\nu}{\Gamma(1-\alpha)\Gamma(2+\alpha)}
\quad\text{and}\quad
\frac{\lambda_{t}}{t^{\alpha}}\rightarrow\nu\cdot\frac{\sin\pi\alpha}{\pi\alpha},
\quad\text{a.s. as $t\rightarrow\infty$}.
\end{equation}
\end{theorem}

\begin{proof}
Let $X_{t}=N_{t}-\mathbb{E}[N_{t}]$. Then, $X_{t}$ satisfies (see Bacry et al. \cite{Bacry})
\begin{equation}
X_{t}=M_{t}+\int_{0}^{t}\psi(t-s)M_{s}ds,
\end{equation}
where $M_{t}=N_{t}-\int_{0}^{t}\lambda_{s}ds$ and $\psi=\sum_{n}h^{\ast n}$. Then, by Doob's maximal inequality,
it is not hard to see that
\begin{align}
\mathbb{E}\left[\left(\frac{N_{t}-\mathbb{E}[N_{t}]}{t^{1+\alpha}}\right)^{2}\right]
&\leq\frac{1}{t^{2+2\alpha}}\mathbb{E}\left[\sup_{s\leq t}M_{s}^{2}\right]\left(1+\int_{0}^{t}\psi(t-s)ds\right)^{2}
\\
&\leq\frac{C}{t^{2+2\alpha}}t^{1+\alpha}(t^{\alpha})^{2}\rightarrow 0,
\end{align}
as $t\rightarrow\infty$ since $0<\alpha<1$. Hence, as $t\rightarrow\infty$,
\begin{equation}
\frac{N_{t}}{t^{1+\alpha}}\rightarrow\frac{\nu}{\Gamma(1-\alpha)\Gamma(2+\alpha)},
\quad\text{in $L^{2}$ as $t\rightarrow\infty$.}
\end{equation}
To show the almost sure convergence, we need only to show that $\frac{1}{t}\sup_{s\leq t}M_{s}\rightarrow 0$
a.s. as $t\rightarrow\infty$. Define $Y_{t}=\int_{0}^{t}\frac{1}{1+s}dM_{s}$. Then by Lemma \ref{HeavyExpectation},
\begin{equation}
\sup_{t>0}\mathbb{E}[Y_{t}^{2}]=\int_{0}^{\infty}\frac{\mathbb{E}[\lambda_{s}]}{(1+s)^{2}}ds<\infty.
\end{equation}
By the martingale convergence theorem, $Y_{t}\rightarrow Y_{\infty}$ a.s. as $t\rightarrow\infty$.
It follows that
\begin{equation}
\frac{M_{t}}{t+1}=Y_{t}-\frac{1}{t+1}\int_{0}^{t}Y_{s}ds\rightarrow 0,
\end{equation}
a.s. as $t\rightarrow\infty$. From here, it is easy to show that $\frac{1}{t}\sup_{s\leq t}M_{s}\rightarrow 0$ a.s.
Finally, since $\lambda_{t}=\nu+\int_{0}^{t}h(t-s)N(ds)$, we conclude that
\begin{equation}
\frac{\lambda_{t}}{t^{\alpha}}\rightarrow\nu\cdot\frac{\sin\pi\alpha}{\pi\alpha},\quad
\text{a.s. as $t\rightarrow\infty$.}
\end{equation}
\end{proof}

\section{Super-Critical Regime}\label{super}

In this section, we are interested in the super-critical regime, i.e. $\lim_{z\rightarrow\infty}\frac{\lambda(z)}{z}=1$
and $\Vert h\Vert_{L^{1}}>1$. First, let us compute the asymptotics for the expectations.
Let $\theta>0$ be the unique positive number such that $\int_{0}^{\infty}e^{-\theta t}h(t)dt=1$. $\theta$ is sometimes
referred to as the Malthusian parameter in the literature.
Let us also define
\begin{equation}
\overline{h}(t)=h(t)e^{-\theta t},
\quad
\overline{m}=\int_{0}^{\infty}th(t)e^{-\theta t}dt.
\end{equation}
Clearly under our assumptions $0<\overline{m}<\infty$ and $\Vert\overline{h}\Vert_{L^{1}}=1$.

\begin{lemma}\label{superexpectation}
(i) Assume $\lambda(z)=\nu+z$, $\nu>0$ being a constant. Then, 
\begin{equation}
\lim_{t\rightarrow\infty}\frac{\mathbb{E}[\lambda_{t}]}{e^{\theta t}}=\frac{\nu}{\theta\overline{m}}.
\end{equation}

(ii) Assume $\lim_{z\rightarrow\infty}\frac{\lambda(z)}{z}=1$ and let $\lambda(\cdot)$ be bounded below by a positive constant. Then,
\begin{equation}
\lim_{t\rightarrow\infty}\frac{1}{t}\log\mathbb{E}[\lambda_{t}]
=\lim_{t\rightarrow\infty}\frac{1}{t}\log\mathbb{E}[N_{t}]=\theta.
\end{equation}
\end{lemma}

\begin{proof}
(i) Let $f(t)=\mathbb{E}[\lambda_{t}]$. We have
\begin{equation}
\frac{f(t)}{e^{\theta t}}=\frac{\nu}{e^{\theta t}}+\int_{0}^{t}h(t-s)e^{-\theta(t-s)}\frac{f(s)}{e^{\theta s}}ds
=\frac{\nu}{e^{\theta t}}+\int_{0}^{t}\overline{h}(t-s)\frac{f(s)}{e^{\theta s}}ds
\end{equation}
Taking Laplace transform, we get
\begin{equation}
\widehat{f(t)e^{-\theta t}}(\sigma)=\frac{\nu}{\theta(1-\hat{\overline{h}}(\sigma))}\sim\frac{\nu}{\theta\overline{m}}\cdot\frac{1}{\sigma},
\end{equation}
as $\sigma\downarrow 0$. By the Tauberian theorem, we have
\begin{equation}
\lim_{t\rightarrow\infty}\frac{\mathbb{E}[\lambda_{t}]}{e^{\theta t}}=\frac{\nu}{\theta\overline{m}}.
\end{equation}

(ii) is a direct consequence of (i).
\end{proof}

This is consistent with the exponential case when $h(t)=ae^{-bt}$ and $a>b$. We have
\begin{equation}
\mathbb{E}[\lambda_{t}]=-\frac{\nu b}{a-b}+\frac{\nu a}{a-b}e^{(a-b)t},
\quad
\mathbb{E}[N_{t}]=-\frac{\nu bt}{a-b}+\frac{\nu a}{(a-b)^{2}}(e^{(a-b)t}-1).
\end{equation}
Indeed, in the exponential case, $\theta=a-b$ and
\begin{equation}
d(Z_{t}e^{-(a-b)t})=e^{-(a-b)t}dZ_{t}+Z_{t}de^{-(a-b)t}
=-aZ_{t}e^{-(a-b)t}dt+ae^{-(a-b)t}dN_{t}.
\end{equation}
Let $Y_{t}=Z_{t}e^{-(a-b)t}$. We have
\begin{equation}
dY_{t}=-aY_{t}dt+ae^{-(a-b)t}dN_{t}=\nu ae^{-(a-b)t}dt+ae^{-(a-b)t}dM_{t}.
\end{equation}
If we assume that $N(-\infty,0]=0$, then $Z_{0}=0$ and
\begin{equation}
Y_{t}=\int_{0}^{t}\nu ae^{-(a-b)s}ds+a\int_{0}^{t}e^{-(a-b)s}dM_{s}.
\end{equation}
Clearly, $\int_{0}^{t}\nu e^{-(a-b)s}ds\rightarrow\frac{\nu a}{a-b}$ and $\int_{0}^{t}ae^{-(a-b)s}dM_{s}$ is a martingale and
\begin{equation}
\sup_{t>0}\mathbb{E}\left[\left(\int_{0}^{t}e^{-(a-b)s}dM_{s}\right)^{2}\right]
=\int_{0}^{\infty}e^{-2(a-b)s}\mathbb{E}[\lambda_{s}]ds=\frac{\nu(2a-b)}{2(a-b)^{2}}<\infty.
\end{equation}
Therefore, by the martingale convergence theorem, there exists some $W$ in $L^{2}(\mathbb{P})$ such that
\begin{equation}
\frac{\lambda_{t}}{e^{(a-b)t}}\rightarrow\frac{\nu a}{a-b}+aW,
\end{equation}
as $t\rightarrow\infty$. The convergence is a.s. and also in $L^{2}(\mathbb{P})$.

For the general $h(\cdot)$ such that $\Vert h\Vert_{L^{1}}>1$, we may even consider the case when $\Vert h\Vert_{L^{1}}=\infty$.
For instance, if we assume that $h(\cdot)$ is decreasing and continuous and then $h(\cdot)$ is bounded and all the arguments for
the case $1<\Vert h\Vert_{L^{1}}<\infty$ would work for the case $\Vert h\Vert_{L^{1}}=\infty$ as well. 

\begin{theorem}\label{superlambda}
Assume $\lambda(z)=\nu+z$, $\nu>0$. We have,
\begin{equation}
\frac{\lambda_{t}}{e^{\theta t}}\rightarrow\frac{\nu}{\theta\overline{m}}+\frac{W}{\overline{m}},
\quad\text{a.s. as $t\rightarrow\infty$},
\end{equation}
where $W=\int_{0}^{\infty}e^{-\theta t}dM_{t}$.
\end{theorem}

\begin{proof}
It is not very hard to observe that
\begin{align}
\frac{\lambda_{t}}{e^{\theta t}}
&=\frac{\nu}{e^{\theta t}}+\int_{0}^{t}\frac{h(t-s)}{e^{\theta t}}dM_{s}+\int_{0}^{t}\frac{h(t-s)}{e^{\theta t}}\lambda_{s}ds
\\
&=\frac{\nu}{e^{\theta t}}+\int_{0}^{t}h(t-s)e^{-\theta(t-s)}\left(\frac{dM_{s}}{e^{\theta s}}\right)
+\int_{0}^{t}h(t-s)e^{-\theta(t-s)}\frac{\lambda_{s}}{e^{\theta s}}ds\nonumber
\\
&=\frac{\nu}{e^{\theta t}}+\int_{0}^{t}\overline{h}(t-s)d\overline{M}_{s}
+\int_{0}^{t}\overline{h}(t-s)\frac{\lambda_{s}}{e^{\theta s}}ds.\nonumber
\end{align}
Taking Laplace transform, we get
\begin{align}
\widehat{\lambda_{t}e^{-\theta t}}(\sigma)&=\frac{\frac{\nu}{\theta+\sigma}+\hat{\overline{h}}(\sigma)
\int_{0}^{\infty}e^{-\sigma t}d\overline{M}_{t}}{1-\hat{\overline{h}}(\sigma)}
=\frac{\frac{\nu}{\theta+\sigma}+\hat{\overline{h}}(\sigma)\int_{0}^{\infty}e^{-(\sigma+\theta)t}dM_{t}}{1-\hat{\overline{h}}(\sigma)}
\\
&\sim\frac{\frac{\nu}{\theta}+W}{\overline{m}}\cdot\frac{1}{\sigma},\nonumber
\end{align}
as $\sigma\downarrow 0$, where $W=\int_{0}^{\infty}e^{-\theta t}dM_{t}$. Notice that $W$ is well defined a.s. because
$\int_{0}^{t}e^{-\theta s}dM_{s}$ is a martingale and
\begin{equation}
\sup_{t>0}\mathbb{E}\left[\left(\int_{0}^{t}e^{-\theta s}dM_{s}\right)^{2}\right]
=\int_{0}^{\infty}e^{-2\theta s}\mathbb{E}[\lambda_{s}]ds<\infty
\end{equation}
by Lemma \ref{superexpectation}.
Hence, by the Tauberian theorem, we conclude that, as $t\rightarrow\infty$,
\begin{equation}
\frac{\lambda_{t}}{e^{\theta t}}\rightarrow\frac{\nu}{\theta\overline{m}}+\frac{W}{\overline{m}}.\quad\text{a.s.}
\end{equation}
\end{proof}

\begin{corollary}
$\frac{N_{t}}{e^{\theta t}}\rightarrow\frac{\nu}{\theta^{2}\overline{m}}
+\frac{W}{\theta\overline{m}}$ a.s. as $t\rightarrow\infty$.
\end{corollary}

\begin{proof}
Let $M_{t}=N_{t}-\int_{0}^{t}\lambda_{s}ds$. Then, since $M_{t}$ is a martingale
and $\mathbb{E}[M_{t}^{2}]=\int_{0}^{t}\mathbb{E}\lambda_{s}ds\leq Ce^{\theta t}$ for some $C>0$, 
it is easy to see that $\frac{M_{t}}{e^{\theta t}}\rightarrow 0$ a.s. as $t\rightarrow\infty$. 
On the other hand,
\begin{equation}
\frac{1}{e^{\theta t}}\int_{0}^{t}\lambda_{s}ds=\int_{0}^{t}e^{-\theta(t-s)}\frac{\lambda_{s}}{e^{\theta s}}ds
\rightarrow\frac{1}{\theta}\left[\frac{\nu}{\theta\overline{m}}+\frac{W}{\overline{m}}\right],
\end{equation}
by Theorem \ref{superlambda}. Hence, we get the desired result.
\end{proof}

\begin{remark}
It would be interesting to study the properties of $W$ defined in Theorem \ref{superlambda}.
Observe that
\begin{align}
\mathbb{E}\left[e^{-\sigma\int_{0}^{t}e^{-\theta s}dM_{s}}\right]
&=\mathbb{E}\left[e^{-\sigma\left(\int_{0}^{t}e^{-\theta s}dN_{s}-\int_{0}^{t}\int_{0}^{s}h(s-u)N(du)e^{-\theta s}ds
\right)}\right]e^{\frac{\sigma}{\theta}\nu(1-e^{-\theta t})}
\\
&=\mathbb{E}\left[e^{-\sigma\int_{0}^{t}\left(e^{-\theta s}-\int_{s}^{t}h(u-s)e^{-\theta u}du\right)N(ds)}\right]
e^{\frac{\sigma}{\theta}\nu(1-e^{-\theta t})}\nonumber
\\
&=\mathbb{E}\left[e^{-\sigma\int_{0}^{t}e^{-\theta s}\overline{H}(t-s)N(ds)}\right]
e^{\frac{\sigma}{\theta}\nu(1-e^{-\theta t})},\nonumber
\end{align}
where $\overline{H}(t)=\int_{t}^{\infty}\overline{h}(s)ds$. Hence, 
\begin{equation}
\mathbb{E}[e^{-\sigma W}]
=e^{\frac{\sigma\nu}{\theta}}\lim_{t\rightarrow\infty}e^{\nu\int_{0}^{t}g_{t}(s)ds},
\end{equation}
where $g_{t}(s)=\exp\left\{-\frac{\sigma}{e^{\theta t}}e^{\theta s}\overline{H}(s)
+\int_{0}^{s}h(s-u)g_{t}(u)du\right\}-1$.
\end{remark}

\section{Explosive Regime}\label{explosive}

In this section, we will provide an explosion, non-explosion criterion for nonlinear Hawkes processes, together with some
asymptotics for the explosion time in the explosive regime. Let $\tau_{\ell}=\inf\{t>0:\lambda_{t}\geq\ell\}$. The quantity
\begin{equation}
\lim_{\ell\rightarrow\infty}\mathbb{P}(\tau_{\ell}\leq t)=F(t)=\mathbb{P}(\tau\leq t),
\end{equation}
is defined as the distribution function of the explosion time $\tau$. We say there is no explosion if $F\equiv 0$, 
otherwise there is explosion. For a short introduction to explosion, non-explosion, we refer to Varadhan \cite{VaradhanIII}.

Next, we provide an explosion, non-explosion criterion for nonlinear Hawkes processes. The proof is based on a well known result for the
explosion, non-explosion criterion for a class of point processes which can be found in the book by Kallenberg \cite{Kallenberg}.

\begin{theorem}[Explosion, Non-Explosion Criterion]\label{criterion}
Assume that $\lambda(\cdot)$ is increasing and that $h(\cdot)$ is integrable
and decreasing, then there is explosion if and only if
\begin{equation}
\sum_{n=0}^{\infty}\frac{1}{\lambda(n)}<\infty.
\end{equation}
\end{theorem}

\begin{proof}
Observe that, for any $T>0$,
\begin{equation}
\mathbb{P}^{h(T)}(\tau\leq T)\leq\mathbb{P}(\tau\leq T)\leq\mathbb{P}^{h(0)}(\tau\leq T),
\end{equation}
where $\mathbb{P}^{h(0)}$ denotes the probability measure for the point process such that initially the rate function is $\lambda(0)$
and after $n$th jumps, the rate function becomes $\lambda(nh(0))$; $\mathbb{P}^{h(T)}$ is defined similarly. It is well known that
the point process with intensity $\lambda(N_{t-})$ is explosive if and only if
\begin{equation}
\sum_{n=0}^{\infty}\frac{1}{\lambda(n)}<\infty.
\end{equation}
For the details and proof of the above result, we refer to Kallenberg \cite{Kallenberg}.
But it is clear under our assumptions that $\sum_{n=0}^{\infty}\frac{1}{\lambda(n)}<\infty$ 
if and only if $\sum_{n=0}^{\infty}\frac{1}{\lambda(cn)}<\infty$,
where $c>0$ is any positive constant.
Therefore, there is explosion if and only if
\begin{equation}
\sum_{n=0}^{\infty}\frac{1}{\lambda(n)}<\infty.
\end{equation}
\end{proof}

Evaluating the exact probability distribution of the explosion time $\tau$, i.e. $\mathbb{P}(\tau\leq t)$, is hard and almost impossible. 
Nevertheless, one can still study its asymptotic behavior, i.e.

(i) $\mathbb{P}(\tau\geq t)$ for large time $t$;

(ii) $\mathbb{P}(\tau\leq\epsilon)$ for small time $\epsilon$.

In the rest of this section, we will use 
Proposition \ref{explosionlarge} to answer (i) and Proposition \ref{explosionsmall} to answer (ii).

\begin{proposition}\label{explosionlarge}
Under the assumptions of Theorem \ref{criterion} satisfying the explosion criterion, we have
\begin{equation}
\lim_{t\rightarrow\infty}\frac{1}{t}\log\mathbb{P}^{\emptyset}(\tau\geq t)=\inf_{t>0}\frac{1}{t}\log\mathbb{P}^{\emptyset}(\tau\geq t)=-\sigma,
\end{equation}
where $0<\sigma<\infty$.
\end{proposition}

\begin{proof}
For a nonlinear Hawkes process with empty history, i.e. $N(-\infty,0]=0$, we have
\begin{equation}
\mathbb{P}^{\emptyset}(\tau\geq t+s)=\mathbb{P}^{\emptyset}(\tau\geq t+s|\tau\geq s)\mathbb{P}^{\emptyset}(\tau\geq s)
\leq\mathbb{P}^{\emptyset}(\tau\geq t)\mathbb{P}^{\emptyset}(\tau\geq s).
\end{equation}
Therefore, $\log\mathbb{P}^{\emptyset}(\tau\geq t)$ is sub-additive and we know that
\begin{equation}
\lim_{t\rightarrow\infty}\frac{1}{t}\log\mathbb{P}^{\emptyset}(\tau\geq t)=\inf_{t>0}\frac{1}{t}\log\mathbb{P}^{\emptyset}(\tau\geq t)=-\sigma
\end{equation}
exists. And we also know that $0<\sigma<\infty$. For example, it is easy to see that $\sigma\leq\lambda(0)$. That is because
$\mathbb{P}^{\emptyset}(\tau\geq t)\geq\mathbb{P}^{\emptyset}(N[0,t]=0)=e^{-\lambda(0)t}$. To see that $\sigma>0$, choose $M$ large 
enough so that $\mathbb{P}(\tau\geq M)<1$ and then $\sigma\geq-\frac{1}{M}\log\mathbb{P}(\tau\geq M)>0$.
\end{proof}

\begin{remark}
Indeed, in the Markovian case, we can say something more
about $\sigma$ defined in Proposition \ref{explosionlarge}. 
When $h(t)=ae^{-bt}$, $Z_{t}=\sum_{\tau<t}ae^{-b(t-\tau)}$ is Markovian and by noticing that
\begin{equation}
\exp\left\{f(Z_{t})-f(Z_{0})-\int_{0}^{t}\frac{\mathcal{A}e^{f}}{e^{f}}(Z_{s})ds\right\}
\end{equation}
is a martingale and that $N_{t}$ explodes if and only if $Z_{t}$ explodes, we have
\begin{equation}
\lim_{t\rightarrow\infty}\frac{1}{t}\log\mathbb{P}^{\emptyset}(\tau\geq t)=-\sigma,
\end{equation}
where $\sigma$ is the principal eigenvalue for
\begin{equation}
\mathcal{A}u=-\sigma u,\quad u\geq 1.
\end{equation}
Note that here you have to choose the test function $u\geq 1$ rather than $u\geq 0$.
\end{remark}

\begin{proposition}\label{explosionsmall}
Assume that $\lambda(z)=\gamma z^{k}+\delta$, where $\gamma,\delta>0$ and $k>1$. According to Theorem \ref{criterion}, it is 
in the explosive regime. We have the following asymptotics for small time $\epsilon$.
\begin{equation}
\lim_{\epsilon\rightarrow 0}\epsilon^{\frac{1}{k-1}}\log\mathbb{P}(\tau\leq\epsilon)
=C_{k}^{\frac{k}{k-1}}(k^{-\frac{1}{k-1}}-k^{-\frac{k}{k-1}}),
\end{equation}
where $C_{k}=\int_{0}^{\infty}\log\left(\frac{\gamma y^{k}h(0)^{k}}{\gamma y^{k}h(0)^{k}+1}\right)dy$.
\end{proposition}

Before we proceed, let us first quote de Bruijn's Tauberian theorem from the book by Bingham, Goldie and Teugels \cite{Bingham}, which will
be used in the proof of Proposition \ref{explosionsmall}.

\begin{theorem}[de Bruijn's Tauberian theorem]
Let $\mu$ be a measure on $(0,\infty)$ whose Laplace transform $M(\lambda):=\int_{0}^{\infty}e^{-\lambda x}d\mu(x)$ converges for all $\lambda>0$.
If $\alpha<0$, $\phi\in\mathcal{R}_{\alpha}(0+)$, i.e. $\phi(\lambda t)/\phi(t)\sim\lambda^{\alpha}$ as $t\sim 0+$, 
put $\psi(\lambda):=\phi(\lambda)/\lambda\in\mathcal{R}_{\alpha-1}(0+)$, then, for $B>0$,
\begin{equation}
-\log\mu(0,x]\sim\frac{B}{\bar{\phi}(1/x)},\quad x\rightarrow 0+,
\end{equation}
if and only if
\begin{equation}
-\log M(\lambda)\sim(1-\alpha)\left(\frac{B}{-\alpha}\right)^{\frac{\alpha}{\alpha-1}}\frac{1}{\bar{\psi}(\lambda)},\quad\lambda\rightarrow\infty.
\end{equation}
Here, $\bar{\phi}(\lambda):=\sup\{t:\phi(t)>\lambda\}$ and similarly for $\bar{\psi}$.
\end{theorem}

\begin{proof}[Proof of Proposition \ref{explosionsmall}]
First, let us observe that since we are considering the event $\{\tau\leq\epsilon\}$ for $\epsilon>0$ very small. It is sufficient
to consider the point process with intensity $\lambda(h(0)N_{t-})$ at time $t$.

To apply de Bruijin's Tauberian theorem, notice that
\begin{equation}
-\log M(\sigma)=-\sum_{i=0}^{\infty}\log\left(\frac{\lambda(ih(0))}{\lambda(ih(0))+\sigma}\right).
\end{equation}
Recall that $\lambda(z)=\gamma z^{k}+\delta$, where $\gamma,\delta>0$ and $k>1$. Then,
\begin{align}
-\log M(\sigma)&=-\sum_{i=0}^{\infty}\log\left(\frac{\gamma i^{k}h(0)^{k}+\delta}{\gamma i^{k}h(0)^{k}+\delta+\sigma}\right)
\\
&\geq-\int_{1}^{\infty}\log\left(\frac{\gamma x^{k}h(0)^{k}+\delta}{\gamma x^{k}h(0)^{k}+\delta+\sigma}\right)dx\nonumber
\\
&=-\sigma^{1/k}\int_{1/\sigma^{1/k}}^{\infty}\log\left(\frac{\gamma\sigma y^{k}h(0)^{k}+\delta}{\gamma\sigma y^{k}h(0)^{k}+\delta+\sigma}\right)dy
\nonumber
\\
&\sim-\sigma^{1/k}\int_{0}^{\infty}\log\left(\frac{\gamma y^{k}h(0)^{k}}{\gamma y^{k}h(0)^{k}+1}\right)dy,
\quad\text{as $\sigma\rightarrow\infty$.}\nonumber
\end{align}
Similarly,
\begin{align}
-\log M(\sigma)&\leq-\int_{0}^{\infty}\log\left(\frac{\gamma x^{k}h(0)^{k}+\delta}{\gamma x^{k}h(0)^{k}+\delta+\sigma}\right)dx
\\
&\sim-\sigma^{1/k}\int_{0}^{\infty}\log\left(\frac{\gamma y^{k}h(0)^{k}}{\gamma y^{k}h(0)^{k}+1}\right)dy
\quad\text{as $\sigma\rightarrow\infty$.}\nonumber
\end{align}
Now let $C_{k}=\int_{0}^{\infty}\log\left(\frac{\gamma y^{k}h(0)^{k}}{\gamma y^{k}h(0)^{k}+1}\right)dy$, $\phi(t)=t^{1-k}$, $\psi(t)=t^{-k}$
and $\alpha=1-k<0$. Then $\bar{\phi}(1/\epsilon)=(1/\epsilon)^{-\frac{1}{k-1}}$ and $\bar{\psi}(\sigma)=\sigma^{-\frac{1}{k}}$. To apply the theorem,
we need to solve $B$ such that
\begin{equation}
(1-\alpha)\left(\frac{B}{-\alpha}\right)^{\frac{\alpha}{\alpha-1}}=k\left(\frac{B}{k-1}\right)^{\frac{k-1}{k}}=C_{k},
\end{equation}
for $B=(k-1)(C_{k}/k)^{\frac{k}{k-1}}$. Therefore,
\begin{equation}
\lim_{\epsilon\rightarrow 0}\epsilon^{\frac{1}{k-1}}\log\mathbb{P}(\tau\leq\epsilon)
=C_{k}^{\frac{k}{k-1}}(k^{-\frac{1}{k-1}}-k^{-\frac{k}{k-1}}).
\end{equation}
\end{proof}

%% file: chap6.tex
\chapter{Limit Theorems for Marked Hawkes Processes\label{chap:six}}

\section{Introduction and Main Results}

\subsection{Introduction}

We consider in this chapter a linear Hawkes process with random marks. Let $N_{t}$ be a simple point process.
$N_{t}$ denotes the number of points in the interval $[0,t)$. Let $\mathcal{F}_{t}$ be
the natural filtration up to time $t$.
We assume that $N(-\infty,0]=0$. At time $t$, the point process has $\mathcal{F}_{t}$-predictable intensity
\begin{equation}
\lambda_{t}:=\nu+Z_{t},\quad Z_{t}:=\sum_{\tau_{i}<t}h(t-\tau_{i},a_{i}),\label{dynamicsmarked}
\end{equation}
where $\nu>0$, the $(\tau_{i})_{i\geq 1}$ are arrival times of the points, and the $(a_{i})_{i\geq 1}$ are i.i.d. random marks,
$a_{i}$ being independent of previous
arrival times $\tau_{j}$, $j\leq i$. Let us assume that $a_{i}$ has a common distribution $q(da)$ on a metric space $\mathbb{X}$.
Here, $h(\cdot,\cdot):\mathbb{R}^{+}\times\mathbb{X}\rightarrow\mathbb{R}^{+}$ is integrable,
i.e. $\int_{0}^{\infty}\int_{\mathbb{X}}h(t,a)q(da)dt<\infty$. 
Let $H(a):=\int_{0}^{\infty}h(t,a)dt$ for any $a\in\mathbb{X}$. We also assume that
\begin{equation}
\int_{\mathbb{X}}H(a)q(da)<1.\label{lessthanone}
\end{equation}
Let $\mathbb{P}^{q}$ denote the probability measure for the $a_{i}$'s with the common law $q(da)$. 
Under assumption \eqref{lessthanone}, it is well known that there exists a unique stationary version of the linear
marked Hawkes process satisfying the dynamics \eqref{dynamicsmarked} and that by ergodic theorem, a law
of large numbers holds, 
\begin{equation}
\lim_{t\rightarrow\infty}\frac{N_{t}}{t}=\frac{\nu}{1-\mathbb{E}^{q}[H(a)]}.
\end{equation}
This chapter is organized as follows. In Section \ref{MainResults},
we will introduce the main results of this paper, i.e. the central limit theorem and the large deviation principle
for linear marked Hawkes processes. The proof of the central limit theorem will be given
in Section \ref{CLTProof} and the proof of the large deviation principle will be given in
Section \ref{LDPProof}. Finally, we will discuss an application of our results to a risk model in finance
in Section \ref{RiskModel}.

\subsection{Main Results}\label{MainResults}

For a linear marked Hawkes process satisfying
the dynamics \eqref{dynamicsmarked}, we have the following large deviation principle.

\begin{theorem}[Large Deviation Principle]\label{LDP}
Assume the conditions \eqref{lessthanone} and
\begin{equation}
\lim_{x\rightarrow\infty}\left\{\int_{\mathbb{X}}e^{H(a)x}q(da)-x\right\}=\infty.\label{anothercondition}
\end{equation}
Then, 
$(N_{t}/t\in\cdot)$ satisfies a large deviation principle with rate function,
\begin{align}
\Lambda(x)
&:=
\begin{cases}
\inf_{\hat{q}}\left\{x\mathbb{E}^{\hat{q}}[H(a)]+\nu-x+x\log\left(\frac{x}{x\mathbb{E}^{\hat{q}}[H(a)]+\nu}\right)
+x\mathbb{E}^{\hat{q}}\left[\log\frac{d\hat{q}}{dq}\right]\right\} &\text{$x\geq 0$}
\\
+\infty &\text{$x<0$}
\end{cases}\nonumber
\\
&=
\begin{cases}
\theta_{\ast}x-\nu(x_{\ast}-1) &\text{$x\geq 0$}
\\
+\infty &\text{$x<0$}
\end{cases},\nonumber
\end{align}
where the infimum of $\hat{q}$
is taken over $\mathcal{M}(\mathbb{X})$,
the space of probability measures on $\mathbb{X}$
such that $\hat{q}$ is absolutely continuous w.r.t. $q$. 
Here, $\theta_{\ast}$ and $x_{\ast}$
satisfy the following equations
\begin{equation}
\begin{cases}
x_{\ast}=\mathbb{E}^{q}\left[e^{\theta_{\ast}+(x_{\ast}-1)H(a)}\right]
\\
\frac{x}{\nu}=x_{\ast}+\frac{x}{\nu}\mathbb{E}^{q}\left[H(a)e^{\theta_{\ast}+(x_{\ast}-1)H(a)}\right]
\end{cases}.
\end{equation}
\end{theorem}

\begin{theorem}[Central Limit Theorem]\label{CLT}
Assume $\lim_{t\rightarrow\infty}t^{1/2}\int_{t}^{\infty}\mathbb{E}^{q}[h(s,a)]ds=0$ and
that \eqref{lessthanone} holds. Then, 
\begin{equation}
\frac{N_{t}-\frac{\nu t}{1-\mathbb{E}^{q}[H(a)]}}{\sqrt{t}}
\rightarrow N\left(0,\frac{\nu(1+\text{Var}^{q}[H(a)])}{(1-\mathbb{E}^{q}[H(a)])^{3}}\right),
\end{equation}
in distribution as $t\rightarrow\infty$.
\end{theorem}

\section{Proof of Central Limit Theorem}\label{CLTProof}

\begin{proof}[Proof of Theorem \ref{CLT}]
First, let us observe that
\begin{align}
\int_{0}^{t}\lambda_{s}ds
&=\nu t+\sum_{\tau_{i}<t}\int_{\tau_{i}}^{t}h(s-\tau_{i},a_{i})ds
\\
&=\nu t+\sum_{\tau_{i}<t}H(a_{i})-\mathcal{E}_{t},\nonumber
\end{align}
where the error term $\mathcal{E}_{t}$ is given by
\begin{equation}
\mathcal{E}_{t}:=\sum_{\tau_{i}<t}\int_{t}^{\infty}h(s-\tau_{i},a_{i})ds.
\end{equation}
Therefore, 
\begin{align}
\frac{N_{t}-\int_{0}^{t}\lambda_{s}ds}{\sqrt{t}}
&=\frac{N_{t}-\nu t-\sum_{\tau_{i}<t}H(a_{i})}{\sqrt{t}}+\frac{\mathcal{E}_{t}}{\sqrt{t}}\label{rearrange}
\\
&=(1-\mathbb{E}^{q}[H(a)])\frac{N_{t}-\mu t}{\sqrt{t}}+\frac{\mathbb{E}^{q}[H(a)]N_{t}-\sum_{\tau_{i}<t}H(a_{i})}{\sqrt{t}}
+\frac{\mathcal{E}_{t}}{\sqrt{t}}\nonumber,
\end{align}
where $\mu:=\frac{\nu}{1-\mathbb{E}^{q}[H(a)]}$. Rearranging the terms in \eqref{rearrange},
we get
\begin{equation}
\frac{N_{t}-\mu t}{\sqrt{t}}
=\frac{1}{1-\mathbb{E}^{q}[H(a)]}
\left[\frac{N_{t}-\int_{0}^{t}\lambda_{s}ds}{\sqrt{t}}+\frac{\sum_{\tau_{i}<t}(H(a_{i})-\mathbb{E}^{q}[H(a)])}{\sqrt{t}}
-\frac{\mathcal{E}_{t}}{\sqrt{t}}\right].
\end{equation}
It is easy to check that $\frac{\mathcal{E}_{t}}{\sqrt{t}}\rightarrow 0$ in probability as $t\rightarrow\infty$. To see this,
first notice that $\mathbb{E}[\lambda_{t}]\leq\frac{\nu}{1-\mathbb{E}^{q}[H(a)]}$ uniformly in $t$.
Let $g(t,a):=\int_{t}^{\infty}h(s,a)ds$. We have $\mathcal{E}_{t}=\sum_{\tau_{i}<t}g(t-\tau_{i},a_{i})$
and thus
\begin{align}
\mathbb{E}[\mathcal{E}_{t}]
&=\int_{0}^{t}\int_{\mathbb{X}}g(t-s,a)q(da)\mathbb{E}[\lambda_{s}]ds
\\
&\leq\frac{\nu}{1-\mathbb{E}^{q}[H(a)]}\int_{0}^{t}\int_{\mathbb{X}}g(t-s,a)q(da)ds\nonumber
\\
&=\frac{\nu}{1-\mathbb{E}^{q}[H(a)]}\int_{0}^{t}\mathbb{E}^{q}[g(s,a)]ds.\nonumber
\end{align}
Hence, by L'H\^{o}pital's rule,
\begin{align}
\lim_{t\rightarrow\infty}\frac{1}{t^{1/2}}\int_{0}^{t}\mathbb{E}^{q}[g(s,a)]ds
&=\lim_{t\rightarrow\infty}\frac{\mathbb{E}^{q}[g(t,a)]}{\frac{1}{2}t^{-1/2}}
\\
&=\lim_{t\rightarrow\infty}2t^{1/2}\int_{t}^{\infty}\mathbb{E}^{q}[h(s,a)]ds=0.\nonumber
\end{align}
Hence, $\frac{\mathcal{E}_{t}}{\sqrt{t}}\rightarrow 0$ in probability as $t\rightarrow\infty$.

Furthermore, $M_{1}(t):=N_{t}-\int_{0}^{t}\lambda_{s}ds$ and $M_{2}(t):=\sum_{\tau_{i}<t}(H(a_{i})-\mathbb{E}^{q}[H(a)])$
are both martingales. 

Moreover, since $\int_{0}^{t}\lambda_{s}ds$ is of finite variation, the quadratic
variation of $M_{1}(t)+M_{2}(t)$ is the same as the quadratic variation of $N_{t}+M_{2}(t)$.
And notice that $N_{t}+M_{2}(t)=\sum_{\tau_{i}<t}(1+H(a_{i})-\mathbb{E}^{q}[H(a)])$ which
has quadratic variation
\begin{equation}
\sum_{\tau_{i}<t}(1+H(a_{i})-\mathbb{E}^{q}[H(a)])^{2}.
\end{equation}
By the standard law of large numbers, we have
\begin{align}
\frac{1}{t}\sum_{\tau_{i}<t}(1+H(a_{i})-\mathbb{E}^{q}[H(a)])
&=\frac{N_{t}}{t}\cdot\frac{1}{N_{t}}\sum_{\tau_{i}<t}(1+H(a_{i})-\mathbb{E}^{q}[H(a)])^{2}
\\
&\rightarrow\frac{\nu}{1-\mathbb{E}^{q}[H(a)]}\cdot\mathbb{E}^{q}\left[(1+H(a)-\mathbb{E}^{q}[H(a)])^{2}\right]\nonumber
\\
&=\frac{\nu(1+\text{Var}^{q}[H(a)])}{1-\mathbb{E}^{q}[H(a)]},\nonumber
\end{align}
a.s. as $t\rightarrow\infty$. By a standard martingale central limit theorem, we conclude that
\begin{equation}
\frac{N_{t}-\frac{\nu t}{1-\mathbb{E}^{q}[H(a)]}}{\sqrt{t}}
\rightarrow N\left(0,\frac{\nu(1+\text{Var}^{q}[H(a)])}{(1-\mathbb{E}^{q}[H(a)])^{3}}\right),
\end{equation}
in distribution as $t\rightarrow\infty$.
\end{proof}

\section{Proof of Large Deviation Principle}\label{LDPProof}

\subsection{Limit of a Logarithmic Moment Generating Function}

In this subsection, we prove the existence of the limit of the logarithmic moment generating function
$\lim_{t\rightarrow\infty}\frac{1}{t}\log\mathbb{E}[e^{\theta N_{t}}]$ and give a variational formula 
and a more explicit formula for this limit.

\begin{theorem}\label{logarithmic}
The limit $\Gamma(\theta)$ of the logarithmic moment generating function is
\begin{equation}
\Gamma(\theta)=\lim_{t\rightarrow\infty}\frac{1}{t}\log\mathbb{E}[e^{\theta N_{t}}]
=
\begin{cases}
\nu(f(\theta)-1) &\text{if $\theta\in(-\infty,\theta_{c}]$}
\\
+\infty &\text{otherwise}
\end{cases},
\end{equation}
where $f(\theta)$ is the minimal solution to $x=\int_{\mathbb{X}}e^{\theta+H(a)(x-1)}q(da)$ and
\begin{equation}\label{thetastar}
\theta_{c}=-\log\int_{\mathbb{X}}H(a)e^{H(a)(x_{c}-1)}q(da)>0,
\end{equation}
where $x_{c}>1$ satisfies the equation $x\int_{\mathbb{X}}H(a)e^{H(a)(x-1)}q(da)=\int_{\mathbb{X}}e^{H(a)(x-1)}q(da)$.
\end{theorem}

We will break the proof of Theorem \ref{logarithmic} into the proof of the lower bound, i.e. Lemma \ref{Markedlowerbound} and the 
proof of the upper bound, i.e. Lemma \ref{Markedupperbound}.

Before we proceed, let us first prove Lemma \ref{formula}, which will be repeatedly used.

\begin{lemma}\label{formula}
Consider a linear marked Hawkes process with intensity
\begin{equation}
\lambda_{t}:=\alpha+\beta Z_{t}:=\alpha+\beta\sum_{\tau_{i}<t}h(t-\tau_{i},a_{i}),
\end{equation}
and $\beta\mathbb{E}^{q}[H(a)]<1$, where the $a_{i}$ are i.i.d. random marks with the common law $q(da)$
independent of the previous arrival times, then
there exists a unique invariant measure $\pi$ for $Z_{t}$ such that
\begin{equation}
\int\lambda(z)\pi(dz)=\frac{\alpha}{1-\beta\mathbb{E}^{q}[H(a)]}.
\end{equation}
\end{lemma}

\begin{proof}
The ergodicity of $Z_{t}$ is well known. Let $\pi$ be the invariant probability measure
for $Z_{t}$. Then
\begin{equation}
\int\lambda(z)\pi(dz)=\alpha+\beta\int_{\mathbb{X}}\int_{0}^{\infty}h(t,a)dtq(da)\int\lambda(z)\pi(dz).
\end{equation}
\end{proof}

\begin{lemma}[Lower Bound]\label{Markedlowerbound}
\begin{equation}
\liminf_{t\rightarrow\infty}\frac{1}{t}\log\mathbb{E}[e^{\theta N_{t}}]
\geq
\begin{cases}
\nu(f(\theta)-1) &\text{if $\theta\in(-\infty,\theta_{c}]$}
\\
+\infty &\text{otherwise}
\end{cases},
\end{equation}
where $f(\theta)$ is the minimal solution to $x=\int e^{\theta+H(a)(x-1)}q(da)$ and $\theta_{c}$ is defined in
\eqref{thetastar}.
\end{lemma}

\begin{proof}
The intensity at time $t$ is $\lambda_{t}:=\lambda(Z_{t})$ where $\lambda(z)=\nu+z$
and $Z_{t}=\sum_{\tau_{i}<t}h(t-\tau_{i},a_{i})$.
We tilt $\lambda$ to $\hat{\lambda}$ and $q$ to $\hat{q}$ such that by Girsanov formula
the tilted probability measure $\hat{\mathbb{P}}$ is given by
\begin{equation}
\frac{d\hat{\mathbb{P}}}{d\mathbb{P}}\bigg|_{\mathcal{F}_{t}}
=\exp\left\{\int_{0}^{t}(\lambda(Z_{s})-\hat{\lambda}(Z_{s}))ds+\int_{0}^{t}\log\left(\frac{\hat{\lambda}(Z_{s})}{
\lambda(Z_{s})}\right)
+\log\left(\frac{d\hat{q}}{dq}\right)dN_{s}\right\}.
\end{equation}
Let $\mathcal{Q}_{e}$ be the set of $(\hat{\lambda},\hat{q},\hat{\pi})$ such that the marked Hawkes process with 
intensity $\hat{\lambda}(Z_{t})$ and random 
marks distributed as $\hat{q}$ is ergodic with $\hat{\pi}$ as the invariant measure of $Z_{t}$.

By the ergodic theorem and Jensen's inequality, for any $(\hat{\lambda},\hat{q},\hat{\pi})\in\mathcal{Q}_{e}$, 
\begin{align}
&\liminf_{t\rightarrow\infty}\frac{1}{t}\log\mathbb{E}[e^{\theta N_{t}}]
\\
&\geq\liminf_{t\rightarrow\infty}\hat{\mathbb{E}}\left[\frac{1}{t}\theta
N_{t}-\frac{1}{t}\int_{0}^{t}(\lambda-\hat{\lambda})ds
-\frac{1}{t}\int_{0}^{t}\left[\log(\hat{\lambda}/\lambda)+\log(d\hat{q}/dq)\right]\hat{\lambda}ds\right]\nonumber
\\
&=\int\theta\hat{\lambda}\hat{\pi}(dz)+\int(\hat{\lambda}-\lambda)\hat{\pi}(dz)
-\iint\left[\log(\hat{\lambda}/\lambda)+\log(d\hat{q}/dq)\right]\hat{\lambda}\hat{q}\hat{\pi}(dz).\nonumber
\end{align}
Hence, 
\begin{align}
&\liminf_{t\rightarrow\infty}\frac{1}{t}\log\mathbb{E}[e^{\theta N_{t}}]
\\
&\geq\sup_{(\hat{\lambda},\hat{q},\hat{\pi})\in\mathcal{Q}_{e}}\left\{\int\theta\hat{\lambda}\hat{\pi}
+\int(\hat{\lambda}-\lambda)\hat{\pi}-\iint\left[\log(\hat{\lambda}/\lambda)+\log(d\hat{q}/dq)\right]\hat{\lambda}\hat{q}
\hat{\pi}\right\}.\nonumber
\\
&\geq\sup_{(K\lambda,\hat{q},\hat{\pi})\in\mathcal{Q}_{e}}\int\left[\left(\theta-\mathbb{E}^{\hat{q}}[\log(d\hat{q}/dq)]
\right)\hat{\lambda}
+\hat{\lambda}-\lambda-\hat{\lambda}\log\left(\hat{\lambda}/\lambda\right)\right]\hat{\pi}\nonumber
\\
&\geq\sup_{0<K<\mathbb{E}^{\hat{q}}[H(a)]^{-1},(K\lambda,\hat{q},\hat{\pi})\in\mathcal{Q}_{e}}
\int\left[\left(\theta-\mathbb{E}^{\hat{q}}[\log(d\hat{q}/dq)]\right)
+1-\frac{1}{K}-\log K\right]\hat{\lambda}\hat{\pi}\nonumber
\\
&=\sup_{\hat{q}}\sup_{0<K<\mathbb{E}^{\hat{q}}[H(a)]^{-1}}\left[\left(\theta-\mathbb{E}^{\hat{q}}[\log(d\hat{q}/dq)]
\right)+1-\frac{1}{K}-\log K\right]
\cdot\frac{K\nu}{1-K\mathbb{E}^{\hat{q}}[H(a)]},\nonumber
\end{align}
where the last equality is obtained by applying Lemma \ref{formula}. The supremum
of $\hat{q}$ is taken over $\mathcal{M}(\mathbb{X})$ such that $\hat{q}$
is absolutely continuous w.r.t. $q$. Optimizing over $K>0$, we get
\begin{align}
&\liminf_{t\rightarrow\infty}\frac{1}{t}\log\mathbb{E}[e^{\theta N_{t}}]
\\
&\geq
\begin{cases}
\sup_{\hat{q}}\nu(\hat{f}(\theta)-1) 
&\text{if
$\theta\in\big(-\infty,\mathbb{E}^{\hat{q}}\left[\log\frac{d\hat{q}}{dq}\right]+\mathbb{E}^{\hat{q}}[H(a)]-1-\log\mathbb{E}^{\hat{q}}[H(a)]\big]$
}
\\
+\infty &\text{otherwise}
\end{cases},\nonumber
\end{align}
where $\hat{f}(\theta)$ is the minimal solution to the equation
\begin{align}
x&=e^{\theta+\mathbb{E}^{\hat{q}}[\log(dq/d\hat{q})]+\mathbb{E}^{\hat{q}}[H(a)](x-1)}
\\
&\leq\mathbb{E}^{\hat{q}}\left[e^{\theta+H(a)(x-1)}\frac{dq}{d\hat{q}}\right]=\int e^{\theta+H(a)(x-1)}q(da).\nonumber
\end{align}
The last inequality is satisfied by Jensen's inequality; the equality holds if and only if
\begin{equation}
\frac{d\hat{q}}{dq}=\frac{e^{H(a)(x-1)}}{\mathbb{E}^{q}[e^{H(a)(x-1)}]}.
\end{equation}

Optimizing over $\hat{q}$, we get
\begin{equation}
\liminf_{t\rightarrow\infty}\frac{1}{t}\log\mathbb{E}[e^{\theta N_{t}}]
\geq
\begin{cases}
\nu(f(\theta)-1) &\text{if $\theta\in(-\infty,\theta_{c}]$}
\\
+\infty &\text{otherwise},
\end{cases}
\end{equation}
where $\theta_{c}$ is some critical value to be determined. Let
\begin{equation}
G(x)=e^{\theta}\int e^{H(a)(x-1)}q(da)-x.
\end{equation}
If $\theta=0$, then $G(x)=\int e^{H(a)(x-1)}q(da)-x$ satisfies $G(1)=0$, $G(\infty)=\infty$ (by \eqref{anothercondition})
and $G'(1)=\mathbb{E}^{q}[H(a)]-1<0$ which implies $\min_{x>1}G(x)<0$.
Hence, there exists some critical $\theta_{c}>0$ such that $\min_{x>1}G(x)=0$.
The critical values $x_{c}$ and $\theta_{c}$ satisfy $G(x_{c})=G'(x_{c})=0$, which implies
\begin{equation}
\theta_{c}=-\log\int H(a)e^{H(a)(x_{c}-1)}q(da),
\end{equation}
where $x_{c}>1$ satisfies the equation $x\int H(a)e^{H(a)(x-1)}q(da)=\int e^{H(a)(x-1)}q(da)$.

It is easy to check that indeed, for $dq_{\ast}=\frac{e^{H(a)(x_{\ast}-1)}}{\mathbb{E}^{q}[e^{H(a)(x_{\ast}-1)}]}dq$,
\begin{equation}
\mathbb{E}^{q_{\ast}}\left[\log\frac{dq_{\ast}}{dq}\right]+\mathbb{E}^{q_{\ast}}[H(a)]-1-\log\mathbb{E}^{q_{\ast}}[H(a)]
)
=-\log\int H(a)e^{H(a)(x_{\ast}-1)}q(da).
\end{equation}
\end{proof}

\begin{lemma}[Upper Bound]\label{Markedupperbound}
\begin{equation}
\limsup_{t\rightarrow\infty}\frac{1}{t}\log\mathbb{E}[e^{\theta N_{t}}]
\leq
\begin{cases}
\nu(f(\theta)-1) &\text{if $\theta\in(-\infty,\theta_{c}]$}
\\
+\infty &\text{otherwise}
\end{cases},
\end{equation}
where $f(\theta)$ is the minimal solution to $x=\int e^{\theta+H(a)(x-1)}q(da)$ and $\theta_{c}$ is defined in
\eqref{thetastar}.
\end{lemma}

\begin{proof} 
It is well known that a linear Hawkes process has an immigration-birth representation.
The immigrants (roots) arrive via a standard Poisson process with constant intensity $\nu>0$. Each
immigrant generates children according to a Galton-Watson tree. 
(See for example Hawkes and Oakes \cite{HawkesII} and Karabash \cite{KarabashII}.)
Consider a random, rooted tree (with root, i.e. immigrant, at time $0$) associated to the Hawkes process via the Galton-Watson
interpretation. Note the root is unmarked at the start of the process so the marking goes into the expectation calculation later.
Let $K$ be the number of
children of the root node, and let $S^{(1)}_t,
S^{(2)}_t,
\ldots,S^{(K)}_t$ be the number of descendants of root's $k$-th child that were born before time $t$ (including
$k$-th child if an only
if it was born before time $t$). Let $S_t$ be the total number of children in tree before time $t$ including root node.
Then
\begin{align}
F_S(t) 	&:=\E[\exp(\theta S_t)]\\
	&= \sum_{k=0}^{\infty} \E[\exp(\theta S_t)| K=k ]\P(K=k)\nonumber \\
	&= \exp(\theta) \sum_{k=0}^{\infty} \P(K=k) \prod_{i=1}^k \E\left[\exp\left(\theta S_t^{(i)}\right)\right]\nonumber\\
	&= \exp(\theta) \sum_{k=0}^{\infty} \E\left[\exp\left(\theta S_{t}^{(1)}\right)\right]^k \P(K=k)\nonumber\\
	&= \exp(\theta) \sum_{k=0}^{\infty} \int_{\mathbb{X}} \left[ \left( \int_{0}^{t} \frac{h(s,a)}{H(a)}F_S(t-s)ds  \right)^k 
e^{-H(a)}\frac{H(a)^k}{k!} \right] q(da)\nonumber\\
	&= \int_{\mathbb{X}} \exp  \left( \theta+\int_0^t h(s,a)(F_S(t-s)-1)ds \right) q(da).\nonumber
\end{align}
Now observe that $F_S(t)$ is strictly increasing and hence must approach to the smaller solution $x_*$ of the following
equation
\begin{equation} \label{eq:x_*}
x=\int_{\mathbb{X}} \exp \left[ \theta+H(a)(x-1) \right] q(da).
\end{equation} 
Finally, since random roots arrive according to a Poisson process with constant intensity $\nu>0$,
we have
\begin{equation}
F_N(t) 	:=\E[\exp(\theta N_t)]= \exp  \left[ \nu \int_0^t (F_S(t-s)-1) ds  \right].
\end{equation}
But since $F_S(s) \uparrow x_*$ as $s \to \infty$ we obtain the main result 
\begin{equation}
\frac 1t \log F_N(t)=\nu \frac 1t \left[ \int_0^t \left( F_S(s)-1 \right) ds  \right] \mathop{\longrightarrow}_{t \to \infty}
\nu(x_*-1) ,
\end{equation}
which proves the desired formula. Note that $x_*=\infty$ when there is no solution to \eqref{eq:x_*}.
The proof is complete.
\end{proof}

\subsection{Large Deviation Principle}

In this section, we prove the main result, i.e. Theorem \ref{LDP} by using the G\"{a}rtner-Ellis theorem for the upper bound
and tilting method for the lower bound.

\begin{proof}[Proof of Theorem \ref{LDP}]
For the upper bound, since we have Theorem \ref{logarithmic}, we can simply apply G\"{a}rtner-Ellis theorem. 
To prove the lower bound, it suffices to show that for any $x>0$, $\epsilon>0$, we have
\begin{equation}
\liminf_{t\rightarrow\infty}\frac{1}{t}\log\mathbb{P}\left(\frac{N_{t}}{t}\in B_{\epsilon}(x)\right)
\geq-\sup_{\theta\in\mathbb{R}}\{\theta x-\Gamma(\theta)\},
\end{equation}
where $B_{\epsilon}(x)$ denotes the open ball centered at $x$ with radius $\epsilon$. 
Let $\hat{\mathbb{P}}$ denote the tilted probability measure with rate $\hat{\lambda}$ and marks distributed by $\hat{q}(da)$
as defined in Lemma \ref{Markedlowerbound}. By Jensen's inequality,
\begin{align}
&\frac{1}{t}\log\mathbb{P}\left(\frac{N_{t}}{t}\in B_{\epsilon}(x)\right)
\\
&\geq\frac{1}{t}\log\int_{\frac{N_{t}}{t}\in B_{\epsilon}(x)}\frac{d\mathbb{P}}{d\hat{\mathbb{P}}}d\hat{\mathbb{P}}\nonumber
\\
&=\frac{1}{t}\log\hat{\mathbb{P}}\left(\frac{N_{t}}{t}\in B_{\epsilon}(x)\right)
-\frac{1}{t}\log\left[\frac{1}{\hat{\mathbb{P}}\left(\frac{N_{t}}{t}\in B_{\epsilon}(x)\right)}
\int_{\frac{N_{t}}{t}\in B_{\epsilon}(x)}\frac{d\hat{\mathbb{P}}}{d\mathbb{P}}d\hat{\mathbb{P}}\right]\nonumber
\\
&\geq\frac{1}{t}\log\hat{\mathbb{P}}\left(\frac{N_{t}}{t}\in B_{\epsilon}(x)\right)
-\frac{1}{\hat{\mathbb{P}}\left(\frac{N_{t}}{t}\in B_{\epsilon}(x)\right)}
\cdot\frac{1}{t}\cdot\hat{\mathbb{E}}\left[1_{\frac{N_{t}}{t}\in B_{\epsilon}(x)}\log\frac{d\hat{\mathbb{P}}}{d\mathbb{P}}\right].\nonumber
\end{align}
By the ergodic theorem,
\begin{equation}
\liminf_{t\rightarrow\infty}\frac{1}{t}\log\mathbb{P}\left(\frac{N_{t}}{t}\in B_{\epsilon}(x)\right)
\geq-\mathop{\inf_{0<K<\mathbb{E}^{\hat{q}}[H(a)]^{-1}}}_{ (K\lambda,\hat{q},\hat{\pi})\in\mathcal{Q}_{e}^{x}} 
\mathcal{H}(\hat{\lambda},\hat{q},\hat{\pi}).
\end{equation}
where $\mathcal{Q}_{e}^{x}$ is defined by
\begin{equation}
\mathcal{Q}_{e}^{x}=\left\{(\hat{\lambda},\hat{q},\hat{\pi})\in\mathcal{Q}_{e}:\int\hat{\lambda}(z)\hat{\pi}(dz)=x\right\}.
\end{equation}
and the relative entropy $\mathcal{H}$ is
\begin{equation}
\mathcal{H}(\hat{\lambda},\hat{q},\hat{\pi})=\int(\lambda-\hat{\lambda})\hat{\pi}+\int\log(\hat{\lambda}/\lambda)\hat{\lambda}\hat{\pi}
+\iint\log(d\hat{q}/dq)\hat{q}\hat{\lambda}\hat{\pi}.
\end{equation}
By Lemma \ref{formula}, 
\begin{align}
&\inf_{0<K<\mathbb{E}^{\hat{q}}[H(a)]^{-1},x=\frac{\nu K}{1-K\mathbb{E}^{\hat{q}}[H(a)]},(K\lambda,\hat{q},\hat{\pi})\in\mathcal{Q}_{e}}
\mathcal{H}(\hat{\lambda},\hat{q},\hat{\pi})
\\
&=\inf_{K=\frac{x}{x\mathbb{E}^{\hat{q}}[H(a)]+\nu},(K\lambda,\hat{q},\hat{\pi})\in\mathcal{Q}_{e}}\left\{\frac{1}{K}-1+\log K
+\mathbb{E}^{\hat{q}}\left[\log\frac{d\hat{q}}{dq}\right]\right\}\int\hat{\lambda}\hat{\pi}\nonumber
\\
&=\inf_{\hat{q}}\left\{\mathbb{E}^{\hat{q}}[H(a)]+\frac{\nu}{x}-1+\log\left(\frac{x}{x\mathbb{E}^{\hat{q}}[H(a)]+\nu}\right)
+\mathbb{E}^{\hat{q}}\left[\log\frac{d\hat{q}}{dq}\right]\right\}x\nonumber
\\
&=\inf_{\hat{q}}\left\{x\mathbb{E}^{\hat{q}}[H(a)]+\nu-x+x\log\left(\frac{x}{x\mathbb{E}^{\hat{q}}[H(a)]+\nu}\right)
+x\mathbb{E}^{\hat{q}}\left[\log\frac{d\hat{q}}{dq}\right]\right\}.\nonumber
\end{align}
Next, let us find a more explict form for the Legendre-Fenchel transform of $\Gamma(\theta)$.
\begin{equation}\label{theta}
\sup_{\theta\in\mathbb{R}}\{\theta x-\Gamma(\theta)\}=\sup_{\theta\in\mathbb{R}}\{\theta x-\nu(f(\theta)-1)\},
\end{equation}
where $f(\theta)=\mathbb{E}^{q}[e^{\theta+(f(\theta)-1)H(a)}]$. Here,
\begin{equation}
f'(\theta)=\mathbb{E}^{q}\left[(1+f'(\theta)H(a))e^{\theta+(f(\theta)-1)H(a)}\right].
\end{equation}
So the optimal $\theta_{\ast}$ for \eqref{theta} would satisfy $f'(\theta_{\ast})=\frac{x}{\nu}$ and $\theta_{\ast}$ and $x_{\ast}=f(\theta_{\ast})$
satisfy the following equations
\begin{equation}
\begin{cases}
x_{\ast}=\mathbb{E}^{q}\left[e^{\theta_{\ast}+(x_{\ast}-1)H(a)}\right]
\\
\frac{x}{\nu}=x_{\ast}+\frac{x}{\nu}\mathbb{E}^{q}\left[H(a)e^{\theta_{\ast}+(x_{\ast}-1)H(a)}\right]
\end{cases},
\end{equation}
and $\sup_{\theta\in\mathbb{R}}\{\theta x-\Gamma(\theta)\}=\theta_{\ast}x-\nu(x_{\ast}-1)$. 

On the other hand, letting $dq_{\ast}=\frac{e^{(x_{\ast}-1)H(a)}}{\mathbb{E}^{q}[e^{(x_{\ast}-1)H(a)}]}dq$, we have
\begin{equation}
\mathbb{E}^{q_{\ast}}[H(a)]=\frac{\mathbb{E}^{q}\left[e^{\theta_{\ast}+(x_{\ast}-1)H(a)}\right]}
{\mathbb{E}^{q}\left[e^{(x_{\ast}-1)H(a)}\right]}=\frac{1}{x_{\ast}}-\frac{\nu}{x},
\end{equation}
and $\mathbb{E}^{q_{\ast}}[\log\frac{dq_{\ast}}{dq}]=(x_{\ast}-1)\mathbb{E}^{q_{\ast}}[H(a)]-\log\mathbb{E}^{q}[e^{(x_{\ast}-1)H(a)}]$, 
which imply
\begin{align}
&\liminf_{t\rightarrow\infty}\frac{1}{t}\log\mathbb{P}\left(\frac{N_{t}}{t}\in B_{\epsilon}(x)\right)
\\
&\geq-\inf_{\hat{q}}\left\{x\mathbb{E}^{\hat{q}}[H(a)]+\nu-x+x\log\left(\frac{x}{x\mathbb{E}^{\hat{q}}[H(a)]+\nu}\right)
+x\mathbb{E}^{\hat{q}}\left[\log\frac{d\hat{q}}{dq}\right]\right\}\nonumber
\\
&\geq-\left\{x\mathbb{E}^{q_{\ast}}[H(a)]+\nu-x+x\log\left(\frac{x}{x\mathbb{E}^{q_{\ast}}[H(a)]+\nu}\right)
+x\mathbb{E}^{q_{\ast}}\left[\log\frac{dq_{\ast}}{dq}\right]\right\}\nonumber
\\
&=\theta_{\ast}x-\nu(x_{\ast}-1)=\sup_{\theta\in\mathbb{R}}\{\theta x-\Gamma(\theta)\}.\nonumber
\end{align}
\end{proof}

\section{Risk Model with Marked Hawkes Claims Arrivals}\label{RiskModel}

We consider the following risk model for the surplus process $R_{t}$ of an insurance portfolio,
\begin{equation}
R_{t}=u+\rho t-\sum_{i=1}^{N_{t}}C_{i},
\end{equation}
where $u>0$ is the initial reserve, $\rho>0$ is the constant premium and the $C_{i}$'s are i.i.d. positive random variables
with the common distribution $\mu(dC)$. $C_{i}$ represents the claim size at the $i$th arrival time, 
these being independent of $N_{t}$, a marked Hawkes process.

For $u>0$, let
\begin{equation}
\tau_{u}=\inf\{t>0: R_{t}\leq 0\},
\end{equation}
and denote the infinite and finite horizon ruin probabilities by
\begin{equation}
\psi(u)=\mathbb{P}(\tau_{u}<\infty),\quad\psi(u,uz)=\mathbb{P}(\tau_{u}\leq uz),\quad u,z>0.
\end{equation}

By the law of large numbers,
\begin{equation}
\lim_{t\rightarrow\infty}\frac{1}{t}\sum_{i=1}^{N_{t}}C_{i}=\frac{\mathbb{E}^{\mu}[C]\nu}{1-\mathbb{E}^{q}[H(a)]}.
\end{equation}
Therefore, to exclude the trivial case, we need to assume that
\begin{equation}\label{between}
\frac{\mathbb{E}^{\mu}[C]\nu}{1-\mathbb{E}^{q}[H(a)]}<\rho<\frac{\nu(x_{c})-1}{\theta_{c}},
\end{equation}
where the critical values $\theta_{c}$ and $x_{c}=f(\theta_{c})$ satisfy
\begin{equation}\label{ftheta}
\begin{cases}
x_{c}=\int_{\mathbb{R}^{+}}\int_{\mathbb{X}}e^{\theta_{c}C+H(a)(x_{c}-1)}q(da)\mu(dC)
\\
1=\int_{\mathbb{R}^{+}}\int_{\mathbb{X}}H(a)e^{H(a)(x_{c}-1)+\theta_{c}C}q(da)\mu(dC)
\end{cases}.
\end{equation}

Let us first assume that the claim sizes following light tails, i.e. there exists some $\theta>0$
such that $\int_{\mathbb{R}^{+}}e^{\theta C}\mu(dC)<\infty$.

Following the proofs of large deviation results in Section \ref{LDPProof}, we have
\begin{equation}
\Gamma_{C}(\theta):=\lim_{t\rightarrow\infty}\frac{1}{t}\log\mathbb{E}\left[e^{\theta\sum_{i=1}^{N_{t}}C_{i}}\right]
=
\begin{cases}
\nu(x-1) &\text{if $\theta\in(-\infty,\theta_{c}]$}
\\
+\infty &\text{otherwise}
\end{cases},
\end{equation}
where $x$ is the minimal solution to the equation
\begin{equation}
x=\int_{\mathbb{R}^{+}}\int_{\mathbb{X}}e^{\theta C+(x-1)H(a)}q(da)\mu(dC).
\end{equation}

Before we proceed, let us quote a result from Glynn and Whitt \cite{Glynn}, which
will be used in our proof Theorem \ref{InfiniteHorizon}.

\begin{theorem}[Glynn and Whitt \cite{Glynn}]\label{GlynnThm}
Let $S_{n}$ be random variables. $\tau_{u}=\inf\{n: S_{n}>u\}$ and $\psi(u)=\mathbb{P}(\tau_{u}<\infty)$. 
Assume that there exist $\gamma,\epsilon>0$ such that

(i) $\kappa_{n}(\theta)=\log\mathbb{E}[e^{\theta S_{n}}]$ is well defined and finite for $\gamma-\epsilon<\theta<\gamma+\epsilon$.

(ii) $\limsup_{n\rightarrow\infty}\mathbb{E}[e^{\theta(S_{n}-S_{n-1})}]<\infty$ for $-\epsilon<\theta<\epsilon$.

(iii) $\kappa(\theta)=\lim_{n\rightarrow\infty}\frac{1}{n}\kappa_{n}(\theta)$ exists and is finite for $\gamma-\epsilon<\theta<\gamma+\epsilon$.

(iv) $\kappa(\gamma)=0$ and $\kappa$ is differentiable at $\gamma$ with $0<\kappa'(\gamma)<\infty$.

Then, $\lim_{u\rightarrow\infty}\frac{1}{u}\log\psi(u)=-\gamma$.
\end{theorem}

\begin{remark}
We claim that $\Gamma_{C}(\theta)=\rho\theta$ has a unique positive solution $\theta^{\dagger}<\theta_{c}$. 
Let $G(\theta)=\Gamma_{C}(\theta)-\rho\theta$. Notice that $G(0)=0$, $G(\infty)=\infty$, and that $G$ is convex. 
We also have $G'(0)=\frac{\mathbb{E}^{\mu}[C]\nu}{1-\mathbb{E}^{q}[H(a)]}-\rho<0$ and $\Gamma_{C}(\theta_{c})-\rho\theta_{c}>0$ since 
we assume that $\rho<\frac{\nu(f(\theta_{c})-1)}{\theta_{c}}$. Therefore, there exists only one
solution $\theta^{\dagger}\in(0,\theta_{c})$
of $\Gamma_{C}(\theta^{\dagger})=\rho\theta^{\dagger}$.
\end{remark}

\begin{theorem}[Infinite Horizon]\label{InfiniteHorizon}
Assume all the assumptions in Theorem \ref{LDP} and in addition \eqref{between}, we have
$\lim_{u\rightarrow\infty}\frac{1}{u}\log\psi(u)=-\theta^{\dagger}$, where $\theta^{\dagger}\in(0,\theta_{c})$ 
is the unique positive solution of $\Gamma_{C}(\theta)=\rho\theta$.
\end{theorem}

\begin{proof}
Take $S_{t}=\sum_{i=1}^{N_{t}}C_{i}-\rho t$ and $\kappa_{t}(\theta)=\log\mathbb{E}[e^{\theta S_{t}}]$. 
Then $\lim_{t\rightarrow\infty}\frac{1}{t}\kappa_{t}(\theta)=\Gamma_{C}(\theta)-\rho\theta$. 
Consider $\{S_{nh}\}_{n\in\mathbb{N}}$. We have $\lim_{n\rightarrow\infty}\frac{1}{n}\kappa_{nh}(\theta)=h\Gamma_{C}(\theta)-h\rho\theta$. 
Checking the conditions in Theorem \ref{GlynnThm} and applying it, we get 
\begin{equation}
\lim_{u\rightarrow\infty}\frac{1}{u}\log\mathbb{P}\left(\sup_{n\in\mathbb{N}}S_{nh}>u\right)=-\theta^{\dagger}.
\end{equation}
Finally, notice that
\begin{equation}
\sup_{t\in\mathbb{R}^{+}}S_{t}\geq\sup_{n\in\mathbb{N}}S_{nh}\geq\sup_{t\in\mathbb{R}^{+}}S_{t}-\rho h.
\end{equation}
Hence, $\lim_{u\rightarrow\infty}\frac{1}{u}\log\psi(u)=-\theta^{\dagger}$.
\end{proof}

\begin{theorem}[Finite Horizon]
Under the same assumptions as in Theorem \ref{InfiniteHorizon}, we have
\begin{equation}
\lim_{u\rightarrow\infty}\frac{1}{u}\log\psi(u,uz)=-w(z),\quad\text{for any $z>0$}.
\end{equation}
Here
\begin{equation}
w(z)=
\begin{cases}
z\Lambda_{C}\left(\frac{1}{z}+\rho\right) &\text{if $0<z<\frac{1}{\Gamma'(\theta^{\dagger})-\rho}$}
\\
\theta^{\dagger} &\text{if $z\geq\frac{1}{\Gamma'(\theta^{\dagger})-\rho}$}
\end{cases},
\end{equation}
$\Lambda_{C}(x)=\sup_{\theta\in\mathbb{R}}\{\theta x-\Gamma_{C}(\theta)\}$
and $\theta^{\dagger}\in(0,\theta_{c})$ 
is the unique positive solution of $\Gamma_{C}(\theta)=\rho\theta$, as before.
\end{theorem}

\begin{proof}
The proof is similar to that in Stabile and Torrisi \cite{Stabile} and we omit it here.
\end{proof}

Next, we are interested to study the case when the claim sizes have heavy tails, i.e. $\int_{\mathbb{R}^{+}}e^{\theta C}\mu(dC)=+\infty$
for any $\theta>0$.

A distribution function $B$ is subexponential, i.e. $B\in\mathcal{S}$ if 
\begin{equation}
\lim_{x\rightarrow\infty}\frac{\mathbb{P}(C_{1}+C_{2}>x)}{\mathbb{P}(C_{1}>x)}=2,
\end{equation}
where $C_{1}$, $C_{2}$ are i.i.d. random variables with distribution function $B$. 
Let us denote $B(x):=\mathbb{P}(C_{1}\geq x)$
and let us assume that $\mathbb{E}[C_{1}]<\infty$ and define $B_{0}(x):=\frac{1}{\mathbb{E}[C]}\int_{0}^{x}\overline{B}(y)dy$,
where $\overline{F}(x)=1-F(x)$ is the complement of any distribution function $F(x)$.

Goldie and Resnick \cite{Goldie} showed that if $B\in\mathcal{S}$ and satisfies some smoothness
conditions, then $B$ belongs to the maximum domain of attraction of either the Frechet distribution
or the Gumbel distribution. In the former case, $\overline{B}$ is regularly varying,
i.e. $\overline{B}(x)=L(x)/x^{\alpha+1}$, for some $\alpha>0$ and we write
it as $\overline{B}\in\mathcal{R}(-\alpha-1)$, $\alpha>0$.

We assume that $B_{0}\in\mathcal{S}$ and either $\overline{B}\in\mathcal{R}(-\alpha-1)$ or $B\in\mathcal{G}$, i.e.
the maximum domain of attraction of Gumbel distribution. $\mathcal{G}$ includes Weibull and lognormal distributions.

When the arrival process $N_{t}$ satisfies a large deviation result, the probability that it deviates away
from its mean is exponentially small, which is dominated by subexonential distributions. By using the techniques
for the asymptotics of ruin probabilities for risk processes with non-stationary, 
non-renewal arrivals and subexponential claims from Zhu \cite{ZhuVI},
we have the following infinite-horizon and finite-horizon ruin probability estimates when the claim sizes
are subexponential.

\begin{theorem}
Assume the net profit condition $\rho>\mathbb{E}[C_{1}]\frac{\nu}{1-\mathbb{E}^{q}[H(a)]}$.

(i) (Infinite-Horizon)
\begin{equation}
\lim_{u\rightarrow\infty}\frac{\psi(u)}{\overline{B}_{0}(u)}
=\frac{\nu\mathbb{E}[C_{1}]}{\rho(1-\mathbb{E}^{q}[H(a)])-\nu\mathbb{E}[C_{1}]}.
\end{equation}

(ii) (Finite-Horizon) For any $T>0$,
\begin{align}
&\lim_{u\rightarrow\infty}\frac{\psi(u,uz)}{\overline{B}_{0}(u)}
\\
&=
\begin{cases}
\frac{\nu\mathbb{E}[C_{1}]}{\rho(1-\mathbb{E}^{q}[H(a)])-\nu\mathbb{E}[C_{1}]}
\left[1-\left(1+\left(\frac{\rho(1-\mathbb{E}^{q}[H(a)])-\nu\mathbb{E}[C_{1}]}{\rho(1-\mathbb{E}^{q}[H(a)])}\right)
\frac{T}{\alpha}\right)^{-\alpha}\right]
&\text{if $\overline{B}\in\mathcal{R}(-\alpha-1)$}
\\
\frac{\nu\mathbb{E}[C_{1}]}{\rho(1-\mathbb{E}^{q}[H(a)])-\nu\mathbb{E}[C_{1}]}
\left[1-e^{-\frac{\rho(1-\mathbb{E}^{q}[H(a)])-\nu\mathbb{E}[C_{1}]}{\rho(1-\mathbb{E}^{q}[H(a)])}T}\right]&\text{if $B\in\mathcal{G}$}
\end{cases}.\nonumber
\end{align}
\end{theorem}

\section{Examples with Explicit Formulas}

In this section, we discuss two examples where an explicit formula exists.

Example \ref{EX1} is about the exponential asymptotics of the infinite-horizon ruin probability when $H(a)$ 
and the claim size $C$ are exponentially distributed.
Example \ref{EX2} gives an explicit expression for the rate function of the large deviation principle
when $H(a)$ is exponentially distributed.

\begin{example}\label{EX1}
Recall that $x$ is the minimal solution of
\begin{equation}
x=\int_{\mathbb{R}^{+}}\int_{\mathbb{X}}e^{\theta C+(x-1)H(a)}q(da)\mu(dC).
\end{equation}
Now, assume that $H(a)$ is exponentially distributed with parameter $\lambda>0$, then, we have
\begin{equation}
x=\mathbb{E}^{\mu}[e^{\theta C}]\frac{\lambda}{\lambda-(x-1)},
\end{equation}
which implies that
\begin{equation}
x=\frac{1}{2}\left\{\lambda+1-\sqrt{(\lambda+1)^{2}-4\lambda\mathbb{E}^{\mu}[e^{\theta C}]}\right\}.
\end{equation}
Now, assume that $C$ is exponentially distributed with parameter $\gamma>0$. Then, 
\begin{equation}
x=\frac{1}{2}\left\{\lambda+1-\sqrt{(\lambda+1)^{2}-4\lambda\frac{\gamma}{\gamma-\theta}}\right\}.
\end{equation}
The infinite horizon probability satisfies $\lim_{u\rightarrow\infty}\frac{1}{u}\log\psi(u)=-\theta^{\dagger}$,
where $\theta^{\dagger}$ satisfies
\begin{equation}
\rho\theta^{\dagger}
=\nu\left(\frac{1}{2}\left\{\lambda+1-\sqrt{(\lambda+1)^{2}-4\lambda\frac{\gamma}{\gamma-\theta^{\dagger}}}\right\}-1\right),
\end{equation}
which implies
\begin{equation}
\frac{2\rho\theta^{\dagger}}{\nu}+1-\lambda=-\sqrt{(\lambda+1)^{2}-\frac{4\lambda\gamma}{\gamma-\theta^{\dagger}}},
\end{equation}
and thus
\begin{equation}
\frac{\rho^{2}}{\nu^{2}}(\theta^{\dagger})^{2}+\frac{\rho\theta^{\dagger}}{\nu}(1-\lambda)
=\lambda-\frac{\lambda\gamma}{\gamma-\theta^{\dagger}}=\frac{-\lambda\theta^{\dagger}}{\gamma-\theta^{\dagger}}.
\end{equation}
Since we are looking for positive $\theta^{\dagger}$, we get the quadratic equation,
\begin{equation}
\rho^{2}(\theta^{\dagger})^{2}-(\rho^{2}\gamma-\rho\nu(1-\lambda))\theta^{\dagger}-(\rho\nu\gamma(1-\lambda)+\lambda\nu^{2})=0.
\end{equation}
Since $\rho>\frac{\mathbb{E}^{\mu}[C]\nu}{1-\mathbb{E}^{q}[H(a)]}=\frac{\nu\lambda}{\gamma(\lambda-1)}$, we have
$\rho\nu\gamma(1-\lambda)+\lambda\nu^{2}>0$. Therefore,
\begin{equation}
\theta^{\dagger}=\frac{(\rho^{2}\gamma-\rho\nu(1-\lambda))+\sqrt{(\rho^{2}\gamma-\rho\nu(1-\lambda))^{2}
+4\rho^{2}(\rho\nu\gamma(1-\lambda)+\lambda\nu^{2})}}{2\rho^{2}}.
\end{equation}
\end{example}

\begin{example}\label{EX2}
Now, let $H(a)$ be exponentially distributed with parameter $\lambda>0$. We want
an explicit expression for the rate function of the large deviation principle for $(N_{t}/t\in\cdot)$.
Notice that,
\begin{equation}
\Gamma(\theta)
=
\begin{cases}
\nu\left(\frac{1}{2}\left\{\lambda+1-\sqrt{(\lambda+1)^{2}-4\lambda e^{\theta}}\right\}-1\right)
&\text{for $\theta\leq\log\left(\frac{(\lambda+1)^{2}}{4\lambda}\right)$}
\\
+\infty &\text{otherwise}
\end{cases}.
\end{equation}
To get $I(x)=\sup_{\theta\in\mathbb{R}}\{\theta x-\Gamma(\theta)\}$, we optimize over $\theta$
and consider $x=\Gamma'(\theta)$. Evidently,
\begin{equation}
x+\frac{1}{2}\nu(-4\lambda)e^{\theta}\frac{1}{2\sqrt{(\lambda+1)^{2}-4\lambda e^{\theta}}}=0,
\end{equation}
which gives us
\begin{equation}
\theta=\log\left(\frac{-2x^{2}+x\sqrt{4x^{2}+\nu^{2}(\lambda+1)^{2}}}{\lambda\nu^{2}}\right),
\end{equation}
whence,
\begin{equation}
I(x)
=
\begin{cases}
x\log\left(\frac{-2x^{2}+x\sqrt{4x^{2}+\nu^{2}(\lambda+1)^{2}}}{\lambda\nu^{2}}\right)
\\
\qquad\qquad\qquad
-\nu\left(\frac{1}{2}\left\{\lambda+1-\frac{-2x+\sqrt{4x^{2}+\nu^{2}(\lambda+1)^{2}}}{\nu}\right\}-1\right)
&\text{if $x\geq 0$}
\\
+\infty &\text{otherwise}
\end{cases}.
\end{equation}
\end{example}

%% file: biblio.tex
%

%% file: thesis.bbl
\begin{thebibliography}{120}\addcontentsline{toc}{chapter}{Bibliography}

\bibitem{BacryIV} Bacry, E., Dayri, K. and J. F. Muzy. (2012).
Non-parametric kernel estimation for symmetric Hawkes processes. Applications to high frequency financial data.
\textit{Eur. Phys. J.} B \textbf{85}:157

\bibitem{Bacry} Bacry, E., Delattre, S., Hoffmann, M. and J. F. Muzy.
Scaling limits for Hawkes processes and application to financial statistics. Preprint.
arXiv:1202.0842.

\bibitem{BacryII} Bacry, E., Delattre, S., Hoffmann, M. and J. F. Muzy.
Modeling microstructure noise with mutually exciting point processes.
To appear in \textit{Quantitative Finance}.
arXiv:1101.3422.

\bibitem{BacryIII} Bacry, E. and J. F. Muzy.
Hawkes model for price and trades high-frequency dynamics.
Preprint.
arXiv:1301.1135.

\bibitem{BartlettI}
Bartlett, M. S. (1963).
Statistical estimation of density functions.
\textit{Sankhy\={a}} A. \textbf{25}, 245-254.

\bibitem{BartlettII}
Bartlett, M. S. (1963).
The spectral analysis of point processes.
\textit{J. R. Statist. Soc.} B \textbf{25}, 264-296.

\bibitem{Bauwens}
Bauwens, L. and N. Hautsch.
Modelling financial high frequency data using point processes.
\textit{Handbook of Financial Time Series}. 953-979, 2009.

\bibitem{Billingsley} Billingsley, P. 
\textit{Convergence of Probability Measures}, 2nd edition.
Wiley-Interscience, New York, 1999.

\bibitem{Bingham} Bingham, N. H., Goldie, C. M. and J. L. Teugels, \emph{Regular Variation}, Cambridge University Press, 1989 

\bibitem{Blundell} Blundell, C., Heller, K. A. and J. M. Beck.
Modelling reciprocating relationships with Hawkes processes.
Preprint, 2012.

\bibitem{Bordenave} Bordenave, C. and G. L. Torrisi. (2007).
Large deviations of Poisson cluster processes. \textit{Stochastic Models}. \textbf{23}, 593-625. 

\bibitem{Bormetti} Bormetti, G., Calcagnile, L. M., Treccani, M., Corsi, F., Marmi, S. and F. Lillo.
Modelling systemic cojumps with Hawkes factor models.
Preprint. arXiv:1301.6141.

\bibitem{Bowsher}
Bowsher, C. G. (2007).
Modelling security market events in continuous time: intensity based,
multivariate point process models.
\textit{Journal of Econometrics}. \textbf{141}, 876-912. 

\bibitem{Bremaud} Br\'{e}maud, P. and L. Massouli\'{e}. (1996).
Stability of nonlinear Hawkes processes. \textit{Ann. Probab.}. \textbf{24}, 1563-1588.

\bibitem{BremaudII} Br\'{e}maud, P., Nappo, G. and G. L. Torrisi. (2002).
Rate of convergence to equilibrium of marked Hawkes processes. 
\textit{J. Appl. Prob.} \textbf{39}, 123-136.

\bibitem{BremaudIII} Br\'{e}maud, P. and L. Massouli\'{e}. (2001). 
Hawkes branching point processes without ancestors.
\textit{J. Appl. Prob.} \textbf{38}, 122-135.

\bibitem{BremaudIV} Br\'{e}maud, P. and S. Foss. (2010).
Ergodicity of a stress release point process seismic model with aftershocks.
\textit{Markov Processes Relat. Fields}. \textbf{16}, 389-408.

\bibitem{BremaudV} Br\'{e}maud, P. and L. Massouli\'{e}. (2002).
Power spectra of general shot noises and Hawkes point processes with a random excitation.
\textit{Advances in Applied Probability}. \textbf{34}, 205-222.

\bibitem{BremaudVI} Br\'{e}maud, P., L. Massouli\'{e} and A. Ridolfi. (2005).
Power spectra of random spike fields and related processes.
\textit{Advances in Applied Probability}. \textbf{37}, 1116-1146.

\bibitem{Brix} Brix, A. and W. S. Kendall. (2002).
Simulation of cluster point processes without edge effects. 
\textit{Advances in Applied Probability}. \textbf{34}, 267-280.

\bibitem{Carstensen} Carstensen, L., Sandelin, A., Winther, O. and N. R. Hansen. (2010).
Multivariate Hawkes process models of the occurrence of regulatory elements.
\textit{BMC Bioinformatics}. \textbf{11}:456.

\bibitem{Cartea} Cartea, \'{A}., Jaimungal, S. and J. Ricci.
Buy low sell high: a high frequency trading perspective.
\textit{SSRN eLibrary}, 2011.

\bibitem{Chavez} Chavez-Demoulin, V., Davison, A. C. and A. J. McNeil. (2005).
Estimating value-at-risk: a point process approach.
\textit{Quantitative Finance}. \textbf{5}, 227-234.

\bibitem{ChavezII} Chavez-Demoulin, V. and J. A. McGill. (2012).
High-frequency financial data modeling using Hawkes processes.
\textit{Journal of Banking \& Finance}. \textbf{36}, 3415-3426.

\bibitem{Chornoboy} Chornoboy, E. S., Schramm, L. P. and A. F. Karr. (1988).
Maximum likelihood identification of neural point process systems.
\textit{Biol. Cybern.} \textbf{59}, 265-275.

\bibitem{Crane} Crane, R. and D. Sornette. (2008).
Robust dynamic classes revealed by measuring the response function of a social system.
Proc. Nat. Acad. Sci. USA \textbf{105}, 15649.

\bibitem{Daley} Daley, D. J. and D. Vere-Jones. 
\textit{An Introduction to the Theory of Point Processes}. Volume I and II, 2nd edition. Springer-Verlag, New York, 2003.

\bibitem{Dassios} Dassios, A. and H. Zhao. (2011).
A dynamic contagion process. 
\textit{Advances in Applied Probability}. \textbf{43}, 814-846.

\bibitem{DassiosII}
Dassios, A. and H. Zhao. (2012).
Ruin by dynamic contagion claims.
\textit{Insurance: Mathematics and Economics}. \textbf{51}, 93-106.

\bibitem{Dembo} Dembo, A. and O. Zeitouni. 
\textit{Large Deviations Techniques and Applications}, 2nd edition. 
Springer, New York, 1998.

\bibitem{Donsker} Donsker, M. D. and S. R. S. Varadhan. (1983). 
Asymptotic evaluation of certain Markov process expectations for large time. IV. 
\textit{Communications of Pure and Applied Mathematics}. \textbf{36}, 183-212.

\bibitem{Echeverria} Echeverr\'{i}a, P. (1982).
A criterion for invariant measures of Markov processes.
\textit{Probability Theory and Related Fields}. \textbf{61}, 1-16.

\bibitem{Egami} Egami, M., Kato, Y. and T. Sawaki.
An analysis of CDS market liquidity by the Hawkes process.
\textit{SSRN eLibrary}, 2013.

\bibitem{Egesdal} Egesdal, M., Fathauer, C., Louie, K. and J. Neuman. 
Statistical and stochastic modelling of gang rivalries in Los Angeles.
\textit{SIAM Undergraduate Research Online}. 2010.

\bibitem{Embrechts} Embrechts, P., Liniger, T. and L. Lin. (2011).
Multivariate Hawkes processes: an application to financial data.
\textit{J. Appl. Prob. Spec. Vol.} \textbf{48A}, 367-378.

\bibitem{Errais} Errais, E., Giesecke, K. and L. Goldberg. (2010).
Affine point processes and portfolio credit risk.
\textit{SIAM J. Financial Math}. \textbf{1}, 642-665.

\bibitem{Fan} Fan, K. (1953). 
Minimax theorems. 
\textit{Proc. Natl. Acad. Sci. USA}. \textbf{39}, 42-47.

\bibitem{Feller} Feller, W., \textit{An Introduction to Probability Theory and Its Applications}, 
Volume I and Volume II, 2nd edition, New York, 1971.

\bibitem{Filimonov} Filimonov, V. and D. Sornette. (2012).
Quantifying reflexivity in financial markets: Toward a prediction of flash crashes. 
\textit{Physical Review E} \textbf{85} 056108

\bibitem{Frenk} Frenk, J. B. G. and G. Kassay. 
\textit{The Level Set Method of Jo\'{o} and Its Use in Minimax Theory}.
Technical Report E.I 2003-03, 
Econometric Institute, Erasmus University, Rotterdam, 2003.

\bibitem{GieseckeII} Giesecke, K. and L. Goldberg, L. and X. Ding (2011).
A top-down approach to multi-name credit.
\textit{Operations Research}. \textbf{59}, 283-300.

\bibitem{Giesecke} Giesecke, K. and P. Tomecek. 
Dependent events and changes of time. Working paper, Cornell University, July 2005.

\bibitem{Glynn} Glynn, P. W. and W. Whitt. (1994). 
Logarithmic asymptotics for steady-state tail probabilities in a single-server queue. 
\textit{J. Appl. Probab.} \textbf{31}, 131-156.

\bibitem{Goldie}
Goldie, C. M. and S. Resnick. (1988). 
Distributions that are both subexponential and in the domain
of attraction of an extreme value distribution. 
\textit{Adv. Appl. Probab}. \textbf{20}, 706-718.

\bibitem{Grandell} Grandell, J. (1977)
Point processes and random measures.  
\textit{Advances in Applied Probability}. \textbf{9}, 502-526.

\bibitem{Gusto} Gusto, G. and S. Schbath. (2005).
F.A.D.O.: a statistical method to detect favored or avoided distances between occurrences of motifs using the Hawkes model.
\textit{Stat. Appl. Genet. Mol. Biol.}, \textbf{4}, Article 24.

\bibitem{Hairer} Hairer, M.
\textit{Convergence of Markov Processes}, 
Lecture Notes, University of Warwick, available at \texttt{http://www.hairer.org/notes/Convergence.pdf}, 2010.

\bibitem{Halpin} Halpin, P. F. and P. De Boeck.
Modelling dyadic interaction with Hawkes process.
Preprint, 2012. To appear in \textit{Psychometrika}.

\bibitem{Hansen} Hansen, N. R., Reynaud-Bouret, P. and V. Rivoirard.
Lasso and probabilistic inequalities for multivariate point processes.
Preprint. arXiv:1208.0570.

\bibitem{Hardiman} Hardiman, S. J., Bercot, N. and J-P. Bouchaud.
Critical reflexivity in financial markets: a Hawkes process analysis.
Preprint. arXiv:1302.1405.

\bibitem{Hawkes} Hawkes, A. G. (1971).
Spectra of some self-exciting and mutually exciting point processes.
\textit{Biometrika}. \textbf{58}, 83-90.

\bibitem{HawkesIII} Hawkes, A. G. (1971).
Point spectra of some mutually exciting point processes.
\textit{J. Roy. Statist. Soc. Ser. B} \textbf{33}, 438-443.

\bibitem{HawkesIV} Hawkes, A. G. and L. Adamopoulos. (1973).
Cluster models for earthquakes-regional comparisons.
\textit{Bull. Int. Statist. Inst.} \textbf{45}, 454-461.

\bibitem{HawkesII} Hawkes, A. G. and D. Oakes. (1974). 
A cluster process representation of a self-exciting process. 
\textit{J. Appl. Prob}. \textbf{11}, 93-503.

\bibitem{Hegemann} Hegemann, R. A., Lewis, E. A. and A. L. Bertozzi. (2013).
An ``Estimate \& Score Algorithm'' for simultaneous parameter estimation and 
reconstruction of incomplete data on social networks.
\textit{Security Informatics}. \textbf{2}:1.

\bibitem{Hewlett} Hewlett, P.
Clustering of order arrivals, price impact and trade path optimisation. 
Workshop on Financial Modeling with Jump processes,
\'{E}cole Polytechnique, 2006.

\bibitem{Heyde} Heyde, C. C. and Scott, D. J. (1973).
Invariance principle for the law of the iterated logarithm for martingales and processes with stationary increments.
\textit{Ann. Probab}. \textbf{1}, 428-436. 

\bibitem{Jagers} Jagers, P. 
\textit{Branching Processes with Biological Applications}. John Wiley, London, 1975.

\bibitem{Johnson} Johnson, D. H. (1996).
Point process models of single-neuron discharges.
\textit{J. Computational Neuroscience}. \textbf{3}, 275-299.

\bibitem{Joo} Jo\'{o}, I. (1984). 
Note on my paper ``A simple proof for von Neumann's minmax theorem''. 
\textit{Acta. Math. Hung.} \textbf{44}, 363-365.

\bibitem{Kallenberg} Kallenberg, O. \emph{Foundations of Modern Probability}, Springer, 2nd edition, 2002.

\bibitem{Karabash} Karabash, D. and L. Zhu.
Limit theorems for marked Hawkes processes with application to a risk model. 
Preprint. arXiv:1211.4039.

\bibitem{KarabashII} Karabash, D. 
On stability of Hawkes process.
Preprint. arXiv:1201.1573.

\bibitem{Koralov} Koralov, L. B. and Ya. G. Sinai. 
\textit{Theory of Probability and Random Processes}.
Springer, 2nd edition, 2012.

\bibitem{Krumin} Krumin, M., Reutsky I. and S. Shoham. (2010).
Correlation-based analysis and generation of multiple spike trains using Hawkes models with an exogenous input.
\textit{Frontiers in Computational Neuroscience}. \textbf{4}, article 147.

\bibitem{Kwiecinski} Kwieci\'{n}ski, A. and R. Szekli. (1996).
Some monotonicity and dependence properties of self-exciting point processes.
\textit{Annals of Applied Probability}. \textbf{6}, 1211-1231.

\bibitem{Large}
Large, J. (2007).
Measuring the resiliency of an electronic limit order book.
\textit{Journal of Financial Markets}. \textbf{10}, 1-25. 

\bibitem{LewisII} Lewis, E. and G. Mohler.
A nonparametric EM algorithm for multiscale Hawkes processes.
Preprint, 2011.

\bibitem{LewisIII} Lewis, E., Mohler, Brantingham, P. J. and A. Bertozzi. (2011).
Self-exciting point process of insergency in Iraq.
\textit{Security Journal}. \textbf{25}, 0955-1662.

\bibitem{Lewis} Lewis, P. A. W. and G. S. Shedler. (1979).
Simulation of nonhomogeneous Poisson processes by thinning. 
\textit{Naval Research Logistics Quarterly}. \textbf{26}, 403-413.

\bibitem{Liniger} Liniger, T.
\textit{Multivariate Hawkes Processes}. PhD thesis, ETH, 2009.

\bibitem{Lipster} Lipster, R. S. and A. N. Shiryaev. 
\textit{Statistics of Random Processes II. Applications}, 2nd edition. 
Springer, 2001.

\bibitem{Marsan} Marsan, D. and O. Lengline. (2008).
Extending earthquakes' reach through cascading. 
\textit{Science}. \textbf{319} (5866), 1076.

\bibitem{Massoulie} Massouli\'{e}, L. (1998).
Stability results for a general class of interacting point
processes dynamics, and applications.
\textit{Stochastic Processes and their Applications}. \textbf{75}, 1-30.

\bibitem{Mehrdad} Mehrdad, B., Sen, S. and L. Zhu.
The speed of a biased walk on a Galton-Watson tree is monotonic
with respect to progeny distributions for high values of bias. 
Preprint. arXiv:1212.3004.

\bibitem{Meyer} Meyer, S., Elias J. and M. H\"{o}hle. (2012).
A space-time conditional intensity model for invasive meningococcal disease occurence.
\textit{Biometrics}. \textbf{68}, 607-616.

\bibitem{Mitchell} Mitchell, L. and M. E. Cates. (2010).
Hawkes process as a model of social interactions: a view on video dynamics.
\textit{Journal of Physics A: Mathematical and Theoretical}. \textbf{43}, 045101

\bibitem{Mohler} Mohler, G. O., Short, M. B., Brantingham, P. J., Schoenberg F. P. and G. E. Tita. (2011).
Self-exciting point process modelling of crime.
\textit{Journal of the American Statistical Association}. \textbf{106}, 100-108.

\bibitem{MollerI} M\o ller, J. and J. G. Rasmussen. (2005).
Perfect simulation of Hawkes processes.
\textit{Advances in Applied Probability}. \textbf{37}, 629-646.

\bibitem{MollerII} M\o ller, J. and J. G. Rasmussen. (2006).
Approximate simulation of Hawkes processes.
\textit{Methodology and Computing in Applied Probability} \text{8}, 53-64.

\bibitem{MollerIII} M\o ller, J. and G. L. Torrisi. (2005).
Second order analysis for spatial Hawkes processes. 
Technical Report R-2005-20, Department of Mathematical Sciences, Aalborg University.

\bibitem{MollerIV} M\o ller, J. and G. L. Torrisi. (2007).
The pair correlation function of spatial Hawkes processes.
\textit{Statistics \& Probability Letters}. \textbf{77}, 995-1003.

\bibitem{Muni} Muni Toke, I. and F. Pomponio.
Modelling trades-through in a limited order book using Hawkes processes.
\textit{SSRN eLibrary}, 2011.

\bibitem{Musmeci} Musmeci, F. and D. Vere-Jones. (1992).
A space-time clustering model for historical earthquakes. 
\textit{Annals of the Institute of Statistical Mathematics}. \textbf{44}, 1-11.

\bibitem{Oakes} Oakes, D. (1975). 
The Markovian self-exciting process. 
\textit{J. Appl. Prob.} \textbf{12}, 69-77.

\bibitem{Ogata} Ogata, Y. (1978).
The asymptotic behavior of maximum likelihood estimates for stationary point processes.
\textit{Ann. Inst. Statist. Math.} \textbf{30}, 243-261.

\bibitem{OgataII} Ogata, Y. (1988).
Statistical models for earthquake occurrences and residual analysis for point processes.
\textit{J. Amer. Statist. Assoc.} \textbf{83}, 9-27.

\bibitem{OgataIII} Ogata, Y. (1998).
Space-time point-process models for earthquake occurrences.
\textit{Ann. Inst. Statist. Math.} \textbf{50}, 379-402.

\bibitem{OgataIV} Ogata, Y. (1981).
On Lewis' simulation method for point processes. 
\textit{IEEE Transactions on Information Theory}. \textbf{27}, 23-31.

\bibitem{OgataV} Ogata, Y., Akaike, H. and K. Katsura. (1982).
The application of linear intensity models to the investigation of causal relations between a
point process and another stochastic process. 
\textit{Annals of the Institute of Statistical Mathematics}. \textbf{34}, 373-387.

\bibitem{Ozaki} Ozaki, T. (1979).
Maximum likelihood estimation of Hawkes' self-exciting point processes.
\textit{Ann. Inst. Statist. Math.} \textbf{31}, 145-155.

\bibitem{Peng} Peng, X. and S. Kou. 
\textit{Default Clustering and Valuation of Collateralized Debt Obligations}. 
Working Paper, Columbia University, January 2009.

\bibitem{PerniceI} Pernice, V., Staude B., Carndanobile, S. and S. Rotter. (2012).
How structure determines correlations in neuronal networks.
\textit{PLoS Computational Biology}. \textbf{85}:031916.

\bibitem{PerniceII} Pernice, V., Staude B., Carndanobile, S. and S. Rotter. (2011).
Recurrent interactions in spiking networks with arbitrary topology.
\textit{Physical Review E}. \textbf{7}:e1002059.

\bibitem{Porter} Porter, M. and G. White. (2012).
Self-exciting hurdle models for terrorist activity.
\textit{Annals of Applied Statistics} \textbf{6}, 106-124. 

\bibitem{Reynaud} Reynaud-Bouret, P. and S. Schbath. (2010).
Adaptive estimation for Hawkes processes; application to genome analysis.
\textit{Ann. Statist.} \textbf{38}, 2781-2822.

\bibitem{ReynaudII} Reynaud-Bouret, P. and E. Roy. (2007).
Some non asymptotic tail estimates for Hawkes processes.
\textit{Bull. Belg. Math. Soc. Simon Stevin} \textbf{13}, 883-896.

\bibitem{ReynaudIII} Reynaud-Bouret, P., Tuleau-Malot, C., Rivoirard, V. and F. Grammont. 
Spike trains as (in)homogeneous Poisson processes or 
Hawkes processes: non-parametric adaptive estimation and goodness-of-fit tests.
Preprint.

\bibitem{Rubin} Rubin, I. (1972). 
Regular point processes and their detection. 
\textit{IEEE Transactions on Information Theory}. \textbf{18}, 547-557.

\bibitem{Sen} Sen, S. and L. Zhu. 
Large deviations for self-correcting point processes. 
In Preparation. 

\bibitem{Sornette} Sornette, D. and S. Utkin. (2009).
Limits of declustering methods for disentangling exogenous from
endogenous events in time series with foreshocks, main shocks, and aftershocks. 
\textit{Physical Review E}. \textbf{79} (6), 61110.

\bibitem{Stabile} Stabile, G. and G. L. Torrisi. (2010). 
Risk processes with non-stationary Hawkes arrivals. 
\textit{Methodol. Comput. Appl. Prob.} \textbf{12}, 415-429.

\bibitem{Torrisi} Torrisi, G. L. (2002).
A class of interacting marked point processes: Rate of convergence to equilibrium. 
\textit{Journal of Applied Probability} \textbf{39}, 137-160.

\bibitem{Wang} Wang, T., Bebbington, M. and D. Harte. (2012).
Markov-modulated Hawkes process with stepwise decay.
\textit{Annals of the Institute of Statistical Mathematics} \textbf{64} 521-544. 

\bibitem{Wei} Wei, C. Z. and J. Winnicki. (1989).
Some asymptotic results for the branching process with immigration. 
\textit{Stoch. Proc. Appl.} \textbf{31}, 261-282.

\bibitem{Varadhan} Varadhan, S. R. S. (2008).
Large deviations. 
\textit{Annals of Probability}. \textbf{36}, 397-419.

\bibitem{VaradhanII} Varadhan, S. R. S. 
\textit{Large Deviations and Applications}, SIAM, Philadelphia, 1984.

\bibitem{VaradhanIV} Varadhan, S. R. S.
\textit{Probability Theory}, Courant Lecture Notes, American Mathematical Society, 2007.

\bibitem{VaradhanIII} Varadhan, S. R. S. 
\textit{Stochastic Processes}, Courant Lecture Notes, American Mathematical Society, 2007.

\bibitem{Vere} Vere-Jones, D. (1978).
Earthquake prediction: A statistician's view. 
\textit{Journal of Physics of the Earth} \textbf{26}, 129-146.

\bibitem{Zheng} Zheng, B., Roueff, F. and F. Abergel.
Ergodicity and scaling limit of a constrained multivariate Hawkes process.
\textit{SSRN eLibrary}, 2013.

\bibitem{ZhuI} Zhu, L. 
Large deviations for Markovian nonlinear Hawkes processes.
Preprint. arXiv:1108.2432.

\bibitem{ZhuII} Zhu, L.
Process-level large deviations for nonlinear Hawkes point processes. 
To appear in \textit{Annales de l'Institut Henri Poincar\'{e}}.
arXiv:1108.2431.

\bibitem{ZhuIII} Zhu, L.
Central limit theorem for nonlinear Hawkes processes. 
To appear in \textit{Journal of Applied Probability}.
arXiv:1204.1067.

\bibitem{ZhuIV} Zhu, L. (2013).
Moderate deviations for Hawkes processes.
\textit{Statistics \& Probability Letters} \textbf{83}, 885-890.

\bibitem{ZhuVII} Zhu, L.
Limit theorems for a Cox-Ingersoll-Ross process with Hawkes jumps. Preprint.

\bibitem{ZhuV} Zhu, L.
Asymptotics for nonlinear Hawkes processes.
In Preparation.

\bibitem{ZhuVI} Zhu, L.
Ruin probabilities for risk processes with non-stationary arrivals and subexponential claims.
Preprint. arXiv:1304.1940.

\end{thebibliography}
